\documentclass[a4paper, 11pt]{report}
\pdfoutput=1
\usepackage[utf8]{inputenc}
\usepackage[T1]{fontenc}
\usepackage{lmodern}
\usepackage{graphicx}
\usepackage[french, english]{babel}
\usepackage{array}
\usepackage{textcomp}
\usepackage{pdfpages}
\usepackage{stmaryrd}
\usepackage[all]{xy}
\usepackage{tikz}
\usepackage{url}
\usepackage[nottoc]{tocbibind}
\usetikzlibrary{patterns}
\usetikzlibrary{shapes}
\usepackage{ccaption}
\usepackage{float}
\usepackage{amssymb,amsmath,amsthm}
\usepackage{hyperref}
\usepackage{enumitem}
\setlist[enumerate,1]{label={(\arabic*)}}
\usepackage{tkz-graph}
\usepackage{makeidx}
\makeindex
\usepackage{nomencl}
\makenomenclature
\usepackage{fancyhdr}
\theoremstyle{plain}
\newtheorem{theo}{Théorème}[section]
\theoremstyle{plain}
\newtheorem*{the0}{Théorème}
\theoremstyle{plain}
\newtheorem*{thep}{Théorème principal}
\theoremstyle{definition}
\newtheorem{defi}[theo]{Définition}
\theoremstyle{definition}
\newtheorem*{def0}{Définition}
\theoremstyle{plain}
\newtheorem{defiprop}[theo]{Définition-Propriété}
\theoremstyle{plain}
\newtheorem{prop}[theo]{Propriété}
\theoremstyle{plain}
\newtheorem{propn}[theo]{Proposition}
\theoremstyle{plain}
\newtheorem*{propn0}{Proposition}
\theoremstyle{plain}
\newtheorem{notaprop}[theo]{Notation-Propriété}
\theoremstyle{plain}
\newtheorem{lem}[theo]{Lemme}
\theoremstyle{plain}

\theoremstyle{plain}
\newtheorem{conj}{Conjecture}
\theoremstyle{definition}
\newtheorem{nota}[theo]{Notation}
\theoremstyle{plain}
\newtheorem{coro}[theo]{Corollaire}
\theoremstyle{definition}
\newtheorem{exem}[theo]{Exemple}
\theoremstyle{definition}

\theoremstyle{remark}
\newtheorem{rema}[theo]{Remarque}
\theoremstyle{remark}

\theoremstyle{plain}
\newtheorem{theI}{Théorème}
\theoremstyle{plain}
\newtheorem{theE}{Theorem}
\theoremstyle{plain}
\newtheorem*{the0E}{Theorem}
\theoremstyle{plain}
\newtheorem*{thepE}{Main theorem}
\theoremstyle{plain}

\theoremstyle{definition}

\theoremstyle{definition}

\theoremstyle{plain}

\theoremstyle{plain}

\theoremstyle{plain}

\theoremstyle{plain}
\newtheorem*{propn0E}{Proposition}
\theoremstyle{plain}

\theoremstyle{plain}

\theoremstyle{plain}

\theoremstyle{plain}

\theoremstyle{definition}

\theoremstyle{plain}

\theoremstyle{definition}

\theoremstyle{definition}

\theoremstyle{remark}

\theoremstyle{remark}

\numberwithin{equation}{chapter}
\DeclareMathOperator{\id}{id}
\DeclareMathOperator{\inv}{inv}
\DeclareMathOperator{\GK}{GKdim}
\DeclareMathOperator{\Spec}{Specm}
\DeclareMathOperator{\M}{M}
\DeclareMathOperator{\gl}{\mathfrak{gl}}
\DeclareMathOperator{\spl}{\mathfrak{sl}}
\DeclareMathOperator{\so}{\mathfrak{so}}
\DeclareMathOperator{\syp}{\mathfrak{sp}}
\DeclareMathOperator{\GL}{GL}

\DeclareMathOperator{\Bij}{Bij}

\DeclareMathOperator{\eav}{{}^e \! \!}

\DeclareMathOperator{\trace}{tr}

\DeclareMathOperator{\Card}{Card}

\DeclareMathOperator{\vect}{Vect}

\DeclareMathOperator{\Der}{Der}

\DeclareMathOperator{\codim}{codim}
\DeclareMathOperator{\degtr}{degtr}
\DeclareMathOperator{\bideg}{bideg}
\DeclareMathOperator{\ind}{ind}
\DeclareMathOperator{\Sym}{S}
\DeclareMathOperator{\Y}{Y}
\DeclareMathOperator{\FY}{C}
\DeclareMathOperator{\Sy}{Sy}

\DeclareMathOperator{\A}{\mathcal{A}}
\DeclareMathOperator{\B}{\mathcal{B}}

\DeclareMathOperator{\gota}{\mathfrak{a}}
\DeclareMathOperator{\gotb}{\mathfrak{b}}

\DeclareMathOperator{\g}{\mathfrak{g}}
\DeclareMathOperator{\h}{\mathfrak{h}}
\DeclareMathOperator{\gotk}{\mathfrak{k}}
\DeclareMathOperator{\gotl}{\mathfrak{l}}

\DeclareMathOperator{\gotn}{\mathfrak{n}}
\DeclareMathOperator{\gotp}{\mathfrak{p}}
\DeclareMathOperator{\gotq}{\mathfrak{q}}
\DeclareMathOperator{\gotr}{\mathfrak{r}}
\DeclareMathOperator{\gots}{\mathfrak{s}}

\DeclareMathOperator{\gotz}{\mathfrak{z}}
\DeclareMathOperator{\C}{\mathbb{K}}

\DeclareMathOperator{\N}{\mathbb{N}}
\DeclareMathOperator{\rat}{\mathbb{Z}}
\DeclareMathOperator{\Q}{\mathbb{Q}}

\DeclareMathOperator{\ad}{ad}
\DeclareMathOperator{\End}{End}

\DeclareMathOperator{\Frac}{Frac}

\DeclareMathOperator{\mor}{Hom}
\DeclareMathOperator{\rg}{rg}
\DeclareMathOperator{\rk}{rk}

\DeclareMathOperator{\pr}{pr}
\DeclareMathOperator{\res}{res}
\DeclareMathOperator{\V}{\mathcal{V}}
\DeclareMathOperator{\supp}{supp}

\newcommand{\imax}{i_{\text{max}}}
\newcommand{\nonvide}{\neq \emptyset}

\newcommand{\longtwoheadrightarrow}{}% teste si deja defini
\DeclareRobustCommand{\longtwoheadrightarrow}{\relbar\joinrel\twoheadrightarrow}
\title{Semi-invariants symétriques de contractions paraboliques}
\author{Kenny Phommady}
\begin{document}
\selectlanguage{french}
\includepdf{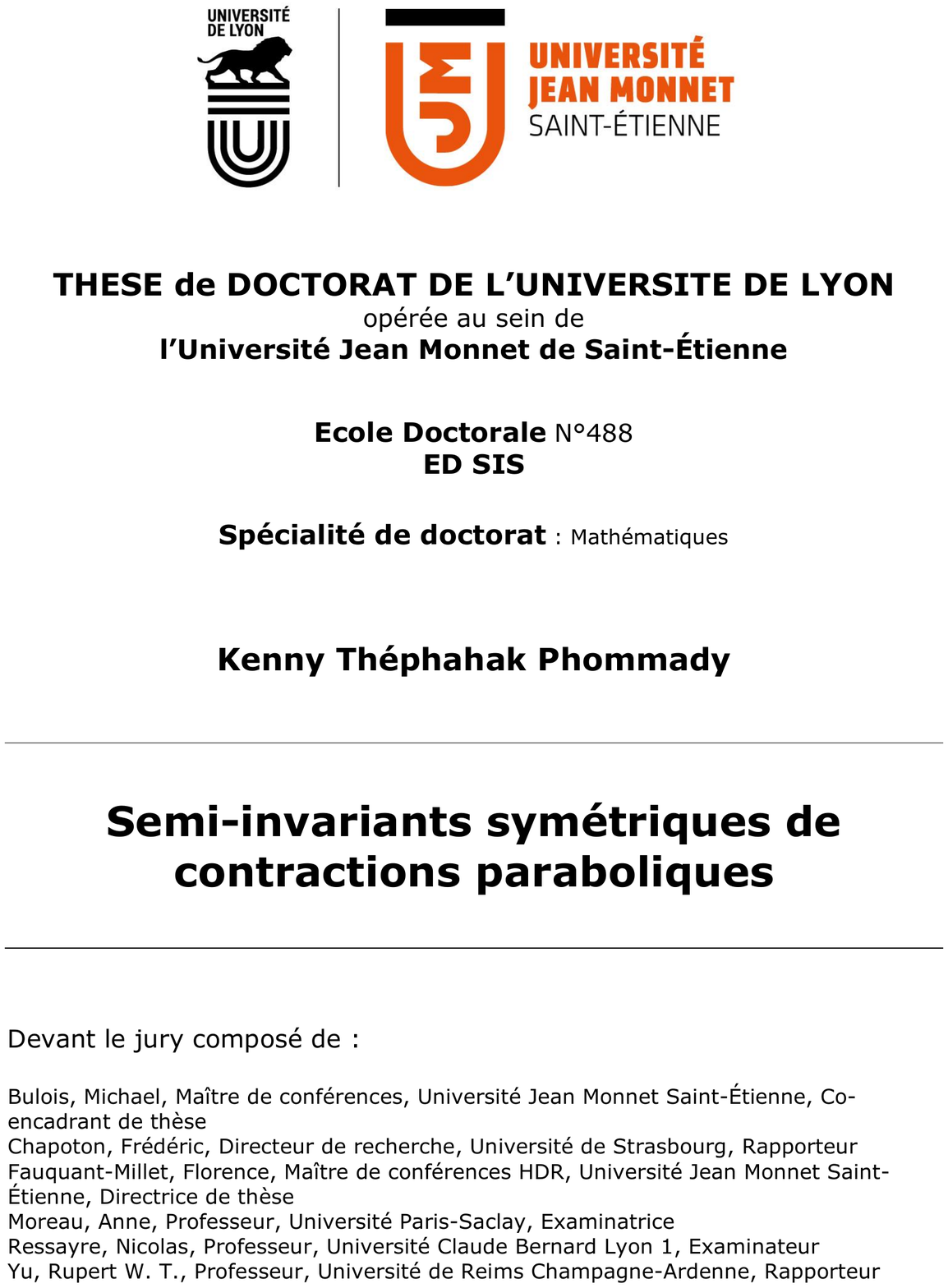}
\addcontentsline{toc}{chapter}{Remerciements}
\chapter*{Remerciements}

Je tiens tout d'abord à remercier Florence Fauquant-Millet et Michael Bulois pour avoir encadré ma thèse durant ces trois (et un peu plus) années. Leurs interactions tant professionnelles que mathématiques ont été très souvent complémentaires et m'ont beaucoup apporté, et j'ai vraiment beaucoup apprécié travailler sous leur direction, humainement et mathématiquement.\par
Je tiens à remercier Nicolas Ressayre pour s'être intéressé très tôt à ma thèse et pour avoir accepté de faire partie du jury de la thèse.\par
Je tiens également à remercier Frédéric Chapoton et Rupert Yu pour  avoir donné de leur temps et accepté de rapporter cette thèse. Je remercie également Anne Moreau pour avoir accepté de faire partie du jury de ma thèse.\par
Je veux tout particulièrement remercier Philippe Gille qui a été d'une aide inestimable dans mon choix de thèse à la fin de mon cursus de master.\par
Je remercie tous les membres de l'Institut Camille Jordan à l'Université Jean Monnet à Saint-Étienne pour m'avoir accueilli dans une ambiance professionnelle incomparable, que je n'oublierai pas de sitôt. Je ne pourrai citer tout le monde, mais je tenais tout de même à remercier en particulier Stéphane Gaussent et Driss Essouabri pour leurs conseils professionnels toujours très justes, ainsi que toutes celles et ceux qui ont contribué à l'ambiance conviviale au bureau des doctorants : Clément, Nohra, Auguste, Farshid, Christopher, Benjamin, Romain, Rita et Simon.\par
Je remercie tous les membres de l'Institut Camille Jordan à l'Université Claude Bernard Lyon 1 pour leur accueil chaleureux à l'occasion notamment des groupes de travail et des séminaires lorsqu'ils sont organisés à Lyon, et je veux également remercier les organisateurs des dits groupes de travail et séminaires, lesquels m'ont permis de découvrir des domaines des mathématiques connexes à cette thèse et les domaines de recherche de certains de mes collègues.\par
Je souhaite remercier ma famille pour leur soutien, ainsi que mes amis, en particulier Carine et Ruben, qui m'ont permis de sortir mon esprit de la thèse à certains moments, et trouver de nouvelles idées pour la thèse à d'autres. Je remercie Etoiles et sa communauté qui m'ont permis de tenir dans les temps difficiles.
\tableofcontents
\newpage

\chapter{Introduction}
(English version of the introduction from page \pageref{english})\par
Le but de cette thèse est d'étudier une classe importante d'algèbres de Lie non réductives, que sont les contractions paraboliques (des déformations d'algèbres de Lie réductives). Plus précisément, nous étudions l'algèbre des semi-invariants symétriques associée à de telles contractions paraboliques. Notre étude se base sur des résultats importants de Panyushev et Yakimova, concernant la polynomialité de l'algèbre des invariants symétriques (en types $A$ et $C$) de ces contractions paraboliques.
\section*{Algèbres d'invariants et de semi-invariants}
Commençons par définir les algèbres d'invariants et de semi-invariants. Soit $\C$ un corps de caractéristique $0$ algébriquement clos, $(\gotk, [\, , \, ])$ une algèbre de Lie de dimension finie sur $\C$ et $\gotk^*$ son espace vectoriel dual. On définit l'indice $\ind \gotk$ comme la plus petite dimension du noyau de $\phi_f$ pour $f \in \gotk^*$, où $\phi_f$ est la forme bilinéaire définie par $\phi_f : (x,y) \rightarrow f([x,y])$. La représentation adjointe de $\gotk$ sur lui-même est la représentation d'algèbre de Lie naturelle définie à partir du crochet de Lie $\ad(g)(h) = g \cdot h = [g,h]$ pour tous $g, h \in \gotk$. On étend cette représentation en une unique représentation, que l'on note toujours $\ad$, de $\gotk$ sur l'algèbre symétrique $\Sym(\gotk)$ telle que $\ad(g)$ est une dérivation sur $\Sym(\gotk)$ pour tout $g \in \gotk$.\par
On s'intéresse aux algèbres des invariants $\Y(\gotk)$ et des semi-invariants $\Sy(\gotk)$, qui sont des sous-algèbres de $\Sym(\gotk)$. L'algèbre des invariants $\Y(\gotk)$ est l'algèbre formée de l'ensemble des invariants (dits encore invariants symétriques), c'est-à-dire tous les éléments de $\Sym(\gotk)$ annulés par tous les $\ad(g)$, $g \in \gotk$. L'algèbre des semi-invariants $\Sy(\gotk)$ est l'algèbre \emph{engendrée par} tous les semi-invariants (dits encore semi-invariants symétriques), c'est-à-dire les éléments de $\Sym(\gotk)$ qui sont vecteurs propres pour tous les $\ad(g)$, $g \in \gotk$. Si $s$ est un semi-invariant (non nul), il existe alors $\lambda \in \gotk^*$ tel que $\ad(g)(s)=\lambda(g)s$ pour tout $g \in \gotk$, et on appelle $\lambda$ le poids de $s$. On peut munir $\Sym(\gotk)$ d'un crochet de Poisson $\{ \, , \}$ qui étend le crochet de Lie sur $\gotk$, de sorte que l'algèbre associative $\Sym(\gotk)$ devient une \emph{algèbre de Poisson}. Les algèbres $\Y(\gotk)$ et $\Sy(\gotk)$ correspondent alors au centre et au semi-centre de cette algèbre de Poisson. Leurs dimensions de Gelfand-Kirillov $\GK \Y(\gotk)$ et $\GK \Sy(\gotk)$ sont les degrés de transcendance des corps des fractions associés. Par un théorème de Rosenlicht, on a $\GK \Y(\gotk) \leq \ind \gotk$ \cite{ros63}, avec l'égalité si $\Y(\gotk)=\Sy(\gotk)$ (voir \cite[Chap. I, Sec. B, 5.12]{fau14}). En théorie des invariants associée à des algèbres de Lie, une des questions cruciales relative aux algèbres $\Y(\gotk)$ et $\Sy(\gotk)$ est la polynomialité, c'est-à-dire déterminer si oui ou non ces algèbres sont isomorphes à des algèbres de polynômes.\par
Géométriquement, plaçons-nous dans le cas où $\gotk$ est l'algèbre de Lie d'un groupe algébrique connexe $K$. L'algèbre $\Y(\gotk)$ s'identifie alors à l'algèbre $\C[\gotk^*]^K$ des fonctions polynomiales sur $\gotk^*$ invariantes par l'action coadjointe de $K$. Si $\Y(\gotk)$ est de type fini, la variété associée à cette algèbre est alors $ \Spec \Y(\gotk) = \gotk^* \sslash K$. La polynomialité de l'algèbre $\Y(\gotk)$ implique alors que $\gotk^* \sslash K$ est un espace affine. De manière analogue aux algèbres, on peut également définir le corps des invariants $\FY(\gotk)$, qui est un sous-corps de $\Frac(\Sym(\gotk))$, le corps des fractions de $\Sym(\gotk)$. Si $\FY(\gotk)$ est une extension transcendante pure de $\C$, alors géométriquement, cela revient à affirmer que $\gotk^* \sslash K$ est une variété rationnelle. Si $\FY(\gotk)=\Frac(\Y(\gotk))$, alors la polynomialité de $\Y(\gotk)$ entraine clairement la rationalité de $\FY(\gotk)$. En revanche, la réciproque n'est pas vérifiée : Dixmier a exhibé un contre-exemple dans lequel $\Y(\gotk)$ n'est pas polynomiale et n'est même pas de type fini mais $\FY(\gotk)=\Frac(\Y(\gotk))$ est bien une extension transcendante pure de $\C$ \cite[4.4.8 et 4.9.20]{Dix}.\par
Le premier résultat important de polynomialité d'algèbre d'invariants est celui de Chevalley \cite{Dix}. Ce résultat affirme que si $\g$ est une algèbre de Lie réductive, c'est-à-dire le produit d'une algèbre de Lie semi-simple par une algèbre de Lie abélienne, alors $\Y(\g)$ est polynomiale. Dans le cas où $\g$ est semi-simple, la preuve repose sur un isomorphisme entre l'algèbre des invariants $\Y(\g) \simeq \Y(\g^*)$ et l'algèbre $\Sym(\h^*)^W$, où $\h$ est une sous-algèbre de Cartan de $\g$ et $W$ est le groupe de Weyl du système de racines associé à $(\g,\h)$. En particulier, $W$ peut être vu comme un sous-groupe fini de $\GL(\h^*)$ engendré par des réflexions. Chevalley identifie alors $\Sym(\h^*)^W$ à une algèbre de polynômes par un résultat d'algèbre commutative (voir \cite[théorème 31.1.6]{TYu} pour une exposition). Un résultat de Kostant \cite{kos63} donne une autre interprétation à ce résultat. Il affirme qu'il existe un sous-espace affine $\mathcal{V}$ de $\g^*$ tel que le morphisme $\Y(\g) \simeq \C[\g^*]^G \rightarrow \C[\mathcal{V}]$ de restriction à $\mathcal{V}$ est un isomorphisme de $\Y(\g)$ sur l'algèbre $\C[\mathcal{V}]$ des fonctions polynomiales sur $\mathcal{V}$. De tels isomorphismes sont appelés "sections de Kostant-Weierstrass"\footnote{On rencontre souvent les désignations de "tranches de Kostant" ou de "sections de Weierstrass". On peut aussi désigner par ces termes le sous-espace affine $\mathcal{V}$.}.\par % Ces sections permettent d'une certaine manière de linéariser l'étude de l'algèbre des invariants.\par
Si $\gotk=\g^e$ est le centralisateur d'un élément nilpotent $e$ pour une algèbre de Lie $\g$ simple, Premet, Panyushev et Yakimova ont montré la polynomialité de $\Y(\g^e)$ lorsque $\g$ est de type $A$ ou $C$ \cite{ppy07}. Leur démonstration repose sur le théorème suivant de Panyushev \cite[théorème 1.2]{pan07}.
\begin{theI}
Soit $\gotk$ une algèbre de Lie. Supposons que $\gotk$ vérifie la propriété de codimension $2$, et que $\GK \Y(\gotk) = \ind \gotk$. Soit $l= \ind \gotk$, et soient $f_1, \ldots, f_l$ des éléments homogènes algébriquement indépendants de $\Y(\gotk)$. Alors
\begin{enumerate}
\item $\sum_{i=1}^l \deg f_i \geq (\dim \gotk + \ind \gotk )/2$,
\item si $\sum_{i=1}^l \deg f_i = (\dim \gotk + \ind \gotk )/2$, alors $\Y(\gotk)=\C[f_1, \ldots, f_l]$.
\end{enumerate}
\label{pan13}
\end{theI}
On définit la \emph{propriété de codimension $2$} comme suit :
\begin{def0}
Soit $\gotk$ une algèbre de Lie. Pour tout $f \in \gotk^*$, on dit que $f$ est \emph{régulier} dans $\gotk^*$ si le noyau de la forme bilinéaire\footnote{introduite plus tôt} $\phi_f$ est de dimension minimale. Sinon, on dit que $f$ est \emph{singulier} dans $\gotk^*$.\par
On dit alors que $\gotk$ \emph{vérifie la propriété de codimension $2$} si l'ensemble $\gotk^*_{\text{sing}}$ des éléments singuliers dans $\gotk^*$ est de codimension au moins $2$ dans $\gotk^*$.
\end{def0}
Ainsi la question de la polynomialité de $\Y(\gotk)$ dans le cas $\gotk=\g^e$ a été longuement étudiée, et c'est dans ce cadre que plusieurs cas de non-polynomialité ont été trouvés. Lorsque $e$ est un vecteur de plus haut poids pour $\g$ simple de type $E_8$, alors on a $\Y(\g^e) \simeq \Sy(\gotp)$ pour une certaine sous-algèbre parabolique $\gotp$ de $\g$ et Yakimova a montré que cette algèbre n'était pas polynomiale \cite{yak07}. Charbonnel et Moreau ont prolongé l'étude en donnant des conditions nécessaires et suffisantes à la polynomialité de $\Y(\g^e)$. Si $\Y(\g^e)$ est polynomiale, ils montrent également que $\Sym(\g^e)$ est un $\Y(\g^e)$-module libre. Ils utilisent alors ces conditions pour étudier la polynomialité de $\Y(\g^e)$ en fonction de l'élément $e \in \g$ dans les cas où $\g$ est simple de type $B$ ou $D$. Par exemple, ils trouvent un cas où $\Y(\g^e)$ n'est pas polynomiale, lorsque $\g$ est de type $D_7$ \cite[\S 7.3]{chm14}.\par
Les résultats de polynomialité d'algèbres de semi-invariants $\Sy(\gotk)$ sont plus parcellaires. Leur étude peut passer par l'étude de la \emph{troncation canonique} de $\gotk$. Pour une algèbre de Lie $\gotk$ quelconque, on appelle troncation canonique de $\gotk$, que l'on note $\gotk_\Lambda$, la plus grande sous-algèbre de Lie de $\gotk$ sur laquelle s'annulent tous les poids de semi-invariants. Fauquant-Millet et Joseph\footnote{reprenant Borho, Gabriel et Rentschler dans le cas nilpotent} ont montré dans un cadre assez général\footnote{en particulier si $\gotk$ est une algèbre de Lie ad-algébrique} que l'on a $\Sy(\gotk)=\Y(\gotk_\Lambda)=\Sy(\gotk_\Lambda)$ \cite[Appendice B.2]{fmj08}. Ainsi l'étude de polynomialité d'une algèbre de semi-invariants revient à étudier la polynomialité d'une algèbre d'invariants. Par exemple, si $\g$ est une algèbre de Lie réductive, on a $\g_\Lambda=\g$, de sorte que $\Sy(\g)=\Y(\g)$, ainsi $\Sy(\g)$ est polynomiale. Joseph \cite{jos77} a montré que dans le cas d'une sous-algèbre de Borel $\gotb$ d'une algèbre de Lie simple $\g$, l'algèbre des semi-invariants $\Sy(\gotb)$ est polynomiale engendrée par $\rg \g$ générateurs ($\rg \g$ est le rang de $\g$) calculés explicitement. Dans le cas où $\gotp$ est une sous-algèbre biparabolique en type $A$ ou $C$, Joseph et Fauquant-Millet \cite{FMJ05} ont montré la polynomialité de $\Sy(\gotp)$ en encadrant le caractère formel de $\Sy(\gotp)$ par deux bornes dont ils montrent l'égalité notamment en types $A$ et $C$. En types $A$ et $C$, ils parviennent à calculer les poids et les degrés de générateurs homogènes.
%Pour cela, ils exhibent un encadrement $\mathcal{A} \subset \gr'(\gr''(\Sy(\gotp))) \subset \mathcal{B}$ par deux algèbres polynomiales $\mathcal{A}$ et $\mathcal{B}$, où $\gr'$ et $\gr''$ sont des graduations de $\Sym(\gotp)$. Ils transposent alors cet encadrement en inégalités de \emph{caractères formels} (voir \cite[2.8]{jos07} pour une définition) et montrent que les caractères formels de $\mathcal{A}$ et $\mathcal{B}$ sont égaux en types $A$ et $C$. Ils en déduisent que les inclusions de l'encadrement sont en fait des égalités, et donc que $\Sy(\gotp)$ est polynomiale.
En type $A$, Joseph a également exhibé des sections de Kostant-Weierstrass pour $\Y(\gotp_\Lambda)=\Sy(\gotp)$ \cite{jos08}.\par
\section*{Contractions paraboliques}
On s'intéresse dans cette thèse à une classe particulière d'algèbres de Lie dites \emph{contractions paraboliques}, qui sont des déformations d'algèbres de Lie réductives. Ces algèbres ont été introduites par E. Feigin, dans le cas d'une contraction par une sous-algèbre de Borel \cite{Fei12}, puis par Panyushev et Yakimova pour une sous-algèbre parabolique quelconque \cite{py13}. Soit $\g$ une algèbre de Lie réductive et $\gotp$ une sous-algèbre parabolique de $\g$. On note $\gotn$ le radical nilpotent de $\gotp$ et $\gotn^-$ le radical nilpotent de la sous-algèbre parabolique opposée de $\gotp$ de sorte que, si $\gotl$ est un facteur de Levi de $\gotp$ on a $\g=\gotn \oplus \gotl \oplus \gotn^-$. On a alors $\gotp=\gotn \oplus \gotl$ d'où $\g = \gotp \oplus \gotn^-$. On définit $\gotq = \gotp \ltimes (\gotn^-)^{\text{a}}$ la contraction parabolique de $\g$ par $\gotp$. Ici, $(\gotn^-)^\text{a}$ est l'espace vectoriel $\gotn^-$ vu comme une sous-algèbre de Lie abélienne\footnote{Pour des raisons pratiques, on omettra dans la suite cette notation et on notera simplement cette algèbre de Lie $\gotn^-$.} de $\gotq$ ainsi qu'un $\gotp$-module par l'identification $\gotn^- \simeq \g/\gotp$. On peut voir cette algèbre de Lie comme une contraction d'Inönu-Wigner\footnote{pour une définition générale des contractions d'Inönu-Wigner, voir \cite{iw53} ou \cite[\S 2.3]{dr85}} de $\g$. Pour tout $t \neq 0$, soit $c_t : \g \rightarrow \g$ l'automorphisme de $\C$-espace vectoriel défini par 
$$\forall p \in \gotp, \forall n \in \gotn^-, \, c_t(p+n)=p+tn$$
On définit alors $\g_{(t)}$ l'algèbre de Lie égale à $\g$ comme espace vectoriel, et de crochet de Lie défini par
$$\forall q,q' \in \g_{(t)}, [q,q']_{\g_{(t)}} = c_t^{-1}([c_t(q),c_t(q')]_{\g}) $$
La limite, lorsque $t$ tend vers $0$, des algèbres de Lie\footnote{Autrement dit, la limite, lorsque $t$ tend vers $0$, du crochet de Lie sur $\g_{(t)}$ est le crochet de Lie sur $\gotq$.} $\g_{(t)}$ donne la contraction parabolique $\gotq$.\par
À partir de ces contractions paraboliques, Feigin définit des dégénérescences de variétés de drapeaux et étudie combinatoirement en détail de telles variétés. Ces dégénérescences permettent par exemple d'obtenir un lien combinatoire avec la suite des nombres de Genocchi \cite{fem2012}.\par
D'autres algèbres de Lie étudiées peuvent s'apparenter dans certains cas particuliers à des contractions paraboliques. Les algèbres de Lie de la forme $\gota \ltimes \gota^*$ avec $\gota$ une algèbre de Lie, sont des cas particuliers de \emph{doubles de Drinfeld} (voir \cite{dri88}, \cite{es01} ou \cite[\S 3.2]{jos95}). Lorsque $\gota=\gotb$ est une sous-algèbre de Borel, l'algèbre de Lie $\gotb \ltimes \gotb^*$ quotientée par son centre\footnote{ici le centre de $\gotb \ltimes \gotb^*$ est isomorphe à la sous-algèbre de Cartan associée à $\gotb^*$} est isomorphe à une contraction parabolique par une sous-algèbre de Borel. Postérieurement aux résultats présentés dans ce présent rapport, Bar-Natan et Van der Veen en type $A$ \cite{bnv20}, puis Bulois et Ressayre pour une sous-algèbre de Borel de type quelconque \cite{br20} étudient les automorphismes d'une telle algèbre. Pour cela, Bulois et Ressayre réinterprètent ces algèbres à l'aide des algèbres de Kac-Moody.\par
En type $A$, une contraction parabolique par une sous-algèbre de Borel peut être également vue comme l'algèbre de Lie associée à un \emph{schéma cyclique de Brauer}, introduit par Knutson et Zinn-Justin \cite[1.1 et 2.1]{kzj07}. Le schéma cyclique de Brauer est défini comme l'ensemble $E$ des matrices $M$ de taille $n \times n$ telles que $M \bullet M = 0$, où $\bullet$ est un produit associatif défini par déformation du produit usuel sur les matrices. La structure d'algèbre de Lie sur $E$ associée au produit associatif $\bullet$ en fait alors une contraction parabolique par une sous-algèbre de Borel en type $A$.\par
Panyushev et Yakimova ont également étudié d'autres types de contractions, notamment les $\rat_2$-contractions (voir par exemple \cite{yak14-2} et \cite{pan07}), qui sont du type $\g_0 \ltimes \g_1$ où $\g=\g_0 \oplus \g_1$ est une $\rat_2$-graduation d'une algèbre de Lie réductive $\g$. De manière plus générale, Panyushev et Yakimova ont pour but de classifier toutes les algèbres de Lie de la forme $\g \ltimes V$ avec $\g$ une algèbre de Lie simple et $V$ une représentation de dimension finie de $\g$ telles que l'algèbre des invariants associée est polynomiale (voir par exemple \cite{py19}).\par
\section*{Algèbres d'invariants et de semi-invariants associées à des contractions paraboliques}
Cette thèse se place dans la continuité des travaux de Panyushev et Yakimova concernant l'étude de $\Y(\gotq)$ \cite{py13} lorsque $\gotq$ est une contraction parabolique et en reprend les résultats essentiels. D'après Panyushev et Yakimova (\cite[\mbox{}1]{py13}) pour une contraction parabolique $\gotq$, il existe un groupe algébrique connexe $Q$ tel que $\gotq$ en soit l'algèbre de Lie. Dans \cite{py13}, ils étudient la polynomialité des algèbres d'invariants $\C[\gotq]^Q$ et $\C[\gotq^*]^Q$. Par un raisonnement de géométrie algébrique\footnote{essentiellement grâce à une variante du lemme d'Igusa}, Panyushev et Yakimova montrent aisément que $\C[\gotq]^Q$ est une algèbre polynomiale. Pour toute contraction parabolique $\gotq$ par une sous-algèbre de Borel, ils montrent également \cite{py12} que $\C[\gotq^*]^Q \simeq \Y(\gotq)$ est polynomiale (cela utilise des propriétés spécifiques au radical nilpotent d'une sous-algèbre de Borel). En revanche, l'étude de l'algèbre $\Y(\gotq) \simeq \C[\gotq^*]^Q$ de manière générale est plus complexe. On pose $\gotq$ une contraction parabolique d'une algèbre de Lie semi-simple $\g$ par une sous-algèbre parabolique $\gotp$ et $G, P$ des groupes algébriques connexes tels que $\g, \gotp$ sont les algèbres de Lie respectivement de $G, P$. À partir de générateurs homogènes $\mathcal{F}_1, \ldots, \mathcal{F}_n$ algébriquement indépendants de l'algèbre $\Y(\g)$, Panyushev et Yakimova construisent des éléments $\mathcal{F}_1^\bullet, \ldots, \mathcal{F}_n^\bullet$ de $\Y(\gotq)$ obtenus par dégénérescence des $\mathcal{F}_1, \ldots, \mathcal{F}_n$ via les $c_t$. Ces $\mathcal{F}_1^\bullet, \ldots, \mathcal{F}_n^\bullet$ engendrent librement $\Y(\gotq)$ dans le cas d'une contraction par une sous-algèbre de Borel ; ils sont également les candidats potentiels à librement engendrer $\Y(\gotq)$ dans le cas général. Pour montrer ce dernier point dans le cas d'une sous-algèbre parabolique quelconque, ils reprennent l'idée de sections de Kostant-Weierstrass en l'adaptant aux cas de contractions paraboliques et construisent une injection\footnote{qui est une application de restriction} $\psi$ de $\Y(\gotq)$ vers une certaine algèbre $\Sym(\g^e)^{P^e}$, où $e$ est un élément nilpotent de $\g$. Les invariants $\mathcal{F}_1^\bullet, \ldots \mathcal{F}_n^\bullet \in \Y(\gotq) \subset \C[\gotq^*]$ peuvent alors être vus dans cette dernière algèbre, et on note $\eav \mathcal{F}_m=\psi(\mathcal{F}_m^\bullet)$. Les $\eav \mathcal{F}_m$ appartiennent en fait à $\Sym(\g^e)^{G^e}=\Y(\g^e)$. Dans les cas des types $A$ et $C$, Panyushev et Yakimova montrent que $\Sym(\g^e)^{P^e}=\Y(\g^e)$ et que les $\eav \mathcal{F}_m$ forment une famille algébriquement indépendante qui engendre $\Y(\g^e)$ (via l'étude avec Premet \cite{ppy07} mentionnée précédemment). De plus, dans ces cas, l'application de restriction $\psi$ précédente, qui était injective, devient un isomorphisme d'algèbres de $\Y(\gotq)$ sur $\Y(\g^e)$. Avec ces derniers résultats, ils parviennent à conclure que $\Y(\gotq)$ est polynomiale et engendrée par $\mathcal{F}_1^\bullet, \ldots \mathcal{F}_n^\bullet$.\par
Yakimova s'est également intéressée à la polynomialité de $\Sy(\gotq)$ où $\gotq$ est une contraction parabolique par une sous-algèbre de Borel. Elle conclut dans ce cas à la polynomialité de $\Sy(\gotq)$ \cite[\S 5.1]{yak12}. Sa démonstration reprend la famille $\mathcal{F}^\bullet_1, \ldots, \mathcal{F}^\bullet_{n}$ qui engendre librement $\Y(\gotq)$. Elle remarque que $\mathcal{F}^\bullet_{n}$ admet une décomposition en facteurs irréductibles\footnote{qui sont nécessairement des semi-invariants} avec des facteurs de degré $1$, c'est-à-dire dans $\gotq \subset \Sym(\gotq)$. En remplaçant ainsi $\mathcal{F}^\bullet_{n}$ dans la famille $\mathcal{F}^\bullet_1, \ldots, \mathcal{F}^\bullet_{n}$ par ses facteurs irréductibles, elle obtient une certaine famille $\mathbf{F}$, dont elle montre qu'elle engendre librement $\Sy(\gotq)$.\par
Dans cette thèse, on reprend l'idée de décomposer un invariant en facteurs irréductibles puis de remplacer l'invariant par ses facteurs irréductibles. L'idée est d'obtenir une famille $\mathbf{F}$ dont on espère qu'elle engendre librement l'algèbre des semi-invariants. En revanche, pour démontrer que $\mathbf{F}$ engendre bien l'algèbre des semi-invariants, l'approche que nous avons choisie diverge de celle de Yakimova. En effet, un point clé de la démonstration de Yakimova est le calcul explicite de ce que l'on appelle un \emph{semi-invariant fondamental}. Son calcul est spécifique au cas du Borel et utilise le fait que la troncation canonique de $\gotq$ est ici égale à l'algèbre de Lie dérivée $\gotq'$, de sorte que $\Sy(\gotq)=\Y(\gotq')$. Or si $\gotq$ est une contraction parabolique quelconque en type $A$, les inclusions $\gotq' \subset \gotq_\Lambda \subset \gotq$ sont dans bien des cas des inclusions strictes\footnote{pour $\gotq$ une contraction parabolique de $\gl_n$, on calculera précisément $\gotq_\Lambda$ à la section \ref{sec5}}. Une adaptation des arguments de Yakimova dans un cadre plus général nous semble difficile.\par
Aussi, une démonstration passant par le théorème \ref{pan13} de Panyushev a été envisagée, en prenant $\gotk=\gotq_\Lambda$\footnote{on rappelle que $\Sy(\gotq)=\Y(\gotq_\Lambda)$}. Une hypothèse cruciale de ce théorème est la \emph{propriété de codimension $2$}. Celle-ci est vérifiée si et seulement si le semi-invariant fondamental associé est constant \cite{js10}. Or dans le cas où $\gotq$ est une contraction par une sous-algèbre de Borel $\gotb$, avec $\gotq'=\gotq_\Lambda$, Yakimova \cite[théorème 5.5]{yak12} montre que, hormis le cas du type $A$, le semi-invariant fondamental de $\gotq'=\gotq_\Lambda$ n'est pas constant, donc l'algèbre de Lie $\gotq_\Lambda$ ne vérifie pas la propriété de codimension $2$. Une démonstration passant par le théorème \ref{pan13} ne pourra donc pas se généraliser en dehors du type $A$.
\section*{Résultats principaux}
Présentons un résumé des résultats de ce rapport. Le résultat principal de cette thèse est celui-ci :
\begin{thep}
Soit $\gotq$ une contraction parabolique d'une algèbre de Lie simple $\g$ par une sous-algèbre parabolique $\gotp$.
\begin{itemize}
\item[$(*)$] Si $\g$ est 
\begin{itemize}
\item soit de type $A$,
\item soit de type $C$ et telle que les facteurs de Levi de $\gotp$ sont de type $A$,\footnote{On entend par "facteur de Levi de type $A$" un facteur ne faisant apparaître que des racines courtes. Par exemple, si le facteur de Levi ne fait apparaître que la racine longue, on ne considèrera pas ce facteur de Levi comme "de type $A$" bien qu'il soit a priori isomorphe à un produit de la forme $\spl_2 \times \gota$ avec $\gota$ une algèbre de Lie abélienne.}
\end{itemize}
alors $\Sy(\gotq)$ est polynomiale.
\item[$(**)$] Il existe une contraction parabolique $\gotq$, avec $\g$ de type $C$ et telle que les facteurs de Levi de $\gotp$ ne sont pas de type $A$, pour laquelle $\Sy(\gotq)$ n'est pas polynomiale.
\end{itemize}
\end{thep}
\subsection*{Contractions paraboliques de $\gl_n$}
Pour montrer la polynomialité dans les cas $(*)$, on se base sur l'étude des contractions paraboliques de l'ensemble $\gl_n$ des matrices de taille $n$. Si $\gotq$ est une contraction parabolique de $\gl_n$, la polynomialité de $\Sy(\gotq)$ se démontre en trois étapes :
\begin{itemize}
\item Étape 1 : exhiber une famille particulière $\mathbf{F}$ algébriquement indépendante de $\Sy(\gotq)$,
\item Étape 2 : montrer que $\mathbf{F}$ est une base de transcendance de $\Sy(\gotq)$,
\item Étape 3 : conclure que $\Sy(\gotq)=\C[\mathbf{F}]$.
\end{itemize}
\paragraph{Étape 1 :} Tout facteur de semi-invariant est encore un semi-invariant. Ainsi, de la même manière que Yakimova dans le cas du Borel, on trouve des semi-invariants propres (c'est-à-dire de poids non nul) en factorisant des invariants connus. Si $F_m \in \Sym(\gl_n)$ est la somme des mineurs principaux de taille $m$, alors $F_1, \ldots, F_n$ engendre librement $\Y(\gl_n)$ et $F_1^\bullet, \ldots, F_n^\bullet$ engendre librement $\Y(\gotq)$\footnote{On déduit ce point de l'étude des $\mathcal{F}_m^\bullet$ en type $A$ par Panyushev et Yakimova \cite{py13} (voir la proposition \ref{fdroit}).}, où l'on note $f^\bullet$ la composante de $f$ de degré maximal en $\gotn^-$. Dans un premier temps, on factorise donc les $F_m^\bullet$ (section \ref{sec4}). Soit $\gotl$ un facteur de Levi associé à $\gotq$. L'algèbre $\gotl$ est isomorphe à un certain produit $\gl_{i_1} \times \ldots \times \gl_{i_s}$. On dit que chaque $\gl_{i_k}$ est un \emph{bloc} de $\gotl$, de sorte que $s$ est le nombre de blocs de $\gotl$. Deux blocs $\gl_{i_k}$ et $\gl_{i_{k'}}$ sont dits \emph{isomorphes} si $i_k=i_{k'}$. On note également $\imax = \max_k i_k$ la taille du plus gros bloc. Pour tout $m \in \llbracket 1,n \rrbracket$,
\begin{itemize}
\item s'il existe $i \in \{i_1, \ldots, i_s\}$ tel que $m=\sum_k \min(i_k,i)$, alors on pose $r_m$ le nombre de blocs de $\gotl$ qui sont de taille $i$, autrement dit $r_m=\Card(\{k \, | \, i_k=i\})$,
\item sinon, on pose $r_m=1$.
\end{itemize}
\begin{the0}[voir théorème \ref{semiinv}]
Pour tout $m \in \llbracket 1,n \rrbracket$, l'invariant $F_m^{\bullet}$ est produit de $r_m$ facteurs homogènes non constants, notés $F_{m,1}, \ldots, F_{m,r_m}$. Ces facteurs sont des semi-invariants et vérifient de plus :
\begin{enumerate}[label=(\arabic*)]
\item pour tout $t \in \llbracket 1, r_m-1 \rrbracket$, $F_{m,t} \in \Sym(\gotn^-)$,
\item notant $\mathfrak{L}_m:={\lambda}_{m,1}, \ldots, {\lambda}_{m,r_m}$ la famille des poids de $F_{m,1}, \ldots, F_{m,r_m}$, on a $\lambda_{m,1} + \ldots + {\lambda}_{m,r_m}=0$ et $\mathfrak{L}_m$ est de rang $r_m-1$,
\item les espaces vectoriels $\vect(({\lambda}_{m,t})_t)$ pour $m \in \llbracket 1,n \rrbracket$ sont en somme directe,
\item la famille $\mathbf{F}$ des $F_{m,t}$ est algébriquement indépendante.
\end{enumerate}
\end{the0}
Avec ce théorème, on obtient ainsi des facteurs non triviaux de $F_m^\bullet$ si $r_m \geq 2$, facteurs dont on peut décrire explicitement le poids et le degré. On obtient la décomposition du théorème ainsi que le point (1) en étudiant combinatoirement les invariants $F_m^\bullet$ (sous-section \ref{sec3}). On montre les points (2) et (3) en calculant explicitement les poids des $F_{m,t}$ (sous-section \ref{ssec4.2}). Ce calcul des poids permet également de montrer le point (4) grâce au théorème plus général suivant :
\begin{the0}[voir théorème \ref{alin}]
Soit $\gotk$ une algèbre de Lie de dimension finie. Soit $d \geq 1$ et $(f_m)_{1 \leq m \leq d}$ une famille d'éléments algébriquement indépendants de $\Y(\gotk)$. Supposons que pour tout $m \in \llbracket 1,d \rrbracket$, l'invariant $f_m$ se décompose en un produit de semi-invariants $f_m = \prod_{t=1}^{r_m} f_{m,t}^{s_{m,t}}$ avec $r_m \geq 1$, et les $s_{m,t} \geq 1$ des entiers premiers entre eux (dans leur ensemble). On suppose de plus que les semi-invariants $f_{m,t}$ vérifient :
\begin{itemize}
\item à $m$ fixé, l'ensemble des poids $\lambda_{m,t}$ des $f_{m,t}$ pour $t \in \llbracket 1,r_m \rrbracket$ forme une famille de rang $r_m-1$,
\item les espaces vectoriels $\vect \left((\lambda_{m,t})_{1 \leq t \leq r_{m}}\right)$ pour $m \in \llbracket 1, d \rrbracket$ sont en somme directe.
\end{itemize}
Alors les $(f_{m,t})_{1 \leq m \leq d, \, 1 \leq t \leq r_m}$ sont algébriquement indépendants. 
\end{the0}
\paragraph{Étape 2 :} L'étape suivante est de montrer que $\mathbf{F}$ est une base de transcendance de $\Sy(\gotq)$. Pour cela, on montre que la dimension de Gelfand-Kirillov $\GK \C[\mathbf{F}]$ de l'algèbre polynomiale $\C[\mathbf{F}]$, engendrée par $\mathbf{F}$, est égale à la dimension de Gelfand-Kirillov de $\Sy(\gotq)$. Comme $\GK \C[\mathbf{F}]$ est le degré de transcendance du corps des fractions de $\C[\mathbf{F}]$, par l'algébrique indépendance de $\mathbf{F}$, c'est également le cardinal de $\mathbf{F}$. Il suffit alors de montrer que $\GK \Sy(\gotq) = \Card(\mathbf{F})$. On a déjà $\GK \Sy(\gotq) = \GK \Y(\gotq_\Lambda) = \ind \gotq_\Lambda$.\par
Ensuite, on calcule $\ind \gotq_\Lambda$, qui est donné par le théorème suivant :
\begin{the0}[voir théorème \ref{dimind}]
Soit $s$ le nombre de blocs d'un facteur de Levi $\gotl$ associé à $\gotq$ et $p$ le nombre de classes d'isomorphisme de blocs de $\gotl$. Autrement dit, $\gotl$ est de la forme $\gotl \simeq \gl_{i_1} \times \ldots \times \gl_{i_s}$ avec $\Card(\{i_1, \ldots, i_s\})=p$. On a
$$\dim \gotq_{\Lambda}=n^2-(s-p) \qquad \text{et} \qquad \ind \gotq_{\Lambda}=\Card(\mathbf{F})=\GK (\C[\mathbf{F}])=n+(s-p)$$
\end{the0}
Pour montrer ce dernier théorème, on utilise une égalité de Ooms et van den Bergh \cite[Proposition 3.1]{oom10}
$$ \dim \gotq + \ind \gotq = \dim \gotq_\Lambda + \ind \gotq_\Lambda$$
On a $\dim \gotq= \dim \g = n^2$ et Panyushev et Yakimova ont montré que $\ind \gotq=\rg \g =n$ \cite[Théorème 3.1]{py13}. On sait que $\ind \gotq_\Lambda = \GK \Sy(\gotq) \geq \GK \C[\mathbf{F}] = \Card \mathbf{F} = n+(s-p)$. Pour conclure pour le théorème, on montre alors que $\dim \gotq_\Lambda \geq n^2-(s-p)$ :
\begin{itemize}
\item on part de l'inclusion vraie pour toute algèbre de Lie $\gotq' \oplus \gotz(\gotq) \subset \gotq_\Lambda$ (où $\gotz(\gotq)$ est le centre de $\gotq$) qui fournit un sous-espace vectoriel de $\gotq_\Lambda$ de dimension $n^2-s+1$,
\item puisque les $F_m^\bullet$ appartiennent à $\Sy(\gotq)\subset\Sym(\gotq_\Lambda)$, on trouve les $p-1$ dimensions restantes en calculant certains facteurs dans $\h_\Lambda$ (où $\h_\Lambda:=\h \cap \gotq_\Lambda$ avec $\h$ une sous-algèbre de Cartan de $\gl_n$) de termes de certains $F_m^\bullet$ bien choisis (section \ref{sec5}).
\end{itemize}
\paragraph{Étape 3 :} Il reste alors à montrer que les $F_{m,t}$ engendrent bien $\Sy(\gotq)$ (section \ref{sec6}). On montre d'abord un premier théorème applicable a priori en dehors du contexte des contractions paraboliques.
\begin{the0}[voir théorème \ref{puissant}]
Soit $\gotk$ une algèbre de Lie de dimension finie. On suppose que $\Y(\gotk)$ est une algèbre factorielle. On suppose également que
\begin{enumerate}
\item[(I)] il existe un morphisme de $\C$-algèbres $\vartheta : \Sym(\gotk) \rightarrow \Y(\gotk)$ tel que $\vartheta_{|\Y(\gotk)}$ est un isomorphisme.
\end{enumerate}
Soit $(f_m)_{1 \leq m \leq d}$ une famille d'invariants irréductibles de $\Y(\gotk)$. Étant donné pour tout $m \in \llbracket 1,d \rrbracket$ une décomposition de $f_m$ dans $\Sym(\gotk)$ de la forme
\[
f_m = \prod_{t=1}^{r_{m}} (f_{m,t})^{\nu_{m,t}} \label{decompintro} \tag{$\clubsuit$}
\]
avec $\nu_{m,t} \in \N^*$, on suppose pour tout $m$ que l'on est dans au moins un des deux cas suivants :
\begin{itemize}
\item la décomposition \eqref{decompintro} est triviale, c'est-à-dire $r_m=1$ et $\nu_{m,1}=1$, autrement dit $f_m = f_{m,1}$,
\item la décomposition \eqref{decompintro} est la décomposition de $f_m$ en éléments irréductibles dans $\Sym(\gotk)$.
\end{itemize}
On note $\mathbf{f}$ l'ensemble des $f_{m,t}$. Si $\GK \Sy(\gotk)=\GK \C[\mathbf{f}]$, alors l'ensemble des poids\footnote{qui en général est un semi-groupe} de $\Sy(\gotk)$ est un groupe et $\Sy(\gotk)$ est engendrée par $\Y(\gotk)$ et les éléments de $\mathbf{f}$. En particulier, si les $f_m$ engendrent $\Y(\gotk)$, alors $\Sy(\gotk)=\C[\mathbf{f}]$.
\end{the0}
On va montrer que les contractions paraboliques $\gotq$ dans les cas $(*)$ du théorème principal vont vérifier les hypothèses du théorème ci-dessus, donc en particulier que l'ensemble des poids de $\Sy(\gotq)$ est un groupe. Ceci est particulier aux contractions paraboliques car une telle propriété n'est pas vraie en général. Par exemple pour une sous-algèbre parabolique $\gotp$ propre d'une algèbre de Lie $\g$ semi-simple, l'ensemble des poids de $\Sy(\gotp)$ n'est pas un groupe (voir \cite[lemme 4.2.4]{FMJ05}).\par
Sous les hypothèses de ce théorème, on voit que la décomposition de $f_m$ est triviale si et seulement si $r_m=1$. De plus, si $m$ vérifie $r_m \geq 2$, alors on remarque que pour tout $t$, $f_{m,t}$ est un semi-invariant\footnote{c'est un facteur d'un invariant donc un semi-invariant, voir la proposition \ref{dixmier}} et est de poids non nul. Dans notre cas, si pour un certain $m$, on a une décomposition $F_m^\bullet=\prod_{t=1}^{r_m} F_{m,t}$ avec $r_m \geq 2$, on ne sait pas a priori si les $F_{m,t}$ sont bien irréductibles. Pour montrer cela, on utilise le résultat technique suivant.
\begin{the0}[voir théorème \ref{inter}]
Soit $\gotk$ une algèbre de Lie de dimension finie. On suppose que $\Y(\gotk)$ est une algèbre factorielle. Soit $(f_m)_{1 \leq m \leq d}$ une famille d'invariants irréductibles dans $\Y(\gotk)$, et pour tout $m$, une décomposition
$$f_m = \prod_{t=1}^{r_{m}} (f_{m,t})^{\nu_{m,t}}$$
de $f_m$ dans $\Sym(\gotk)$, avec $\nu_{m,t} \in \N^*$ pour tout $m,t$. On note $\mathbf{f}$ l'ensemble des $f_{m,t}$ et $\mathbf{f}^\times$ l'ensemble des $f_{m,t}$ pour lesquels $r_m \geq 2$. On suppose que $\GK \C[{\mathbf{f}}]=\GK \Sy(\gotk)$. Si les trois hypothèses suivantes sont vérifiées :
\begin{enumerate}
\item[(I)] il existe un morphisme de $\C$-algèbres $\vartheta : \Sym(\gotk) \rightarrow \Y(\gotk)$ tel que $\vartheta_{|\Y(\gotk)}$ est un isomorphisme.
\item[(II)] tout $f \in \mathbf{f}^\times$ est indivisible dans le semi-groupe multiplicatif $\Sy(\gotk)$,\footnote{c'est-à-dire $f$ n'est pas une puissance $a$\up{ème} dans $\Sy(\gotk)$ pour tout $a \geq 2$}
\item[(III)] pour tout $f \in \mathbf{f}^\times$, il existe une $\C$-algèbre factorielle $\A_f$ et un morphisme de $\C$-algèbres $\vartheta_f : \Sy(\gotk) \rightarrow \A_f$ tel que $\vartheta_f(f)$ n'est pas inversible dans $\A_f$ et est premier avec $\vartheta_f(g)$ pour tout $g \in \mathbf{f}^\times \setminus \{f\}$. \label{iiiiii}
\end{enumerate}
alors les $f_{m,t}$ dans $\mathbf{f}^\times$ sont irréductibles dans $\Sym(\gotk)$.
%$\mathbf{M}_2=\{\mu \in \llbracket 1,\delta \rrbracket \, | \, r'_{\mu} \geq 2 \}$, $\mathcal{E}=\{(\mu,\tau) \, | \, \mu \in \mathbf{M}_2, 1 \leq \tau \leq r'_{\mu}\}$ et
\end{the0}
Pour conclure que $\Sy(\gotq)=\C[\mathbf{F}]$ dans le cas d'une contraction parabolique $\gotq$ de $\gl_n$ avec nos semi-invariants $F_{m,t}$ construits à l'étape 1, on montre que les hypothèses (I), (II) et (III) sont vérifiées.\par
Pour l'hypothèse (II), on montre que les semi-invariants $F_{m,t}$ pour $r_m \geq 2$ sont indivisibles en remarquant qu'ils sont de degré partiel $1$ en certains éléments $e_{v,w}$ de la base canonique de $\gl_n$.\par
On rappelle que les algèbres $\Sym(\gotk)$ et $\C[\gotk^*]$ sont isomorphes. De même, si $\A$ est une $\C$-algèbre, les $\A$-algèbres $\A \otimes_{\C} \Sym(\gotk)$ et $\A[\A \otimes_{\C} \gotk^*]$ sont naturellement isomorphes, ce qui donne un sens à $f(q) \in \A$ où $f \in \A \otimes_{\C} \Sym(\gotk)$ et $q \in \A \otimes_{\C} \gotk^*$. Pour les hypothèses (I) et (III), on démontre en fait des hypothèses (I') et (III') plus fortes, mais plus adaptées à notre cas :
\begin{enumerate}
\item[(I')] il existe $g_1, \ldots, g_{d'}$ qui engendrent librement $\Y(\gotk)$ et $q \in \C[X_1, \ldots, X_{d'}] \otimes \gotk^*$ tel que pour tout $m \in \llbracket 1,d' \rrbracket$, on a $g_m(q) = X_m$,
\item[(III')] pour tout $m \in \llbracket 1,d \rrbracket$ tel que $r_m \geq 2$ et $t \in \llbracket 1, r_m \rrbracket$, il existe $q_{m,t} \in \C[X] \otimes \gotk^*$ tel que $\deg_X f_{m,t}(q_{m,t}) \geq 1$ et $f_{\mu,\tau}(q_{m,t}) \in \C^\times$ pour $(\mu,\tau) \neq (m,t)$.
\end{enumerate}
Avec (I'), le morphisme $\vartheta$ dans (I) est alors $f \in \Sym(\gotk) \mapsto f(q) \in \C[X_1, \ldots, X_{d'}] \simeq \Y(\gotk)$. L'hypothèse (III') est une reformulation de (III) avec $\mathcal{A}_f=\C[X]$. On remarque au passage que l'hypothèse (I') implique l'existence d'une section de Kostant-Weierstrass pour $\Y(\gotq)$. Les démonstrations de (I'), où les $g_1, \ldots, g_{d'}$ sont les $F_1^\bullet, \ldots, F_n^\bullet$, et de (III') sont combinatoires et utilisent des graphes orientés pondérés que l'on appelle \textit{cheminements}. Un cheminement est un graphe orienté pondéré d'ensemble de sommets $\llbracket 1,n \rrbracket$ pour lequel il existe une unique arête allant d'un sommet $x$ à un sommet $y$. Lorsqu'on représente graphiquement un cheminement, on ne représente que les arêtes de poids non nul.\par
Soit $\A=\C[X_1, \ldots, X_n]$. À $q \in \A \otimes \gl_n^*$ fixé, on associe un cheminement $\mathcal{G}(q)$ de poids appartenant à $\A$ appelé \emph{graphe de $q$}. Le cheminement $\mathcal{G}(q)$ est tel que l'arête de $x$ vers $y$ dans $\mathcal{G}(q)$ est de poids $e_{x,y}(q)$. Les $F_m^\bullet$ sont sommes de termes de la forme $\pm \prod_{l \in J} e_{l,\sigma(l)}$ avec $J$ un sous-ensemble de $\llbracket 1,n \rrbracket$ et $\sigma$ une permutation des éléments de l'ensemble $J$. Ainsi pour calculer $F_m^\bullet(q)$, il suffit de
\begin{itemize}
\item déterminer dans $\mathcal{G}(q)$ les \emph{sous-graphes circuits} de longueur $m$, c'est-à-dire les unions de circuits de $\mathcal{G}(q)$ à sommets disjoints dont toutes les arêtes sont de poids non nul, et qui totalisent $m$ sommets,
\item montrer que ces sous-graphes sont bien associés à des monômes de $F_m^\bullet$, c'est-à-dire vérifier qu'ils ont bien le bon degré en $\gotn^-$.
\end{itemize}
Pour le point (I'), on choisit alors $q$ tel que $\mathcal{G}(q)$ est de la forme suivante :
\begin{center}
\begin{tikzpicture}
  \tikzset{LabelStyle/.style = {fill=white}}
  \tikzset{VertexStyle/.style = {%
  shape = circle, minimum size = 28pt,draw}}
  \SetGraphUnit{2.5}
  \Vertex[L=$v_1$]{1}
  \EA[L=$v_2$](1){2}
  \EA[L=$v_3$](2){3}
  \Loop[dist = 2cm, dir = WE, label = $X_1$](1.west)
  \Edge[style= {->}, label = 1](2)(1)
  \Edge[style= {->}, label = 1](3)(2)
  
  \tikzset{EdgeStyle/.style = {->,bend left=30}}
  \Edge[label=$X_2$](1)(2)
  \tikzset{EdgeStyle/.style = {->,bend left=34}}
  \Edge[label=$X_3$](1)(3)
  \SetVertexNormal[Shape = circle, LineColor=white, MinSize=28pt]
  \tikzset{EdgeStyle/.style = {->}}
  \EA[L=$\ldots$](3){4}
  \Edge[style= {->}, label = 1](4)(3)
  %\SetUpVertex[LineColor=black]
  \SetVertexNormal[Shape = circle, LineColor=black, MinSize=28pt]
  \EA[L=$v_{n-1}$](4){n-1}
  \Edge[style= {->}, label = 1](n-1)(4)
  \EA[L=$v_n$](n-1){n}
  \Edge[style= {->}, label = 1](n)(n-1)
  \tikzset{EdgeStyle/.style = {->,bend left=38}}
  \Edge[label=$X_{n-1}$](1)(n-1)
  \tikzset{EdgeStyle/.style = {->,bend left=42}}
  \Edge[label=$X_n$](1)(n)
\end{tikzpicture}
\end{center}
avec $\{v_1, \ldots, v_n\}=\llbracket 1,n \rrbracket$. À $m \in \llbracket 1,n \rrbracket$ fixé, ce cheminement admet un unique sous-graphe circuits de longueur $m$, qui est le circuit suivant :
\begin{center}
\begin{tikzpicture}
  \tikzset{LabelStyle/.style = {fill=white}}
  \tikzset{VertexStyle/.style = {%
  shape = circle, minimum size = 28pt,draw}}
  \SetGraphUnit{2.5}
  \Vertex[L=$v_1$]{1}
  \EA[L=$v_2$](1){2}
  \EA[L=$v_3$](2){3}
  %\Loop[dist = 2cm, dir = WE, label = $X_1$](1.west)
  \Edge[style= {->}, label = $1$](2)(1)
  \Edge[style= {->}, label = $1$](3)(2)
  
  \tikzset{EdgeStyle/.style = {->,bend left=30}}
  %\Edge[label=$-X_2$](1)(2)
  %\tikzset{EdgeStyle/.style = {->,bend left=34}}
  %\Edge[label=$X_3$](1)(3)
  \SetVertexNormal[Shape = circle, LineColor=white, MinSize=28pt]
  \tikzset{EdgeStyle/.style = {->}}
  \EA[L=$\ldots$](3){4}
  \Edge[style= {->}, label = $1$](4)(3)
  %\SetUpVertex[LineColor=black]
  \SetVertexNormal[Shape = circle, LineColor=black, MinSize=28pt]
  \EA[L=$v_{m-1}$](4){m-1}
  \Edge[style= {->}, label = $1$](m-1)(4)
  \EA[L=$v_m$](m-1){m}
  \Edge[style= {->}, label = $1$](m)(m-1)
%  \tikzset{EdgeStyle/.style = {->,bend left=38}}
%  \Edge[label=$(-1)^m X_{m-1}$](1)(m-1)
  \tikzset{EdgeStyle/.style = {->,bend left=42}}
  \Edge[label=$X_m$](1)(m)
\end{tikzpicture}
\end{center}
que l'on fait correspondre à un unique monôme $\mathcal{S}_m$ de $F_m$. Avec de bons choix de la suite $(v_1, \ldots ,v_n)$, c'est-à-dire de l'ordre des $v_i$, le monôme $\mathcal{S}_m$ est un monôme de $F_m^\bullet$ pour tout $m$, de sorte que $F_m^\bullet(q)=\mathcal{S}_m(q)=X_m$.\par
Pour le point (III'), à $m,t$ fixés, on spécialise ce $q$ en $X_m=X$ et $X_{\mu}=1$ pour $\mu \neq m$, de sorte que $F_m^\bullet(q)=X$ et $F_{\mu,\tau}(q) \in \C^\times$ pour $\mu \neq m$. Pour un bon choix de $(v_1, \ldots, v_n)$ parmi ceux qui vérifient le point (I'), on a $F_{m,t}(q) \in \C^\times X$ et $F_{\mu,\tau}(q) \in \C^\times$ pour $(\mu,\tau) \neq (m,t)$. Les hypothèses (I'), (II) et (III') étant vérifiées, on a bien $\Sy(\gotq)=\C[\mathbf{F}]$ et $\mathbf{F}$ engendre donc librement $\Sy(\gotq)$ en type $\gl_n$.
\subsection*{Contractions paraboliques en types $A$ et $C$}
Pour montrer la partie $(*)$ du théorème principal, on étudie la polynomialité de $\Sy(\widetilde{\gotq})$ lorsque $\widetilde{\gotq}$ est une contraction parabolique de $\g=\spl_n$ ou $\g=\syp_n$ (chapitre \ref{chap4}) en remarquant que $\widetilde{\gotq}$ est nécessairement une sous-algèbre de Lie d'une contraction parabolique $\gotq$ de $\gl_n$. Pour cela, une idée est de projeter sur $\widetilde{\gotq}$ les résultats que l'on a sur $\gotq$. Soit $\pr : \Sym(\gotq) \rightarrow \Sym(\widetilde{\gotq})$ une projection telle que l'application linéaire sous-jacente $\pr : \gotq \rightarrow \widetilde{\gotq}$ soit un morphisme de $\widetilde{\gotq}$-modules. Dans les cas que l'on considère, on a $\pr(F_m^\bullet)=\pr(F_m)^\bullet  = \prod_{t=1}^{r_m} \pr(F_{m,t})$.\par
En type $A$ (section \ref{sec7}), on a $\g=\spl_n$ l'algèbre de Lie du groupe spécial linéaire, c'est-à-dire l'ensemble des matrices de taille $n$ de trace nulle. On considère la projection $\pr : \Sym(\gl_n) \rightarrow \Sym(\spl_n)$ induite par la décomposition $\gl_n = \spl_n \oplus \C \id$. Dans ce cas, on a $\pr(F_1^\bullet)=\pr(\id)=0$, mais les $\pr(F_m^\bullet)$ pour $m \in \llbracket 2,n \rrbracket$ engendrent librement $\Y(\widetilde{\gotq})$ par \cite{py13}. Le poids de $\pr(F_{m,t})$ est le poids de $F_{m,t}$ restreint à $\spl_n$, et on montre que les $\pr(F_{m,t})$ sont algébriquement indépendants à nouveau grâce au théorème \ref{alin}. On voit facilement que les $\pr(F_{m,t})$ forment une base de transcendance de $\Sy(\widetilde{\gotq})$. Enfin, si $x \in \Sy(\widetilde{\gotq})$, alors $x \in \Sy(\gotq)$ puisque $\C \id = \gotz(\gotq)$, ainsi $x=\pr(x)$ est un polynôme en les $F_{m,t}$, donc en les $\pr(F_{m,t})$.\par
En type $C$, on a $n$ pair et $\g=\syp_n$ l'algèbre de Lie du groupe symplectique\footnote{On prend la définition de $\syp_n$ donnée par \cite[Chap. VIII]{bou06}.}. On considère la projection $\pr : \Sym(\gl_n) \rightarrow \Sym(\syp_n)$ induite par la décomposition $\gl_n = \syp_n \oplus \left(\syp_n^\perp \oplus \C \id\right)$ où $\syp_n^\perp \subset \spl_n$ est l'orthogonal de $\syp_n$ pour la forme de Killing de $\spl_n$. Dans ce cas, on a $\pr(F_m^\bullet)=0$ pour tout $m$ impair, mais les $\pr(F_m^\bullet)$ pour $m \in \llbracket 1,n \rrbracket$ pair engendrent librement $\Y(\widetilde{\gotq})$ par \cite{py13}.\par
Si l'on se place dans le cas $(*)$ du théorème principal, c'est-à-dire dans le cas où le facteur de Levi associé à la contraction $\widetilde{\gotq}$ n'a pas de facteur de type $C$, alors les facteurs de Levi $\gotl$ associés à la contraction parabolique $\gotq$ de $\gl_n$ sont de la forme $\gotl \simeq \gl_{i_1} \times \ldots \times \gl_{i_s}$ avec $s$ pair (section \ref{sec9}). Dans ce cas, la démonstration de la polynomialité de $\Sy(\widetilde{\gotq})$ reprend globalement le même schéma de preuve qu'en type $\gl_n$. Dans $\gotq$, la décomposition $F_m^\bullet = \prod_{t=1}^{r_m} F_{m,t}$ est non triviale, c'est-à-dire telle que $r_m \geq 2$, seulement si $m$ est \emph{pair}. On obtient alors des décompositions de la forme $\pr(F_m^\bullet) = \pm \prod_{t=1}^{b_m} f_{m,t}^{s_{m,t}}$, où certaines puissances $s_{m,t} \geq 2$ apparaissent en regroupant les facteurs colinéaires après projection. On peut à nouveau utiliser le théorème \ref{alin} pour montrer que la famille de semi-invariants deux à deux non-colinéaires $\mathbf{f}=(f_{m,t})_{m,t}$ que l'on obtient est algébriquement indépendante. Cette famille est de cardinal $n/2+s/2 = \ind \widetilde{\gotq} +s/2$. Ainsi $\dim \widetilde{\gotq}_\Lambda \geq \dim \widetilde{\gotq}' = \dim \widetilde{\gotq} - s/2$ et $\ind \widetilde{\gotq}_\Lambda \geq \ind \widetilde{\gotq} +s/2$. Par l'équation de Ooms et Van den Bergh \eqref{ovdb}, on a donc les égalités, de sorte que la famille $\mathbf{f}$ forme une base de transcendance de $\Sy(\widetilde{\gotq})$. Pour montrer que $\Sy(\widetilde{\gotq})=\C[\mathbf{f}]$, on montre encore les hypothèses (I), (II), et (III) du théorème \ref{inter}, en reprenant les hypothèses (I') et (III') mentionnées précédemment pour montrer (I) et (III).\par
Pour l'hypothèse (II), on montre que pour la plupart des $f_{m,t}$, leur poids $\lambda_{m,t}$ est indivisible dans l'ensemble des poids $\Lambda$ de $\Sy(\widetilde{\gotq})$, c'est-à-dire que $\lambda_{m,t} \notin a \Lambda$ pour $a \geq 2$. En fait, puisque $\Lambda \subset \sum_k \rat \varpi_k$ où les $\varpi_k$ sont les poids fondamentaux de $\syp_n$, il suffit de vérifier que $\lambda_{m,t}$ est irréductible dans $\sum_k \rat \varpi_k$. Il existe des poids $\lambda_{m,t}$ qui ne vérifient pas cette dernière propriété et dans ces cas particuliers, on vérifie l'indivisibilité des semi-invariants directement en montrant qu'ils sont de degré $1$ en certains éléments d'une base de $\syp_n$.\par
Pour l'hypothèse (I'), on réutilise la notion de graphe de $q$, que l'on modifie légèrement en type $C$. Soit $\A=\C[X_1, \ldots, X_{n/2}]$. À $q \in \A \otimes \syp_n^*$ fixé, on associe un cheminement $\mathcal{G}(q)$ de poids dans $\A$ appelé \emph{graphe de $q$}, tel que l'arête de $x$ vers $y$ est de poids $\pr(e_{x,y})(q)$. Cette définition implique une certaine symétrie dans les graphes de $q$, dans le sens où s'il existe une arête $x \rightarrow y$ de poids non nul dans le graphe de $q$, l'arête $n+1-y \rightarrow n+1-x$ est également de poids non nul\footnote{Ce résultat provient de la construction de $\syp_n$ donnée dans \cite{bou06}, dans laquelle la sous-algèbre de Cartan de $\syp_n$ est diagonale.}. On exhibe alors comme en type $\gl_n$ un élément $q$ tel que son graphe $\mathcal{G}(q)$ admet un unique sous-graphe circuits à $m$ sommets, qui est en fait un circuit. Le dit graphe a la forme suivante :
\begin{center}
\begin{tikzpicture}
  \tikzset{LabelStyle/.style = {fill=white}}
  \tikzset{VertexStyle/.style = {%
  shape = circle, minimum size = 33pt,draw}}
  \SetGraphUnit{2.5}
  \Vertex[L=$v_n$]{1}
  \EA[L=$v_{n-1}$](1){2}
  \EA[L=$v_{n-2}$](2){3}
  %\Loop[dist = 2cm, dir = WE, label = $X_1$](1.west)
  \Edge[style= {->}](1)(2)
  \Edge[style= {->}](2)(3)
  \SetVertexNormal[Shape = circle, LineColor=white, MinSize=33pt]
  \tikzset{EdgeStyle/.style = {->}}
  \EA[L=$\ldots$](3){4}
  \NO[L=$\ldots$](4){4'}
  \Edge(3)(4)
  %\SetUpVertex[LineColor=black]
  \SetVertexNormal[Shape = circle, LineColor=black, MinSize=33pt]
  \EA[L=$v_{n/2+2}$](4){m-1}
  \Edge[style= {->}](4)(m-1)
  \EA[L=$v_{n/2+1}$](m-1){m}
  \Edge[style= {->}](m-1)(m)
  \NO[L=$v_1$](1){1'}
  \NO[L=$v_2$](2){2'}
  \NO[L=$v_3$](3){3'}
  \NO[L=$v_{n/2-1}$](m-1){m-1'}
  \NO[L=$v_{n/2}$](m){m'}
  \Edge[style= {->}](2')(1')
  \Edge[style= {->}](3')(2')
  \Edge[style= {->}](4')(3')
  \Edge[style= {->}](m-1')(4')
  \Edge[style= {->}](m')(m-1')
  \Edge[style= {->,bend right=30}, label =$1$](1')(1)
  \Edge[style= {->,bend right=30}, label =$X_1$](1)(1')
  \Edge[style= {->,bend right=30}, label =$X_2$](2)(2')
  \Edge[style= {->,bend right=30}, label =$X_3$](3)(3')
  \Edge[style= {->,bend right=30}, label =$X_{n/2-1}$](m-1)(m-1')
  \Edge[style= {->,bend right=30}, label =$X_{n/2}$](m)(m')
\end{tikzpicture}
\end{center}
avec $\{v_1, \ldots, v_{n}\}=\llbracket 1,n \rrbracket$.\par
Pour le point (III'), il n'existe pas suffisamment de suites $(v_1, \ldots, v_n)$ qui conviennent au point (I'), donc on ne peut pas procéder comme pour une contraction parabolique de $\gl_n$ en adaptant simplement le $q$ obtenu au point (I'). On choisit d'abord un certain $q' \in \widetilde{\gotq}^*$, pour lequel $\mathcal{G}(q')$ est un graphe circuits union disjointe de circuits $\mathcal{C}_1, \ldots, \mathcal{C}_{l}$. On note $\mathcal{S}_1, \ldots, \mathcal{S}_{l}$ les monômes associés, et de ce choix de $q'$ on tire $ \mathcal{S}_j(q')=1$ pour tout $j$ et $\pr(F_\mu^\bullet)(q')=\prod_{j=1}^{i} \mathcal{S}_j(q')=1$ pour tout $\mu$ (pair) tel que $r_\mu \geq 2$, où $i$ est tel que $\mu=\sum_{k} \min(i_k,i)$. Fixons $m,t$. Pour vérifier l'hypothèse (III'), l'idée est de modifier $q'$ en multipliant un terme de $q'$ par $X$. Si $q$ est le $q'$ modifié, on a alors $\mathcal{S}_{i}(q)=X^p$ avec $p \in \{1,2 \}$ et $\mathcal{S}_{j}(q)=1$ pour $j \neq i$, et donc $\pr(F_m^\bullet)(q)=X^p$. Or avec cette seule opération, on obtient également $\pr(F_\mu^\bullet)(q)=X^p$ pour $\mu > m$ tel que $r_\mu \geq 2$, ce qui ne concorde pas avec l'hypothèse (III'), à moins que $i=\imax:=\max i_k$, c'est-à-dire la plus grande taille de bloc de $\gotl$. Si $i \neq \imax$, pour éviter cela, on modifie un peu plus $q'$ et la structure des circuits $\mathcal{C}_{i}$ et $\mathcal{C}_{i+1}$ dans le graphe circuits $\mathcal{G}(q')$, pour cette fois, obtenir $q$ qui vérifie $\deg_X \pr(F_m^\bullet)(q) \geq 1$ et $\deg_X \pr(F_\mu^\bullet)(q)=0$ pour $\mu \neq m$ tel que $r_\mu \geq 2$. En choisissant soigneusement sur quelles arêtes effectuer la modification, on fait en sorte que $q$ vérifie la condition (III').\par
Plaçons-nous maintenant dans le cas où le facteur de Levi associé à la contraction $\widetilde{\gotq}$ a un facteur de type $C$, c'est-à-dire dans le cas où les facteurs de Levi $\gotl$ associés à la contraction parabolique $\gotq$ de $\gl_n$ sont de la forme $\gotl=\gl_{i_1} \times \ldots \times \gl_{i_s}$ avec $s$ impair. Dans ce cas, il peut exister dans $\gl_n$ des décompositions non triviales $F_m^\bullet = \prod_{t=1}^{r_m} F_{m,t}$ avec $m$ impair. Or pour $m$ impair, $\pr(F_m)=0$ d'où $\pr(F_m^\bullet)=0$ et il existe des cas avec $m$ impair où $\pr(F_{m,t}) \neq 0$ pour certains $t$. Dans ces cas, si l'on ne considère que les $\pr(F_{m,t})$ pour $m$ pair, la famille $\mathbf{f}$ qu'ils forment n'est pas une base de transcendance de $\Sy(\gotq)$ : il "manque" des éléments. Prenons un exemple (section \ref{sec9}) et considérons $\gotq$ la contraction parabolique de $\gl_8$ définie par le diagramme suivant\footnote{la ligne pleine sépare les $e_{x,y}$ dans $\gotp$ et ceux dans $\gotn^-$, et les $e_{x,y}$ entre la ligne pleine et la ligne en pointillés forment une base d'un facteur de Levi $\gotl$} :
\begin{center}
\begin{tikzpicture}[scale=0.6]
\foreach \k in {1,2,...,7}
	{\draw[color=gray!20]  (0,\k)--(8,\k);
	\draw[color=gray!20] (\k,0)--(\k,8);}
\draw (0,0)--(8,0);
\draw (0,0)--(0,8);
\draw (0,8)--(8,8);
\draw (8,0)--(8,8);
\draw[ultra thick] (0,7)--(1,7)--(1,5)--(3,5)--(3,3)--(5,3)--(5,1)--(7,1)--(7,0);
\draw (2,2) node{\LARGE $\gotn^-$};
\draw (6,6) node{\LARGE $\gotp$};
\draw[dashed] (1,8)--(1,7)--(3,7)--(3,5)--(5,5)--(5,3)--(7,3)--(7,1)--(8,1);
%\draw (0.5,7.5) node{\Large $\g_{I_1}$};
%\draw (2,6) node{\Large $\g_{I_2}$};
%\draw (4,4) node{\Large $\g_{I_3}$};
%\draw (6,2) node{\Large $\g_{I_4}$};
%\draw (7.5,0.5) node{\Large $\g_{I_5}$};
\end{tikzpicture}
\end{center}
On note $\widetilde{\gotq}$ la contraction parabolique de $\syp_8$ correspondante. Dans ce cas-ci, on a une décomposition $F_8^\bullet=F_{8,1}F_{8,2}F_{8,3}$ qui se projette en $\pr(F_8^\bullet)=\pr(F_{8,1})^2\pr(F_{8,2})$ ainsi qu'une décomposition $F_5^\bullet=F_{5,1} F_{5,2}$. En projetant, on obtient $\pr(F_5^\bullet)=\pr(F_{5,1})=0$ alors que $\pr(F_{5,2}) = \pr(e_{1,8}) \neq 0$. Or $\pr(F_{5,2})$ n'est proportionnel à aucun des $\pr(F_{m,t})$ avec $m$ pair. Pour appliquer la même démarche que dans les cas précédents, on souhaiterait retrouver ce $\pr(F_{5,2})$ comme facteur d'un invariant, c'est-à-dire ici un polynôme en $\pr(F_2^\bullet), \pr(F_4^\bullet), \pr(F_6^\bullet), \pr(F_8^\bullet)$. On trouve que $\pr(F_{5,2})$ divise $\pr(F_8^\bullet)-\frac{1}{4} \pr(F_4^\bullet)^2$, et on note $\pr(F_8^\bullet)-\frac{1}{4} \pr(F_4^\bullet)^2=\widetilde{F} \pr(F_{5,2})$. On montre alors que la famille $\pr(F_2^\bullet), \pr(F_6^\bullet), \pr(F_{8,1}), \pr(F_{8,2}), \widetilde{F}, \pr(F_{5,2})$ est algébriquement indépendante en calculant les poids des semi-invariants et en reprenant l'idée des cas précédents. On montre enfin que la famille $\pr(F_2^\bullet), \pr(F_4^\bullet), \pr(F_6^\bullet), \pr(F_{8,1}), \pr(F_{8,2}), \widetilde{F}, \pr(F_{5,2})$ engendre l'algèbre $\Sy(\widetilde{\gotq})$ en montrant à la main les hypothèses (I'), (II) et (III)\footnote{et non (III') qui ne peut pas être simplement montré ici}. Puisque $\widetilde{F} \pr(F_{5,2})+\frac{1}{4}\pr(F_4^\bullet)^2=\pr(F_{8,1})^2\pr(F_{8,2})$, l'algèbre
$$\Sy(\widetilde{\gotq})=\C[\pr(F_2^\bullet), \pr(F_4^\bullet), \pr(F_6^\bullet), \pr(F_{8,1}), \pr(F_{8,2}), \widetilde{F}, \pr(F_{5,2})]$$
est isomorphe à l'algèbre $\C[X_1, X_2, X_3, X_4, X_5, X_6, X_7]/(X_6X_7+\frac{1}{4}X_2^2-X_4^2X_5)$ qui n'est pas polynomiale.\par
Les phénomènes menant à ce contre-exemple ne sont pas isolés. On présente à la section \ref{sec11} quelques conjectures dans ce sens.

\subsection*{Quelques pistes pour des contractions paraboliques en d'autres types}

On discute à la section \ref{sec12} les limites de l'approche que nous avons menée et quelques pistes possibles pour étudier des cas en d'autres types. Le schéma de preuve de polynomialité ou de non-polynomialité de l'algèbre des semi-invariants associée à une contraction parabolique $\gotq$ présenté ici repose a minima sur la connaissance d'une famille génératrice minimale de $\Y(\gotq)$. Une telle famille est donnée notamment en types $A$ et $C$ par Panyushev et Yakimova \cite{py13}. En types $B$ et $D$ ainsi qu'en types exceptionnels, la méconnaissance d'une telle famille de manière générale rend difficile une telle étude. De plus, certaines propriétés des contractions paraboliques qui sont vraies en types $A$ et $C$ ne le sont plus nécessairement en types $B$ ou $D$. Par exemple, il est possible que $\pr(F_m)^\bullet \neq \pr(F_m^\bullet)$ c'est-à-dire $\pr(F_m^\bullet)=0$ alors que $\pr(F_m)^\bullet \neq 0$, ce qui complique la combinatoire pour trouver des semi-invariants.\par
À l'inverse, dans certains cas pour lesquels l'approche de Panyushev et Yakimova a échoué, on peut espérer trouver une base de transcendance d'invariants en étudiant des semi-invariants connus. Par exemple, on peut utiliser la proposition suivante qui donne les semi-invariants de degré $1$ pour une contraction parabolique.
\begin{propn0}[voir proposition \ref{degre1}]
Soit $\gotq=\gotp \ltimes \gotn^-$ une contraction parabolique d'une algèbre simple $\g$ associée à une certaine sous-algèbre de Cartan, une certaine base $\pi$ du système de racines $R$ associé et un certain sous-ensemble $\pi' \varsubsetneq \pi$. On note $\theta$ la plus grande racine de $R$. Soit $e \in \gotq$ non nul. Alors $e$ est un semi-invariant (de degré $1$) de $\Sym(\gotq)$ si et seulement si $e \in \g_{\beta}$, avec $\beta \in (-\pi \setminus -\pi') \cup \{\theta\}$ tel que le sommet associé à $-\beta$ dans le diagramme de Dynkin étendu de $R$ n'est relié à aucun sommet associé à un élément de $\pi'$.
\end{propn0}
Dans l'exemple \cite[remarque 4.6]{py13} en type $D_6$ avec $\pi'=\{\alpha_3,\alpha_4,\alpha_5\}$, la famille $\pr(F_2)^\bullet$, $\ldots$, $\pr(F_{10})^\bullet, f^\bullet$ (avec $f$ une racine carrée de $\pr(F_{12})$) n'est pas algébriquement indépendante et est donc "trop petite". En observant les semi-invariants de degré $1$, on remarque que $\pr(F_4)^\bullet$ appartient à l'idéal de $\Sym(\gotq)$ engendré par des semi-invariants $x_1, x_2$ de degré $1$. En poussant les calculs on obtient $\pr(F_4)^\bullet=x_1y_1 + x_2y_2 + x_1x_2y_{1,2}$, où $y_1, y_2, y_{1,2}$ sont des semi-invariants et $x_1y_1$,  $x_2y_2$,  $x_1x_2y_{1,2}$ sont des invariants de même degré et degré en $\gotn^-$ que $\pr(F_4)^\bullet$. On peut espérer avec ces trois invariants obtenir une famille "plus grande" voire algébriquement indépendante.\par
Le travail de cette thèse peut être relié à d'autres articles récemment publiés. Rappelons que l'on décrit explicitement les générateurs de $\Sy(\gotq)$ lorsque $\Sy(\gotq)$ est polynomial (avec leur poids et leur degré) en type $A$ et en type $C$ avec un facteur de Levi de type $A$. Cela semble également donner une écriture explicite des générateurs de l'algèbre polynomiale $\Sym(\gotn^-)^{\gotp'}$. Cette question a été évoquée par Joseph et Fittouhi (voir \cite[\S 2.2.3]{FJ20}). Toujours selon Joseph et Fittouhi, le nombre de ces générateurs est aussi égal au nombre de variétés orbitales d'hypersurface dans le radical nilpotent $\gotn$ de la sous-algèbre parabolique $\gotp$ (voir \cite[\S 2.3.4]{FJ20}). Ces variétés orbitales d'hypersurface ont été décrites en type $A$ par Joseph et Melnikov \cite{JM03} et par Perelman \cite{Per03} pour les autres types classiques.
% sont bien algébriquement indépendants et on montre que les projections des $F_m^\bullet$ (hormis $F_1^\bullet$ dont la projection est nulle) engendrent l'algèbre des semi-invariants (section \ref{sec7}). En type $C$, les projections des $F_m^\bullet$ sont non nulles si et seulement si $m$ est pair, ce qui ne pose pas de problème lorsque le Levi associé à la sous-algèbre de Lie parabolique sous-jacente est de type $A$ (section \ref{sec9}), mais en pose lorsque ce n'est pas le cas, et il arrive que l'algèbre des semi-invariants ne soit pas polynomiale (section \ref{sec10}). Les calculs heuristiques qyant permis d'arriver à ces résultats amènent à des conjectures pour une condition nécessaire et suffisante à la polynomialité en type $C$. On discute les limites du schéma de preuve en évoquant un exemple exhibé par Panyushev et Yakimova présentant les limites de leur méthode pour montrer la polynomialité de l'algèbre des invariants (section \ref{sec11}).
\newpage
\selectlanguage{english}
\chapter*{Introduction in English}
\label{english}
\addcontentsline{toc}{chapter}{Introduction in English}
The goal of this PhD thesis is to study an important class of non-reductive Lie algebras called parabolic contractions (which are deformations of reductive Lie algebras). More precisely, we study the symmetric algebra of semi-invariants associated with such parabolic contractions. Our study relies on crucial results from Panyushev and Yakimova about the polynomiality of the symmetric algebra of invariants  (in types $A$ and $C$) associated with these parabolic contractions.

\section*{Algebras of invariants and semi-invariants}
First we define algebras of invariants and semi-invariants. Let $\C$ be an algebraically closed field with characteristic zero, $(\gotk, [\, , \, ])$ a finite-dimensional Lie algebra over $\C$ and $\gotk^*$ its dual vector space. Define the index $\ind \gotk$ as the smallest dimension of the kernel of $\phi_f$ for $f \in \gotk^*$, where $\phi_f$ is the bilinear form defined by $\phi_f : (x,y) \rightarrow f([x,y])$. The adjoint representation of $\gotk$ on itself is the natural Lie algebra representation defined from its Lie bracket $\ad(g)(h) = g \cdot h = [g,h]$ for all $g, h \in \gotk$. Extend this representation to a unique representation, which is still denoted by $\ad$, from $\gotk$ to the symmetric algebra $\Sym(\gotk)$ of $\gotk$ such that $\ad(g)$ is a derivation on $\Sym(\gotk)$ for all $g \in \gotk$.\par
We focus on the algebra of invariants  $\Y(\gotk)$ and the algebra of semi-invariants  $\Sy(\gotk)$, which are subalgebras of $\Sym(\gotk)$. The algebra of invariants $\Y(\gotk)$ is the algebra of all invariants (or symmetric invariants), \emph{i.e.} all elements of $\Sym(\gotk)$ in the kernel of every $\ad(g)$, $g \in \gotk$. The algebra of semi-invariants $\Sy(\gotk)$ is the algebra \emph{generated by} all semi-invariants (or symmetric semi-invariants), \emph{i.e.} the elements of $\Sym(\gotk)$ which are eigenvectors of every $\ad(g)$, $g \in \gotk$. Let $s$ be a non-zero semi-invariant, then there exists $\lambda \in \gotk^*$ such that $\ad(g)(s)=\lambda(g)s$ for every $g \in \gotk$. Call $\lambda$ the weight of $s$. Endow $\Sym(\gotk)$ with a Poisson bracket $\{ \, , \}$ extending the Lie bracket on $\gotk$, such that the associative algebra $\Sym(\gotk)$ becomes a \emph{Poisson algebra}. The algebras $\Y(\gotk)$ and $\Sy(\gotk)$ are respectively the centre and the semi-centre of this Poisson algebra. Their Gelfand-Kirillov dimensions $\GK \Y(\gotk)$ and $\GK \Sy(\gotk)$ coincide with the transcendence degrees of their associated fields of fractions. By a theorem of Rosenlicht, we get $\GK \Y(\gotk) \leq \ind \gotk$ \cite{ros63}, with equality when $\Y(\gotk)=\Sy(\gotk)$ (see \cite[Chap. I, Sec. B, 5.12]{fau14}). In invariant theory associated with Lie algebras, one of the crucial questions related to the algebras $\Y(\gotk)$ and $\Sy(\gotk)$ is the polynomiality, \emph{i.e.} determining whether these algebras are isomorphic to some algebras of polynomials.\par
Geometrically, let us consider the case when $\gotk$ is the Lie algebra of a connected algebraic group $K$. The algebra $\Y(\gotk)$ identifies with the algebra $\C[\gotk^*]^K$ of polynomial functions on $\gotk^*$ which are invariant under the coadjoint action of $K$. If $\Y(\gotk)$ has finite type, the variety associated with the algebra of invariants  is then $ \Spec \Y(\gotk) = \gotk^* \sslash K$. The polynomiality of $\Y(\gotk)$ then implies that $\gotk^* \sslash K$ is an affine space. Similarly to algebras, one can also define the field of invariants $\FY(\gotk)$, which is a subfield of the fraction field $\Frac(\Sym(\gotk))$ of $\Sym(\gotk)$. If $\FY(\gotk)$ is a purely transcendental extension of $\C$, then geometrically, $\gotk^* \sslash K$ is a rational variety. If $\FY(\gotk)=\Frac(\Y(\gotk))$, then the polynomiality of $\Y(\gotk)$ clearly implies the rationality of $\FY(\gotk)$. However, the converse is not true: Dixmier has given a counterexample in which $\Y(\gotk)$ is not polynomial and even not finitely generated but $\FY(\gotk)=\Frac(\Y(\gotk))$ is indeed rational \cite[4.4.8 and 4.9.20]{Dix}.\par
The first well-known result of polynomiality of algebra of invariants is due to Chevalley \cite{Dix}. It states that if $\g$ is a reductive Lie algebra, \emph{i.e.} the product of a semisimple and an abelian Lie algebra, then $\Y(\g)$ is polynomial. In the case where $\g$ is semisimple, the proof relies on an isomorphism between the algebra of invariants $\Y(\g) \simeq \Y(\g^*)$ and the algebra $\Sym(\h^*)^W$, where $\h$ is a Cartan subalgebra of $\g$ and $W$ is the Weyl group associated with the root system of $(\g,\h)$. In particular, $W$ can be seen as a finite subgroup of $\GL(\h^*)$ generated by reflections. Chevalley then identifies $\Sym(\h^*)^W$ with a polynomial algebra thanks to a result of commutative algebra (see \cite[theorem 31.1.6]{TYu} for more details). A result of Kostant \cite{kos63} gives another interpretation of this result. According to it, there exists an affine subspace $\mathcal{V}$ of $\g^*$ such that the restriction morphism $\Y(\g) \simeq \C[\g^*]^G \rightarrow \C[\mathcal{V}]$ to $\mathcal{V}$ is an isomorphism from $\Y(\g)$ to the algebra $\C[\mathcal{V}]$ of polynomial functions on $\mathcal{V}$. Such isomorphisms are called "Kostant-Weierstrass sections"\footnote{They are also often called "Kostant slices" or "Weierstrass sections". These terms can also designate the affine subspace $\mathcal{V}$.}.\par % Ces sections permettent d'une certaine manière de linéariser l'étude de l'algèbre des invariants.\par
If $\gotk=\g^e$ is the centraliser of a nilpotent element $e$ of a simple Lie algebra $\g$, Premet, Panyushev and Yakimova have shown the polynomiality of $\Y(\g^e)$ when $\g$ is simple of type $A$ or $C$ \cite{ppy07}. Their proof relies on the following theorem of Panyushev \cite[theorem 1.2]{pan07}.
\begin{theE}
Let $\gotk$ be a finite-dimensional Lie algebra. Assume that $\gotk$ verifies the codimension $2$ property, and $\GK \Y(\gotk) = \ind \gotk$. Let $l= \ind \gotk$, and $f_1, \ldots, f_l$ be algebraically independent homogeneous elements of $\Y(\gotk)$. Then
\begin{enumerate}
\item $\sum_{i=1}^l \deg f_i \geq (\dim \gotk + \ind \gotk )/2$,
\item if $\sum_{i=1}^l \deg f_i = (\dim \gotk + \ind \gotk )/2$, then $\Y(\gotk)=\C[f_1, \ldots, f_l]$.
\end{enumerate}
\label{pan13}
\end{theE}
Here the \emph{codimension $2$ property} is defined via the following:
\begin{def0}
Let $\gotk$ be a Lie algebra. For every $f \in \gotk^*$, we say that $f$ is \emph{regular} in $\gotk^*$ if the kernel of the bilinear form\footnote{introduced earlier} $\phi_f$ has minimal dimension. Otherwise, we say that $f$ is \emph{singular} in $\gotk^*$.\par
We say that $\gotk$ \emph{verifies the codimension $2$ property} if the set of singular elements $\gotk^*_{\text{sing}}$ in $\gotk^*$ has codimension at least $2$ in $\gotk^*$.
\end{def0}
Thus the question of the polynomiality of $\Y(\gotk)$ when $\gotk=\g^e$ has been thoroughly studied, and several cases of non-polynomiality have been discovered in this case. If $e$ is a highest root vector for $\g$ simple of type $E_8$, then one has $\Y(\g^e) \simeq \Sy(\gotp)$ for a certain parabolic subalgebra $\gotp$ of $\g$ and Yakimova has shown that this algebra was not polynomial \cite{yak07}. Charbonnel and Moreau extended this study by giving a necessary and sufficient condition for the polynomiality of $\Y(\g^e)$. If $\Y(\g^e)$ is polynomial, they also showed that $\Sym(\g^e)$ is a free $\Y(\g^e)$-module. They then use these conditions to study the polynomiality of $\Y(\g^e)$ depending on the element $e \in \g$ when $\g$ is simple of type $B$ or $D$. For example they find a case where $\Y(\g^e)$ is not polynomial, when $\g$ is of type $D_7$ \cite[\S 7.3]{chm14}.\par
Polynomiality results concerning the algebra of semi-invariants $\Sy(\gotk)$ are more fragmented. These algebras can be studied via the \emph{canonical truncation} of $\gotk$. For a Lie algebra $\gotk$, the canonical truncation $\gotk_\Lambda$ of $\gotk$ is the largest Lie subalgebra of $\gotk$ which vanishes on all weights of $\Sy(\gotk)$. Fauquant-Millet and Joseph\footnote{following Borho, Gabriel and Rentschler in the nilpotent case} showed in a quite general frame\footnote{in particular if $\gotk$ is an ad-algebraic Lie algebra} that $\Sy(\gotk)=\Y(\gotk_\Lambda)=\Sy(\gotk_\Lambda)$ \cite[Appendice B.2]{fmj08}. Thus, studying the polynomiality of an algebra of semi-invariants reduces to, in a certain sense, the study of the polynomiality of an algebra of invariants . For example, if $\g$ is a reductive Lie algebra, one has $\g_\Lambda=\g$, so that $\Sy(\g)=\Y(\g)$, thus $\Sy(\g)$ is polynomial. Joseph \cite{jos77} showed that in the case of a Borel subalgebra $\gotb$ of a simple Lie algebra $\g$, the algebra of semi-invariants  $\Sy(\gotb)$ is polynomial and generated by $\rk \g$ generators ($\rk \g$ being the rank of $\g$) which are explicitly computed. In the case where $\gotp$ is a biparabolic Lie subalgebra in type $A$ or $C$, Joseph and Fauquant-Millet \cite{FMJ05} showed the polynomiality of $\Sy(\gotp)$ by finding an upper bound and a lower bound to the formal character of $\Sy(\gotp)$ and by showing the bounds are equal notably in types $A$ and $C$. In types $A$ and $C$, they successfully computed the weights and degrees of homogeneous generators.
%Pour cela, ils exhibent un encadrement $\mathcal{A} \subset \gr'(\gr''(\Sy(\gotp))) \subset \mathcal{B}$ par deux algèbres polynomiales $\mathcal{A}$ et $\mathcal{B}$, où $\gr'$ et $\gr''$ sont des graduations de $\Sym(\gotp)$. Ils transposent alors cet encadrement en inégalités de \emph{caractères formels} (voir \cite[2.8]{jos07} pour une définition) et montrent que les caractères formels de $\mathcal{A}$ et $\mathcal{B}$ sont égaux en types $A$ et $C$. Ils en déduisent que les inclusions de l'encadrement sont en fait des égalités, et donc que $\Sy(\gotp)$ est polynomiale.
In type $A$, Joseph also exhibited Kostant-Weierstrass sections for $\Y(\gotp_\Lambda)=\Sy(\gotp)$ \cite{jos08}.\par
\section*{Parabolic contractions}
In this PhD thesis, we focus on a particular class of Lie algebras called \emph{parabolic contractions}, which are deformations of reductive Lie algebras. These algebras have been introduced by E. Feigin, in the case of a contraction by a Borel subalgebra \cite{Fei12}, and Panyushev and Yakimova for all parabolic subalgebras \cite{py13}. Let $\g$ be a reductive Lie algebra and $\gotp$ a parabolic subalgebra of $\g$. Denote by $\gotn$ the nilradical of $\gotp$ and $\gotn^-$ the nilradical of the  opposite parabolic subalgebra of $\gotp$ so that, if $\gotl$ is a Levi factor of $\gotp$, one has $\g=\gotn \oplus \gotl \oplus \gotn^-$. Then $\gotp=\gotn \oplus \gotl$ thus $\g = \gotp \oplus \gotn^-$. Define $\gotq = \gotp \ltimes (\gotn^-)^{\text{a}}$ the parabolic contraction of $\g$ by $\gotp$. Here, $(\gotn^-)^\text{a}$ is the vector subspace $\gotn^-$ seen as an abelian ideal\footnote{For practical reasons, we will omit this notation after this and simply note this Lie algebra $\gotn^-$.} of $\gotq$ as well as a $\gotp$-module with the identification $\gotn^- \simeq \g/\gotp$. One can see this Lie algebra as an Inönu-Wigner contraction\footnote{for a general definition of Inönu-Wigner contractions, see \cite{iw53} or \cite[\S 2.3]{dr85}} of $\g$. For every $t \neq 0$, let $c_t : \g \rightarrow \g$ be the automorphism of $\C$-vector space defined by
$$\forall p \in \gotp, \forall n \in \gotn^-, \, c_t(p+n)=p+tn$$
Define then $\g_{(t)}$ the Lie algebra equal to $\g$ as vector space, with its Lie bracket defined as
$$\forall q,q' \in \g_{(t)}, [q,q']_{\g_{(t)}} = c_t^{-1}([c_t(q),c_t(q')]_{\g}) $$
The limit, when $t$ goes to $0$, of the Lie algebras\footnote{In other words, the limit, when $t$ goes to $0$, of the Lie brackets on $\g_{(t)}$ gives the Lie bracket on $\gotq$.} $\g_{(t)}$ gives the parabolic contraction $\gotq$.\par
From these parabolic contractions, Feigin defines flag variety degenerations and studies them combinatorially in detail. With these degenerations, he gets a combinatorial link with the sequence of Genocchi numbers \cite{fem2012}.\par
Other previously studied Lie algebras are similar to some parabolic contractions in particular cases. Lie algebras of the form $\gota \ltimes \gota^*$ with $\gota$ a Lie algebra, are particular cases of \emph{Drinfeld doubles} (see \cite{dri88}, \cite{es01} or \cite[\S 3.2]{jos95}). When $\gota=\gotb$ is a Borel subalgebra, the quotient of the Lie algebra $\gotb \ltimes \gotb^*$ by its centre\footnote{here the centre of $\gotb \ltimes \gotb^*$ is isomorphic to the Cartan subalgebra associated with $\gotb^*$} is isomorphic to a parabolic contraction by a Borel subalgebra. After the results presented in this dissertation, Bar-Natan and Van der Veen in type $A$ \cite{bnv20}, then Bulois and Ressayre for a Borel subalgebra in all types \cite{br20} study the automorphisms of such algebras. For this, Bulois and Ressayre redefine these algebras in the setting of Kac-Moody algebras.\par
In type $A$, a parabolic contraction by a Borel subalgebra can be seen as the Lie algebra associated with a \emph{Brauer cyclic scheme}, introduced by Knutson and Zinn-Justin \cite[1.1 and 2.1]{kzj07}. The Brauer cyclic scheme is defined as the set $E$ of all matrices $M$ with size $n \times n$ such that $M \bullet M = 0$, where $\bullet$ is an associative product which deforms the usual product on matrices. The Lie algebra structure on $E$ associated with the associative product $\bullet$ gives rise to a parabolic contraction by a Borel subalgebra in type $A$.\par
Panyushev and Yakimova have also studied other types of contractions, e.g., $\rat_2$-contractions (see for instance \cite{yak14-2} and \cite{pan07}), of the form $\g_0 \ltimes \g_1$ where $\g=\g_0 \oplus \g_1$ is a $\rat_2$-grading of a reductive Lie algebra $\g$. More generally, Panyushev and Yakimova are aiming at classifying all Lie algebras of the form $\g \ltimes V$ with $\g$ a simple Lie algebra and $V$ a finite-dimensional representation of $\g$ such that the associated algebra of invariants  is polynomial (e.g., see \cite{py19}).\par
\section*{Algebras of invariants and semi-invariants associated with parabolic contractions}
This PhD thesis lies in the continuity of the work of Panyushev and Yakimova concerning the study of $\Y(\gotq)$ \cite{py13} where $\gotq$ is a parabolic contraction and we use their essential results. According to Panyushev and Yakimova (\cite[\mbox{}1]{py13}) for a parabolic contraction $\gotq$, there exists a connected algebraic group $Q$ such that $\gotq$ is its Lie algebra. In \cite{py13}, they study the polynomiality of the algebras of invariants $\C[\gotq]^Q$ and $\C[\gotq^*]^Q$. Using algebraic geometry\footnote{essentially some form of the Igusa lemma}, Panyushev and Yakimova easily show that $\C[\gotq]^Q$ is a polynomial algebra. For every parabolic contraction $\gotq$ by a Borel subalgebra, they also show \cite{py12} that $\C[\gotq^*]^Q \simeq \Y(\gotq)$ is polynomial (using properties which are specific to the nilradical of a Borel subalgebra). However, studying the algebra $\Y(\gotq) \simeq \C[\gotq^*]^Q$ more generally is more involved. Set $\gotq$ a parabolic contraction of a semisimple Lie algebra $\g$ by a parabolic subalgebra $\gotp$ and $G, P$ connected algebraic groups such that $\g, \gotp$ are the Lie algebras respectively of $G, P$. From algebraically independent homogeneous generators $\mathcal{F}_1, \ldots, \mathcal{F}_n$ of the algebra $\Y(\g)$, Panyushev and Yakimova construct some elements $\mathcal{F}_1^\bullet, \ldots, \mathcal{F}_n^\bullet$ of $\Y(\gotq)$ by degenerating $\mathcal{F}_1, \ldots, \mathcal{F}_n$ via $c_t$. These $\mathcal{F}_1^\bullet, \ldots, \mathcal{F}_n^\bullet$ freely generate $\Y(\gotq)$ in the case of a contraction by a Borel subalgebra ; they also are potential candidates to freely generate $\Y(\gotq)$ in the general case. To show this last point in the case of any parabolic subalgebra, they use the idea of Kostant-Weierstrass sections and adapt it to the parabolic contractions case. They construct an injection\footnote{which is a restriction map} $\psi$ from $\Y(\gotq)$ to a certain algebra $\Sym(\g^e)^{P^e}$, where $e$ is a nilpotent element of $\g$. The invariants $\mathcal{F}_1^\bullet, \ldots \mathcal{F}_n^\bullet \in \Y(\gotq) \subset \C[\gotq^*]$ can then be seen in the algebra $\Sym(\g^e)^{P^e}$ by setting $\eav \mathcal{F}_m=\psi(\mathcal{F}_m^\bullet)$. In fact, the elements $\eav \mathcal{F}_m$ belong to $\Sym(\g^e)^{G^e}=\Y(\g^e)$. In types $A$ and $C$, Panyushev and Yakimova showed that $\Sym(\g^e)^{P^e}=\Y(\g^e)$ and that the $\eav \mathcal{F}_m$ are algebraically independent and generate $\Y(\g^e)$ (via their study with Premet \cite{ppy07} previously mentioned). Furthermore, in these cases, the restriction map $\psi$, which was injective, becomes an isomorphism of algebras between $\Y(\gotq)$ and $\Y(\g^e)$. With these latest results, they successfully show that $\Y(\gotq)$ is polynomial and generated by $\mathcal{F}_1^\bullet, \ldots \mathcal{F}_n^\bullet$.\par
Yakimova has also considered the question of the polynomiality of $\Sy(\gotq)$ where $\gotq$ is a parabolic contraction by a Borel subalgebra. In this case, she shows the polynomiality of $\Sy(\gotq)$ \cite[\S 5.1]{yak12}. The proof again uses the family of free generators $\mathcal{F}^\bullet_1, \ldots, \mathcal{F}^\bullet_{n}$ of $\Y(\gotq)$. She notes emarks that $\mathcal{F}^\bullet_{n}$ has a unique factorisation as a product of irreducible elements\footnote{which are then semi-invariants} of degree $1$, \emph{i.e.} in $\gotq \subset \Sym(\gotq)$. Thus, by replacing $\mathcal{F}^\bullet_{n}$ in the family $\mathcal{F}^\bullet_1, \ldots, \mathcal{F}^\bullet_{n}$ by its irreducible factors, she gets a certain family $\mathbf{F}$, and shows that it freely generates $\Sy(\gotq)$.\par
In this thesis, we take this idea of decomposing an invariant into a product of irreducible elements then replacing the invariant with its irreducible factors. The idea is to get a family $\mathbf{F}$ which would hopefully freely generate the algebra of semi-invariants . However, to show that $\mathbf{F}$ does generate the algebra of semi-invariants, our approach differs from the one of Yakimova. In fact, one of the key points of Yakimova's proof is the explicit computation of what is called a \emph{fundamental semi-invariant}. This computation is specific to the Borel case and uses the fact that in this specific case, the canonical truncation $\gotq_\Lambda$ is equal to the derived Lie algebra $\gotq'$, so that $\Sy(\gotq)=\Y(\gotq')$. Yet if $\gotq$ is any parabolic contraction in type $A$, the inclusions $\gotq' \subseteq \gotq_\Lambda \subseteq \gotq$ are in many cases proper\footnote{for $\gotq$ a parabolic contraction of $\gl_n$, we compute precisely $\gotq_\Lambda$ in section \ref{sec5}}. We think that adapting Yakimova's arguments in a more general setting may be difficult.\par
Also, we considered a proof using Panyushev's theorem \ref{pan13}, with $\gotk=\gotq_\Lambda$\footnote{remember that $\Sy(\gotq)=\Y(\gotq_\Lambda)$}. In this theorem, a crucial hypothesis is the \emph{codimension $2$ property}. This property is verified if and only if the associated fundamental semi-invariant is scalar \cite{js10}. Yet in the case when $\gotq$ is a contraction by a Borel subalgebra $\gotb$, for which $\gotq'=\gotq_\Lambda$, Yakimova \cite[theorem 5.5]{yak12} shows that, except in type $A$, the fundamental semi-invariant associated with $\gotq_\Lambda$ is not scalar, so the Lie algebra $\gotq_\Lambda$ does not satisfy the codimension $2$ property. Thus a proof using theorem \ref{pan13} cannot be generalised outside type $A$.
\section*{Main results}
We give an outline of the results of this dissertation. Our main result is the following:
\begin{thepE}
Let $\gotq$ be a parabolic contraction of a simple Lie algebra $\g$ by a parabolic subalgebra $\gotp$.
\begin{itemize}
\item[$(*)$] If $\g$ is
\begin{itemize}
\item either of type $A$,
\item or of type $C$ with a Levi factor of $\gotp$ of type $A$,\footnote{By "Levi factor of type $A$", we mean a factor only associated with short roots. For example, if the Levi factor is associated only with the long root, we will not consider this factor as a "type $A$" Levi factor, even if it is isomorphic to a certain product $\spl_2 \times \gota$ with $\gota$ an abelian Lie algebra.}
\end{itemize}
then $\Sy(\gotq)$ is polynomial.
\item[$(**)$] There exists a parabolic contraction $\gotq$, with $\g$ of type $C$ and such that the Levi factor of $\gotp$ is not of type $A$, for which $\Sy(\gotq)$ is not polynomial.
\end{itemize}
\end{thepE}
\subsection*{Parabolic contractions of $\gl_n$}
To show the polynomiality in the $(*)$ cases, we rely on the study on the parabolic contractions of $\gl_n$, which is the Lie algebra of all matrices with size $n$. If $\gotq$ is a parabolic contraction of $\gl_n$, the polynomiality of $\Sy(\gotq)$ is shown in three steps:
\begin{itemize}
\item Step 1: finding a specific algebraically independent family $\mathbf{F}$ of $\Sy(\gotq)$,
\item Step 2: showing that $\mathbf{F}$ is a transcendence basis of $\Sy(\gotq)$,
\item Step 3: conclude that $\Sy(\gotq)=\C[\mathbf{F}]$.
\end{itemize}
\paragraph{Step 1:} Any factor of a semi-invariant is still a semi-invariant. Thus, in the same way as Yakimova in the Borel case, we find proper semi-invariants (\emph{i.e.} with a non-zero weight) by factorising some known invariants. If $F_m \in \Sym(\gl_n)$ is the sum of the principal minors of size $m$, then $F_1, \ldots, F_n$ freely generate $\Y(\gl_n)$ and $F_1^\bullet, \ldots, F_n^\bullet$ freely generate $\Y(\gotq)$\footnote{This point can be deduced from the study of the $\mathcal{F}_m^\bullet$ in type $A$ by Panyushev and Yakimova \cite{py13} (see proposition \ref{fdroit}).}, where $f^\bullet$ is the component of $f$ with maximal degree in $\gotn^-$. So firstly, we factorise the $F_m^\bullet$ (section \ref{sec4}). Let $\gotl$ be a Levi factor associated with $\gotq$. The algebra $\gotl$ is isomorphic to a certain product $\gl_{i_1} \times \ldots \times \gl_{i_s}$. We say that each $\gl_{i_k}$ is a \emph{block} of $\gotl$, so that $s$ is the number of blocks of $\gotl$. Two blocks $\gl_{i_k}$ and $\gl_{i_{k'}}$ are said to be \emph{isomorphic} if $i_k=i_{k'}$. Also set $\imax = \max_k i_k$ the size of the biggest block. For every $m \in \llbracket 1,n \rrbracket$,
\begin{itemize}
\item if there exists $i \in \{i_1, \ldots, i_s\}$ such that $m=\sum_k \min(i_k,i)$, then set $r_m$ as the number of blocks of $\gotl$ with size $i$, \emph{i.e.} $r_m=\Card(\{k \, | \, i_k=i\})$,
\item otherwise, set $r_m=1$.
\end{itemize}
\begin{the0E}[see theorem \ref{semiinv}]
For every $m \in \llbracket 1,n \rrbracket$, the invariant $F_m^{\bullet}$ is a product of $r_m$ non-constant homogeneous factors, denoted by $F_{m,1}, \ldots, F_{m,r_m}$. These factors are semi-invariants and also verify the following properties:
\begin{enumerate}[label=(\arabic*)]
\item for all $t \in \llbracket 1, r_m-1 \rrbracket$, $F_{m,t} \in \Sym(\gotn^-)$,
\item setting $\mathfrak{L}_m:={\lambda}_{m,1}, \ldots, {\lambda}_{m,r_m}$ for the weights of $F_{m,1}, \ldots, F_{m,r_m}$, one has $\lambda_{m,1} + \ldots + {\lambda}_{m,r_m}=0$ and $\mathfrak{L}_m$ has rank $r_m-1$,
\item the vector spaces $\vect(({\lambda}_{m,t})_t)$ for $m \in \llbracket 1,n \rrbracket$ are in direct sum,
\item the family $\mathbf{F}$ of all $F_{m,t}$ is algebraically independent.
\end{enumerate}
\end{the0E}
With this theorem, we get non-trivial factors of $F_m^\bullet$ if $r_m \geq 2$, with their weights and their degrees. The decomposition of $F_m^{\bullet}$ in the above theorem together with (1) is obtained by a combinatorial study of the invariants $F_m^{\bullet}$ (subsection \ref{sec3}). We show (2) and (3) by explicitly computing the weights of the semi-invariants $F_{m,t}$ (subsection \ref{ssec4.2}). This computation also induces (4) with the following general theorem:
\begin{the0E}[see theorem \ref{alin}]
Let $\gotk$ be a finite-dimensional Lie algebra. Let $d \geq 1$ and $(f_m)_{1 \leq m \leq d}$ be a family of algebraically independent elements of $\Y(\gotk)$. Assume that for every $m \in \llbracket 1,d \rrbracket$, the invariant $f_m$ decomposes into a product of semi-invariants $f_m = \prod_{t=1}^{r_m} f_{m,t}^{s_{m,t}}$ with $r_m \geq 1$, and $s_{m,t} \geq 1$ some (setwise) coprime numbers. Also assume that the semi-invariants $f_{m,t}$ verify the following:
\begin{itemize}
\item for a fixed $m$, the family of weights $\lambda_{m,t}$ of $f_{m,t}$ for $t \in \llbracket 1,r_m \rrbracket$ has rank $r_m-1$,
\item the sum of the vector spaces $\vect \left((\lambda_{m,t})_{1 \leq t \leq r_{m}}\right)$ for $m \in \llbracket 1, d \rrbracket$ is a direct sum.
\end{itemize}
Then the $(f_{m,t})_{1 \leq m \leq d, \, 1 \leq t \leq r_m}$ are algebraically independent. 
\end{the0E}
\paragraph{Step 2:} The next step is to show that $\mathbf{F}$ is a transcendence basis of $\Sy(\gotq)$. For that purpose, we show that the Gelfand-Kirillov dimension $\GK \C[\mathbf{F}]$ of the polynomial algebra $\C[\mathbf{F}]$ generated by $\mathbf{F}$, is equal to the Gelfand-Kirillov dimension of $\Sy(\gotq)$. Since $\GK \C[\mathbf{F}]$ is the transcendence degree of the field of fractions of $\C[\mathbf{F}]$, and since $\mathbf{F}$ is algebraically independent, then $\GK \C[\mathbf{F}]$ is also the cardinality of $\mathbf{F}$. Thus it is sufficient to show that $\GK \Sy(\gotq) = \Card(\mathbf{F})$. We already have $\GK \Sy(\gotq) = \GK \Y(\gotq_\Lambda) = \ind \gotq_\Lambda$.\par
Then, we compute $\ind \gotq_\Lambda$, which is given by the following theorem:
\begin{the0E}[see theorem \ref{dimind}]
Let $s$ be the number of blocks of a Levi factor $\gotl$ associated with $\gotq$ and $p$ the number of isomorphism classes of blocks of $\gotl$. That is, $\gotl$ verifies $\gotl \simeq \gl_{i_1} \times \ldots \times \gl_{i_s}$ with $\Card(\{i_1, \ldots, i_s\})=p$. Then
$$\dim \gotq_{\Lambda}=n^2-(s-p) \qquad \text{and} \qquad \ind \gotq_{\Lambda}=\Card(\mathbf{F})=\GK (\C[\mathbf{F}])=n+(s-p)$$
\end{the0E}
To show the above theorem, we use an equality proved by Ooms and van den Bergh \cite[Proposition 3.1]{oom10}
$$ \dim \gotq + \ind \gotq = \dim \gotq_\Lambda + \ind \gotq_\Lambda$$
One has $\dim \gotq= \dim \g = n^2$, and Panyushev and Yakimova have shown that $\ind \gotq=\rg \g =n$ \cite[Theorem 3.1]{py13}. Also $\ind \gotq_\Lambda = \GK \Sy(\gotq) \geq \GK \C[\mathbf{F}] = \Card \mathbf{F} = n+(s-p)$. To conclude, we then show that $\dim \gotq_\Lambda \geq n^2-(s-p)$:
\begin{itemize}
\item we start with the inclusion $\gotq' \oplus \gotz(\gotq) \subset \gotq_\Lambda$ (where $\gotz(\gotq)$ is the centre of $\gotq$) which is true for any Lie algebra, this inclusion gives us a vector subspace of $\gotq_\Lambda$ with dimension $n^2-s+1$,
\item since the invariants $F_m^\bullet$ belong to $\Sy(\gotq)\subset\Sym(\gotq_\Lambda)$, the $p-1$ remaining dimensions are given by computing some factors in $\h_\Lambda$ (where $\h_\Lambda:=\h\cap \gotq_\Lambda$ with $\h$ a Cartan subalgebra of $\gl_n$) of some terms for well-chosen $F_m^{\bullet}$ (section \ref{sec5}).
\end{itemize}
\paragraph{Step 3:} It suffices then to show that the $F_{m,t}$ generate $\Sy(\gotq)$ (section \ref{sec6}). We first show a theorem which can be applied outside the frame of parabolic contractions.
\begin{the0E}[see theorem \ref{puissant}]
Let $\gotk$ be a finite-dimensional Lie algebra. Assume that $\Y(\gotk)$ is a UFD and that
\begin{enumerate}
\item[(I)] there exists a $\C$-algebra homomorphism $\vartheta : \Sym(\gotk) \rightarrow \Y(\gotk)$ such that $\vartheta_{|\Y(\gotk)}$ is an isomorphism.
\end{enumerate}
Let $(f_m)_{1 \leq m \leq d}$ be a family of invariants which are irreducible in $\Y(\gotk)$. Given for all $m \in \llbracket 1,d \rrbracket$ a decomposition of $f_m$ in $\Sym(\gotk)$ as
\[
f_m = \prod_{t=1}^{r_{m}} (f_{m,t})^{\nu_{m,t}} \label{decompintro} \tag{$\clubsuit$}
\]
with $\nu_{m,t} \in \N^*$, assume that for each $m$, we are in one of the two following cases:
\begin{itemize}
\item the decomposition \eqref{decompintro} is trivial, \emph{i.e.} $r_m=1$ and $\nu_{m,1}=1$, that is $f_m = f_{m,1}$,
\item the décomposition \eqref{decompintro} is the decomposition of $f_m$ as a product of irreducible elements in $\Sym(\gotk)$.
\end{itemize}
Let $\mathbf{f}$ be the set of all $f_{m,t}$. If $\GK \Sy(\gotk)=\GK \C[\mathbf{f}]$, then the set of weights\footnote{which is a semi-group in the general case} of $\Sy(\gotk)$ is a group and $\Sy(\gotk)$ is generated by $\Y(\gotk)$ and the elements of $\mathbf{f}$. In particular, if the elements $f_m$ generate $\Y(\gotk)$, then $\Sy(\gotk)=\C[\mathbf{f}]$.
\end{the0E}
We will show that parabolic contractions $\gotq$ in the $(*)$ cases of the main theorem verify the hypotheses from the above theorem, so that in particular, the set of weights of $\Sy(\gotq)$ is a group. This is specific to parabolic contractions since such property is not true in general. For instance, for a proper parabolic subalgebra $\gotp$ of a semi-simple Lie algebra $\g$, the set of weights of $\Sy(\gotp)$ is not a group (see \cite[lemma 4.2.4]{FMJ05}).\par
Under the hypotheses of this theorem, we see that the decomposition of $f_m$ is trivial if and only if $r_m=1$. Moreover, if $m$ is such that $r_m \geq 2$, then we see that for all $t$, $f_{m,t}$ is a semi-invariant\footnote{it is a factor of an invariant so a semi-invariant, see proposition \ref{dixmier}} and has non-zero weight. However in our case, when we obtain a decomposition of $F_m^\bullet=\prod_{t=1}^{r_m} F_{m,t}$ with $r_m \geq 2$, we do not know a priori whether all $F_{m,t}$ are indeed irreducible. To show this, we use the more technical following result.
\begin{the0E}[see theorem \ref{inter}]
Let $\gotk$ be a finite-dimensional Lie algebra. Assume that $\Y(\gotk)$ is a UFD. Let $(f_m)_{1 \leq m \leq d}$ be a family of invariants which are irreducible in $\Y(\gotk)$, and for every $m$, assume that $f_m$ has a decomposition as
$$f_m = \prod_{t=1}^{r_{m}} (f_{m,t})^{\nu_{m,t}}$$
with $\nu_{m,t}\in \N^*$. Set $\mathbf{f}$ the set of all $f_{m,t}$ and $\mathbf{f}^\times$ the set of $f_{m,t}$ for which $r_m \geq 2$. Assume that $\GK \C[{\mathbf{f}}]=\GK \Sy(\gotk)$. Assume further that the following three hypotheses are satisfied:
\begin{enumerate}
\item[(I)] there exists a homomorphism of $\C$-algebras $\vartheta : \Sym(\gotk) \rightarrow \Y(\gotk)$ such that $\vartheta_{|\Y(\gotk)}$ is an isomorphism.
\item[(II)] each $f \in \mathbf{f}^\times$ is indivisible in the multiplicative semi-group $\Sy(\gotk)$,\footnote{\emph{i.e.} $f$ is not an $a$\up{th} power in $\Sy(\gotk)$ for any $a \geq 2$}
\item[(III)] for each $f \in \mathbf{f}^\times$, there exists a $\C$-algebra $\A_f$, which is a UFD, and a $\C$-algebra homomorphism $\vartheta_f : \Sy(\gotk) \rightarrow \A_f$ such that $\vartheta_f(f)$ is not invertible in $\A_f$ and such that $\vartheta_f(f)$ is coprime to every $\vartheta_f(g)$ for all $g \in \mathbf{f}^\times \setminus \{f\}$. \label{iiiiii}
\end{enumerate}
then the $f_{m,t}$ in $\mathbf{f}^\times$ are irreducible in $\Sym(\gotk)$.
%$\mathbf{M}_2=\{\mu \in \llbracket 1,\delta \rrbracket \, | \, r'_{\mu} \geq 2 \}$, $\mathcal{E}=\{(\mu,\tau) \, | \, \mu \in \mathbf{M}_2, 1 \leq \tau \leq r'_{\mu}\}$ et
\end{the0E}
To conclude that $\Sy(\gotq)=\C[\mathbf{F}]$ in the case of a parabolic contraction $\gotq$ of $\gl_n$ with the semi-invariants $F_{m,t}$ we constructed in step 1, we show that the three hypotheses (I), (II) and (III) are satisfied.\par
For the hypothesis (II), we show that the semi-invariants $F_{m,t}$ for $r_m \geq 2$ are indivisible by noting that they have degree $1$ in some $e_{v,w}$ where $(e_{v,w})_{1 \leq v,w \leq n}$ is the canonical basis of $\gl_n$.\par
Recall that the algebras $\Sym(\gotk)$ and $\C[\gotk^*]$ are isomorphic. In the same way, if $\A$ is a $\C$-algebra, the $\A$-algebras $\A \otimes_{\C} \Sym(\gotk)$ and $\A[\A \otimes_{\C} \gotk^*]$ are naturally isomorphic, which allows to write $f(q) \in \A$ where $f \in \A \otimes_{\C} \Sym(\gotk)$ and $q \in \A \otimes_{\C} \gotk^*$. For hypotheses (I) and (III), in fact we show some stronger hypotheses (I') and (III'), which are more adapted to our case:
\begin{enumerate}
\item[(I')] there exists $g_1, \ldots, g_{d'}$ which freely generate $\Y(\gotk)$ and $q \in \C[X_1, \ldots, X_{d'}] \otimes \gotk^*$ such that for every $m \in \llbracket 1,d' \rrbracket$, one has $g_m(q) = X_m$,
\item[(III')] for all $m \in \llbracket 1,d \rrbracket$ such that $r_m \geq 2$ and $t \in \llbracket 1, r_m \rrbracket$, there exists $q_{m,t} \in \C[X] \otimes \gotk^*$ such that $\deg_X f_{m,t}(q_{m,t}) \geq 1$ and $f_{\mu,\tau}(q_{m,t}) \in \C^\times$ for $(\mu,\tau) \neq (m,t)$.
\end{enumerate}
With (I'), the homomorphism $\vartheta$ in (I) is $f \in \Sym(\gotk) \mapsto f(q) \in \C[X_1, \ldots, X_{d'}] \simeq \Y(\gotk)$. The hypothesis (III') is a reformulation of (III) with $\mathcal{A}_f=\C[X]$. Note that hypothesis (I') implies the existence of a Kostant-Weierstrass section for $\Y(\gotq)$. Proofs of (I'), with the $g_1, \ldots, g_{d'}$ being the $F_1^\bullet, \ldots, F_n^\bullet$, and of (III') are combinatorial and use weighted digraphs\footnote{We denote by "digraph" a directed graph.} we call \textit{pathways}. A pathway is a weighted digraph with set of vertices $\llbracket 1,n \rrbracket$ and for which there is exactly one arrow going from any vertex $x$ to any vertex $y$. When we graphically represent a pathway, we only represent arrows with non-zero weights.\par
Let $\A=\C[X_1, \ldots, X_n]$. Fix $q \in \A \otimes \gl_n^*$. We define a pathway $\mathcal{G}(q)$ associated with $q$ with weights belonging to $\A$ which we call the \emph{graph of $q$}. The pathway $\mathcal{G}(q)$ is such that the arrow going from $x$ to $y$ in $\mathcal{G}(q)$ has  weight $e_{x,y}(q)$. The $F_m^\bullet$ are sums of terms of the form $\pm \prod_{l \in J} e_{l,\sigma(l)}$ with $J$ a subset of $\llbracket 1,n \rrbracket$ and $\sigma$ a permutation of the set $J$. Thus to compute $F_m^\bullet(q)$, it suffices to
\begin{itemize}
\item determine the \emph{dicyclic subgraphs} of $\mathcal{G}(q)$ with length $m$, \emph{i.e.} unions of directed cycles in $\mathcal{G}(q)$ with disjoint sets of vertices such that the arrows of all considered cycles have a non-zero weight, with the total subgraph having $m$ vertices,
\item show that these dicyclic subgraphs are associated with monomials of $F_m^\bullet$, \emph{i.e.} verify that they have the right degree in $\gotn^-$.
\end{itemize}
For (I'), we choose $q$ such that $\mathcal{G}(q)$ is represented graphically as follows:
\begin{center}
\begin{tikzpicture}
  \tikzset{LabelStyle/.style = {fill=white}}
  \tikzset{VertexStyle/.style = {%
  shape = circle, minimum size = 28pt,draw}}
  \SetGraphUnit{2.5}
  \Vertex[L=$v_1$]{1}
  \EA[L=$v_2$](1){2}
  \EA[L=$v_3$](2){3}
  \Loop[dist = 2cm, dir = WE, label = $X_1$](1.west)
  \Edge[style= {->}, label = 1](2)(1)
  \Edge[style= {->}, label = 1](3)(2)
  
  \tikzset{EdgeStyle/.style = {->,bend left=30}}
  \Edge[label=$X_2$](1)(2)
  \tikzset{EdgeStyle/.style = {->,bend left=34}}
  \Edge[label=$X_3$](1)(3)
  \SetVertexNormal[Shape = circle, LineColor=white, MinSize=28pt]
  \tikzset{EdgeStyle/.style = {->}}
  \EA[L=$\ldots$](3){4}
  \Edge[style= {->}, label = 1](4)(3)
  %\SetUpVertex[LineColor=black]
  \SetVertexNormal[Shape = circle, LineColor=black, MinSize=28pt]
  \EA[L=$v_{n-1}$](4){n-1}
  \Edge[style= {->}, label = 1](n-1)(4)
  \EA[L=$v_n$](n-1){n}
  \Edge[style= {->}, label = 1](n)(n-1)
  \tikzset{EdgeStyle/.style = {->,bend left=38}}
  \Edge[label=$X_{n-1}$](1)(n-1)
  \tikzset{EdgeStyle/.style = {->,bend left=42}}
  \Edge[label=$X_n$](1)(n)
\end{tikzpicture}
\end{center}
with $\{v_1, \ldots, v_n\}=\llbracket 1,n \rrbracket$. If we fix $m \in \llbracket 1,n \rrbracket$, this pathway has a unique dicyclic subgraph of length $m$, which is the following directed cycle:
\begin{center}
\begin{tikzpicture}
  \tikzset{LabelStyle/.style = {fill=white}}
  \tikzset{VertexStyle/.style = {%
  shape = circle, minimum size = 28pt,draw}}
  \SetGraphUnit{2.5}
  \Vertex[L=$v_1$]{1}
  \EA[L=$v_2$](1){2}
  \EA[L=$v_3$](2){3}
  %\Loop[dist = 2cm, dir = WE, label = $X_1$](1.west)
  \Edge[style= {->}, label = $1$](2)(1)
  \Edge[style= {->}, label = $1$](3)(2)
  
  \tikzset{EdgeStyle/.style = {->,bend left=30}}
  %\Edge[label=$-X_2$](1)(2)
  %\tikzset{EdgeStyle/.style = {->,bend left=34}}
  %\Edge[label=$X_3$](1)(3)
  \SetVertexNormal[Shape = circle, LineColor=white, MinSize=28pt]
  \tikzset{EdgeStyle/.style = {->}}
  \EA[L=$\ldots$](3){4}
  \Edge[style= {->}, label = $1$](4)(3)
  %\SetUpVertex[LineColor=black]
  \SetVertexNormal[Shape = circle, LineColor=black, MinSize=28pt]
  \EA[L=$v_{m-1}$](4){m-1}
  \Edge[style= {->}, label = $1$](m-1)(4)
  \EA[L=$v_m$](m-1){m}
  \Edge[style= {->}, label = $1$](m)(m-1)
%  \tikzset{EdgeStyle/.style = {->,bend left=38}}
%  \Edge[label=$(-1)^m X_{m-1}$](1)(m-1)
  \tikzset{EdgeStyle/.style = {->,bend left=42}}
  \Edge[label=$X_m$](1)(m)
\end{tikzpicture}
\end{center}
This cycle corresponds to a unique monomial $\mathcal{S}_m$ of $F_m$. With a good choice of the sequence $(v_1, \ldots ,v_n)$, \emph{i.e.} of the order of the $v_i$, then the monomial $\mathcal{S}_m$ is a monomial of $F_m^\bullet$ for every $m$, so that $F_m^\bullet(q)=\mathcal{S}_m(q)=X_m$.\par
For (III'), fix $m,t$. We take $q$ as constructed above and set $X_m=X$ and $X_{\mu}=1$ for $\mu \neq m$, so that $F_m^\bullet(q)=X$ and $F_{\mu,\tau}(q) \in \C^\times$ for $\mu \neq m$. With a good choice of the sequence $(v_1, \ldots, v_n)$ among those that satisfy point (I'), we get $F_{m,t}(q) \in \C^\times X$ and $F_{\mu,\tau}(q) \in \C^\times$ for $(\mu,\tau) \neq (m,t)$. With hypotheses (I'), (II) and (III') satisfied, we eventually have $\Sy(\gotq)=\C[\mathbf{F}]$ so $\mathbf{F}$ freely generate $\Sy(\gotq)$ in type $\gl_n$.
\subsection*{Parabolic contractions in type $A$ and $C$}
To show the polynomiality in cases $(*)$ of the main theorem, we study the polynomiality of $\Sy(\widetilde{\gotq})$ when $\widetilde{\gotq}$ is a parabolic contraction of $\g=\spl_n$ or $\g=\syp_n$ (chapter \ref{chap4}). Note that $\widetilde{\gotq}$ is a Lie subalgebra of a parabolic contraction $\gotq$ of $\gl_n$. To study the polynomiality, an idea is to project on $\widetilde{\gotq}$ the results we got on $\gotq$. Let $\pr : \Sym(\gotq) \rightarrow \Sym(\widetilde{\gotq})$ be a projection such that its restriction $\pr : \gotq \rightarrow \widetilde{\gotq}$ is a $\widetilde{\gotq}$-module homomorphism. In the cases we consider here, one has $\pr(F_m^\bullet)=\pr(F_m)^\bullet  = \prod_{t=1}^{r_m} \pr(F_{m,t})$.\par
In type $A$ (section \ref{sec7}), one has $\g=\spl_n$ the Lie algebra of the special linear group, \emph{i.e.} the set of all matrices of size $n$ with trace zero. We consider the projection $\pr : \Sym(\gl_n) \rightarrow \Sym(\spl_n)$ coming from the decomposition $\gl_n = \spl_n \oplus \C \id$. In this case, we get $\pr(F_1^\bullet)=\pr(\id)=0$, and the $\pr(F_m^\bullet)$ for $m \in \llbracket 2,n \rrbracket$ freely generate $\Y(\widetilde{\gotq})$ according to \cite{py13}. The weight of $\pr(F_{m,t})$ is the weight of $F_{m,t}$ restricted to $\spl_n$, and we show that the $\pr(F_{m,t})$ are again algebraically independent making use of the theorem \ref{alin}. We then easily see that the set of $\pr(F_{m,t})$ provides a transcendence basis of $\Sy(\widetilde{\gotq})$. Finally, if $x \in \Sy(\widetilde{\gotq})$, then $x \in \Sy(\gotq)$ since $\C \id = \gotz(\gotq)$, thus $x=\pr(x)$ is a polynomial in the $F_{m,t}$, so it is also a polynomial in the $\pr(F_{m,t})$.\par
In type $C$, the integer $n$ is even and we take $\g=\syp_n$ the Lie algebra of the symplectic group\footnote{We take the definition of $\syp_n$ given by \cite{bou06}.}. Consider the projection $\pr : \Sym(\gl_n) \rightarrow \Sym(\syp_n)$ coming from the decomposition $\gl_n = \syp_n \oplus \left(\syp_n^\perp \oplus \C \id\right)$ where $\syp_n^\perp \subset \spl_n$ is the orthogonal of $\syp_n$ with respect to the Killing form of $\spl_n$. In this case, we get $\pr(F_m^\bullet)=0$ for every odd integer $m$, but the $\pr(F_m^\bullet)$ for even $m \in \llbracket 1,n \rrbracket$ freely generate $\Y(\widetilde{\gotq})$ according to \cite{py13}.\par
In the $(*)$ case of the main theorem, \emph{i.e.} when the Levi factor associated with the contraction $\widetilde{\gotq}$ has no factor of type $C$, then the Levi factors $\gotl$ associated with the parabolic contraction $\gotq$ of $\gl_n$ can be written as $\gotl \simeq \gl_{i_1} \times \ldots \times \gl_{i_s}$ with $s$ even (section \ref{sec9}). In this case, the proof of the polynomiality of $\Sy(\widetilde{\gotq})$ mainly follows the pattern of the type $\gl_n$. In $\gotq$, the decomposition $F_m^\bullet = \prod_{t=1}^{r_m} F_{m,t}$ is non-trivial, \emph{i.e.} such that $r_m \geq 2$, only if $m$ is \emph{even}. We then get decompositions $\pr(F_m^\bullet) = \pm \prod_{t=1}^{b_m} f_{m,t}^{s_{m,t}}$, where some powers $s_{m,t} \geq 2$ appear by gathering the factors that become colinear after projection. Again, we can use theorem \ref{alin} to show that the family $\mathbf{f}=(f_{m,t})_{m,t}$ of pairwise non colinear semi-invariants is algebraically independent. This family has cardinality $n/2+s/2 = \ind \widetilde{\gotq} +s/2$. Thus $\dim \widetilde{\gotq}_\Lambda \geq \dim \widetilde{\gotq}' = \dim \widetilde{\gotq} - s/2$ and $\ind \widetilde{\gotq}_\Lambda \geq \ind \widetilde{\gotq} +s/2$. With Ooms and Van den Bergh's equality \eqref{ovdb}, these inequalities become equalities, so that the family $\mathbf{f}$ becomes a trascendence basis of $\Sy(\widetilde{\gotq})$. To show that $\Sy(\widetilde{\gotq})=\C[\mathbf{f}]$, we once again show that the hypotheses (I), (II), and (III) from theorem \ref{inter} are verified, and again, we use the hypotheses (I') and (III') previously mentioned to show (I) and (III).\par
For (II), we show that most of $f_{m,t}$'s weights, say $\lambda_{m,t}$, are indivisible in the set of weights $\Lambda$ of $\Sy(\widetilde{\gotq})$, \emph{i.e.} $\lambda_{m,t} \notin a \Lambda$ for $a \geq 2$. In fact, since $\Lambda \subset \sum_k \rat \varpi_k$ with $\varpi_k$ being the fundamental weights of $\syp_n$, it is enough to prove that $\lambda_{m,t}$ is indivisible in $\sum_k \rat \varpi_k$. However, there remain some weights $\lambda_{m,t}$ which do not verify this last property. In these particular cases, we verify the semi-invariants indivisibility directly by proving that they have degree $1$ in certain elements of a basis of $\syp_n$.\par
For (I'), we reuse the notion of pathway, which we modify slightly in type $C$. Let $\A=\C[X_1, \ldots, X_{n/2}]$. Fix $q \in \A \otimes \syp_n^*$. Define a pathway $\mathcal{G}(q)$ with weights in $\A$ called the \emph{graph of $q$}, such that the arrow from $x$ to $y$ has weight $\pr(e_{x,y})(q)$. This definition implies that all graphs in type $C$ have a symmetry, in that an arrow $x \rightarrow y$ with non-zero weight in the graph of $q$ implies an arrow $n+1-y \rightarrow n+1-x$ with non-zero weight\footnote{This fact comes from the construction of $\syp_n$ given in \cite{bou06}, where the Cartan subalgebra of $\syp_n$ is diagonal.}. As in type $\gl_n$, we then construct an element $q$ such that its graph $\mathcal{G}(q)$ has a unique dicyclic subgraph of length $m$, which is in fact a directed cycle. Said graph has the following representation:
\begin{center}
\begin{tikzpicture}
  \tikzset{LabelStyle/.style = {fill=white}}
  \tikzset{VertexStyle/.style = {%
  shape = circle, minimum size = 33pt,draw}}
  \SetGraphUnit{2.5}
  \Vertex[L=$v_n$]{1}
  \EA[L=$v_{n-1}$](1){2}
  \EA[L=$v_{n-2}$](2){3}
  %\Loop[dist = 2cm, dir = WE, label = $X_1$](1.west)
  \Edge[style= {->}](1)(2)
  \Edge[style= {->}](2)(3)
  \SetVertexNormal[Shape = circle, LineColor=white, MinSize=33pt]
  \tikzset{EdgeStyle/.style = {->}}
  \EA[L=$\ldots$](3){4}
  \NO[L=$\ldots$](4){4'}
  \Edge(3)(4)
  %\SetUpVertex[LineColor=black]
  \SetVertexNormal[Shape = circle, LineColor=black, MinSize=33pt]
  \EA[L=$v_{n/2+2}$](4){m-1}
  \Edge[style= {->}](4)(m-1)
  \EA[L=$v_{n/2+1}$](m-1){m}
  \Edge[style= {->}](m-1)(m)
  \NO[L=$v_1$](1){1'}
  \NO[L=$v_2$](2){2'}
  \NO[L=$v_3$](3){3'}
  \NO[L=$v_{n/2-1}$](m-1){m-1'}
  \NO[L=$v_{n/2}$](m){m'}
  \Edge[style= {->}](2')(1')
  \Edge[style= {->}](3')(2')
  \Edge[style= {->}](4')(3')
  \Edge[style= {->}](m-1')(4')
  \Edge[style= {->}](m')(m-1')
  \Edge[style= {->,bend right=30}, label =$1$](1')(1)
  \Edge[style= {->,bend right=30}, label =$X_1$](1)(1')
  \Edge[style= {->,bend right=30}, label =$X_2$](2)(2')
  \Edge[style= {->,bend right=30}, label =$X_3$](3)(3')
  \Edge[style= {->,bend right=30}, label =$X_{n/2-1}$](m-1)(m-1')
  \Edge[style= {->,bend right=30}, label =$X_{n/2}$](m)(m')
\end{tikzpicture}
\end{center}
with $\{v_1, \ldots, v_{n}\}=\llbracket 1,n \rrbracket$.\par
For (III'), there are not enough sequences $(v_1, \ldots, v_n)$ satisfying (I'), so we cannot proceed as for a parabolic contraction of $\gl_n$ by simply adapting the $q$ constructed in (I'). Instead we first choose $q' \in \widetilde{\gotq}^*$, such that $\mathcal{G}(q')$ is a dicyclic graph, \emph{i.e.} a disjoint union of cycles $\mathcal{C}_1, \ldots, \mathcal{C}_{l}$. Set $\mathcal{S}_1, \ldots, \mathcal{S}_{l}$ the associated monomials, and from this choice of cycles we get $ \mathcal{S}_j(q')=1$ for all $j$ and $\pr(F_\mu^\bullet)(q')=\prod_{j=1}^{i} \mathcal{S}_j(q')=1$ for every (even) $\mu$ verifying $r_\mu \geq 2$, where $i$ is such that $\mu=\sum_{k} \min(i_k,i)$. Fix $m,t$. To verify (III'), the idea is to modify $q'$ by multiplying a term of $q'$ by $X$. If we note $q$ the modified $q'$, we then have $\mathcal{S}_{i}(q)=X^p$ with $p \in \{1,2 \}$ and $\mathcal{S}_{j}(q)=1$ for $j \neq i$, and so $\pr(F_m^\bullet)(q)=X^p$. Yet we also get $\pr(F_\mu^\bullet)(q)=X^p$ for $\mu > m$ such that $r_\mu \geq 2$, which does not coincide with (III'), unless $i=\imax:=\max(i_k)$ \emph{i.e.} the maximal size of blocks for the Levi factor $\gotl$. To avoid this when $i \neq \imax$, we modify $q'$ further and the structure of the cycles $\mathcal{C}_{i}$ and $\mathcal{C}_{i+1}$ in the dicyclic graph $\mathcal{G}(q')$, to obtain $q$ which satisfies $\deg_X \pr(F_m^\bullet)(q) \geq 1$ et $\deg_X \pr(F_\mu^\bullet)(q)=0$ for $\mu \neq m$ such that $r_\mu \geq 2$. By carefully choosing the edges on which we make this modification, the newly constructed $q$ satisfies (III').\par
Now we prove part $(**)$ of the main theorem. Consider the case where the Levi factor associated with the contraction $\widetilde{\gotq}$ has a type $C$ factor, \emph{i.e.} the case where the Levi factors $\gotl$ associated with the parabolic contraction $\gotq$ of $\gl_n$ can be written $\gotl=\gl_{i_1} \times \ldots \times \gl_{i_s}$ with $s$ odd. In this case, non-trivial decompositions $F_m^\bullet = \prod_{t=1}^{r_m} F_{m,t}$ in $\gl_n$ can occur for some odd $m$. However, for an odd $m$, one gets $\pr(F_m)=0$ thus $\pr(F_m^\bullet)=0$ and there are cases with $m$ odd where $\pr(F_{m,t}) \neq 0$ for certain $t$. In these cases, if we only consider the set of $\pr(F_{m,t})$ for even $m$, the resulting family $\mathbf{f}$ is not a transcendence basis of $\Sy(\widetilde{\gotq})$: some elements are "missing". Take the following example (section \ref{sec9}) and consider $\gotq$ the parabolic contraction of $\gl_8$ defined by the following diagram\footnote{the full line separates elements $e_{x,y}$ in $\gotp$ from those in $\gotn^-$, and the $e_{x,y}$ between the full line and the dashed line form a basis of a Levi factor $\gotl$}:
\begin{center}
\begin{tikzpicture}[scale=0.6]
\foreach \k in {1,2,...,7}
	{\draw[color=gray!20]  (0,\k)--(8,\k);
	\draw[color=gray!20] (\k,0)--(\k,8);}
\draw (0,0)--(8,0);
\draw (0,0)--(0,8);
\draw (0,8)--(8,8);
\draw (8,0)--(8,8);
\draw[ultra thick] (0,7)--(1,7)--(1,5)--(3,5)--(3,3)--(5,3)--(5,1)--(7,1)--(7,0);
\draw (2,2) node{\LARGE $\gotn^-$};
\draw (6,6) node{\LARGE $\gotp$};
\draw[dashed] (1,8)--(1,7)--(3,7)--(3,5)--(5,5)--(5,3)--(7,3)--(7,1)--(8,1);
%\draw (0.5,7.5) node{\Large $\g_{I_1}$};
%\draw (2,6) node{\Large $\g_{I_2}$};
%\draw (4,4) node{\Large $\g_{I_3}$};
%\draw (6,2) node{\Large $\g_{I_4}$};
%\draw (7.5,0.5) node{\Large $\g_{I_5}$};
\end{tikzpicture}
\end{center}
Set $\widetilde{\gotq}=\pr(\gotq)$ the corresponding parabolic contraction of $\syp_8$. In this case, we get a decomposition $F_8^\bullet=F_{8,1}F_{8,2}F_{8,3}$ projecting to $\pr(F_8^\bullet)=\pr(F_{8,1})^2\pr(F_{8,2})$ as well as a decomposition $F_5^\bullet=F_{5,1} F_{5,2}$. Projecting the latter decomposition gives $\pr(F_5^\bullet)=\pr(F_{5,1})=0$, however $\pr(F_{5,2}) = \pr(e_{1,8}) \neq 0$. Yet $\pr(F_{5,2})$ is not proportional to any $\pr(F_{m,t})$ with $m$ even. To apply the same pattern of proof as in the previous cases, we would like to find $\pr(F_{5,2})$ as a factor of an invariant, \emph{i.e.} here a polynomial in $\pr(F_2^\bullet), \pr(F_4^\bullet), \pr(F_6^\bullet), \pr(F_8^\bullet)$. We find that $\pr(F_{5,2})$ divides $\pr(F_8^\bullet)-\frac{1}{4} \pr(F_4^\bullet)^2$, and we set $\pr(F_8^\bullet)-\frac{1}{4} \pr(F_4^\bullet)^2=\widetilde{F} \pr(F_{5,2})$. We then show that the family $\pr(F_2^\bullet), \pr(F_6^\bullet), \pr(F_{8,1}), \pr(F_{8,2}), \widetilde{F}, \pr(F_{5,2})$ is algebraically independent by computing their weights, and using the same kind of proof as in the previous cases. We then show that the family $\pr(F_2^\bullet), \pr(F_4^\bullet), \pr(F_6^\bullet), \pr(F_{8,1}), \pr(F_{8,2}), \widetilde{F}, \pr(F_{5,2})$ generates the algebra $\Sy(\widetilde{\gotq})$ by verifying hypotheses (I'), (II) et (III)\footnote{instead of (III') which cannot be simply showed here}. Since $\widetilde{F} \pr(F_{5,2})+\frac{1}{4}\pr(F_4^\bullet)^2=\pr(F_{8,1})^2\pr(F_{8,2})$, the algebra
$$\Sy(\widetilde{\gotq})=\C[\pr(F_2^\bullet), \pr(F_4^\bullet), \pr(F_6^\bullet), \pr(F_{8,1}), \pr(F_{8,2}), \widetilde{F}, \pr(F_{5,2})]$$
is isomorphic to the algebra $\C[X_1, X_2, X_3, X_4, X_5, X_6, X_7]/(X_6X_7+\frac{1}{4}X_2^2-X_4^2X_5)$ which is not polynomial.\par
The phenomenon yielding this counterexample is not isolated. We present in chapter \ref{sec11} some conjectures in this direction.

\subsection*{Some tracks for parabolic contractions in other types}

We discuss in section \ref{sec12} the limits of our approach and some possible tracks to study some cases in other types. The pattern to prove the polynomiality (or not) of the algebra of semi-invariants associated with a parabolic contraction $\gotq$ which we presented here is at least based on the knowledge of a minimal generating family of $\Y(\gotq)$. Such family  is given in types $A$ and $C$ by Panyushev and Yakimova \cite{py13}. In types $B$ and $D$ as well as in exceptional types, this type of study is much harder because we do not know such family. Moreover, some properties of parabolic contractions which are true in types $A$ and $C$ may not be true anymore in types $B$ or $D$. For instance, one may have $\pr(F_m)^\bullet \neq \pr(F_m^\bullet)$ \emph{i.e.} $\pr(F_m^\bullet)=0$ while $\pr(F_m)^\bullet \neq 0$, which makes the combinatorial study of the semi-invariants considerably harder.\par
Conversely, in some cases where the approach of Panyushev and Yakimova failed, we can hope to find a trascendence basis of the algebra of invariants by studying known semi-invariants. For instance, we can use the following proposition which gives the semi-invariants with degree $1$ for a parabolic contraction.
\begin{propn0E}[see proposition \ref{degre1}]
Let $\gotq=\gotp \ltimes \gotn^-$ be a parabolic contraction of a simple Lie algebra $\g$ associated with a certain Cartan subalgebra, a certain basis $\pi$ of the associated root system $R$ and a certain subset $\pi' \varsubsetneq \pi$. Set $\theta$ the highest root of $R$. Let $e \in \gotq$ be non zero. Then $e$ is a semi-invariant (of degree $1$) of $\Sym(\gotq)$ if and only if $e \in \g_{\beta}$, with $\beta \in (-\pi \setminus -\pi') \cup \{\theta\}$ such that the vertex associated with $-\beta$ in the extended Dynkin diagram of $R$ is not linked to any vertex associated with an element of $\pi'$.
\end{propn0E}
In the example \cite[remark 4.6]{py13} in type $D_6$ with $\pi'=\{\alpha_3,\alpha_4,\alpha_5\}$, the family $\pr(F_2)^\bullet$, $\ldots$, $\pr(F_{10})^\bullet, f^\bullet$ (with $f$ a square root of $\pr(F_{12})$) is not algebraically independent thus is "too small". By studying the semi-invariants of degree $1$, we remark that $\pr(F_4)^\bullet$ is in the ideal of $\Sym(\gotq)$ generated by some semi-invariants $x_1, x_2$ of degree $1$. More explicitly, we get $\pr(F_4)^\bullet=x_1y_1 + x_2y_2 + x_1x_2y_{1,2}$, where $y_1, y_2, y_{1,2}$ are semi-invariants and $x_1y_1$,  $x_2y_2$,  $x_1x_2y_{1,2}$ are invariants with the same degree and degree in $\gotn^-$ as $\pr(F_4)^\bullet$. With these three invariants, we can hope to get a "wider" family, maybe even an algebraically independent one.\par
The work of this thesis can be linked to other recently published articles. Recall that we explicitly describe the generators of $\Sy(\gotq)$ when $\Sy(\gotq)$ is polynomial (with their weights and their degrees) in type $A$ and in type $C$ with a Levi factor of type $A$. This also seems to give an explicit formula for the generators of the polynomial algebra $\Sym(\gotn^-)^{\gotp'}$. This question has been mentioned by Fittouhi and Joseph (see \cite[\S 2.2.3]{FJ20}). Also, according to Joseph and Fittouhi, the number of these generators is equal to the number of hypersurface orbital varieties in the nilradical $\gotn$ of the parabolic subalgebra $\gotp$ (see \cite[\S 2.3.4]{FJ20}). These hypersurface orbital varieties have been described in type $A$ by Joseph and Melnikov \cite{JM03} and by Perelman \cite{Per03} for the other classical types.
% sont bien algébriquement indépendants et  on montre que les projections des $F_m^\bullet$ (hormis $F_1^\bullet$ dont la projection est nulle) engendrent l'algèbre des semi-invariants (section \ref{sec7}). En type $C$, les projections des $F_m^\bullet$ sont non nulles si et seulement si $m$ est pair, ce qui ne pose pas de problème lorsque le Levi associé à la sous-algèbre de Lie parabolique sous-jacente est de type $A$ (section \ref{sec9}), mais en pose lorsque ce n'est pas le cas, et il arrive que l'algèbre des semi-invariants ne soit pas polynomiale (section \ref{sec10}). Les calculs heuristiques qyant permis d'arriver à ces résultats amènent à des conjectures pour une condition nécessaire et suffisante à la polynomialité en type $C$. On discute les limites du schéma de preuve en évoquant un exemple exhibé par Panyushev et Yakimova présentant les limites de leur méthode pour montrer la polynomialité de l'algèbre des invariants (section \ref{sec11}).

\selectlanguage{french}
\chapter{Notations, premières définitions}
\section{Notations, définitions générales}
\label{sec2}
On rappelle que $\C$ est un corps de caractéristique $0$ algébriquement clos. Tous les espaces vectoriels et algèbres considérés ont pour corps de base $\C$. Soit $\N=\{0,1,2,\ldots\}$ l'ensemble des entiers naturels, $\rat$ l'ensemble des entiers relatifs, $\Q$ l'ensemble des nombres rationnels. On note $\N^*=\N \setminus \{0\}$. Pour tous $a<b$ dans $\rat$, on note $\llbracket a,b \rrbracket$ \index{$\llbracket a,b \rrbracket$} l'ensemble des entiers $k$ tels que $a \leq k \leq b$. Si $a > b$, on pose par convention $\llbracket a,b \rrbracket=\emptyset$. Pour un ensemble $J$ fini, on notera son cardinal $\vert J \vert$ ou $\Card(J)$. On note $\mathfrak{S}(J)$ l'ensemble des permutations de $J$, et pour tout $k \in \N^*$, $\mathfrak{S}_k:=\mathfrak{S}_{\llbracket 1,k \rrbracket}$. On note $\delta$ le symbole de Kronecker, c'est-à-dire pour tous $u,v$ dans un ensemble $E$, $\delta_{u,v}$ vaut $1$ si $u=v$ et $0$ sinon. Si $\mathcal{R}$ est un anneau, on note $\mathcal{R}^\times$ l'ensemble des éléments inversibles de $\mathcal{R}$.\par
Pour $V$ un $\C$-espace vectoriel de dimension finie on note $\dim V$ la dimension de $V$, et pour $a,b \in V$, on note $a \propto b$ \index{$a \propto b$} s'il existe $c \in \C^\times$ tel que $a=cb$. On note $V^*$ le dual de $V$ et $\Sym(V)$ l'algèbre symétrique de $V$, qui est une algèbre de polynômes de $\dim V$ générateurs. On munit $\Sym(V)$ \index{$\Sym(V)$} de sa structure d'algèbre graduée. Si $V_1$ et $V_2$ sont des sous-espaces vectoriels supplémentaires de $V$, on note $V_1 \oplus V_2$ la somme directe de $V_1$ et $V_2$ ainsi que $\pr_{V_1,V_2}$ la projection de $V$ sur $V_1$ parallèlement à $V_2$. Si $\mathcal{X}$ est un semi-groupe abélien, on dit que $x \in \mathcal{X}$ est un élément \textbf{indivisible} (ou \textbf{non divisible}) de $\mathcal{X}$ si pour tout entier $k \geq 2$, on a $x \notin k\mathcal{X}:=\{kx \, | \, x \in \mathcal{X}\}$. Si $Y$ est une variété algébrique, on note $\C[Y]$ \index{$\C[Y]$} l'algèbre des fonctions régulières sur $Y$.\par
Soit $n \in \N^*$. Soit $\g:=\gl_n$ \index{$\g=\gl_n$} l'ensemble des matrices de taille $n \times n$ à coefficients dans $\C$ muni d'une structure d'algèbre de Lie avec le crochet $[x,y]=xy-yx$ (où ici le produit $xy$ est le produit matriciel) pour tous $x,y \in \gl_n$. On note également $\widehat{\g}$\index{$\widehat{\g}$} le sous-espace vectoriel de $\g$ formé des matrices de diagonale nulle. On fixe $I:=\llbracket 1,n \rrbracket$\index{$I$}, $(e_{p,q})_{p,q \in I}$ \index{$e_{p,q}$, $e^*_{p,q}$} la base canonique de $\gl_n$ et $(e^*_{p,q})_{p,q \in I}$ la base duale de $\gl_n^*$ associée. Pour tous $J,J' \subset I$, on note $\g_{J,J'}$ \index{$\g_{J,J'}$, $\g_{J}$, $V_{J,J'}$, $V_{J}$} le sous-espace vectoriel de $\g=\gl_n$ engendré par les $e_{p,q}$ avec $p \in J, q \in J'$. On note aussi $\g_J:=\g_{J,J}$. Si $V$ est un sous-espace vectoriel de $\g$, on notera $V_{J,J'}:=V \cap \g_{J,J'}$ (et de même que précédemment, $V_{J}:=V_{J,J}$). Pour $J,J' \subset I$ de même cardinal, on écrit $\Delta_{J,J'} \in \Sym(\gl_n)$ \index{$\Delta_{J,J'}$, $\Delta_{J}$} le mineur associé aux coefficients de ligne $J$ et de colonne $J'$, autrement dit :
$$\Delta_{J,J'}=\sum_{\sigma \in \Bij(J,J')} \varepsilon(\sigma)\prod_{j \in J} e_{j,\sigma(j)}$$
(le produit est cette fois-ci le produit tensoriel dans $\Sym(\gl_n)$) où $\Bij(J,J')$ est l'ensemble des bijections de $J$ dans $J'$ et pour tout $\sigma \in \Bij(J,J')$, on note $\varepsilon(\sigma)=\varepsilon(\sigma)=(-1)^{\inv(\sigma)}$ où $\inv(\sigma)$ est le nombre d'inversions de $\sigma$. On note également $\Delta_{J} :=\Delta_{J,J}$. Un \textbf{mineur principal} est un mineur de la forme $\Delta_{J}$, pour $J \subset I$. On note alors
\begin{equation}
F_j=\sum_{|J|=j} \Delta_J \label{Fm}
\end{equation}\index{$F_m$}
Soit $\B$ une $\C$-algèbre associative, commutative intègre et unitaire, on appelle monôme en $a_1, \ldots, a_d \in \B \setminus \{0\}$ une expression de la forme $k \, a_1^{r_1} \ldots a_d^{r_d}$, avec $k \in \C^\times$ et $r_i \in \N$ pour tout $1 \leq i \leq d$. Si $P$ est un polynôme en les $a_1, \ldots, a_d$, qui s'écrit $P = \sum_{r_1, \ldots, r_d} k_{r_1, \ldots, r_d} \, a_1^{r_1} \ldots a_d^{r_d}$, on appellera \textbf{monôme de $P$} (en les $a_1, \ldots, a_d$) un $k_{(r_i)_i} \, a_1^{r_1} \ldots a_d^{r_d}$ avec $k_{r_1, \ldots, r_d} \neq 0$. On note également $\C[a_1, \ldots, a_d]$ la $\C$-algèbre engendrée par $a_1, \ldots, a_d$.\par
Soit $\A$ une $\C$-algèbre associative commutative unitaire intègre de type fini. La \textbf{dimension de Gelfand-Kirillov} de $\A$, notée $\GK \A$ \index{$\GK \A$} coïncide avec le degré de transcendance sur $\C$ du corps des fractions de $\A$. C'est en particulier le cardinal maximal d'une famille algébriquement indépendante de $\A$. Si $\mathbf{a}=(a_1, \ldots, a_d)$ est une famille finie algébriquement indépendante d'éléments de $\A$, la famille $\mathbf{a}$ est \textbf{une base de transcendance} de $\A$ si $d=\GK \A$. De manière équivalente, la famille $\mathbf{a}$ est une base de transcendance de $\A$ si $\mathbf{a}$ est algébriquement indépendante et pour tout $x \in \A$, la famille $(a_1, \ldots, a_d,x)$ est algébriquement liée. La famille $\mathbf{a}=(a_1, \ldots, a_d)$ \textbf{engendre librement} $\A$ si $\mathbf{a}$ est algébriquement indépendante et $\A=\C[a_1, \ldots, a_d]$. En particulier, toute famille engendrant librement $\A$ est une base de transcendance de $\A$. On dit qu'une algèbre est \textbf{polynomiale} lorsqu'elle est isomorphe à une algèbre de polynômes, c'est-à-dire lorsqu'elle admet un sous-ensemble fini qui l'engendre librement.

\section{Rappels sur les algèbres de Lie, contractions paraboliques, algèbres d'invariants et de semi-invariants}
Toutes les algèbres de Lie considérées par la suite seront de dimension finie. Pour toute algèbre de Lie $\gotk$, on note $\gotk'=[\gotk,\gotk]$ l'algèbre de Lie dérivée de $\gotk$ et $\gotz(\gotk)$ le centre de $\gotk$. Dans toute cette section, on considère $\gotk$ une algèbre de Lie (de dimension finie). Comme dans \cite{Dix}, si $\gotr$ et $\gots$ sont deux idéaux de $\gotk$ d'intersection nulle, on note $\gotr \oplus \gots$ la somme des deux espaces vectoriels, et $\gotr \times \gots$ l'algèbre de Lie produit.
 
\subsection{Rappels sur les algèbres de Lie}
On reprend quelques définitions de base de \cite{Dix}.
%\subsubsection{Algèbres de Lie nilpotentes}
%\begin{defi}
%On définit $\mathcal{C}^i(\gotk)$ par récurrence pour tout $i \in \N$ par $\mathcal{C}^0(\gotk)=\gotk$ et $\mathcal{C}^{i+1}(\gotk)=[\gotk,\mathcal{C}^i(\gotk)]$. On dit que $\gotk$ est nilpotente s'il existe $i$ tel que $\mathcal{C}^i(\gotk)=0$
%\end{defi}
%\begin{prop}
%L'ensemble des idéaux nilpotents de $\gotk$ admet un maximum pour l'inclusion, qui est le \textbf{plus grand idéal nilpotent} de $\gotk$.
%\end{prop}
%\subsubsection{Indice d'une algèbre de Lie}
\begin{defi}
Pour tout $f \in \gotk^*$, on définit la forme bilinéaire alternée $\phi_f$ par $\phi_f(x,y)=f([x,y])$ pour tous $x,y \in \gotk$. Soit $\gotk^{(f)} \subset \gotk$ le noyau de la forme bilinéaire $\phi_f$. On définit l'\textbf{indice} de $\gotk$ par $\ind \gotk=\min \{ \dim \gotk^{(f)} \, | \, f \in \gotk^* \}$. \index{$\ind \gotk$}
\end{defi}
\begin{defi}
Soient $\gota$ et $\gotb$ deux algèbres de Lie, et $\phi : \gota \rightarrow \Der \gotb$ un morphisme d'algèbres de Lie entre $\gota$ et l'algèbre de Lie des dérivations $\Der \gotb$ de $\gotb$. On définit le produit semi-direct $\gota \ltimes_\phi \gotb$ comme étant l'espace vectoriel $\gota \oplus \gotb$ muni du crochet de Lie défini par :
$$\forall a, a' \in \gota, \forall b,b' \in \gotb, [a+b,a'+b']=[a,a']+[b,b']+\phi(a)(b')-\phi(a')(b)$$
En particulier, $\gotb$ est un idéal de $\gota \ltimes_\phi \gotb$. On note souvent $\gota \ltimes \gotb$ pour $\gota \ltimes_\phi \gotb$ lorsqu'il n'y a pas d'ambiguïté sur le morphisme $\phi$.
\end{defi}
\subsection{Contractions paraboliques}
\begin{defi}
Soit $\g= \g' \oplus \gotz(\g)$ une algèbre de Lie réductive de crochet de Lie $[\, , ]_{\g}$, et $\h:=\h^{\text{ss}} \oplus \gotz(\g)$ une sous-algèbre de Cartan de $\g$. Soit $R$ le système de racines associé à $(\g',\h^{\text{ss}})$, et $\pi$ une base de $R$. Pour tout $\beta \in R$, on note $\g_{\beta}=\left\{x \in \g \; : \; \forall h \in \h, [h,x]_{\g}=\beta(h)x\right\}$\index{$\g_{\beta}$}. On note $R=R^+ \sqcup R^-$ la décomposition de $R$ en racines positives et négatives induite par $\pi$. Soit $\pi' \subset \pi$\index{$\pi$, $\pi'$}. On note $R_{\pi'}$ le sous-système de racines de $R$ engendré par $\pi'$, et $R_{\pi'} = R_{\pi'}^+ \sqcup R_{\pi'}^-$ la décomposition en racines positives et négatives induite par $\pi'$. Soit $\gotp=\bigoplus_{\beta \in R^+ \sqcup R^-_{\pi'}} \g_{\beta} \oplus \h$. On dit que $\gotp$ est \textbf{la sous-algèbre parabolique} associée à $\h$, $\pi$ et $\pi'$. On appelle sous-algèbre parabolique de $\g$ une sous-algèbre de Lie de $\g$ ainsi définie pour certains choix de $\h$ et $\pi'$. Soit $\gotn^+=\bigoplus_{\beta \in R^+ \setminus R^+_{\pi'}} \g_{\beta}$ et $\gotn^-=\bigoplus_{\beta \in R^- \setminus R^-_{\pi'}} \g_{\beta}$. On a $\g=\gotp \oplus \gotn^-$. Soit $(\gotn^-)^{\text{a}}$ l'espace vectoriel $\gotn^-$ muni d'une structure d'algèbre de Lie abélienne. On définit alors $\gotq := \gotp \ltimes (\gotn^-)^{\text{a}}$ la \textbf{contraction parabolique} de $\g$ par $\gotp$, où le produit semi-direct découle du morphisme $\phi : \gotp \rightarrow \Der \, (\gotn^-)^\text{a} \simeq \End \gotn^-$ défini par
$$\phi(p)(n)=p \centerdot n =[p,n]_{\gotq}= \pr_{\gotn^-,\gotp}([p,n]_{\g}), \, p \in \gotp, n \in \gotn^-.$$
Pour des raisons pratiques, on omettra souvent l'exposant $( \; \; \; \; )^\text{a}$ de l'algèbre de Lie $(\gotn^-)^\text{a}$ et on notera simplement cette algèbre de Lie $\gotn^-$. Pour tous $p_1, p_2 \in \gotp, n_1, n_2 \in \gotn^-$, on a donc :
$$[p_1,p_2]_{\gotq}=[p_1,p_2]_{\g} \qquad [p_1,n_1]_{\gotq}=\pr_{\gotn^-,\gotp} \left([p_1,n_1]_{\g}\right) \qquad [n_1,n_2]_{\gotq}=0$$
On remarque en particulier que pour tout $h \in \h$ et tout $x \in \g$, on a $[h,x]_{\gotq}=[h,x]_{\g}$. On appelle \textbf{contraction parabolique de $\g$} toute construction d'un tel $\gotq$ pour tout choix de sous-algèbre parabolique $\gotp$, c'est-à-dire à conjugaison près, pour tout choix de $\pi' \subset \pi$.
\label{conpara}
\end{defi}
\begin{rema}
L'algèbre de Lie $\gotq$ n'est en général pas réductive. La sous-algèbre $\gotn^-$ est un idéal abélien de $\gotq$. Les espaces vectoriels $\gotq$ et $\g$ sont égaux et ont la même structure de $\h$-module. Un calcul simple donne $\gotq'=\gotp' \ltimes \gotn^-$.
\end{rema}
\begin{prop}[Panyushev, Yakimova \cite{py13}]
Une contraction parabolique $\gotq$ d'une algèbre réductive $\g$ est algébrique, c'est-à-dire l'algèbre de Lie d'un groupe algébrique $Q$.
\end{prop}
\begin{theo}[Panyushev, Yakimova \cite{py13}]
Pour toute contraction parabolique $\gotq$ d'une algèbre de Lie réductive $\g$, l'indice de $\gotq$ est égal au rang de $\g$.
\end{theo}
On définit certains cas particuliers importants de contractions paraboliques.
\begin{defi}
On prend $\g=\gl_n$ et $\h \subset \g$ l'ensemble des matrices diagonales, qui est une sous-algèbre de Cartan de $\g$. On a $\gl_n'=\spl_n:=\{x \in \gl_n \, | \, \trace(x)=0\}$ (où $\trace(x)$ désigne la trace de $x$) et $\gotz(\gl_n)=\C \id$ (où $\id$ est la matrice identité). On note $\h^A:=\h \cap \g'$ (qu'on a noté $\h^{\text{ss}}$ dans la définition \ref{conpara}). Pour tout $\ell \in \llbracket 1,n-1 \rrbracket$, soit $h_\ell=e_{\ell,\ell}-e_{\ell+1,\ell+1}$ et $\varpi_{\ell} \in (\h^A)^*$ le poids fondamental associé à $h_{\ell}$, \index{$h_\ell$}\index{$\varpi_\ell$} c'est-à-dire que la base $(\varpi_{\ell})_\ell$ est la base de $(\h^A)^*$ duale de la base $(h_\ell)_\ell$ de $\h^A$. On étend alors ces poids fondamentaux à des éléments de $\g^*$ en posant pour tout $\ell$, $(\varpi_\ell)_{|\widehat{\g} \oplus \C \id}=0$ où $\widehat{\g}$ est l'ensemble des matrices de diagonale nulle. Par convention, on pose également $\varpi_0=\varpi_n=0$. On note $(\epsilon_i)_{i \in \llbracket 1,n \rrbracket}$ la base de $\h^*$ duale de la base $(e_{i,i})_{i \in \llbracket 1,n \rrbracket}$ de $\h$, et pour $i \in \llbracket 1,n-1 \rrbracket$, on note $\alpha_i=\epsilon_i-\epsilon_{i+1}$. On considère la base $\pi=\left\{\alpha_1, \ldots, \alpha_{n-1}\right\}$ de $R$. Soit $\pi' \subset \pi$. Soit $\gotp$ la sous-algèbre parabolique associée à $\h$, $\pi$ et $\pi'$. On appelle \textbf{sous-algèbre parabolique standard} une sous-algèbre parabolique définie dans ces conditions. On appelle également \textbf{contraction parabolique standard de $\g$} une contraction parabolique de $\g$ par une sous-algèbre parabolique standard.
\label{canonique}
\end{defi}
\begin{prop}
Toute contraction parabolique de $\g=\gl_n$ est la conjuguée d'une contraction parabolique standard de $\g$.
\end{prop}
Jusqu'à la fin du chapitre \ref{chap4}, $\gotp$ désignera une sous-algèbre parabolique standard et $\gotq$ une contraction parabolique de $\g$ par $\gotp$, et on reprendra souvent les notations de la définition \ref{canonique}. On choisit la sous-algèbre parabolique $\gotp$ propre, autrement dit $\pi' \neq \pi$. En particulier, on a $n \geq 2$.

\subsection{Algèbres d'invariants et de semi-invariants}
On reprend $\gotk$ une algèbre de Lie quelconque de dimension finie. L'action adjointe de $\gotk$ s'étend en une unique action d'algèbre de Lie de $\gotk$ sur son algèbre symétrique $\Sym(\gotk)$. On note $q \cdot f$ \index{$q \cdot f$} si $q \in \gotk$ et $f \in \Sym(\gotk)$ pour désigner cette action. On rappelle que cette action d'algèbre de Lie est caractérisée par les points suivants :
\begin{itemize}
\item pour $q, f \in \gotk$, on a $q \cdot f =[q,f]$,
\item pour tout $q \in \gotk$, l'application $f \in \Sym(\gotk) \mapsto q \cdot f$ est une dérivation de $\Sym(\gotk)$.
\end{itemize}
Désormais, lorsqu'il n'y a pas de précision, si $q_1, q_2 \in \gotk$, le produit $q_1 q_2$ sera le produit dans l'algèbre commutative $\Sym(\gotk)$, à ne pas confondre avec un éventuel produit associatif dans le cas où $\gotk$ est une sous-algèbre de Lie de $\g=\gl_n$.
\begin{defi}
On définit l'\textbf{algèbre des invariants} de $\Sym(\gotk)$ comme la sous-algèbre $\Y(\gotk) \subset \Sym(\gotk)$ \index{$\Y(\gotk)$} définie $\Y(\gotk)=\left\{s \in \Sym(\gotk) \; | \; \forall q \in \gotk, q \cdot s =0\right\}$. Un élément de $\Y(\gotk)$ est appelé \textbf{invariant} de $\Sym(\gotk)$.
\end{defi}
\begin{defiprop}
Pour $\lambda \in \gotk^*$, soit $\Sym(\gotk)_{\lambda}=\left\{ s \in \Sym(\gotk) \; | \; \forall q \in \gotk, q \cdot s = \lambda(q)s\right\}$. \index{$\Sym(\gotk)_\lambda$}\index{$\Sy(\gotk)$} On appelle \textbf{semi-invariant} de $\Sym(\gotk)$ un élément $f$ non nul appartenant à $\Sym(\gotk)_{\lambda}$ pour un certain $\lambda \in \gotk^*$, et on dit que $\lambda$ est le \textbf{poids} de $f$. On note $\Lambda(\gotk)$\index{$\Lambda(\gotk), \Lambda$} (ou $\Lambda$ s'il n'y a pas d'ambiguïté sur l'algèbre de Lie $\gotk$) l'ensemble des poids de semi-invariants de $\Sym(\gotk)$ (ou par abus, l'ensemble des poids de $\Sym(\gotk)$), c'est-à-dire $\Lambda(\gotk) := \left\{\lambda \in \gotk^* \; | \; \Sym(\gotk)_{\lambda}\neq 0\right\}$. L'espace vectoriel $\Sy(\gotk):=\bigoplus_{\lambda \in \Lambda} \Sym(\gotk)_{\lambda}$ est alors une algèbre, appelée \textbf{algèbre des semi-invariants} de $\Sym(\gotk)$.
\end{defiprop}
La démonstration vient du point suivant :
\begin{prop}
Pour tous $\mu, \nu \in \gotk^*$, on a $\Sym(\gotk)_\mu \Sym(\gotk)_\nu \subset \Sym(\gotk)_{\mu+\nu}$. En particulier, $\Lambda$ admet une structure de semi-groupe. On a $\Y(\gotk)=\Sym(\gotk)_0 \subset \Sy(\gotk)$, de sorte que $\Sy(\gotk)$ ainsi que tous les 
$\Sym(\gotk)_\mu$ sont des $\Y(\gotk)$-modules. \label{algsemiinv}
\end{prop}
\begin{rema}
Si $\gotq$ est une contraction parabolique définie comme à la définition \ref{conpara}, alors $\Lambda(\gotq) \subset ({\h^{\text{ss}}})^*$ (la notation $\h^{\text{ss}}$ a été introduite dans la définition \ref{conpara}).
\end{rema}
La proposition suivante a été montrée par Moeglin (\cite{Moe}).
\begin{propn}
Soit $f \in \Sym(\gotk)$ un semi-invariant non nul, alors tout facteur de $f$ est un semi-invariant.\label{dixmier}
\end{propn}
On donne une démonstration élémentaire pour le lecteur.
\begin{proof}
On peut supposer que $f$ est non constant. Par factorialité de $\Sym(\gotk)$, $f$ s'écrit
$$f=\prod_{u=1}^d f_u^{p_u}$$
avec les $f_u \in \Sym(\gotk)$ irréductibles premiers entre eux deux à deux, et $p_u \in \N^*$.
Comme tout produit de semi-invariants est un semi-invariant par la propriété \ref{algsemiinv}, il suffit de montrer que chaque $f_u$ est un semi-invariant et donc de montrer que :
\begin{enumerate}
\item si $g,h$ sont non nuls et premiers entre eux et $gh$ est un semi-invariant, alors $g$ et $h$ sont des semi-invariants,
\item si $g^p$, avec $p \in \N^*$ et $g \neq 0$, est un semi-invariant, alors $g$ est un semi-invariant.
\end{enumerate}\par
Pour (1), supposons $gh$ de poids $\lambda$ alors pour tout $q \in \gotk$, $$q\cdot (gh)=\lambda(q) \, gh = (q \cdot g) \, h + g \, (q \cdot h).$$
Ainsi $(\lambda(q) \, g - q \cdot g) \, h = g \, (q \cdot h)$. Puisque $g$ et $h$ sont premiers entre eux, $h $ divise $ q \cdot h$. Or $\deg(q \cdot h) \leq \deg(h)$ donc il existe $\mu \in \gotk^*$ tel que $q \cdot h = \mu(q) \, h$ et on obtient alors $q \cdot g = (\lambda - \mu)(q) \, g$.\par
Pour (2), supposons $g^p$ de poids $\lambda$, alors pour tout $q \in \gotq$
$$q \cdot g^p = \lambda(q)g^p = p \, ( q \cdot g )\, g^{p-1}$$
d'où $q \cdot g = \frac{\lambda(q)}{p} g $.
\end{proof}
\section{Troncation canonique et ad-algébricité}
Cette sous-section reprend en grande partie les résultats exposés dans \cite[\mbox{}2.4.5]{jos95} et \cite[Sec. B, Chap. I]{fau14}.
\begin{defi}
Soit $\gotk$ une algèbre de Lie et $\Lambda$ l'ensemble des poids de $\Sym(\gotk)$. La \textbf{troncation canonique} $\gotk_{\Lambda}$ est la sous-algèbre de Lie de $\gotk$ définie par $\gotk_{\Lambda} = \bigcap_{\lambda \in \Lambda} \ker \lambda$. \index{$\gotk_\Lambda$}
\label{troncation}
\end{defi}
\begin{prop}
Soit $\gotk$ une algèbre de Lie, on note $\Lambda = \Lambda(\gotk)$. On a $\gotk' \oplus \gotz(\gotk) \subset \gotk_\Lambda$.
\label{inclusion}
\end{prop}
\begin{proof}
Soit $\lambda \in \Lambda$, et $s$ un semi-invariant non nul de poids $\lambda$. Pour tous $x,y \in \gotk$ on a $\lambda([x,y])s=[x,y] \cdot s=(x \cdot (y \cdot s))-(y \cdot (x \cdot s))=(\lambda(x)\lambda(y)-\lambda(y)\lambda(x))s=0$. Si $x \in \gotz(\gotk)$, on a $\lambda(x) s = x \cdot s = 0$ donc $\lambda(x)=0$.
\end{proof}
La définition suivante est tirée de \cite[\mbox{}24.8.2]{TYu}.
\begin{defi}
L'algèbre de Lie $\gotk$ est \textbf{ad-algébrique} si l'algèbre de Lie $\ad \gotk$ est algébrique, c'est-à-dire une sous-algèbre de Lie algébrique de $\gl(\gotk)$.
\end{defi}
En particulier, une algèbre de Lie algébrique est ad-algébrique. On obtient alors le théorème suivant, présenté dans \cite[Appendice B.2]{fmj08} dans une version un peu plus générale, qui généralise un résultat de Borho, Gabriel et Rentschler \cite[Satz 6.1]{bor73}.
\begin{theo}[Joseph, Millet]
Soit $\gotk$ une algèbre de Lie de dimension finie. Si $\gotk$ est ad-algébrique, alors $\Sy(\gotk)=\Y(\gotk_\Lambda)=\Sy(\gotk_\Lambda)$.
\label{jm}
\end{theo}
En particulier, dans le cas où $\gotq$ est une contraction parabolique d'une algèbre de Lie réductive, $\gotq$ est ad-algebrique (par \cite[\S 2.4.5]{jos95}, voir également \cite[chap. I, Sec. B, 6.3]{fau14}). On a donc $\Sy(\gotq) = \Sy(\gotq_\Lambda)=\Y(\gotq_{\Lambda})$. De plus, par un résultat de Chevalley et Dixmier \cite[lemme 7]{dix57}, qui est un cas particulier d'un théorème de Rosenlicht \cite{ros63}, on a $\degtr \Frac(\Y(\gotq_\Lambda)) \leq \degtr \, (\Frac(\Sym(\gotq_\Lambda)))^{\gotq_\Lambda}=\ind \gotq_\Lambda$. Comme $\Y(\gotq_\Lambda)=\Sy(\gotq_\Lambda)$, on a (voir \cite[Chap. I, Sec. B, 5.12]{fau14})
%\begin{theo}[Borho, Gabriel, Rentschler]
%Pour toute algèbre de Lie $\gotq$ presque ad-algébrique, on a $\Sy(\gotq)=\Y(\gotq_{\Lambda})$. Ainsi $\Sy(\gotq) \subset \Sym(\gotq_{\Lambda})$. \label{bgr}
%\end{theo}
\begin{theo}
$\GK \Sy(\gotq)=\GK \Y(\gotq_\Lambda)=\ind \gotq_\Lambda$.\label{ai}
\end{theo}
%\section{Résultats préliminaires cruciaux}
%On rassemble dans cette section des résultats cruciaux qui seront fréquemment utilisés dans la suite.
%\begin{defi}
%Pour tout $j \in \llbracket 1,n \rrbracket$, on définit 
%$$F_j:=\sum_{J \subset I, |J|=j} \Delta_J$$ \index{025@$F_j$}
%qui est un élément de $\Sym(\g) \simeq_\text{alg} \Sym(\gotq)$ de degré $j$. En particulier, $F_1=\id$. Remarquons que l'ensemble des monômes de $F_j$ forment une famille linéairement indépendante. \label{Fm}
%\end{defi}
%\begin{defi}
%On définit les projections suivantes :
%\begin{itemize}
%\item $\pr^A : \gl_n \rightarrow \spl_n$ la projection sur $\spl_n$ parallèlement à $\C \id$,
%\item $\pr^C : \gl_n \rightarrow \syp_n$ la projection sur $\syp_n$ parallèlement à $\syp_n^\perp \oplus \C \id$, où $\syp_n^\perp$ est l'orthogonal de $\syp_n$ dans $\spl_n$ pour la forme de Killing de $\spl_n$.
%\end{itemize}
%Ces deux projections s'étendent en des morphismes de $\C$-algèbres sur les algèbres symétriques associées.
%\end{defi}
%\begin{theo}[Panyushev, Yakimova \cite{py13}]
%Soit $\gotq^A$ (respectivement $\gotq^C$) une contraction parabolique de $\spl_n$ (respectivement $\syp_n$). L'algèbre des invariants $\Y(\gotq^A)$ (respectivement $\Y(\gotq^C)$) est polynomiale et librement engendrée par les $\pr(F_m^\bullet)$
%\end{theo}

\chapter{Polynomialité de l'algèbre des semi-invariants associée à une contraction parabolique de $\gl_n$}
On rappelle que $\gotq = \gotp \ltimes \gotn^-$ est une contraction parabolique standard de $\gl_n$.
\section{Propriétés relatives à une contraction parabolique}
\label{sec3}
Dans cette section, on développe une combinatoire sur la contraction parabolique qui permettra d'exhiber des semi-invariants comme facteurs d'invariants symétriques. 

\subsection{Combinatoire de la contraction parabolique}
\label{ssec3.1}
%Si $\gotq$ est une contraction parabolique de $\spl_n$ par un Borel, on a $\mathcal{F}_{n-1}^\bullet=e_{2,1} \, e_{3,2} \ldots e_{n,n-1} \, e_{1,n}$, et Yakimova \cite[lemme 5.8, théorème 5.9]{yak12} a montré que la famille $\{\mathcal{F}_j^{\bullet} \, | \, 1 \leq j \leq n-2 \} \cup \{e_{2,1}, \, e_{3,2}, \ldots, e_{n,n-1}, \, e_{1,n} \}$ est une famille d'éléments algébriquement indépendants qui engendre $\Sy(\widetilde{\gotq})$, donc $\Sy(\widetilde{\gotq})$ est polynomiale (voir notation de la définition \ref{fipt}). Ce manuscrit généralise ce type de construction à un parabolique quelconque de $\gotq$ (pour $\gl_n$).
Sous les notations de la section \ref{sec2} et la définition \ref{canonique}, on introduit les notations suivantes qui seront utiles pour la suite.
\begin{nota} On rappelle que $\g=\gl_n$ et $I=\llbracket 1,n \rrbracket$ avec $n \geq 2$.
\label{nota} 
\begin{itemize}
\item On pose $\pi \setminus \pi'=\{\alpha_{\iota_1}, \ldots, \alpha_{\iota_{s-1}}\}$ avec $s \geq 2$, et $I_{\gotp}=\left\{0=\iota_0 < \iota_1 < \ldots < \iota_{s-1}< \iota_{s}=n\right\}$\index{$I_{\gotp}$}. Pour tout $k \in \llbracket 1,s \rrbracket$, on note $I_k=\llbracket \iota_{k-1}+1, \iota_{k}\rrbracket$\index{$I_k, J_k, i_k,j_k$}. On obtient alors une partition $I = I_1 \sqcup \ldots \sqcup I_s$, telle que $\gotl:=\g_{I_1} \times \ldots \times \g_{I_s}$ est un facteur de Levi de $\gotp$. Soit $i_k=\Card I_k=\iota_k-\iota_{k-1}$, de sorte que $n=i_1+\ldots+i_s$ et $\g_{I_k} \simeq \gl_{i_k}$ ; on note également $\imax:=\max_{1 \leq k \leq s}(i_k)$\index{$\imax$}. Pour tout $x \in \llbracket 1,n \rrbracket$, on note $k(x)$\index{$k(x), i(x)$} l'unique indice de $\llbracket 1,s \rrbracket$ tel que $x \in I_{k(x)}$ et $i(x)=\Card(I_{k(x)})$. Pour tout $i \in \N$, on pose $K(i)=\{k \in \llbracket 1,s \rrbracket \, | \, i_k \geq i\}$\index{$K(i)$}. Les $K(i)$, $1\leq i \leq \imax$ forment une suite décroissante (pour l'inclusion) de parties de $\llbracket 1,s \rrbracket$. On note $\kappa_i=\left\{k_{i,1} < \ldots < k_{i,\rho_i}\right\}$\index{$\kappa_i$, $k_{i,t}$}\index{$\kappa$, $k_t$|see{$\kappa_i$, $k_{i,t}$}} l'ensemble des $k$ tels que $i_k=i$, c'est-à-dire $\kappa_i=K(i) \setminus K(i+1)$.
\item Pour tout $J \subset I$, l'algèbre $\gotp_J=\gotp \cap \g_J$ est une sous-algèbre parabolique de $\g_J$. On note alors $J_k=J \cap I_k$ pour tout $k$, de sorte que $J = J_1 \sqcup \ldots \sqcup J_s$. On appelle \textbf{partition de $J$} (associée à $\gotp_J$) cette partition, et les $J_k$ sont les \textbf{parts} de cette partition. On pose $j_k=\Card J_k$, d'où $\Card J=j_1+\ldots+j_s$. Si $L \subset \llbracket 1,s\rrbracket$, on note $J_L:= \bigsqcup_{k \in L} J_k$\index{$J_L$}.
%, et on entendra par "partition de $J$" la partition $J = J_1 \sqcup \ldots \sqcup J_s$.
\item Pour tout $i \in \llbracket 0, \imax \rrbracket$, on note $\rho_i=\Card \kappa_i = \Card (\left\{k \in \llbracket 1,s \rrbracket \; | \; i_k=i \right\})$\index{$\rho_i$} et $m_i=\sum_{\ell=1}^{\imax} \rho_\ell \min (i,\ell) = \sum_{k=1}^s \min (i,i_k)$\index{$m_i$}. 
%Dans la suite et de manière plus générale, s'il n'y a pas d'ambiguïté, on omettra l'indice $i$.
On pose $\mathbf{I}$\index{$\mathbf{I}$} l'ensemble des $i$ tels que $\rho_i \geq 1$. Soit $\mathbf{M}_k$\index{$\mathbf{M}_0, \mathbf{M}_1, \mathbf{M}_2$, $\mathbf{M}_k$, $k \in \N$} l'ensemble des $m_i$ tels que $\rho_i \geq k$. La famille des $\mathbf{M}_k$ forme une suite décroissante de sous-ensembles de $\llbracket 0,n \rrbracket$. On a en particulier
\begin{itemize}
\item $\mathbf{M}_0$ l'ensemble des $m_i$ pour $i \in \llbracket 0,\imax \rrbracket$,
\item $\mathbf{M}_1$ l'ensemble des $m_i$ pour $i \in \mathbf{I}$,
\item $\mathbf{M}_2$ l'ensemble des $m_i$ avec $i \in \mathbf{I}$ et $\rho_i \geq 2$.
\end{itemize}
%On pose $p=\Card(\mathbf{I})$ et $\mathbf{I}=\left\{i^{(1)} < \ldots < i^{(p)}\right\}$.
\end{itemize}
\end{nota}
Les propriétés élémentaires suivantes seront utiles pour la suite.
\begin{prop}
On a $\gotq'=\bigoplus_{u \neq v} \C e_{u,v} \oplus \bigoplus_{\iota \notin I_{\gotp}} \C h_\iota$ et $\gotz(\gotq)=\C \id$. \label{qprime}
\end{prop}
\begin{prop}\mbox{}
\label{propelem}
\begin{enumerate}
\item On a $\rho_0=m_0=0$, $m_1=s$ et $m_{\imax}=n$. Pour tout $i$, l'entier $\rho_i$ est non nul si et seulement si il existe $k$ tel que $\Card(I_k)=i$.
\item L'application $i \in \llbracket 0, \imax \rrbracket \mapsto m_i \in \mathbf{M}_0$ est bijective croissante. En particulier, $|\mathbf{I}|=|\mathbf{M}_1|$.
\item Pour tout $i$, on a $|K(i)|=m_i-m_{i-1}$,
\item Pour tout $i \in \llbracket 1, \imax \rrbracket$, on a $m_i-m_{i-1}= \Card(\{k \, | \, i_k \geq  i \, \}) > 0$. Ainsi pour tout $i \in \llbracket 1, \imax-1 \rrbracket$, on a 
$$\left\{ \begin{array}{ll}
m_{i+1}-m_{i} = m_i-m_{i-1}&\text{si } i \notin \mathbf{I} \\
m_{i+1}-m_{i} < m_i-m_{i-1}&\text{si } i \in \mathbf{I}
\end{array}\right.$$
En particulier, la suite $(m_i-m_{i-1})$ est décroissante.
\item Si $m \in \mathbf{M}_1$, alors $m-1 \notin \mathbf{M}_0$, sauf dans le cas où $m=m_{\imax}=n$ et $n \notin \mathbf{M}_2$,
\item Pour tout $i$, on a $K(i+1)=K(i) \setminus \kappa_i$.
\end{enumerate}
\end{prop}
\begin{proof}[Preuve de (5)]
Si $m \in \mathbf{M}_1$ et $m \neq n$, alors $m=m_i$ avec $i \in \mathbf{I} \setminus \{\imax\}$. Donc par le point (4), on a $m_i-m_{i-1} > m_{i+1}-m_i \geq 1$, d'où $m_i-m_{i-1} \geq 2$. En particulier, par le point (2), $m_i-1=m-1 \notin \mathbf{M}_0$. Si $i=\imax$ et $n=m_{\imax} \in \mathbf{M}_2$, alors on a également $m_i-m_{i-1}=\Card(\{k \, | \, i_k \geq  i \, \}) \geq 2$ par définition de $\mathbf{M}_2$.
\end{proof}
\begin{exem}
\label{leexemple}
En type $\gl_n$, pour illustrer les différentes constructions, on se basera sur l'exemple suivant : on prend $n=12$, et $\pi'=\{\alpha_1, \alpha_2, \alpha_3, \alpha_6, \alpha_7, \alpha_8, \alpha_{10}\}$, de sorte que $I_1=\{1,2,3,4\}$, $I_2=\{5\}$, $I_3=\{6,7,8,9\}$, $I_4=\{10,11\}$ et $I_5=\{12\}$. On schématise cet exemple comme suit :
\begin{center}
\begin{tikzpicture}[scale=0.4]
\foreach \k in {1,2,...,11}
	{\draw[color=gray!20]  (0,\k)--(12,\k);
	\draw[color=gray!20] (\k,0)--(\k,12);}
\draw (0,0)--(12,0);
\draw (0,0)--(0,12);
\draw (0,12)--(12,12);
\draw (12,0)--(12,12);
\draw[ultra thick] (0,8)--(4,8)--(4,7)--(5,7)--(5,3)--(9,3)--(9,1)--(11,1)--(11,0);
\draw (3,3) node{\LARGE $\gotn^-$};
\draw (9,9) node{\LARGE $\gotp$};
\draw[dashed] (4,12)--(4,8)--(5,8)--(5,7)--(9,7)--(9,3)--(11,3)--(11,1)--(12,1);
\draw (2,10) node{\Large $\g_{I_1}$};
\draw (4.5,7.5) node{\tiny $\g_{I_2}$};
\draw (7,5) node{\Large $\g_{I_3}$};
\draw (10,2) node{\Large $\g_{I_4}$};
\draw (11.5,0.5) node{\tiny $\g_{I_5}$};
\end{tikzpicture}
\end{center}
\end{exem}
\begin{nota}
Soient $x,y \in \llbracket 1,n \rrbracket$.
On définit alors
\begin{itemize}
\item une relation d'équivalence par $x \sim y$\index{$x \sim y$, $x \nsim y$, $x \preceq y$, $x \prec y$} si $k(x)=k(y)$, on note également $x \nsim y$ si $k(x) \neq k(y)$,
\item un préordre $x \preceq y$ si $k(x) \leq k(y)$, on note également $x \prec y$ si $k(x) < k(y)$ ; on pourra également noter la relation dans l'autre sens.
\end{itemize} \label{ordre}
\end{nota}
\begin{prop}
Soient $x,y \in \llbracket 1,n \rrbracket$. On a $x \leq y \Rightarrow x \preceq y$ et $x \prec y \Rightarrow x < y$, les réciproques de ces implications étant fausses en général.\label{propordre}
\end{prop}
\begin{prop}
Soit $p,q \in \llbracket 1,n \rrbracket$. Alors $e_{p,q}$ appartient à $\gotn^{-}$ (respectivement à $\gotp$) si et seulement si $p \succ q$ (respectivement $p \preceq q$). \label{pn}
\end{prop}

\subsection{Degrés en $\gotn^-$ des $F_j$ et mineurs de taille $j$ de degré maximal en $\gotn^-$}
On rappelle que pour tout $j \in \llbracket 1,n \rrbracket$, on note $F_j$ la somme des mineurs principaux de taille $j$ (voir équation \eqref{Fm}). Ces $F_j$ pour $j \in \llbracket 1,n \rrbracket$ engendrent librement $\Y(\gl_n)$ et vont jouer un rôle crucial dans la description de l'algèbre des invariants symétriques $\Y(\gotq)$ (voir proposition \ref{fdroit}).
\begin{defi} \mbox{}
\label{invbas}
\begin{itemize}
\item Soit $N$ un sous-espace vectoriel de $\gl_n$ engendré par des éléments de la base canonique $(e_{i,j})_{1 \leq i,j \leq n}$. Pour un monôme de la forme $a \, e_{i_1, j_1} \ldots e_{i_r, j_r}\in \Sym(\gl_n)$ avec $a \in \C^\times$, \textbf{son degré en $N$}, noté $\deg_{N}(a \, e_{i_1, j_1} \ldots e_{i_r, j_r})$, est le nombre de $k$ tels que $e_{i_k, j_k}$ est dans $N$ (on pose par convention $\deg_{N}(0)=-\infty$). On pose $\deg_{e_{v,w}}$ pour $\deg_{\C e_{v,w}}$. Si $F \in \Sym(\gotq)$ est une somme de tels monômes linéairement indépendants, le degré $\deg_{N}(F)$ \index{$\deg_{N}(F)$} de $F$ en $N$ est défini comme le maximum des degrés en $N$ de ses monômes. On dit que $F$ est \textbf{homogène en $N$} lorsque tous ses monômes sont de même degré en $N$. On dit que $F$ est \textbf{bihomogène en $N$} lorsque $F$ est homogène pour le degré total sur $\Sym(\gotq)$ et homogène en $N$.
\item Pour $F \in \Sym(\gotq)$, on note $F^{\bullet} \in \Sym(\gotq)$ \index{$F^\bullet$} la composante de plus haut degré en $\gotn^-$ de $F$, c'est-à-dire la somme de ses monômes de degré maximal en $\gotn^-$.
\item Pour $F \in \Sym(\gotq)$ bihomogène en $\gotn^-$, on appelle \textbf{bidegré} (en $\gotp$ et $\gotn^-$) de $F$ le couple $(\deg_{\gotp} F, \deg_{\gotn^-} F)$, et on le note $\bideg_{\gotp, \gotn^-} F$, ou $\bideg F$ \index{$\bideg F$} s'il n'y a pas d'ambiguïté. Comme $\deg F = \deg_{\gotp} F + \deg_{\gotn^-} F$, deux éléments de même bidegré ont même degré total.
\end{itemize}
\end{defi}
\begin{prop}\mbox{}
\begin{enumerate}
\item Soient $F, F' \in \Sym(\gotq)$ bihomogènes en $\gotn^-$. Les conditions suivantes sont alors équivalentes :
\begin{enumerate}
\item $\bideg F = \bideg F'$,
\item $(\deg F, \deg_{\gotn^-} F)=(\deg F', \deg_{\gotn^-} F')$,
\item $(\deg F, \deg_{\gotp} F)=(\deg F', \deg_{\gotp} F')$.
\end{enumerate}
\item Pour tous $F_1, F_2 \in \Sym(\gotq)$, on a $(F_1F_2)^\bullet = F_1^\bullet F_2^\bullet$.
\end{enumerate}
\end{prop}
Dans toute la suite, on va chercher à calculer les $F_j^\bullet$. Cela nécessite de connaître les degrés en $\gotn^-$ des $F_j$. Pour les calculer, on a besoin de déterminer les degrés en $\gotn^-$ des mineurs principaux.
\begin{lem}
On se place sous les notations de \ref{nota}. Soit $J \subset I$. On rappelle que $\Delta_J=\sum_{\sigma \in \mathfrak{S}_J} \varepsilon(\sigma)  \prod_{\ell \in J} e_{\ell, \sigma(\ell)}$. Alors le degré en $\gotn^-$ de $\Delta_J$ est égal à $|J|-\max (j_k)$.\label{clé}
\end{lem}
\begin{rema}
Comme les $\left(\prod_{l \in J} e_{l, \sigma(l)}\right)_{\sigma \in \mathfrak{S}_J}$ sont linéairement indépendants, on a $$\deg_{\gotn^-} \Delta_J=\max_{\sigma \in \mathfrak{S}_J} \left( \deg_{\gotn^-} \left(\prod_{l \in J} e_{l, \sigma(l)}\right) \right)$$ Autrement dit, le degré en $\gotn^-$ de $\Delta_J$ est le maximum sur l'ensemble des $\sigma \in \mathfrak{S}_J$ du nombre de $e_{l, \sigma(l)}$ dans $\gotn^-$.\label{rq1}
\end{rema}
\begin{nota}
\label{notaop}
Soit $J \subset I$, $j=\Card(J)$. Soit $\theta_J$\index{$\theta_J, \Theta_J$} l'unique bijection strictement croissante de $J$ dans $\llbracket 1,j \rrbracket$. Cette bijection induit un isomorphisme d'algèbres de Lie $\Theta_J : \g_J \stackrel{\simeq}{\rightarrow} \gl_j$.
\end{nota}
\begin{proof}[Preuve du lemme \ref{clé}]
Comme $\deg_{\gotn^-} \Delta_J = \deg_{\gotn_J^-} \Delta_J$, il suffit de travailler sur $\g_J$. La bijection $\theta_J$ et l'isomorphisme $\Theta_J$ permettent alors de se ramener au cas où $J=\llbracket 1,j \rrbracket$. Sans perdre de généralité, on supposera donc dans la suite que $J=I$. En particulier, $\Delta_J=:\Delta$ est simplement le déterminant sur $\g$, $j=n$ et pour tout $k$, on a $j_k=i_k$.
%On schématise alors $\gotq$ comme suit (le trait plein sépare $\gotn^-$ et $\gotp$) :
%\begin{center}
%\begin{tikzpicture}[scale=0.6]
%\draw (0,4)--(9,4)--(9,13)--(0,13)--(0,4);
%\draw (0,11)--(2,11)--(2,10)--(3,10);
%\draw[dashed] (2,13)--(2,11)--(3,11)--(3,10);
%\draw[dotted] (3.2,9.8)--(4.8,8.2);
%\draw (5,8)--(5,5)--(8,5)--(8,4);
%\draw[dashed] (5,8)--(8,8)--(8,5)--(9,5);
%\draw (1,12) node{$\g_{I_1}$};
%\draw (2.5,10.5) node[scale=0.7]{$\g_{I_2}$};
%\draw (6.5,6.5) node{$\g_{I_{s-1}}$};
%\draw (8.5,4.5) node[scale=0.7]{$\g_{I_s}$};
%\draw (2,6) node{$\gotn^-$};
%\draw (7,11) node{$\gotp$};
%\end{tikzpicture}
%\end{center}
Soit $k_0 \in \kappa_{\imax}$ (c'est-à-dire tel que $i_{k_0}=\imax$). Le déterminant $\Delta$ est donné par :
$$\Delta=\sum_{\sigma \in \mathfrak{S}_n}\varepsilon(\sigma) \; e_{1, \sigma(1)} \ldots e_{n, \sigma(n)}$$
Chaque $e_{p,q}$ appartient soit à $\gotp$, soit à $\gotn^-$. En particulier, par la propriété \ref{pn}, tout élément de la forme $e_{p,q}$ avec $p-q=\imax$ est dans $\gotn^{-}$ ; en effet, $p >q$ donc $p \succeq q$, et tout $I_k$ est de la forme $\llbracket a_k, b_k \rrbracket$ avec $b_k-a_k=i_k-1 \leq \imax-1$, donc $p \nsim q $. Ainsi les éléments $e_{\imax+1,1}, \, \ldots, \, e_{n,n-\imax}$ sont bien dans $\gotn^-$. Puisque les termes $e_{1, n-\imax+1}, \, \ldots, \, e_{\imax,n}$ ne sont pas dans $\gotn^-$ (ils sont de la forme $e_{p,q}$ avec $p<q$), le monôme $e_{\imax+1,1} \ldots e_{n,n-\imax} e_{1, n-\imax+1} \ldots e_{\imax,n}$, qui (au signe près) est un monôme de $\Delta$, est de degré $n-\imax$ en $\gotn^-$.\par
Montrons maintenant que ce degré est maximal sur les monômes de $\Delta$. Soit $\sigma \in \mathfrak{S}_n$. Le produit $e_{1, \sigma(1)} \ldots e_{n, \sigma(n)}$ est (au signe près) un monôme de $\Delta$. Supposons que $e_{l, \sigma(l)}, e_{l', \sigma(l')} \in \gotn^-$ avec $l < l'$ et $\sigma(l)>\sigma(l')$. Alors on a $l' > l \succ \sigma(l) > \sigma(l')$. Ainsi, avec les propriétés \ref{propordre} et \ref{pn}, en prenant $\tau=(\sigma(l) \; \sigma(l'))$, on a également $e_{l, \tau\sigma(l)}, e_{l', \tau\sigma(l')} \in \gotn^-$. Ceci implique par récurrence qu'il existe $\sigma' \in \mathfrak{S}_n$ vérifiant les deux conditions suivantes :
\begin{itemize}
\item le monôme $e_{1, \sigma'(1)} \ldots e_{n, \sigma'(n)}$ est de même degré en $\gotn^-$ que $e_{1, \sigma(1)} \ldots e_{n, \sigma(n)}$,
\item si $e_{l_1, \sigma'(l_1)}, \ldots, e_{l_d, \sigma'(l_d)}$ avec $l_1 < \ldots < l_d$ est l'ensemble des $e_{l,\sigma'(l)}$ dans $\gotn^-$, alors l'application $\sigma'_{|\left\{l_1, \ldots, l_d\right\}}$ est strictement croissante.
\end{itemize}
On considère de tels $l_1 < \ldots < l_d$ et on va montrer que $d \leq n-\imax$. Soit $t$ l'indice maximal tel que $l_t \in \bigsqcup_{k=1}^{k_0} I_k$. Alors $\left\{l_{t+1}, \ldots, l_d\right\}\subset \bigsqcup_{k=k_0+1}^s I_k$. En prenant les cardinaux, on obtient $d-t \leq \sum_{k=k_0+1}^{s} i_k$. De plus, puisque $e_{l_t, \sigma'(l_t)}\in \gotn^-$, par la propriété \ref{pn}, on a $\sigma'(l_t) \in \bigsqcup_{k=1}^{k_0-1} I_k$. Alors, par croissance de $\sigma'$, on a $\left\{\sigma'(l_1), \ldots, \sigma'(l_t)\right\} \subset \bigsqcup_{k=1}^{k_0-1} I_k$. En prenant les cardinaux, on obtient $t \leq \sum_{k=1}^{k_0-1} i_k$. Finalement, 
\[
d=t+(d-t)
\leq \sum_{k=1}^{k_0-1} i_k + \sum_{k=k_0+1}^{s} i_k
=n-i_{k_0}=n-\imax.\qedhere
\]
\end{proof}
\begin{rema}
Comme les $(\Delta_J)_{J \subset I, |J|=j}$ sont linéairement indépendants, on a par le lemme \ref{clé}
\[\deg_{\gotn^-} F_j=\max_{J \subset I, |J|=j} \deg_{\gotn^-} \Delta_J =\max_{J \subset I, |J|=j} \left(j-\max_k(j_k)\right)=j-\min_{J \subset I, |J|=j}\max_k(j_k).\]
Autrement dit,
$$\deg_{\gotp} F_j = \min_{J \subset I, |J|=j}\max_k(j_k).$$
\label{degnFj}
\end{rema}
\begin{exem}
On reprend l'exemple \ref{leexemple}. Posons $j=6$. Si $J=\llbracket 1, 6 \rrbracket$, alors $j_1=4$, $j_2=1$, $j_3=1$, $j_4=j_5=0$, d'où $\max_k j_k=4$. Si l'on prend un $J$ construit comme suit :
\begin{itemize}
\item on prend un entier dans chaque $I_k$ pour tout $k$, ce qui donne un ensemble à $5$ éléments,
\item pour le dernier élément de $J$, on prend un autre entier quelconque de $I_1$, $I_3$ ou $I_4$,
\end{itemize}
on a $\max_k j_k=2$, et on ne peut pas choisir $J \subset I = \llbracket 1,12 \rrbracket$ de cardinal $6$ de sorte que $\max_k j_k$ soit strictement plus petit : si $J$ était tel que $\max_k j_k \leq 1$, alors $j=\sum_{k=1}^5 j_k \leq \sum_{k=1}^5 1 = 5$ ce qui est absurde. Dans ce cas, on a donc $\deg_{\gotn^-} F_6=6-2=4$.
\end{exem}
\begin{coro}
Pour tout $j \in \llbracket 1,n \rrbracket$, l'entier $j- \deg_{\gotn^-} F_j$ est l'unique $i \in \llbracket 1,\imax \rrbracket$ tel que $$m_{i-1} < j \leq m_{i }$$
(voir le point (2) de la propriété \ref{propelem}). La suite $(\deg_{\gotn^-} F_j)_{1 \leq j \leq n}$ vérifie alors la relation de récurrence suivante : on a $\deg_{\gotn^-} F_1=0$ et pour tout $j \in \llbracket 1,n-1 \rrbracket$, 
\begin{equation}
deg_{\gotn^-} F_{j+1} = \left\{ \begin{array}{ll}
\deg_{\gotn^-} F_{j}&\text{si } j=m_i, i \in \llbracket 0,\imax-1 \rrbracket \\
\deg_{\gotn^-} F_{j}+1&\text{sinon}
\end{array}\right.\label{degrec}
\end{equation}
\label{coro}
\end{coro}
\begin{proof}
On note $i$ l'unique entier de $\llbracket 1, \imax \rrbracket$ tel que $m_{i-1} < j \leq m_i$. Pour le deuxième point, on a $F_1=\id$ de degré $0$ en $\gotn^-$ (puisque $\h \cap \gotn^-=\{0\}$), donc la relation de récurrence est une conséquence du premier point. Par la remarque \ref{degnFj}, il s'agit de montrer que pour tout $J$ de cardinal $j$, on a $\max{(j_k)} \geq i$ et qu'il existe $J$ pour lequel l'égalité est vérifiée (voir \ref{typej}).\par
Soit $J \subset I$ de cardinal $j$. Si $\max(j_k) < i$, alors pour tout $k$, on a $j_k \leq i-1$. Comme par définition, $j_k \leq i_k$, on a $j_k \leq \min (i-1,i_k)$. Ainsi
\begin{equation}
j=\sum_{k=1}^s j_k \leq \sum_{k=1}^s \min (i-1,i_k)=m_{i-1}<j \label{padide}
\end{equation}
ce qui est absurde.
\begin{defiprop}
Soit $i$ l'unique entier de $\llbracket 1, \imax \rrbracket$ tel que $m_{i-1} < j \leq m_i$. On a $j-m_{i-1} \leq \Card(\{k \in \llbracket 1,s \rrbracket \, | \, i_k \geq i \})$.\par
Soit $K \subset \{k \in \llbracket 1,s \rrbracket \, | \, i_k \geq i \}$ de cardinal $j-m_{i-1}$. Soit alors un sous-ensemble $J \subset I$ défini par les choix de $J_k$ suivants :
\begin{itemize}
\item si $k \notin K$, on choisit $J_k \subset I_k$ de cardinal $\min(i-1,i_k)$,
\item si $k \in K$, on choisit $J_k \subset I_k$ de cardinal $i$.
\end{itemize}
On note $\mathcal{J}(j)$\index{$\mathcal{J}(j)$} l'ensemble des $J \subset I$ ainsi définis, pour tous choix de $K$ et de $J_k$ pour $k \in \llbracket 1,s \rrbracket$. Pour tout $J \in \mathcal{J}(j)$, on a $\Card(J)=j$ et $\max(j_k)=i$. En particulier, pour tout $J \in \mathcal{J}(j)$, on a $\deg_{\gotn^-} \Delta_J=\deg_{\gotn^-} F_j=j-i$.
\label{typej}
\end{defiprop}
\begin{proof}\mbox{}
\begin{itemize}
\item \emph{On a $j-m_{i-1} \leq \Card(\{k \in \llbracket 1,s \rrbracket \, | \, i_k \geq i \})$.} En effet, par hypothèse sur $i$, on a 
$$
0<j-m_{i-1} \leq m_i - m_{i-1}= \Card(\{k \in \llbracket 1,s \rrbracket \, | \, i_k \geq i \})
$$
par la propriété \ref{propelem}.
\item \emph{Pour tout $J \in \mathcal{J}(j)$, on a $\Card(J)=j$.} En effet,
\begin{align*}
\Card(J)&=\sum_{k=1}^s j_k \\
&=\sum_{k \in K} i + \sum_{k \notin K} \min(i-1,i_k)\\
&=\sum_{k \in K} 1 + \sum_{k \in K} (i-1) + \sum_{k \notin K} \min(i-1,i_k) \\
&=\Card(K) + \sum_{k \in K} \min(i-1,i_k) + \sum_{k \notin K} \min(i-1,i_k)\\
&=\Card(K) + \sum_{k=1}^s \min(i-1,i_k)\\
&=(j-m_{i-1})+m_{i-1}=j
\end{align*}
\end{itemize}
Par construction, pour de tels $J$, on a alors $\max(j_k)=i$.
\end{proof}
Ceci conclut donc la démonstration du corollaire \ref{coro}.
\end{proof}
\begin{exem}
\label{exnot}
Appliquons la définition-propriété \ref{typej} à deux cas particuliers.
\begin{itemize}
\item Pour $m=m_i \in \mathbf{M}_0$ non nul, on a $m-\deg_{\gotn^-} F_m=i$ (corollaire \ref{coro}). Comme $m-m_{i-1}=\Card (\{k \, | \, i_k \geq i\})$ (propriété \ref{propelem}), l'ensemble $J$ est dans $\mathcal{J}(m)$ si et seulement si $j_k=\min(i,i_k)$ pour tout $k \in \llbracket 1,s \rrbracket$.
\item Supposons que $m=m_i \in \mathbf{M}_0 \setminus \{n\}$. On a $m+1-\deg_{\gotn^-} F_{m+1}=i+1$ (corollaire \ref{coro}). Dans ce cas, l'ensemble $J \subset I$ appartient à $\mathcal{J}({m+1})$ si et seulement si les conditions suivantes sont vérifiées (définition \ref{typej}) :
\begin{itemize}
\item il existe un unique $k_J$\index{$k_J$} tel que $j_{k_J}=i +1 \leq i_{k_J}$,
\item pour tout $k \neq k_J$, $j_k=\min(i,i_k)$.
\end{itemize}
\end{itemize}
\end{exem}
\begin{coro}
Soit $i \in \llbracket 1, \imax \rrbracket$ et $m:=m_i$.
\begin{enumerate}
\item Soit $J \subset I$ de cardinal $m$. On a $\max (j_k) \geq i$. Alors l'ensemble $J$ est dans $\mathcal{J}(m)$ si et seulement si $\max (j_k) = i$. 
\item Si $J \in \mathcal{J}(m+1)$, alors pour tout $l \in J_{k_J}$ (notation de l'exemple \ref{exnot}), on a $J \setminus \left\{l\right\} \in \mathcal{J}(m)$. Inversement, si $J \in \mathcal{J}(m)$, $m \neq n$, alors pour tout $l \in I \setminus J$, $J \sqcup \left\{l\right\} \in \mathcal{J}(m+1)$.
\end{enumerate}
\label{comb}
\end{coro}
\begin{proof}
Le point (2) est clair. Montrons le point (1). Comme dans la démonstration du corollaire \ref{coro} (voir équation \ref{padide}), on montre que $\max (j_k) \geq i$. Si $J \in \mathcal{J}(m)$, alors $\max(j_k)=i$ par construction. Il reste à montrer la réciproque. Si $\max (j_k) = i$, pour tout $k$, on a $j_k \leq i$. Comme $J_k \subset I_k$, on a $j_k \leq i_k$, d'où $j_k \leq \min(i,i_k)$. Puisque $\sum_k j_k=m=\sum_k \min(i,i_k)$, ceci implique que $j_k = \min(i,i_k)$ pour tout $k$, c'est-à-dire $J \in \mathcal{J}(m)$ (exemple \ref{exnot}).
\end{proof}
\begin{nota}
Soit $k \in \llbracket 1,s \rrbracket$, et $j \in \llbracket 1,n \rrbracket$. On note $\mathcal{J}(j)_k$\index{$\mathcal{J}(j)_k$, $\mathcal{J}_k$}, ou plutôt $\mathcal{J}_k$ quand il n'y aura pas d'ambiguïté sur $j$, l'ensemble des $J_k =J \cap I_k$ pour $J \in \mathcal{J}(j)$. En particulier, par l'exemple \ref{exnot}, pour tout $m=m_i \in \mathbf{M}_0$, on a $\mathcal{J}(m)_k=\{J_k \subset I_k \, | \, |J_k|= \min(i,i_k)\}$ pour tout $k \in \llbracket 1,s \rrbracket$, et on a une bijection
\begin{align*}
\prod_{k=1}^s \mathcal{J}(m)_k &\longrightarrow \mathcal{J}(m) \\
(J_k)_{1 \leq k \leq s} &\longmapsto \bigsqcup_{k=1}^s J_k
\end{align*}
\label{jrond}
\end{nota}
\begin{exem}
\label{exemple}
On reprend l'exemple \ref{leexemple}. On rappelle que $\gotl=\gl_4 \times \gl_1 \times \gl_4 \times \gl_2 \times \gl_1$ est un facteur de Levi de $\gotp$. On a alors $\mathbf{I}=\{1,2,4\}$, et $m_1=5$, $m_2=8$ et $m_4=12$. Remarquons que $3 \notin \mathbf{I}$, et donc que $m_3=10$ n'est pas dans $\mathbf{M}_1$.\par
On représente $\gotp$ comme suit :
\begin{center}
\begin{tikzpicture}[scale=0.5]
\draw (0,0)--(0,5)--(4,5)--(4,4)--(1,4)--(1,3)--(4,3)--(4,2)--(2,2)--(2,1)--(1,1)--(1,0)--(0,0);
\draw (0,1)--(1,1);
\draw (0,2)--(2,2);
\draw (0,3)--(1,3);
\draw (0,4)--(1,4);
\draw (1,1)--(1,3);
\draw (1,4)--(1,5);
\draw (2,2)--(2,3);
\draw (2,4)--(2,5);
\draw (3,4)--(3,5);
\draw (3,2)--(3,3);
%\draw[dashed] (2,-0.5)--(2,5.5);
%\draw (2,-1) node{$m_2=8$};
%\fill [pattern=north east lines]
%	(0,0)--(0,5)--(2,5)--(2,4)--(1,4)--(1,3)--(2,3)--(2,1)--(1,1)--(1,0)--(0,0);
\draw (-0.5,4.5) node{$4$};
\draw (-0.5,3.5) node{$1$};
\draw (-0.5,2.5) node{$4$};
\draw (-0.5,1.5) node{$2$};
\draw (-0.5,0.5) node{$1$};
\end{tikzpicture}
\end{center}
Le nombres de cases de la ligne $k$ est égal à $i_k$. Par exemple, on représente $i_1=4$ par les quatre cases de la première ligne. On obtient une bijection entre les sous-algèbres paraboliques standards $\gotp$ de $\gl_n$ et de tels diagrammes à $n$ cases.\par
Pour tout $J \subset I$, on représente les valeurs de $j_k$ en hachurant sur la ligne $k$ les $j_k$ premières cases. Par exemple, si $m=m_2 = 8$, et $J \in \mathcal{J}(8)$, on a $j_k=\min(2,i_k)$ pour tout $k$, et les valeurs des $j_k$ sont représentées par le diagramme suivant
\begin{center}
\begin{tikzpicture}[scale=0.5]
\draw (0,0)--(0,5)--(4,5)--(4,4)--(1,4)--(1,3)--(4,3)--(4,2)--(2,2)--(2,1)--(1,1)--(1,0)--(0,0);
\draw (0,1)--(1,1);
\draw (0,2)--(2,2);
\draw (0,3)--(1,3);
\draw (0,4)--(1,4);
\draw (1,1)--(1,3);
\draw (1,4)--(1,5);
\draw (2,2)--(2,3);
\draw (2,4)--(2,5);
\draw (3,4)--(3,5);
\draw (3,2)--(3,3);
\draw[dashed] (2,-0.5)--(2,5.5);
\draw (2,-1) node{$m_2=8$};
\fill [pattern=north east lines]
	(0,0)--(0,5)--(2,5)--(2,4)--(1,4)--(1,3)--(2,3)--(2,1)--(1,1)--(1,0)--(0,0);
\draw (-0.5,4.5) node{$4$};
\draw (-0.5,3.5) node{$1$};
\draw (-0.5,2.5) node{$4$};
\draw (-0.5,1.5) node{$2$};
\draw (-0.5,0.5) node{$1$};
\end{tikzpicture}
\end{center}
Les deux diagrammes suivants illustrent les deux possibilités pour $(j_k)_{1 \leq k \leq 5}$ si $J$ est dans $\mathcal{J}(m_2+1)$, c'est-à-dire $J \in \mathcal{J}(9)$.
\begin{center}
\begin{tikzpicture}[scale=0.5]
\draw (0,0)--(0,5)--(4,5)--(4,4)--(1,4)--(1,3)--(4,3)--(4,2)--(2,2)--(2,1)--(1,1)--(1,0)--(0,0);
\draw (0,1)--(1,1);
\draw (0,2)--(2,2);
\draw (0,3)--(1,3);
\draw (0,4)--(1,4);
\draw (1,1)--(1,3);
\draw (1,4)--(1,5);
\draw (2,2)--(2,3);
\draw (2,4)--(2,5);
\draw (3,4)--(3,5);
\draw (3,2)--(3,3);
\draw[dashed] (2,-0.5)--(2,5.5);
\draw (2,-1) node{$m_2=8$};
\fill [pattern=north east lines]
	(0,0)--(0,5)--(3,5)--(3,4)--(1,4)--(1,3)--(2,3)--(2,1)--(1,1)--(1,0)--(0,0);
\draw (-0.5,4.5) node{$4$};
\draw (-0.5,3.5) node{$1$};
\draw (-0.5,2.5) node{$4$};
\draw (-0.5,1.5) node{$2$};
\draw (-0.5,0.5) node{$1$};
\end{tikzpicture}
\hspace{3cm}
\begin{tikzpicture}[scale=0.5]
\draw (0,0)--(0,5)--(4,5)--(4,4)--(1,4)--(1,3)--(4,3)--(4,2)--(2,2)--(2,1)--(1,1)--(1,0)--(0,0);
\draw (0,1)--(1,1);
\draw (0,2)--(2,2);
\draw (0,3)--(1,3);
\draw (0,4)--(1,4);
\draw (1,1)--(1,3);
\draw (1,4)--(1,5);
\draw (2,2)--(2,3);
\draw (2,4)--(2,5);
\draw (3,4)--(3,5);
\draw (3,2)--(3,3);
\draw[dashed] (2,-0.5)--(2,5.5);
\draw (2,-1) node{$m_2=8$};
\fill [pattern=north east lines]
	(0,0)--(0,5)--(2,5)--(2,4)--(1,4)--(1,3)--(3,3)--(3,2)--(2,2)--(2,1)--(1,1)--(1,0)--(0,0);
\draw (-0.5,4.5) node{$4$};
\draw (-0.5,3.5) node{$1$};
\draw (-0.5,2.5) node{$4$};
\draw (-0.5,1.5) node{$2$};
\draw (-0.5,0.5) node{$1$};
\end{tikzpicture}
\end{center}
On n'a pas $\max(j_k)=i +1 \Rightarrow J \in \mathcal{J}(m_i+1)$. En effet, les $j_k$ représentés par le diagramme suivant donnent également $\max(j_k)=i +1$ :
\begin{center}
\begin{tikzpicture}[scale=0.5]
\draw (0,0)--(0,5)--(4,5)--(4,4)--(1,4)--(1,3)--(4,3)--(4,2)--(2,2)--(2,1)--(1,1)--(1,0)--(0,0);
\draw (0,1)--(1,1);
\draw (0,2)--(2,2);
\draw (0,3)--(1,3);
\draw (0,4)--(1,4);
\draw (1,1)--(1,3);
\draw (1,4)--(1,5);
\draw (2,2)--(2,3);
\draw (2,4)--(2,5);
\draw (3,4)--(3,5);
\draw (3,2)--(3,3);
\draw[dashed] (2,-0.5)--(2,5.5);
\draw (2,-1) node{$m_2=8$};
\fill [pattern=north east lines]
	(0,0)--(0,5)--(3,5)--(3,4)--(1,4)--(1,3)--(3,3)--(3,2)--(1,2)--(1,0)--(0,0);
\draw (-0.5,4.5) node{$4$};
\draw (-0.5,3.5) node{$1$};
\draw (-0.5,2.5) node{$4$};
\draw (-0.5,1.5) node{$2$};
\draw (-0.5,0.5) node{$1$};
\end{tikzpicture}
\end{center}
\end{exem}

\subsection{Les $F_m^\bullet$ engendrent librement $\Y(\gotq)$}
\begin{defi}
Soit $\gotq=\gotp \ltimes \gotn^-$ une contraction parabolique standard de $\gl_n$. Soit $\gotp^A=\gotp \cap \spl_n$ (on a $\gotn^- \subset \spl_n$). On pose alors $\gotq^A=\gotp^A \ltimes \gotn^-$\index{$\gotq^A$}. L'algèbre de Lie $\gotq^A$ est la \textbf{contraction parabolique standard} de $\spl_n$ associée à $\gotq$. Son crochet de Lie $[\; , \; ]_{\gotq^A}$ vérifie $[x,y]_{\gotq^A}=[x,y]_{\gotq}$ pour tous $x,y \in \gotq^A$.
\label{defqA}
\end{defi}
Comme dans $\gl_n$, toute contraction parabolique de $\spl_n$ est la conjuguée d'une contraction parabolique standard de $\spl_n$. Dans la suite, on posera toujours $\gotq^A$ la contraction parabolique standard de $\spl_n$ associée à $\gotq$.
\begin{defi}
Soit $\pr^A : \gl_n \rightarrow \spl_n$ \index{$\pr^A$} la projection sur $\spl_n$ parallèlement à $\C \id$. La projection $\pr^A$ s'étend en un unique morphisme d'algèbres graduées $\pr^A: \Sym(\gl_n) \rightarrow \Sym(\spl_n)$, qui est la projection sur $\Sym(\spl_n)$ parallèlement à $\Sym(\gl_n)\id$. On note aussi $\mathcal{F}_j:=\pr^A(F_{j+1})$ \index{$\mathcal{F}_j$} pour tout $j \in \llbracket 1,n-1 \rrbracket$. \label{prA}
\end{defi}
\begin{prop}
La projection $\pr^A$ vue de $\gotq$ dans $\gotq^A$ (qui sont égales respectivement à $\gl_n$ et $\spl_n$ comme espaces vectoriels) est un morphisme de $\gotq^A$-modules.
\end{prop}
\begin{proof}
Soient $q_1 \in \gotq^A$ et $q_2=\pr^A(q_2) + x \id \in \gotq$, avec $x \in \C$. On a
$$\pr^A([q_1,q_2]_{\gotq})=\pr^A([q_1,\pr^A(q_2)]_{\gotq} + x[q_1,\id]_{\gotq})=\pr^A([q_1,\pr^A(q_2)]_{\gotq})=[q_1,\pr^A(q_2)]_{\gotq^A} $$\qedhere
\end{proof}
\begin{prop}
Pour tout $j \in \llbracket 1,n-1 \rrbracket$, on a $\mathcal{F}_j^\bullet := \pr^A(F_{j+1})^\bullet=\pr^A(F_{j+1}^\bullet)$.
\label{ftypeA}
\end{prop}
\begin{proof}
On pose 
\begin{equation}
F_{j+1}=\sum_{k=0}^d F_{j+1}^{(k)} \label{composante}
\end{equation}
où $F_{j+1}^{(k)}$ est la composante de $F_{j+1}$ de degré $k$ en $\gotn^-$ et on suppose que $F_{j+1}^{(d)} \neq 0$. Alors par définition, on a $F_{j+1}^{(d)} =F_{j+1}^\bullet$. En projetant l'équation \eqref{composante}, on obtient
$$\mathcal{F}_j:=\pr^A(F_{j+1})=\sum_{k=0}^d \pr^A \left(F_{j+1}^{(k)}\right)$$
Pour tout $s \in \Sym(\gotq)$ non nul bihomogène en $\gotn^-$, on a soit $\deg_{\gotn^-} \pr^A(s) = \deg_{\gotn^-} s$, soit $\pr^A(s)=0$. Il s'agit donc de montrer que $\pr^A \left(F_{j+1}^{(d)}\right) = \pr^A \left(F_{j+1}^\bullet\right) \neq 0$. Autrement dit, il s'agit de montrer que $F_{j+1}^\bullet \notin \Sym(\gl_n) \id$. Puisque $\pr^A$ stabilise $\Sym(\gl_n) \h$ et est l'identité sur $\Sym(\widehat{\g})$, où $\widehat{\g}$ est l'ensemble des matrices de diagonale nulle, il suffit de montrer que $F_{j+1}^\bullet$ a une composante dans $\Sym(\widehat{\g})$. Comme les monômes de $F_{j+1}$, a fortiori de $F_{j+1}^\bullet$ sont linéairement indépendants, il suffit de montrer qu'il existe un monôme de $F_{j+1}^\bullet$ qui est dans $\Sym(\widehat{\g})$. Or, si $\Delta_J$ est un mineur de taille $j+1$ de degré maximal en $\gotn^-$, on peut exhiber un monôme de $\Delta_J$ qui soit dans $\Sym(\widehat{\g})$, de la même manière que dans la démonstration du lemme \ref{clé}.
\end{proof}
\begin{theo}[Panyushev, Yakimova]
Les $\mathcal{F}_j^\bullet$, $1 \leq j \leq n-1$ engendrent librement $\Y({\gotq^A})$ (\cite{py13}, théorème 5.1).
\label{fipt}
\end{theo}
\begin{propn}
Les $F_j^{\bullet}$, $1 \leq j \leq n$ engendrent librement $\Y(\gotq)$.
\label{fdroit}
\end{propn}
\begin{proof}
Par une démonstration analogue à \cite{py12}, théorème 1.1, on montre que les $F_j^{\bullet}$, $1 \leq j \leq n$ sont bien dans $\Y(\gotq)$. Comme $\C \id \subset \gotz(\gotq)$, on a $\Y(\gotq)=\Y(\gotq^A) \otimes \Sym(\C \id)$, ainsi $\id, \mathcal{F}_1^\bullet, \ldots, \mathcal{F}_{n-1}^\bullet$ est une famille algébriquement indépendante de polynômes engendrant $\Y(\gotq)$. On remarque que $F_1^\bullet=\id$.\par
Comme $\deg \mathcal{F}_j^\bullet=j+1$, pour tout $j \in \llbracket 2,n \rrbracket$, il existe $P_j \in \C[X_1, \ldots,X_{j-1}]$ tel que
\begin{equation}
F_j^\bullet=\mathcal{F}_{j-1}^\bullet+P_j(\id, \mathcal{F}_1^\bullet, \mathcal{F}_2^\bullet, \ldots, \mathcal{F}_{j-2}^\bullet)\id.\label{pol}
\end{equation}
Si donc $\phi \in \C[X_1, \ldots,X_{n}]$ est tel que $\phi(F_1^\bullet, \ldots, F_n^\bullet)=0$, en posant
$$\psi=\phi(X_1, X_2+P_2(X_1)X_1, \ldots, X_n+P_n(X_1, \ldots, X_{n-1})X_1)$$
on obtient $\psi(\id, \mathcal{F}_1^\bullet, \ldots, \mathcal{F}_{n-1}^\bullet)=\phi(F_1^\bullet, \ldots, F_n^\bullet)=0$, d'où $\psi=0$. On peut alors revenir de $\psi$ à $\phi$ par un changement de variables, ainsi $\phi=0$, donc les $F_j^\bullet$ sont algébriquement indépendants.\par
Avec l'équation \eqref{pol}, on montre par récurrence que les $\mathcal{F}_{j-1}^\bullet$ sont des polynômes en les $F_k^\bullet$, donc les $F_j^\bullet$ engendrent bien $\Y(\gotq)$.
%L'endomorphisme $f$ de $\C[X_1, \ldots, X_n]$ défini par $f(X_1)=X_1$ et $f(X_j)=X_j+P_j(X_1, \ldots, X_{j-1})$ pour $2 \leq j \leq n$ est un automorphisme.
%En effet, si on définit par récurrence $(Y_j)_{1 \leq j \leq n}$ par :
%%\begin{itemize}
%%\item Y_1=X_1,
%\item pour tout $j \in \llbracket 1,n-1 \rrbracket$, $Y_{j+1}=X_{j+1}-P_{j+1}(Y_1, \ldots, Y_j)$.
%Alors l'endomorphisme de $\C[X_1, \ldots, X_n]$ qui à $X_j$ associe $Y_j$ pour tout $j$ est l'inverse de $f$.
\end{proof}

\section{Une famille de semi-invariants algébriquement indépendants}
\label{sec4}
On reprend les notations de \ref{nota}.
\begin{nota}
On rappelle que $i \in \mathbf{I} \rightarrow m_i \in \mathbf{M}_1$ est bijective (voir propriété \ref{propelem}). Pour tout $m \in \llbracket 1,n \rrbracket$, on définit alors $r_m$ par :
\begin{itemize}
\item $r_m=\rho_i$ si $m=m_i \in \mathbf{M}_1$,
\item $r_m=1$ sinon. \index{$r_m$}
\end{itemize}
\end{nota} 
Dans cette section, on va montrer le théorème suivant :
\begin{theo}
Soit $m \in \llbracket 1,n \rrbracket$. Alors $F_m^{\bullet}$ est (à une constante multiplicative non nulle près) produit de $r:=r_m$ facteurs homogènes non constants, notés $F_{m,1}, \ldots, F_{m,r}$. Ces facteurs sont des semi-invariants de $\Sym(\gotq)$ et vérifient de plus :
\begin{enumerate}
\item pour tout $t \in \llbracket 1, r-1 \rrbracket$, $F_{m,t} \in \Sym(\gotn^-)$,
\item notant $\mathfrak{L}_m:={\lambda}_{m,1}, \ldots, {\lambda}_{m,r}$ la famille des poids de $F_{m,1}, \ldots, F_{m,r}$, on a $\lambda_{m,1} + \ldots + {\lambda}_{m,r}=0$ et $\mathfrak{L}_m$ est de rang $r-1$,
\item les espaces vectoriels $\vect(({\lambda}_{m,t})_t)$ pour $m \in \llbracket 1,n \rrbracket$ sont en somme directe,
\item la famille $\mathbf{F}=(F_{m,t})_{1 \leq m \leq n, \, 1 \leq t \leq r_m}$ est algébriquement indépendante.
\end{enumerate}
\label{semiinv}
\end{theo}
Les $F_{m,t}$ sont définis à l'équation \eqref{expmochesemiinv}, à la fin de la sous-section \ref{ssec2.2}. Le point (1) est montré à la suite. Les points (2) et (3) sont montrés dans la sous-section \ref{ssec4.2}. Le point (4) est montré dans la sous-section \ref{ssec4.3}.
\begin{exem}
\mbox{}
\begin{itemize}
\item Dans le cas où $\gotp$ est une sous-algèbre de Borel, pour $i=\imax=1$, on a $r=s=n$ et $m=n$. Pour tout $t \in \llbracket 1, n-1 \rrbracket$, on a $F_{n,t}=e_{t+1,t}$, et $F_{n,n}=e_{1,n}$. Pour tout $t \in \llbracket 1, n-1 \rrbracket$, le semi-invariant $F_{n,t}$ est bien dans $\Sym(\gotn^-)$. Les $F_1^\bullet,  F_2^\bullet, \ldots, F_{n-1}^\bullet, F_{n,1}, \ldots, F_{n,n}$ forment une famille algébriquement indépendante. On retrouve bien la construction de Yakimova \cite[\S 5.1]{yak12} dans sa démonstration de la polynomialité de $\Sy(\gotq)$.
\item Si $n$ est pair, on appelle \emph{cas de la racine centrale} le cas où $\pi'=\pi \setminus \{\alpha_{n/2}\}$, autrement dit le cas représenté par le diagramme suivant :
\begin{center}
\begin{tikzpicture}[scale=0.3]
\draw (0,0)--(9,0)--(9,9)--(0,9)--(0,0);
\draw (0,4.5)--(4.5,4.5)--(4.5,0);
\draw[dashed] (4.5,9)--(4.5,4.5)--(9,4.5);
\draw (2.25,2.25) node{$\gotn^-$};
\draw (2.25,6.75) node{$\g_{I_1}$};
\draw (6.75,6.75) node{$\gotp$};
\draw (6.75,2.25) node{$\g_{I_2}$};
\end{tikzpicture}
\end{center}
Dans ce cas, on a $I_1=\llbracket 1,n/2 \rrbracket$ et $I_2=\llbracket n/2+1,n \rrbracket$, d'où $\mathbf{I}=\{n/2\}$ et $\mathbf{M}_1=\{n\}$. On obtient
$$F_n^\bullet \propto F_{n,1} \; F_{n,2}$$
où $F_{n,1}=\Delta_{I_2, I_1}$ est le déterminant du quadrant bas gauche et $F_{n,2}=\Delta_{I_1, I_2}$ est le déterminant du quadrant haut droite. On n'exhibe pas de semi-invariant non trivial à partir des autres $F_j^\bullet$.
\end{itemize}
\label{raccenetc}
\end{exem}
\subsection{Factorisation des $F_m^{\bullet}$}
Dans cette section, on fixe $i \in \mathbf{I}$, $r:=\rho_i$, $m:=m_i$ ; on note $\kappa:=\kappa_i=\{k_1 < \ldots < k_r\}$ (voir notation \ref{nota}).
\begin{nota}
On rappelle que $\theta_J$ est défini en \ref{notaop}. On note $\mathfrak{T}= \llbracket 1,m \rrbracket$\index{$\mathfrak{T}, \mathfrak{T}_k$}. Pour tous $J,J' \in \mathcal{J}(m)$, et pour tout $k$, par \ref{jrond}, on a $\theta_J(J_k)=\theta_{J'}(J'_k)$. On note alors $\mathfrak{T}_k:=\theta_J(J_k)$ pour n'importe quel $J \in \mathcal{J}(m)$. Ceci implique en particulier que
si $J, J' \in \mathcal{J}(m)$ vérifient $J_{k}=J'_{k}$, alors $(\theta_J)_{|J_{k}}=(\theta_{J'})_{|J'_{k}}$.\label{Tgot}
\end{nota}
Grâce aux résultats de la section \ref{sec3}, on a une formule plus précise pour $F_m^\bullet$.
\begin{propn}
On a $\deg_{\gotn^-}F_m=m-i$, et
$$F_m^{\bullet}=\sum_{J \in \mathcal{J}(m)} \Delta_J^{\bullet}$$
\label{oui}
\end{propn}
\begin{proof}
On rappelle que $F_m=\sum_{J \; \text{tq} \; |J|=m} \Delta_J$ (équation \eqref{Fm}). Le lemme \ref{clé} montre que $\deg_{\gotn^-}\Delta_J=m-\max(j_k)$ et par le corollaire \ref{comb} (1), $\max(j_k) \geq i$, où l'égalité est vérifiée si et seulement si $J \in \mathcal{J}(m)$. On conclut par la remarque \ref{degnFj}.
\end{proof}
\begin{propn}
Soit $J \in \mathcal{J}(m)$. Soit $\sigma \in \mathfrak{S}(J)$ tel que $\prod_{l \in J} e_{l, \sigma(l)}$ soit un monôme de degré $m-i$ en $\gotn^-$. Alors pour tout $t \in \llbracket 1,r-1 \rrbracket$, on a $\sigma(J_{\llbracket k_t+1,k_{t+1}\rrbracket})=J_{\llbracket k_t,k_{t+1}-1\rrbracket}$ et $e_{l,\sigma(l)} \in \gotn^-$ pour tout $l \in J_{\llbracket k_t+1,k_{t+1}\rrbracket}$.
\label{segment}
\end{propn}
\begin{proof}
Si $r=1$, la proposition est vide. On suppose donc dans la suite que $r \geq 2$. Pour tout $k$, on a $j_k=\min(i_k,i)$, donc $j_k=i$ pour tout $k \in \kappa$. Soit $t \in \llbracket 1,r-1 \rrbracket$ et $\gotq_{\text{res}}=\gotp_{\text{res}} \ltimes \gotn^-_{\text{res}}$ la contraction parabolique de $\g_J$ définie par $\gotp_{\text{res}}=\gotp_J \oplus \gotn^-_{J_{\llbracket k_t, k_{t+1}\rrbracket}}=\gotp_J + \g_{J_{\llbracket k_t, k_{t+1}\rrbracket}}$ et $\gotn^-_{\text{res}}= \bigoplus_{e_{p,q} \in \g_J \setminus \gotp_{\text{res}}} \C e_{p,q}$, de sorte que la partition associée à cette contraction parabolique est $J_1 \sqcup \ldots \sqcup J_{k_t-1} \sqcup J_{\llbracket k_t,k_{t+1}\rrbracket} \sqcup J_{k_{t+1}+1} \sqcup \ldots \sqcup J_s$. En particulier, les espaces vectoriels $\gotq_{\text{res}}$ et $\gotq_J$ sont égaux.
\begin{center}
\begin{tikzpicture}[scale=0.6]
\draw (0,4)--(9,4)--(9,13)--(0,13)--(0,4);
\draw[dashed] (0,12)--(1,12);
\draw[dashed] (1,13)--(1,12);
\draw[dotted] (1.2,11.8)--(1.8,11.2);
\draw (0.5,12.5) node[scale=0.7]{$\g_{J_1}$};
\draw[dashed] (2,11)--(2,9)--(4,9);
\draw[dashed] (2,11)--(4,11)--(4,9);
\draw (3,10) node{$\g_{J_{k_t}}$};
\draw[dotted] (4.2,8.8)--(4.8,8.2);
\draw[dashed] (5,8)--(5,6)--(7,6);
\draw[dashed] (5,8)--(7,8)--(7,6);
\draw (6,7) node{$\g_{J_{k_{t+1}}}$};
\draw[dotted] (7.2,5.8)--(7.8,5.2);
\draw[dashed] (8,5)--(8,4);
\draw[dashed] (8,5)--(9,5);
\draw (8.5,4.5) node[scale=0.7]{$\g_{J_s}$};

\draw[color=red] (1.9,11.1)--(1.9,5.9)--(7.1,5.9)--(7.1,11.1)--(1.9,11.1);
\draw[color=red] (5,11.1)--(5.5,11.6);
\draw (6,12) node[color=red, scale=0.9]{$\g_{J_{\llbracket k_t, k_{t+1} \rrbracket}}$};
\draw (9,11)--(9.5,11.5);
\draw (10,12) node[scale=0.9]{$\g_{J}$};
\draw[color=blue] (2.05,8.95)--(2.05,6.05)--(4.95,6.05)--(4.95,8.95)--(2.05,8.95);
\draw[color=blue] (2.05,8.05)--(-2.05,7.05);
\draw (-3.05,6.65) node[color=blue, scale=0.9]{$\g_{J_{\llbracket k_t+1, k_{t+1}\rrbracket},J_{\llbracket k_t, k_{t+1}-1\rrbracket}}$};
\fill [pattern=horizontal lines]
	(2.1,8.9)--(2.1,6.1)--(4.9,6.1)--(4.9,7.9)--(3.9,8.9)--(2.1,8.9);
\fill [color=white]
	(2.3,7.9)--(4.6,7.9)--(4.6,7.0)--(2.3,7.0)--(2.3,7.9);
\fill [pattern=dots]
	(0.1,11.9)--(0.1,4.1)--(7.9,4.1)--(7.9,5.1)--(6.9,5.9)--(1.9,5.9)--(1.9,10.9)--(0.9,11.9)--(0.1,11.9);
\fill [color=white]
	(0.4,4.6)--(0.4,5.3)--(1.6,5.3)--(1.6,4.6)--(0.4,4.6);
\draw (1,5) node{$\gotn^-_{\text{res}}$};
\draw (3.5,7.5) node[scale=0.85]{$\gotn^-_{J_{\llbracket k_t, k_{t+1}\rrbracket}}$};
\draw[color=white] (15.05,8.05)--(15.05,7.05);;
\end{tikzpicture}
\end{center}
On a alors $\gotn^-_{J}=\gotn^-_{J_{\llbracket k_t, k_{t+1}\rrbracket}} \oplus \gotn^-_{{\text{res}}}$, et de plus, tout $e_{u,v}$ appartenant à $\gotn^-_{J}$ est soit dans $\gotn^-_{J_{\llbracket k_t, k_{t+1}\rrbracket}}$, soit dans $\gotn^-_{\text{res}}$. Ainsi 
$$\deg_{\gotn^-} \prod_{l \in J} e_{l, \sigma(l)} = \deg_{\gotn^-_{J_{\llbracket k_t, k_{t+1}\rrbracket}}} \prod_{l \in J} e_{l,\sigma(l)} + \deg_{\gotn^-_{\text{res}}} \prod_{l \in J} e_{l,\sigma(l)}.$$
%$\Card \left(\left\{e_{l,\sigma(l)} \in \gotn_{-J_{\llbracket k_t, k_{t+1}\rrbracket}}\right\}\right)=\Card \left(\left\{e_{l,\sigma(l)} \in \gotn_{-J_{\llbracket k_t, k_{t+1}\rrbracket}}, l \in J_{\llbracket k_t, k_{t+1}\rrbracket} \right\}\right)$
Notons $J'$ l'ensemble des $l \in J_{\llbracket k_t, k_{t+1}\rrbracket}$ tels que $\sigma(l) \in J_{\llbracket k_t, k_{t+1}\rrbracket}$. Alors 
$$\deg_{\gotn^-_{J_{\llbracket k_t, k_{t+1}\rrbracket}}} \prod_{l \in J} e_{l,\sigma(l)} = \deg_{\gotn^-_{J_{\llbracket k_t, k_{t+1}\rrbracket}}} \prod_{l \in J'} e_{l,\sigma(l)}.$$
Le produit $\prod_{l \in J'} e_{l,\sigma(l)}$ est lui-même un facteur d'un produit $\prod_{l \in J_{\llbracket k_t, k_{t+1}\rrbracket}} e_{l,\sigma'(l)}$ pour un certain $\sigma' \in \mathfrak{S}(J_{\llbracket k_t, k_{t+1}\rrbracket})$, d'où
\begin{align*}
\deg_{\gotn^-_{J_{\llbracket k_t, k_{t+1}\rrbracket}}} \prod_{l \in J} e_{l,\sigma(l)} &\leq \deg_{\gotn^-_{J_{\llbracket k_t, k_{t+1}\rrbracket}}} \prod_{l \in J_{\llbracket k_t, k_{t+1}\rrbracket}} e_{l,\sigma'(l)} \\
&\leq \deg_{\gotn^{-}} \Delta_{J_{\llbracket k_t, k_{t+1}\rrbracket}}.
\end{align*}
par la remarque \ref{rq1}. Ainsi par le lemme \ref{clé} (on rappelle que $j_{k_t}=j_{k_{t+1}}=i$ et par construction, $j_k \leq i$ pour tout $k$),
$$\deg_{\gotn^-_{J_{\llbracket k_t, k_{t+1}\rrbracket}}} \prod_{l \in J} e_{l,\sigma(l)} \leq \deg_{\gotn^{-}}\Delta_{J_{\llbracket k_t, k_{t+1}\rrbracket}}=\Card(J_{\llbracket k_t,k_{t+1}\rrbracket})-i.$$
De même, par le lemme \ref{clé}, on a $\deg_{\gotn^-_{\text{res}}} \prod_{l \in J} e_{l,\sigma(l)} \leq m-\Card(J_{\llbracket k_t,k_{t+1}\rrbracket})$. Le terme $\prod_{l \in J} e_{l, \sigma(l)}$ étant supposé de degré maximal en $\gotn^-$ égal à $m-i$, on a donc
$$\deg_{\gotn^-_{J_{\llbracket k_t, k_{t+1}\rrbracket}}} \prod_{l \in J} e_{l,\sigma(l)} = \Card\left(\left\{l \in J \, | \, e_{l,\sigma(l)} \in \gotn^-_{J_{\llbracket k_t, k_{t+1} \rrbracket}}\right\}\right) = \Card(J_{\llbracket k_t,k_{t+1}\rrbracket})-i \quad $$
$$\deg_{\gotn^-_{\text{res}}} \prod_{l \in J} e_{l,\sigma(l)} = \Card\left(\left\{l \in J \, | \, e_{l,\sigma(l)} \in \gotn^-_{\text{res}}\right\}\right) = m-\Card(J_{\llbracket k_t,k_{t+1}\rrbracket}).$$
On a $\gotn^-_{J_{\llbracket k_t, k_{t+1}\rrbracket}} \subset\g_{J_{\llbracket k_t+1, k_{t+1}\rrbracket},J_{\llbracket k_t, k_{t+1}-1\rrbracket}} $ et $\Card(J_{\llbracket k_t+1, k_{t+1}\rrbracket})=\Card(J_{\llbracket k_t, k_{t+1}-1\rrbracket})=\Card(J_{\llbracket k_t,k_{t+1}\rrbracket})-i$. Alors on a $\left\{ l \in J \, | \, e_{l,\sigma(l)} \in \gotn^-_{J_{\llbracket k_t, k_{t+1}\rrbracket}} \right\} \subset J_{\llbracket k_t+1, k_{t+1}\rrbracket}$, d'où l'égalité par égalité des cardinaux. Ainsi pour tout $l \in J_{\llbracket k_t+1, k_{t+1} \rrbracket}$, on a $\sigma(l) \in J_{\llbracket k_t, k_{t+1}-1 \rrbracket}$. Autrement dit, $\sigma\left( J_{\llbracket k_t+1, k_{t+1} \rrbracket}\right) \subset J_{\llbracket k_t, k_{t+1}-1 \rrbracket}$, et on a égalité par injectivité de $\sigma$ et égalité des cardinaux.
\end{proof}
\begin{propn}
Pour tout $J \in \mathcal{J}(m)$, $\Delta_J^{\bullet}$ est produit de $r$ facteurs homogènes non constants.
\label{fact}
\end{propn}
\begin{nota}
\label{nota46}
Pour tout $J \subset I$, on définit deux partitions :
\begin{itemize}
\item $J=J^{(1)} \, \sqcup \, \ldots \, \sqcup \, J^{(r-1)} \, \sqcup J^{(r)}$ où $J^{(t)}:=J_{\left\llbracket k_t+1, k_{t+1} \right\rrbracket}$\index{$J^{(t)}$, $J^{[t]}$, $j^{(t)}$, $j^{[t]}$,$\mathfrak{T}^{(t)}$, $\mathfrak{T}^{[t]}$} pour $1 \leq t \leq r-1$ et $J^{(r)}:=J \setminus \bigsqcup_{t=1}^{r-1} J^{(t)}=J_{\left\llbracket 1, k_1 \right\rrbracket \sqcup \left\llbracket k_r+1, s \right\rrbracket}$ ; pour tout $t \in \llbracket 1,r \rrbracket$, on note $j^{(t)}:=\Card J^{(t)}$ et $\mathfrak{T}^{(t)}:=\theta_J\left(J^{(t)}\right)$ (voir notation \ref{Tgot}),
\item $J=J^{[1]} \, \sqcup \, \ldots \, \sqcup \, J^{[r-1]} \, \sqcup J^{[r]}$ où $J^{[t]}:=J_{\left\llbracket k_t, k_{t+1}-1 \right\rrbracket}$ pour $1 \leq t \leq r-1$ et $J^{[r]}:=J \setminus \bigsqcup_{t=1}^{r-1} J^{[t]}=J_{\left\llbracket 1, k_1-1 \right\rrbracket \sqcup \left\llbracket k_r, s \right\rrbracket}$ ; pour tout $t \in \llbracket 1,r \rrbracket$, on note $j^{[t]}=\Card J^{[t]}$ et $\mathfrak{T}^{[t]}:=\theta_J\left(J^{[t]}\right)$.
\end{itemize}
Si $J \in \mathcal{J}(m)$, par définition de $\kappa$, pour tout $t \in \llbracket 1,r \rrbracket$, on a $j^{(t)}=j^{[t]}$.
%Définissons la partition $J=J^{(1)} \, \sqcup \, \ldots \, \sqcup \, J^{(r-1)} \, \sqcup J^{(r)}$ par $J^{(t)}=J_{\left\llbracket k_t+1, k_{t+1} \right\rrbracket}$ pour $1 \leq t \leq r-1$ et $J^{(r)}=J \setminus \bigsqcup_{t=1}^{r-1} J^{(t)}=J_{\left\llbracket 1, k_1 \right\rrbracket \sqcup \left\llbracket k_r+1, s \right\rrbracket}$. On obtient une partition $\mathfrak{T}= \mathfrak{T}^{(1)} \, \sqcup \, \ldots \, \sqcup \, \mathfrak{T}^{(r-1)} \, \sqcup \mathfrak{T}^{(r)}$ avec
%$$\mathfrak{T}^{(t)} := \theta_J(J^{(t)})$$
%qui ne dépend pas de l'ensemble $J \in \mathcal{J}(m)$ considéré. On note aussi $\mathfrak{T}^{(t)}=\Card \left(J^{(t)} \right)=\Card \left(\mathfrak{T}^{(t)} \right)$. Remarquons que pour tout $t \in \llbracket 1, r-1 \rrbracket$, on a $\mathfrak{T}^{(t)}-i := \left\{ v - i \, | \, v \in \mathfrak{T}^{(t)} \right\} = \theta_J(J_{\left\llbracket k_t, k_{t+1}-1 \right\rrbracket})$. On pose alors par convention
%$$\mathfrak{T}^{(r)}-i:=\mathfrak{T} \setminus \bigsqcup_{t \in \llbracket 1,r-1 \rrbracket} \left(\mathfrak{T}^{(t)}-i\right) = \theta_J\left(J \setminus \bigsqcup_{t \in \llbracket 1,r-1 \rrbracket} J_{\left\llbracket k_t, k_{t+1}-1 \right\rrbracket}\right)=\theta_J\left(J_{\left\llbracket 1, k_1-1 \right\rrbracket \sqcup \left\llbracket k_r, s \right\rrbracket}\right)$$
\label{simp}
\end{nota}
\begin{proof}[Preuve de la proposition \ref{fact}]
Notons $S_J$ le sous-ensemble de $\mathfrak{S}(J)$ tel que 
$$\Delta_J^{\bullet}=\sum_{\sigma \in S_J} \varepsilon(\sigma) \; \prod_{l \in J} e_{l, \sigma(l)}$$
c'est-à-dire l'ensemble des $\sigma \in \mathfrak{S}(J)$ tel que $\deg_{\gotn^-} \prod_{l \in J} e_{l,\sigma(l)} = \deg_{\gotn^-} \Delta_J$. On a montré avec les propositions \ref{oui} et \ref{segment} que la permutation $\sigma$ est dans $S_J$ si et seulement si toutes les conditions suivantes sont vérifiées :
%$\sigma$ est de la forme suivante : $\sigma(\mu_t)=\mu_{f(t)}$ où $f \in \mathfrak{S}_j$ est telle que
%\begin{equation}
%\Card \left( \left\{ l \in \llbracket 1,m \rrbracket \; | \; e_{\theta_J^{-1}(l), \, \theta_J^{-1} \circ \varphi(l)} \in \gotn^-\right\} \right) = m-i \label{trucdudessus}
%\end{equation}
%Notons $T$ l'ensemble des $\varphi \in \mathfrak{S}_m$ vérifiant \eqref{trucdudessus}.
%et on a montré que cette condition était équivalente à :
\begin{enumerate}
\item[(C1)] pour tout $t \in \llbracket 1,r \rrbracket$,
$\sigma\left(J^{(t)} \right) = J^{[t]}$,
\item[(C2)] pour tout $t \in \llbracket 1,r-1 \rrbracket$ et pour tout $l \in J^{(t)}$, $e_{l, \, \sigma(l)} \in \gotn^-$
\item[(C3)] $\Card \left( \left\{ l \in J^{(r)} \; | \; e_{l, \, \sigma(l)} \in \gotn^-\right\} \right) = m-i-\sum_{t \in \llbracket 1,r-1 \rrbracket} j^{(t)}=j^{(r)}-i$
\end{enumerate}
Soit l'application
$$\sigma \in S_J \longmapsto \left(\sigma_{|J^{(t)}}\right)_{1 \leq t \leq r} \in S_J^{(1)} \times \ldots \times S_J^{(r-1)} \times S_J^{(r)}$$
où $S_J^{(t)}$\index{$S_J^{(t)}$} est l'ensemble des bijections de $J^{(t)}$ dans $J^{[t]}$ vérifiant (C2) si $t < r$ et (C3) si $t=r$. Par (C1), cette application est bijective.
Soit $\eta \in \mathfrak{S}_m$\index{$\eta$} définie par :
\begin{itemize}
\item $\eta_{|\mathfrak{T}^{[t]}}=\mathfrak{T}^{(t)}$ pour tout $t \in \llbracket 1,r \rrbracket$,
\item $\eta$ est croissante sur chacun des $\mathfrak{T}^{[t]}$,
\end{itemize}
et $\eta':=\theta_J^{-1} \eta \theta_J \in \mathfrak{S}(J)$. Alors pour tout $\sigma \in \mathfrak{S}_m$ et pour tout $t \in \llbracket 1,r \rrbracket$, on a $(\eta' \circ \sigma)_{|J^{(t)}} = \eta' \circ \sigma_{|J^{(t)}}$ qui est une permutation de $J^{(t)}$, d'où $\varepsilon(\eta' \circ \sigma)=\varepsilon(\eta' \circ \sigma_{|J^{(1)}}) \ldots \varepsilon(\eta' \circ \sigma_{|J^{(r)}})$. Finalement, on obtient
\begin{align*}
\varepsilon(\sigma)&=\varepsilon(\eta')\; \varepsilon(\eta' \circ \sigma_{|J^{(1)}}) \ldots \varepsilon(\eta' \circ \sigma_{|J^{(r)}})\\
&=\varepsilon(\eta)\; \varepsilon(\sigma_{|J^{(1)}}) \ldots \varepsilon(\sigma_{|J^{(r)}})
\end{align*}
(on rappelle que $\varepsilon(\sigma)$ a été défini à la section \ref{sec2} pour une bijection quelconque entre sous-ensembles de $\N$). Ainsi :
\begin{align*}
\Delta_J^{\bullet}&= \sum_{\sigma \in S_J} \varepsilon(\sigma) \; \prod_{l \in J} e_{l,\sigma(l)} \\
&= \varepsilon(\eta) \sum_{(\sigma_1, \ldots, \sigma_{r}) \in S_J^{(1)} \times \ldots \times S_J^{(r)}} \; \prod_{t=1}^{r} \left( \varepsilon(\sigma_t) \prod_{l \in J^{(t)}} e_{l, \, \sigma_t(l)} \right)
\end{align*}
et donc
\begin{equation}
\Delta_J^{\bullet}=\varepsilon(\eta) \prod_{t=1}^{r} \left( \sum_{\sigma_t \in S_J^{(t)}} \varepsilon(\sigma_t) \prod_{l \in J^{(t)}} e_{l, \, \sigma_t(l)} \right)
\label{factorisation}
\end{equation}
Cette dernière égalité découle du fait général suivant : soit $A$ un anneau associatif et commutatif, $E_1, \ldots, E_d$ des ensembles finis et pour tout $y_i \in E_i$ avec $i \in \llbracket 1,d \rrbracket$, soit $a_{y_i}$ un élément de $A$, alors
\begin{equation}
\sum_{(y_1, \ldots, y_d) \in E_1 \times \ldots \times E_d} \prod_{i=1}^d a_{y_i} = \prod_{i=1}^d \sum_{y_i \in E_i} a_{y_i}
\label{prepa}
\end{equation}
\end{proof}
En utilisant les notations de la sous-section \ref{sec2}, l'égalité \eqref{factorisation} se simplifie en
\begin{equation}
\Delta_J^\bullet = \varepsilon(\eta) \prod_{t=1}^r \Delta_{J^{(t)},J^{[t]}}^\bullet \label{deltajpt}
\end{equation}
où l'on a donc
\begin{equation}
\Delta_{J^{(t)},J^{[t]}}^\bullet = \sum_{\sigma_t \in S_J^{(t)}} \varepsilon(\sigma_t) \prod_{l \in J^{(t)}} e_{l, \, \sigma_t(l)} \label{deltattpt}
\end{equation}
Par la condition (C2), on a $\Delta_{J^{(t)},J^{[t]}}^\bullet \in \Sym(\gotn^-)$ pour $t<r$.
\begin{rema}
\label{remadiag}
On présente une manière plus visuelle de voir pourquoi $\Delta_{J^{(t)},J^{[t]}}^\bullet$, de degré $j^{(t)}$ (équation \eqref{deltattpt}), est dans $\Sym(\gotn^-)$ pour $t<r$. Par définition, pour tout $k \in \llbracket k_t+1, k_{t+1}-1 \rrbracket$, on a $|J_k| \leq i = |J_{k_t}|=|J_{k_{t+1}}|$. Pour $t < r$, le sous-espace $\g_{J^{(t)},J^{[t]}}$ est alors de la forme suivante
\begin{center}
\begin{tikzpicture}[scale=0.6]
\draw (0,0)--(0,8)--(8,8)--(8,0)--(0,0);
\draw (0,8)--(8,0);
\draw (2,8)--(2,7)--(3,7)--(3,6)--(4,6)--(4,4)--(6,4)--(6,3)--(7,3)--(7,2)--(8,2);
\fill [pattern=horizontal lines] (0.1,0.1)--(0.1,7.9)--(1.9,7.9)--(1.9,6.9)--(2.9,6.9)--(2.9,5.9)--(3.9,5.9)--(3.9,3.9)--(5.9,3.9)--(5.9,2.9)--(6.9,2.9)--(6.9,1.9)--(7.9,1.9)--(7.9,0.1)--(0.1,0.1);
\fill [color=white] (0.9,1.5)--(3.1,1.5)--(3.1,2.5)--(0.9,2.5)--(0.9,1.5);
\draw (2,2) node{$\gotn^-_{J^{(t)},J^{[t]}}$};
\draw (6,6) node{$\gotp_{J^{(t)},J^{[t]}}$};
%\draw[dashed] (0,12)--(1,12);
%\draw[dashed] (1,13)--(1,12);
%\draw[dotted] (1.2,11.8)--(1.8,11.2);
%\draw (0.5,12.5) node[scale=0.7]{$\g_{J_1}$};
%\draw[dashed] (2,11)--(2,9)--(4,9);
%\draw[dashed] (2,11)--(4,11)--(4,9);
%\draw (3,10) node{$\g_{J_{k_t}}$};
%\draw[dotted] (4.2,8.8)--(4.8,8.2);
%\draw[dashed] (5,8)--(5,6)--(7,6);
%\draw[dashed] (5,8)--(7,8)--(7,6);
%\draw (6,7) node{$\g_{J_{k_{t+1}}}$};
%\draw[dotted] (7.2,5.8)--(7.8,5.2);
%\draw[dashed] (8,5)--(8,4);
%\draw[dashed] (8,5)--(9,5);
%\draw (8.5,4.5) node[scale=0.7]{$\g_{J_s}$};
%
%\draw[color=red] (2,11)--(2,6)--(7,6)--(7,11)--(2,11);
%\draw[color=red] (5,11)--(5.5,11.5);
%\draw (6,12) node[color=red, scale=0.9]{$\g_{J_{\llbracket k_t, k_{t+1} \rrbracket}}$};
%\draw (9,11)--(9.5,11.5);
%\draw (10,12) node[scale=0.9]{$\g_{J}$};
%\draw[color=blue] (2.05,8.95)--(2.05,6.05)--(4.95,6.05)--(4.95,8.95)--(2.05,8.95);
%\draw[color=blue] (2.05,8.05)--(-2.05,7.05);
%\draw (-3.05,6.65) node[color=blue, scale=0.9]{$\g_{J_{\llbracket k_t+1, k_{t+1}\rrbracket},J_{\llbracket k_t, k_{t+1}-1\rrbracket}}$};
%\fill [pattern=horizontal lines]
%	(2.1,8.9)--(2.1,6.1)--(4.9,6.1)--(4.9,7.9)--(3.9,8.9)--(2.1,8.9);
%\fill [color=white]
%	(2.3,7.9)--(4.6,7.9)--(4.6,7.0)--(2.3,7.0)--(2.3,7.9);
%\fill [pattern=dots]
%	(0.1,11.9)--(0.1,4.1)--(7.9,4.1)--(7.9,5.1)--(6.9,5.9)--(1.9,5.9)--(1.9,10.9)--(0.9,11.9)--(0.1,11.9);
%\draw (1,5) node{$\gotn^-_{\text{res}}$};
%\draw (3.5,7.5) node[scale=0.85]{$\gotn^-_{J_{\llbracket k_t, k_{t+1}\rrbracket}}$};
%\draw[color=white] (15.05,8.05)--(15.05,7.05);;
\end{tikzpicture}
\end{center}
Le produit dans $\Sym(\gotq)$ des $e_{u,v}$ "diagonaux" (la diagonale vue dans $\g_{J^{(t)},J^{[t]}}$, tracée dans la figure précédente) est dans $\Sym(\gotn^-)$, et donc les monômes de $\Delta_{J^{(t)},J^{[t]}}$ de degré maximal en $\gotn^-$ sont en fait dans $\Sym(\gotn^-)$. Le facteur $\Delta_{J^{(t)},J^{[t]}}^\bullet$ est donc dans $\Sym(\gotn^-)$.
\end{rema}
\begin{propn}
$F_m^{\bullet}$ est produit de $r$ facteurs homogènes non constants.
\label{Fact}
\end{propn}
\begin{proof}
De ce qui précède, on obtient (voir les notations de la définition \ref{typej} et la notation \ref{jrond}) :
\begin{align*}
F_m^{\bullet}&=\sum_{J \in \mathcal{J}(m)} \Delta_J^{\bullet}=\sum_{(J_k) \in \prod_k \mathcal{J}_k} \Delta_J^{\bullet}= \\
&=\varepsilon(\eta) \sum_{(J_k) \in \prod_k \mathcal{J}_k} \prod_{t=1}^r \Delta_{J^{(t)},J^{[t]}}^\bullet
\end{align*}
(on continue de noter $J=\bigsqcup_k J_k$). Pour tout $J \in \mathcal{J}(m)$ et $k \in \kappa$ (voir notation \ref{nota}), on a $j_k=i_k=i$, ainsi pour $k \in \kappa$, l'ensemble $\mathcal{J}_k=\{I_k\}$ n'a qu'un élément. L'ensemble $J \in \mathcal{J}(m)$ est donc uniquement déterminé par les $J_k, k \notin \kappa$. Si pour $k \in \kappa$, on note toujours $J_k = I_k$ dans la somme, alors $J_k$ n'est plus une variable et on a
$$F_m^{\bullet}=\varepsilon(\eta) \sum_{(J_k) \in \prod_{k \notin \kappa} \mathcal{J}_k} \prod_{t=1}^r \Delta_{J^{(t)},J^{[t]}}^\bullet$$
On a $\llbracket 1,s \rrbracket \setminus \kappa = \left(\bigsqcup_{t=1}^{r-1} \llbracket k_{t}+1,k_{t+1} \rrbracket\right) \sqcup \left(\llbracket 1,k_1-1 \rrbracket \sqcup \llbracket k_r+1,s \rrbracket\right)$. Pour tout $t \in \llbracket 1,r-1 \rrbracket$, les ensembles $J^{(t)}$ et $J^{[t]}$ ne dépendent que des choix de $J_k \in \mathcal{J}_k$ pour $k \in \llbracket k_t +1, k_{t+1}-1 \rrbracket$. De la même manière, les ensembles $J^{(r)}$ et $J^{[r]}$ que des choix de $J_k \in \mathcal{J}_k$ pour $k \in \llbracket 1, k_1-1 \rrbracket \sqcup \llbracket k_r+1, s \rrbracket$. Par l'équation \eqref{prepa}, on peut donc écrire :
$$
F_m^{\bullet}=\varepsilon(\eta) \prod_{t=1}^{r}\sum_{(J_k) \in \prod_{k \in K_t} \mathcal{J}_k} \Delta_{J^{(t)},J^{[t]}}^\bullet
$$
où $K_t=\left\{ \begin{array}{ll}
\llbracket k_t+1, k_{t+1}-1 \rrbracket&\text{si } t<r\\\llbracket 1, k_1-1 \rrbracket \sqcup \llbracket k_r+1, s \rrbracket& \text{si } t=r
\end{array}\right.$\index{$K_t$}. On pose alors pour $1 \leq t \leq r$\index{$F_{m,t}$}
\begin{equation}
F_{m,t}:= \sum_{(J_k) \in \prod_{k \in K_t} \mathcal{J}_{k}} \; \Delta_{J^{(t)},J^{[t]}}^\bullet\label{expmochesemiinv} 
\end{equation}
On a alors
$$ F_m^{\bullet}=c_m \prod_{t =1}^r F_{m,t}$$
avec $c_m=\varepsilon(\eta) \in \C^\times$\index{$c_m$} ($\eta$ dépend de $m$). 
\end{proof}
Les $F_{m,t}$ sont bien des semi-invariants non constants (par linéaire indépendance des $\Delta_{J^{(t)},J^{[t]}}^\bullet$) et dans $\Sym(\gotn^-)$ pour $t < r$ (car les $\Delta_{J^{(t)},J^{[t]}}^\bullet$ sont dans $\Sym(\gotn^-)$). Le point (1) du théorème \ref{semiinv} est donc vérifié.
\begin{defi}
Si $m=m_i \in \mathbf{M}_1$, on note $r_m:=\rho_i$. Pour $m \in \llbracket 1,n \rrbracket \setminus \mathbf{M}_1$, on note $r_m=1$ et $c_m=1$, de sorte que pour tout $m \in \llbracket 1,n \rrbracket$, on a
\begin{equation}
F_m^\bullet=c_m\prod_{t=1}^{r_m} F_{m,t} \label{decomp}
\end{equation}
On note alors $\mathbf{F}$\index{$\mathbf{F}$} la famille composée de l'ensemble des $F_{m,t}$ pour $1 \leq m \leq n$ et $1 \leq t \leq r_m$.
\label{familleF}
\end{defi}
\begin{exem}
Dans le cas de l'exemple \ref{leexemple}, où $\gotl \simeq \gl_4 \times \gl_1 \times \gl_4 \times \gl_2 \times \gl_1$, on a $\mathbf{I}=\{1,2,4\}$ et $\mathbf{M}_1=\{5,8,12\}$. On obtient donc
$$F_5^\bullet \propto F_{5,1} F_{5,2} \qquad F_8^\bullet \propto F_{8,1} \qquad F_{12}^\bullet \propto F_{12,1} F_{12,2}$$
où le nombre de facteurs de $F_{m_i}^\bullet$ est le nombre d'occurrences de $\gl_i$ dans $\gotl$. Par exemple, le semi-invariant $F_{5,1}$ est la somme des $\Delta_{\{a,b,12\},\{5,a,b\}}^\bullet$ pour $a \in I_3$ et $b \in I_4$, et $F_{12,1}=\Delta_{\llbracket 5,9 \rrbracket, \llbracket 1,5 \rrbracket}^\bullet$.
\label{exsi}
\end{exem}
\label{ssec2.2}
\subsection{Poids des semi-invariants $F_{m,t}$}
\label{ssec4.2}
Pour $i \in \mathbf{I}$, on a posé $m:=m_i$, $r:=\rho_i$ et $\lambda_{m,t}$ le poids du semi-invariant $F_{m,t}$ pour tout $t \in \llbracket 1,r \rrbracket$. On rappelle également que $\Lambda:=\Lambda(\gotq)$ est l'ensemble des poids de semi-invariants de $\Sym(\gotq)$.
\begin{nota}
Pour $m \in \llbracket 1,n \rrbracket \setminus \mathbf{M}_1$, on note $\lambda_{m,1}=0$, de sorte que pour tout $m \in \llbracket 1,n \rrbracket$ et $t \in \llbracket 1,r_m \rrbracket$, le semi-invariant $F_{m,t}$ soit de poids $\lambda_{m,t}$.
\end{nota}
Par la proposition \ref{dixmier}, les $F_{m,t}$ sont des semi-invariants. Calculons leurs poids. Si $r=1$, il n'y a qu'un seul facteur (qui est $F_{m,r}$) qui est donc un invariant (donc de poids nul), il suffit donc de s'intéresser au cas $r=\rho_i \geq 2$. Comme $F_m^\bullet$ est un invariant, on a $\sum_{t=1}^r \lambda_{m,t}=0$, donc il suffit de s'intéresser aux poids $\lambda_{m,t}$ pour $1 \leq t \leq r-1$.\par
On rappelle que $I_{\gotp}=\left\{0=\iota_0 < \iota_1 < \ldots < \iota_{s-1}< \iota_{s}=n\right\}$ est la partition de la sous-algèbre parabolique $\gotp$, vérifiant $I_k=\llbracket \iota_{k-1}+1, \iota_{k}\rrbracket$, et que $\kappa=\kappa_i=\left\{k_1 < \ldots < k_r\right\}$ est l'ensemble des $k$ tels que $i_k=i$ (notation \ref{nota}). On rappelle également que les $\varpi_k$ pour $1 \leq k \leq n-1$ sont les poids fondamentaux et que l'on a posé $\varpi_0=\varpi_n=0$.
\begin{lem}
On a $\Lambda \subset \bigoplus_{\iota \in I_{\gotp}} \rat \varpi_{\iota}$.
\label{poidsgen}
\end{lem}
%\begin{rema}
%\label{remadimpor}
%Si $\gotk$ est une algèbre de Lie isomorphe à $\g$ comme $\h$-module, on a $\Lambda(\gotk) \subset \bigoplus_{\ell=1}^{n-1} \rat \varpi_{\ell}$ (l'action de $\h$ dans $\gotq$ est diagonalisable à valeurs propres entières).
%\end{rema}
\begin{proof}
Par la propriété \ref{qprime}, on a
$$\gotq' \oplus \gotz(\gotq) = \left(\bigoplus_{p \neq q} \C e_{p,q}  \oplus \bigoplus_{\iota \notin I_{\gotp}} \C h_\iota\right) \oplus \C \id$$
Par la propriété \ref{inclusion}, on a également $\gotq' \oplus \gotz(\gotq) \subset \gotq_\Lambda$. Ainsi pour tout $\lambda \in \Lambda$, on a $\lambda(\gotq' \oplus \gotz(\gotq))=0$. Autrement dit, $\lambda \in \bigoplus_{\iota \in I_{\gotp}} \C \varpi_{\iota}$. Or pour tout $e_{u,v}$ et pour tout $\iota \in I_{\gotp} \setminus \{0,n\}$, on a $h_{\iota} \cdot e_{u,v}=[h_\iota, e_{u,v}]_{\gotq} \in \rat \, e_{u, v}$, d'où le lemme.
\end{proof}
\begin{propn}
Pour tout $t \in \llbracket 1, r-1 \rrbracket$, le poids de $F_{m,t}$ est\index{$\lambda_{m,t}$}
\begin{equation}
\lambda_{m,t}=w_{{k_t}}-w_{{k_{t+1}}}
\label{pdssi}
\end{equation}
où $w_k=\varpi_{\iota_{k-1}}-\varpi_{\iota_{k}}$. Les $\lambda_{m,t}$ vérifient les propriétés (2) et (3) du théorème \ref{semiinv}. On a également $\lambda_{m,r}=w_{{k_r}}-w_{{k_{1}}}$.
\end{propn}
\begin{proof}
%Puisque $F_m^{\bullet}=\prod_{t =1}^r F_{m,t}$ est un invariant, on a $\lambda_m^{(r)}=-\sum_{t=1}^{r-1}\lambda_{m,t}$, d'où $L_{i}$ est engendré par les $\lambda_{m,t}$, $1 \leq t \leq r-1$. Pour montrer (2) du théorème \ref{semiinv}, il suffit de montrer que les $\left(\lambda_{m,t}\right)_{1 \leq t \leq r-1}$ sont linéairement indépendants.\par
La formule pour $\lambda_{m,r}$ résulte de la formule pour $\lambda_{m,t}$ avec $t < r$ et du fait que $\sum_{t=1}^{r} \lambda_{m,t}=0$.\par
Par le lemme \ref{poidsgen}, il suffit de calculer $h_{\iota_k} \cdot F_{m,t}$ pour $k \in \llbracket 1,s-1 \rrbracket$. Rappelons (équation \eqref{expmochesemiinv}) que pour $t \in \llbracket 1, r-1 \rrbracket$
$$
F_{m,t}:= \sum_{(J_k) \in \prod_{k \in K_t} \mathcal{J}_{k}} \; \Delta_{J^{(t)},J^{[t]}}^\bullet
$$
Alors
$$
h_{\iota_k} \cdot F_{m,t}= \sum_{(J_k) \in \prod_{k \in K_t} \mathcal{J}_{k}} h_{\iota_k} \cdot \Delta_{J^{(t)},J^{[t]}}^\bullet
$$
Fixons $J_k \in \mathcal{J}(m)_k$ pour tout $k \in K_t=\llbracket k_t+1, k_{t+1}-1 \rrbracket$. Alors (par l'équation \eqref{deltattpt})
\begin{align*}
h_{\iota_k} \cdot \Delta_{J^{(t)},J^{[t]}}^\bullet &= h_{\iota_k} \cdot \left( \sum_{\sigma_t \in S_J^{(t)}} \varepsilon(\sigma_t) \prod_{l \in J^{(t)}} e_{l, \, \sigma_t(l)} \right)\\
&= \sum_{\sigma_t \in S_J^{(t)}} \varepsilon(\sigma_t) \sum_{l \in J^{(t)}} h_{\iota_k} \cdot e_{l, \, \sigma_t(l)} \prod_{l' \in J^{(t)} \setminus \{l\} } e_{l', \, \sigma_t(l')}
\end{align*}
Rappelons que $\gotq$ et $\g$ sont isomorphes en tant que $\h$-modules. On a donc :
$$ h_{\iota_k} \cdot e_{l, \, \sigma_t(l)} = \left(\delta_{\iota_k, l}-\delta_{\iota_k, \sigma_t(l)}-\delta_{\iota_k+1, l}+\delta_{\iota_k+1, \sigma_t(l)}\right)  e_{l, \, \sigma_t(l)}$$
Donc 
$$\sum_{l \in J^{(t)}} h_{\iota_k} \cdot e_{l, \, \sigma_t(l)} \prod_{l' \in J^{(t)} \setminus \{l\} } e_{l', \, \sigma_t(l')} = \delta \prod_{l \in J^{(t)}} e_{l, \, \sigma_t(l)}$$
où
\begin{align*}
\delta&= \sum_{l \in J^{(t)}} \left(\delta_{\iota_k, l}-\delta_{\iota_k, \sigma_t(l)}-\delta_{\iota_k+1, l}+\delta_{\iota_k+1, \sigma_t(l)}\right)\\
&=\delta_{\iota_k \in J^{(t)}}-\delta_{\iota_k \in J^{[t]}}-\delta_{\iota_k+1 \in J^{(t)}}+\delta_{\iota_k+1 \in J^{[t]}}\\
&= \delta_{\iota_k \in J_{\llbracket k_t+1, k_{t+1} \rrbracket}} - \delta_{\iota_k \in J_{\llbracket k_t, k_{t+1}-1 \rrbracket}} - \delta_{\iota_k+1 \in J_{\llbracket k_t+1, k_{t+1} \rrbracket}}
+ \delta_{\iota_k+1 \in J_{\llbracket k_t, k_{t+1}-1 \rrbracket}}\\
&=\delta_{\iota_k \in J_{k_{t+1}}} - \delta_{\iota_k \in J_{k_{t}}}- \delta_{\iota_k+1 \in J_{k_{t+1}}}+\delta_{\iota_k+1 \in J_{k_{t}}}
\end{align*}
Or pour tout $u$, $k_u \in \kappa$, donc $J_{k_u}=I_{k_u}$, d'où
\begin{align*}
\delta &=\delta_{\iota_k \in I_{k_{t+1}}} - \delta_{\iota_k \in I_{k_{t}}}- \delta_{\iota_k+1 \in I_{k_{t+1}}}+\delta_{\iota_k+1 \in I_{k_{t}}}\\
&= \delta_{k,k_{t+1}} - \delta_{k,k_t}- \delta_{k,k_{t+1}-1}+\delta_{k,k_t-1}
\end{align*}
par définition des $\iota_k$ (voir notation \ref{nota}). Ainsi $\delta$ ne dépend en fait pas des choix de $J_k$ et $F_{m,t}$ est de poids $\lambda_{m,t}=w_{{k_t}}-w_{{k_{t+1}}}$.
Posons $e_k=\varpi_{\iota_k}$, avec $e_0=e_s=0$. La famille des $e_k$, $1 \leq k \leq s-1$ est libre. La famille des $w_k=e_{k-1} - e_{k}$, $1 \leq k \leq s$ est donc de rang $s-1$, et vérifie $\sum_k w_k=0$.
\begin{exem}
Dans l'exemple \ref{leexemple}, les poids appartiennent au réseau $\rat \varpi_4 \oplus \rat \varpi_5 \oplus \rat \varpi_9 \oplus \rat \varpi_{11}$ que l'on représentera sur le diagramme matriciel comme suit :
\begin{center}
\begin{tikzpicture}[scale=0.5]
\foreach \k in {1,2,...,11}
	{\draw[color=gray!20]  (0,\k)--(10,\k);
	\draw[color=gray!20] (\k,0)--(\k,10);}
\draw (0,0)--(12,0);
\draw (0,0)--(0,12);
\draw (0,12)--(12,12);
\draw (12,0)--(12,12);
\draw[ultra thick] (0,8)--(4,8)--(4,7)--(5,7)--(5,3)--(9,3)--(9,1)--(11,1)--(11,0);
\draw (3,3) node{\LARGE $\gotn^-$};
\draw (9,9) node{\LARGE $\gotp$};
\draw[dashed] (4,12)--(4,8)--(5,8)--(5,7)--(9,7)--(9,3)--(11,3)--(11,1)--(12,1);
%\draw (2,10) node{\Large $\g_{I_1}$};
%\draw (4.5,7.5) node{\Large $\g_{I_2}$};
%\draw (7,5) node{\Large $\g_{I_3}$};
%\draw (10,2) node{\Large $\g_{I_4}$};
%\draw (11.5,0.5) node{\Large $\g_{I_5}$};
\draw[->] (3.5,8.5)--(4.5,7.5);
\draw (4.5,8.5) node{\small $\varpi_4$};
\draw[->] (4.5,7.5)--(5.5,6.5);
\draw (5.5,7.5) node{\small $\varpi_5$};
\draw[->] (8.5,3.5)--(9.5,2.5);
\draw (9.5,3.5) node{\small $\varpi_9$};
\draw[->] (10.5,1.5)--(11.5,0.5);
\draw (11.5,1.5) node{\small $\varpi_{11}$};
\end{tikzpicture}
\end{center}
On a alors
\begin{itemize}
\item $e_0=0$, $e_1=\varpi_4$, $e_2=\varpi_5$, $e_3=\varpi_9$ et $e_4=\varpi_{11}$, $e_5=0$,
\item $w_1=-\varpi_4$, $w_2=\varpi_4-\varpi_5$, $w_3=\varpi_5-\varpi_9$, $w_4=\varpi_9-\varpi_{11}$, $w_5=\varpi_{11}$,
\item $\lambda_{5,1}=w_2-w_5$ et $\lambda_{12,1}=w_1-w_3$ ; on rappelle que $m=5$ (respectivement $m=12$) est associé à $i=1$ (respectivement à $i=4$).
\end{itemize}
\end{exem}
%Montrons alors que quelque soit $\kappa=\{k_1 < \ldots < k_r \} \subset \llbracket 1,s \rrbracket$, la famille des $\lambda_{m,t}=w_{k_t}-w_{k_{t+1}}$, $1 \leq k \leq r-1$ est libre.
Montrons maintenant les points (2) et (3) du théorème \ref{semiinv}. Supposons que $\sum_{t=1}^{r-1} x_t \lambda_{m,t}=0$ avec $x_t \in \C$ et posons $x_0=x_r=0$. On a donc $\sum_{t=1}^{r-1} x_t (w_{k_t}-w_{k_{t+1}})=0$, c'est-à-dire
$$\sum_{t=1}^r (x_t-x_{t-1}) w_{k_t}=0$$
Comme la famille des $w_k$ est de rang $s-1$ et vérifie $\sum_{k=1}^s w_k=0$, on a $x_1-x_0 = x_2-x_1= \ldots = x_{r-1}-x_{r-2}=x_r-x_{r-1}$, ce qui implique que tous les $x_t$ sont nuls. En particulier, les $\lambda_{m,t}$ sont linéairement indépendants, ce qui montre (2) du théorème \ref{semiinv}.\par
Pour tout $j \in \llbracket 1, \imax \rrbracket$, soit $\lambda_j \in \vect((\lambda_{m_j,t})_t)$. Supposons que $\sum_j \lambda_j=0$ et montrons que tous les $\lambda_j$ sont nuls. Pour tout $j \in \llbracket 1, \imax \rrbracket$, on a
$$\vect((\lambda_{m_j,t})_t) = \vect((w_{k_{j,t}}-w_{k_{j,t+1}})_{t<\rho_j}) \subset \vect((w_{k_{j,t}})_t)$$
où l'on rappelle que $k_{j,t}$ est le $k_t$ correspondant à $\kappa_j$.  Pour tout $j$, on peut donc écrire $\lambda_j=\sum_t x_{j,t} w_{k_{j,t}}$, et on a donc
\begin{equation}
\sum_{j,t} x_{j,t} w_{k_{j,t}}=0 \label{summand}
\end{equation}
De plus, l'ensemble des $w_{k_{j,t}}$ pour $j \in \llbracket 1, \imax \rrbracket$ et $t \in \llbracket 1, \rho_j \rrbracket$ est l'ensemble des $w_k$ pour $k \in \llbracket 1,s \rrbracket$ (car $\bigsqcup_j \kappa_j = \llbracket 1,s \rrbracket$). Ainsi l'équation \eqref{summand} se réécrit
$$\sum_{k} x_{k} w_{k}=0$$
où $x_k=x_{j,t}$ si $k=k_{j,t}$ (tout $k \in \llbracket 1,s \rrbracket$ s'écrit de manière unique $k=k_{j,t}$ pour un certain $(j,t)$). Comme les $w_k$ forment une famille de rang $s-1$ telle que $\sum_k w_k=0$, tous les $x_k$ sont égaux, et donc pour tout $j$, on obtient $\lambda_j \in \C \sum_t w_{k_{j,t}}$. Or pour tout $j$, on a $\C \sum_t w_{k_{j,t}} \cap \vect((w_{k_{j,t}}-w_{k_{j,t+1}})_{t<\rho_j})=0$ . En effet, 
\begin{itemize}
\item si l'ensemble des $w_{k_{j,t}}$ pour $t \in \llbracket 1, \rho_j \rrbracket$ est l'ensemble des $w_k$ pour $k \in \llbracket 1,s \rrbracket$ (c'est-à-dire, si $\kappa_j=\llbracket 1,s \rrbracket$), alors $\sum_t w_{k_{j,t}}=\sum_k w_k=0$,
\item sinon, l'ensemble des $w_{k_{j,t}}$ pour $t \in \llbracket 1, \rho_j \rrbracket$ est strictement inclus dans l'ensemble des $w_k$ pour $k \in \llbracket 1,s \rrbracket$, et donc la famille des $w_{k_{j,t}}$ pour $t \in \llbracket 1, \rho_j \rrbracket$ est linéairement indépendante donc $\sum_t w_{k_{j,t}} \notin \vect((w_{k_{j,t}}-w_{k_{j,t+1}})_{t<\rho_j})$.
\end{itemize}
Ainsi $\lambda_j=0$ pour tout $j$ et les $\vect((\lambda_{m_j,t})_t)$ sont bien en somme directe. On a donc le point (3) du théorème \ref{semiinv}.
\end{proof}
\subsection{Indépendance algébrique}
\label{ssec4.3}
%\begin{nota}
%Pour $m=m_i \in \mathbf{M}$, on note $r'_m:=r_i$, et pour $m \notin \mathbf{M}$, on note $r'_m:=1$, $F_{m,1}:=F_m^\bullet$ et $\lambda_{m,1}:=0$.
%\end{nota}
On montre (4) du théorème \ref{semiinv}, c'est-à-dire l'indépendance algébrique des $F_{m,t}$ par le théorème suivant.
\begin{theo}
Soit $\gotk$ une algèbre de Lie. Soit $d \geq 1$ et $(f_m)_{1 \leq m \leq d}$ une famille d'éléments algébriquement indépendants de $\Y(\gotk)$. Supposons que pour tout $m \in \llbracket 1,d \rrbracket$, l'invariant $f_m$ se décompose en un produit de semi-invariants $f_m \propto \prod_{t=1}^{r_m} f_{m,t}^{s_{m,t}}$ avec $r_m \geq 1$ et les $s_{m,t} \geq 1$ des entiers premiers entre eux (dans leur ensemble). On suppose de plus que les semi-invariants $f_{m,t}$ vérifient :
\begin{itemize}
\item à $m$ fixé, l'ensemble des poids $\lambda_{m,t}$ des $f_{m,t}$ pour $t \in \llbracket 1,r_m \rrbracket$ forme une famille de rang $r_m-1$.
\item les espaces vectoriels $\vect \left((\lambda_{m,t})_{1 \leq t \leq r_{m}}\right)$ pour $m \in \llbracket 1, d \rrbracket$ sont en somme directe.
\end{itemize}
Alors les $(f_{m,t})_{1 \leq m \leq d, \, 1 \leq t \leq r_m}$ sont algébriquement indépendants. 
\label{alin}
\end{theo}
\begin{rema}
Les hypothèses du théorème sont vérifiées dans le cas que l'on étudie grâce à la proposition \ref{fdroit} et aux points (2) et (3) du théorème \ref{semiinv}.
\end{rema}
\begin{proof}
Soit $\Lambda({\mathbf{f}})$ le semi-groupe engendré par les poids des $f_{m,t}$. Considérons deux monômes en les $f_{m,t}$ respectivement $\prod_{m=1}^d \prod_{t=1}^{r_m} \left(f_{m,t}\right)^{p_{m,t}}$ et $\prod_{m=1}^d \prod_{t=1}^{r_m} \left(f_{m,t}\right)^{p'_{m,t}}$, et supposons qu'ils sont de même poids, disons $\lambda \in \Lambda(\mathbf{f})$. Pour tout $m$, on note $L_m:=\vect \left((\lambda_{m,t})_{1 \leq t \leq r_{m}}\right)$. On a $\vect(\Lambda(\mathbf{f}))=\bigoplus_m L_m$. On considère alors la décomposition $\lambda=\sum \lambda_m$ avec $\lambda_m \in L_m$. Les $L_m$ étant en somme directe, pour tout $m$, les poids de $\prod_{t=1}^{r_m} \left(f_{m,t}\right)^{p_{m,t}}$ et de $\prod_{t=1}^{r_m} \left(f_{m,t}\right)^{p'_{m,t}}$ sont égaux à $\lambda_m$. On fixe désormais $m$. On a donc
$$ \sum_{t=1}^{r_m} p_{m,t} \lambda_{m,t} = \sum_{t=1}^{r_m} p'_{m,t} \lambda_{m,t}=\lambda_m$$
Or la famille des $\lambda_{m,t}$, $1 \leq t \leq r_m$ est de rang $r_m-1$ et vérifie $ \sum_{t=1}^{r_m} s_{m,t} \lambda_{m,t}=0$ (car $f_m \propto \prod_t f_{m,t}^{s_{m,t}}$). Ainsi il existe $\phi_m \in \Q$ tel que pour tout $t$, on a $p'_{m,t}-p_{m,t}=\phi_m s_{m,t}$. Puisque les $s_{m,t}$ sont premiers entre eux, on a nécessairement $\phi_m \in \rat$.\par
En particulier, l'ordre défini sur $\N^{r_m}$ par $(x_{m,t})_t \leq (y_{m,t})_t \Leftrightarrow \forall t, \, x_{m,t} \leq y_{m,t}$ devient un ordre total sur $\mathcal{P}_m:=\{(x_{m,t})_t \in \N^{r_m} \, | \, \sum_{t=1}^{r_m}x_{m,t} \lambda_{m,t}=\lambda_m \}$ et cet ensemble admet un minimum pour cet ordre. Notons $(b_{m,t})_t$ ce minimum. Il ne dépend que de $\lambda_m$ et notamment pas du choix des $p_{m,t}$ et $p'_{m,t}$.\par
Pour tout monôme $\prod_{t=1}^{r_m} \left(f_{m,t}\right)^{p_{m,t}}$ de poids $\lambda_m$, il existe donc un unique $\varphi_m \in \N$ tel que $p_{m,t}=\varphi_m s_{m,t}+ b_{m,t}$ pour tout $t$. L'application $(p_{m,t})_t \in \mathcal{P}_m \mapsto \varphi_m \in \N$ est clairement bijective. Ainsi pour tout $m$, on a
\begin{align*}
\prod_{t=1}^{r_m} \left(f_{m,t}\right)^{p_{m,t}}&=\prod_{t=1}^{r_m} \left(f_{m,t}\right)^{\varphi_m s_{m,t}+ b_{m,t}} \\
&=\left(\prod_{t=1}^{r_m} \left(f_{m,t}\right)^{s_{m,t}}\right)^{\varphi_m} \prod_{t=1}^{r_m} \left(f_{m,t}\right)^{b_{m,t}}\\
&\propto (f_m)^{\varphi_m} \prod_{t=1}^{r_m} (f_{m,t})^{b_{m,t}}
\end{align*}
et donc 
\begin{equation}
\prod_{m=1}^d \prod_{t=1}^{r_m} \left(f_{m,t}\right)^{p_{m,t}}\propto\left(\prod_{m=1}^d f_m^{\varphi_m}\right) \prod_{m=1}^d \prod_{t=1}^{r_m} \left(f_{m,t}\right)^{b_{m,t}} \label{prodprod}
\end{equation}
Supposons qu'un polynôme en les $f_{m,t}$ de la forme $\sum_{\mathbf{p}=(p_{m,t})_{m,t}} k_{\mathbf{p}} \prod_{m,t} f_{m,t}^{p_{m,t}}$ est nul. Pour conclure, il suffit de montrer que tous les $k_{\mathbf{p}}$ sont nécessairement nuls. Par l'absurde, supposons qu'au moins un des $k_{\mathbf{p}}$ est non nul. Comme les $\Sym(\gotk)_\nu$ sont en somme directe, on peut supposer que tous les monômes deux à deux non colinéaires $k_{\mathbf{p}} \prod_{m,t} f_{m,t}^{p_{m,t}}$ ont le même poids $\lambda$. Autrement dit, à une constante multiplicative non nulle près, ils s'écrivent tous sous la forme de l'équation \eqref{prodprod}. L'application $(p_{m,t})_{m,t} \in \mathcal{P}_1 \times \ldots \times \mathcal{P}_d \mapsto (\varphi_m)_m \in \N^d$ est bijective et $(b_{m,t})_{m,t}$ ne dépend pas des $p_{m,t}$. Ainsi le facteur $\prod_{m=1}^d \prod_{t=1}^{r_m} \left(f_{m,t}\right)^{b_{m,t}}$ est commun à tous ces monômes et peut être simplifié. Ils deviennent donc (à une constante multiplicative non nulle près) des monômes de la forme $\prod_{m=1}^d f_m^{\varphi_m}$ deux à deux distincts. On obtient donc un polynôme $P$ non trivial qui est nul en les $f_m$, ce qui est absurde car les $f_m$ sont algébriquement indépendants.
\end{proof}
\section{Dimension de Gelfand-Kirillov}
\label{sec5}
On a montré que la famille $\mathbf{F}$ de la définition \ref{familleF} (voir l'équation \eqref{expmochesemiinv}) est une famille algébriquement indépendante de semi-invariants.
Par la proposition \ref{ai}, on sait que $\ind \gotq_{\Lambda} = \GK \Y(\gotq_{\Lambda}) \geq \Card(\mathbf{F})$, où
\begin{align*}
\Card(\mathbf{F})&=\Card (\llbracket 1,n \rrbracket \setminus \mathbf{M}_1)+\sum_{i \in \mathbf{I}} \rho_{i} \\
&=n-\Card(\mathbf{M}_1)+\sum_{i \in \mathbf{I}} \rho_i\\
\end{align*}
Comme $\rho_i = \Card (\left\{k \in \llbracket 1,s \rrbracket \; | \; i_k=i \right\})$ et $\mathbf{I}$ et $\mathbf{M}_1$ sont en bijection, on a
\begin{equation}
\ind \gotq_{\Lambda} \geq n+(s-p) \label{nps}
\end{equation}
où on rappelle que $s$ est le nombre de blocs d'un facteur de Levi de $\gotp$ et $p=\Card (\mathbf{I})$, c'est-à-dire le nombre de \emph{classes d'isomorphismes} de blocs d'un facteur de Levi de $\gotp$, d'où $p \leq s$.\par
Le but de cette section est de montrer le théorème suivant :
\begin{theo}
Soit $\mathbf{F}$ la famille des semi-invariants de la définition \ref{familleF} (voir équation \eqref{expmochesemiinv}). Alors 
\begin{itemize}
\item $\dim \gotq_{\Lambda}=n^2-(s-p)$,
\item $\ind \gotq_{\Lambda}=\Card(\mathbf{F})=n+(s-p)$.
\end{itemize}
Ainsi avec le théorème \ref{semiinv} (4), la famille $\mathbf{F}$ forme une base de transcendance dans $\Sy(\gotq)$. En particulier, si $\C[\mathbf{F}]$\index{$\C[\mathbf{F}]$} est l'algèbre (polynomiale) engendrée par $\mathbf{F}$, alors l'extension $\Sy(\gotq) \supset \C[\mathbf{F}]$ est algébrique.
\label{dimind}
\end{theo}
Pour montrer que $\ind \gotq_\Lambda = \Card(\mathbf{F})$, on utilise le résultat suivant de Ooms et van den Bergh \cite[Proposition 3.1]{oom10} 
\begin{equation}
\dim \gotq + \ind \gotq = \dim \gotq_{\Lambda} + \ind \gotq_{\Lambda}
\label{ovdb}
\end{equation}
Dans le cas de la contraction parabolique, on a $\dim \gotq = \dim \g=n^2$, et par un résultat de Panyushev et Yakimova (\cite{py13}, théorème 3.1), on sait que $\ind \gotq = \ind \g=n$.
\begin{rema}
Par construction des semi-invariants comme facteurs des $F_m^{\bullet}$, la famille $\mathbf{F}$ vérifie
$$\sum_{f \in \mathbf{F}} \deg f = \sum_{m=1}^n \deg F_m^{\bullet}= \sum_{m=1}^n \deg F_m = \dfrac{\dim \gotq + \ind \gotq}{2} = \dfrac{\dim \gotq_{\Lambda} + \ind \gotq_{\Lambda}}{2}$$
Ainsi $\mathbf{F}$ vérifie bien la condition sur les degrés du théorème \ref{pan13}. On ne sait pas conclure avec le théorème \ref{pan13} car on n'a pas trouvé de moyen de vérifier la propriété de codimension $2$ pour $\gotq_\Lambda$.
\end{rema}
On rappelle (preuve du lemme \ref{poidsgen}) que
$$\gotq'  = \left(\bigoplus_{p \neq q} \C e_{p,q}  \oplus \bigoplus_{\iota \notin I_{\gotp}} \C h_\iota\right)$$
est de dimension $n^2-s$. On rappelle que $\gotq' \oplus \C \gotz(\gotq) \subset \gotq_\Lambda$ (propriété \ref{inclusion}). Cette inclusion n'est toutefois pas une égalité en général. Pour obtenir plus d'éléments de $\gotq_{\Lambda}$, on va une nouvelle fois utiliser les $F_j^{\bullet}$.\par
Pour tout $j \in \llbracket 1,n \rrbracket$, puisque $F_j^{\bullet}$ est un invariant de $\Sym(\gotq)$, on a $F_j^{\bullet} \in \Y(\gotq) \subset \Sy(\gotq) = \Y(\gotq_{\Lambda}) \subset \Sym(\gotq_{\Lambda})$ par le théorème \ref{jm}. Soit $\widehat{\g}$ le sous-espace vectoriel de $\g$ formé des matrices de diagonale nulle. On a $\gotq_{\Lambda} =\widehat{\g} \oplus \h_{\Lambda}$ où $\h_{\Lambda}:=\h \cap \gotq_{\Lambda}$\index{$\h_{\Lambda}$}. Ainsi, on a les isomorphismes d'algèbres $\Sym(\gotq) \simeq \Sym(\widehat{\g}) \otimes_{\C} \Sym(\h)$ et $\Sym(\gotq_\Lambda) \simeq \Sym(\widehat{\g}) \otimes_{\C} \Sym(\h_\Lambda)$. On utilise alors le résultat d'algèbre tensorielle suivant.
\begin{prop}
Soit $F \in \Sym(\gotq_{\Lambda})$. Si $F$ s'écrit $\sum_i u_i v_i$, où $u_i \in \Sym(\widehat{\g})$, $v_i \in \Sym(\h)$, avec les $u_i$ linéairement indépendants, alors pour tout $i$, $v_i \in \Sym(\h_{\Lambda}) \subset \Sym(\gotq_{\Lambda})$. \label{algcla}
\end{prop}
On fixe alors $m:=m_i \in \mathbf{M}_1$ tel que $i \neq \imax$ (voir notation \ref{nota}). 
%D'après la propriété \ref{propelem}, on a $m+1 \notin \mathbf{M}_2$.
On omettra dans la suite les indices $i$.
\begin{lem}
Si $J \in \mathcal{J}(m)$, alors $\Delta_J^\bullet \in \Sym(\widehat{\g})$. \label{gchap}
\end{lem}
\begin{proof}
On a $\deg_{\gotn^-} F_m=m-i$ (voir la proposition \ref{oui}). Soit $J \in \mathcal{J}(m)$. Supposons par l'absurde que $\Delta_J^\bullet \notin \Sym(\widehat{\g})$, c'est-à-dire qu'il existe $\sigma\in \mathfrak{S}(J)$ avec un point fixe $l'$ tel que $\prod_{l \in J} e_{l,\sigma(l)}$ soit un monôme de degré maximal $m-i$ en $\gotn^-$. Alors $\prod_{l \in J \setminus \{l'\}} e_{l,\sigma(l)}$ est également de degré $m-i$ en $\gotn^-$, et c'est (au signe près) un monôme de $\Delta_{J \setminus \{l'\}}$, qui est lui-même un terme de $F_{m-1}$. Ainsi $m-i=\deg_{\gotn^-} \left( \prod_{l \in J } e_{l,\sigma(l)}\right)\leq \deg_{\gotn^-} F_{m-1}$. Or, comme $m-1 \notin \mathbf{M}_0$ (propriété \ref{propelem}), on a $\deg_{\gotn^-} F_{m-1}=m-i-1$ par l'égalité \eqref{degrec} du corollaire \ref{coro}, ce qui est absurde.
\end{proof}
\begin{propn}
Soit $\mathcal{J}'$ l'ensemble des $J' \subset I$ de cardinal $m+1$ tels que $\deg_{\gotn^-} \Delta_{J'} = \deg_{\gotn^-} F_{m+1}$, de sorte que $F_{m+1}^{\bullet}=\sum_{J' \in \mathcal{J}'} \Delta_{J'}^\bullet$. Par la propriété \ref{typej}, on a $\mathcal{J}(m+1) \subset \mathcal{J}'$ (l'inclusion est en fait stricte).
\begin{enumerate}
\item Pour tout $J \in \mathcal{J}'$, on a 
$\deg_{\h} \Delta_J^\bullet \leq 1$ et l'égalité est vérifiée si et seulement si $J \in \mathcal{J}(m+1)$ (voir définition \ref{typej}). Ainsi $\deg_{\h} F_{m+1}^{\bullet}=1$.
\item Pour $J \in \mathcal{J}(m+1)$, la somme des monômes de degré $1$ en $\h$ dans $\Delta_J^{\bullet}$ est
$$
\sum_{l \in J_{k_J}}  e_{l,l} \Delta_{J \setminus \left\{ l \right\} } ^{\bullet}
$$
où $k_J$ est défini dans l'exemple \ref{exnot}.
\end{enumerate}
\label{comb2}
\label{21}
\end{propn}
\begin{proof}
$\left(\bigoplus_{\sigma \in \mathfrak{S}(J')} \C \prod_{l \in J'} e_{l,\sigma(l)}\right)_{J' \in \mathcal{J}'}$ est une famille d'espaces vectoriels en somme directe. En particulier, $(\Delta_{J'}^\bullet)_{J' \in \mathcal{J}'}$ est une famille d'éléments linéairement indépendants. Par l'égalité \eqref{degrec}, on a $\deg_{\gotn^-} F_{m+1}=m-i$. Si $J \in \mathcal{J}'$, on a donc $\max(j_k)=i+1$ par le lemme \ref{clé}, et on distingue alors deux cas.\par
Si $J \in \mathcal{J}'$ est tel qu'il existe $k_0 \neq k_0'$ tels que $j_{k_0}=j_{k_0'}= i+1$, alors $\Delta_{J}^{\bullet} \in \Sym(\widehat{\g})$. En effet, supposons comme précédemment que $\varepsilon(\sigma) \prod_{l \in J} e_{l,\sigma(l)}$ soit un monôme de $\Delta_{J}$ tel que $\sigma$ admette un point fixe $l'$. On a donc $\deg_{\gotn^-} \left( \prod_{l \in J} e_{l,\sigma(l)}\right)\leq \deg_{\gotn^-} \Delta_{J \setminus \left\{l'\right\}}$. Mais au moins un des deux sous-ensembles $J_{k_0}$ ou $J_{k_0'}$ est une part de la partition de $J \setminus \left\{l'\right\}$, et ces deux sous-ensembles sont de cardinal $i+1$. Par le lemme \ref{clé}, on a alors $\deg_{\gotn^-} \Delta_{J \setminus \left\{l'\right\}}= (m+1-1)-(i+1)=m-i-1$. Le monôme $\prod_{l \in J} e_{l,\sigma(l)}$ n'est donc pas de degré maximal $m-i$ en $\gotn^-$ et n'apparaît pas dans $\Delta_{J}^{\bullet}$.\par
Sinon, on a $J \in \mathcal{J}(m+1)$. Supposons que $\varepsilon(\sigma) \prod_{l \in J} e_{l,\sigma(l)}$ est un monôme de $\Delta_{J}^{\bullet}$ de degré $\geq 1$ en $\h$, c'est-à-dire que $\sigma$ a un point fixe $l'$. Comme $\deg_{\gotn^-} \Delta_J^\bullet =\deg_{\gotn^-} F_{m+1}= m-i$, on a $\deg_{\gotn^-} \left( \prod_{l \in J} e_{l,\sigma(l)}\right) =m-i$. Or $\varepsilon(\sigma) \prod_{l \in J} e_{l,\sigma(l)}=\varepsilon(\sigma_{|J \setminus \{l'\}})e_{l',l'} \prod_{l \in J \setminus \{l'\}} e_{l,\sigma(l)}$ est un monôme de $e_{l',l'} \, \Delta_{J \setminus \left\{l'\right\}}$, ce qui implique que $\deg_{\gotn^-} \Delta_{J \setminus \left\{l'\right\}} \geq m-i$. Comme $\Delta_{J \setminus \left\{l'\right\}}$ est un terme de $F_m$, on a également $\deg_{\gotn^-} \Delta_{J \setminus \left\{l'\right\}} \leq\deg_{\gotn^-} F_{m}= m-i$, et donc d'après ce qui précède, on a l'égalité. Ceci implique, d'après le lemme \ref{clé}, que la partition de $J \setminus \left\{l'\right\}$ a sa plus grande part de cardinal $i$, ce qui implique que $l' \in J_{k_J}$, donc que $J \setminus \{l'\} \in \mathcal{J}(m)$ (corollaire \ref{comb} (2)). Alors par le lemme \ref{gchap}, on a $\Delta_{J \setminus \left\{l'\right\}}^\bullet \in \Sym(\widehat{\g})$. Finalement, $\varepsilon(\sigma) \prod_{l \in J} e_{l,\sigma(l)}$ est un monôme de $e_{l',l'} \, \Delta^\bullet_{J \setminus \left\{l'\right\}}$ et est de degré $1$ en $\h$.
\end{proof}
On note $B_{m+1}$ et $C_{m+1}$ la somme des monômes de $F_{m+1}^\bullet$ de degré respectivement $0$ et $1$ en $\h$. Puisque $\widehat{\g} \subset \gotq_\Lambda$, on a $B_{m+1} \in \Sym(\gotq_\Lambda)$, de plus on rappelle que $F_{m+1}^\bullet=B_{m+1}+C_{m+1} \in \Sym(\gotq_\Lambda)$. Ainsi $C_{m+1} \in \Sym(\gotq_\Lambda)$.
\begin{propn}
Pour tout $i \in \mathbf{I}$, on note $\id^{(i)}= \sum_{k \, | \, \Card(I_k)=i} \id_{\g_{I_k}}$\index{$\id^{(i)}$} avec $\id_{\g_{I_k}}=\sum_{l \in I_k} e_{l,l}$. Alors $\gotq' \oplus \bigoplus_{i \in \mathbf{I}} \C \id^{(i)} \subset \gotq_{\Lambda}$.
\label{qlam}
\end{propn}
\begin{proof}
On a, d'après la proposition précédente : 
$$C_{m+1}=\sum_{J \in \mathcal{J}(m+1)} \sum_{l \in J_{k_J}} e_{l,l} \Delta_{J \setminus \left\{ l \right\} } ^{\bullet}$$
D'après le corollaire \ref{comb} (2), l'application
\begin{align*}
\{(J',l') \, | \, J' \in \mathcal{J}(m+1), l' \in J_{k_J}\} &\rightarrow \{(J,l) \, | \, J \in \mathcal{J}(m), l \in I \setminus J \}\\
(J',l') &\mapsto (J' \setminus \{l'\},l')
\end{align*}
est une bijection. La somme se réécrit donc :
$$C_{m+1}=\sum_{J \in \mathcal{J}(m)} \left(\sum_{l \in I \setminus J} e_{l,l} \right) \Delta_{J}^{\bullet}$$
Par le lemme \ref{gchap}, les $\Delta_{J}^{\bullet}$, $J \in \mathcal{J}(m)$ sont dans $\Sym(\widehat{\g})$ et linéairement indépendants car $\Delta_{J}^{\bullet} \in \bigoplus_{\sigma \in \mathfrak{S}(J)} \C \prod_{j \in J} e_{j,\sigma(j)}$. D'après la propriété \ref{algcla}, pour tout $J \in \mathcal{J}(m)$, le terme $\sum_{l \in I \setminus J} e_{l,l}$ est donc dans $\Sym(\h_{\Lambda})$, et donc a fortiori dans $\gotq_{\Lambda}$.\par
Du fait que $\overline{\mathcal{J}(m)}:=\left\{ I \setminus J \, | \, J \in \mathcal{J}(m)\right\}$ est aussi l'ensemble des $\overline{J} \subset I$ tels que $\overline{\jmath}_k:=\Card (\overline{J}_k)=\max(i_k-i,0)$, le terme
$$\sum_{k=1}^s \sum_{l \in \overline{J}_{k}} e_{l,l}$$
appartient à $\gotq_{\Lambda}$ pour tout $\overline{J} \in \overline{\mathcal{J}(m)}$.\par
Pour tout $k \in \llbracket 1,s \rrbracket$, on définit $\trace_{I_k}$ comme l'application trace sur la sous-algèbre $\g_{I_k} \subset \gotq$, que l'on étend à une forme linéaire sur $\gotq$ (en posant $\trace_{I_k}(e_{v,w})=0$ si $e_{v,w} \notin \g_{I_k}$). Comme $\g_{I_k} \simeq \gl_{i_k}$, on a $\spl_{i_k} \simeq \g_{I_k}' \subset \gotq'$. Ainsi, pour $q \in \g_{I_k} \subset \gotq$, on a $q \in \gotq'$ si et seulement si $\trace_{I_k}(q)=0$. Modulo $\gotq'$, on a donc
$$\sum_{l \in \overline{J}_{k}} e_{l,l} \equiv \dfrac{|\overline{J}_k|}{|I_k|} \id_{\g_{I_{k}}}$$
où $\frac{|\overline{J}_k|}{|I_k|}=\frac{\max(i_{k}-i,0)}{i_{k}}$, et donc modulo $\gotq'$, on a
\begin{align*}
\sum_{k=1}^s \sum_{l \in \overline{J}_{k}} e_{l,l}&\equiv \sum_{k=1}^s \dfrac{\max(i_{k}-i,0)}{i_{k}} \id_{\g_{I_{k}}} \\
&\equiv \sum_{j \in \mathbf{I}} \sum_{\stackrel{k \in \llbracket 1,s \rrbracket}{i_k=j}} \dfrac{\max(j-i,0)}{j} \id_{\g_{I_{k}}} \\
&\equiv \sum_{j \in \mathbf{I}} \dfrac{\max(j-i,0)}{j} \sum_{\stackrel{k \in \llbracket 1,s \rrbracket}{i_k=j}} \id_{\g_{I_{k}}} \\
&\equiv \sum_{j \in \mathbf{I}} \dfrac{\max(j-i,0)}{j} \id^{(j)}
\end{align*}
En considérant tous les termes $\sum_{j \in \mathbf{I}} \frac{\max(j-i,0)}{j} \id^{(j)}$ pour $i \in \mathbf{I} \setminus \{\imax\}$, on obtient un système linéaire triangulaire inversible en les $\id^{(j)}$, pour $j \in \mathbf{I} \setminus \{\min \mathbf{I}\}$, ainsi les $\id^{(i)}$ pour $i \in \mathbf{I} \setminus \{\min \mathbf{I}\}$ sont dans $\gotq_{\Lambda}$. Comme $\sum_{i \in \mathbf{I}} \id^{(i)}=\id \in \gotz(\gotq) \subset \gotq_\Lambda$, le terme $\id^{(\min \mathbf{I})}$ est également dans $\gotq_{\Lambda}$.\par
Montrons enfin que les $\id^{(i)}$, $i \in \mathbf{I}$ forment une famille linéairement indépendante et que le sous-espace vectoriel qu'ils engendrent est en somme directe avec $\gotq'$. Soient $x_i \in \C$ pour tout $i \in \mathbf{I}$ et $q \in \gotq'$ tels que
$$\sum_{i \in \mathbf{I}} x_i \id^{(i)} + q=0$$
D'après ce qui précède, pour tout $q \in \gotq'$, on a $\trace_{I_k}(q)=0$ pour tout $k \in \llbracket 1,s \rrbracket$. Soit donc $i \in \mathbf{I}$ et $k \in \llbracket 1,s \rrbracket$ tel que $i_k=|I_k|=i$. On a donc
$$\trace_{I_k}\left(\sum_{i \in \mathbf{I}} x_i \id^{(i)} + q\right)=0$$
donc $x_{i_k} i_k = 0$, d'où $x_i=x_{i_k}=0$, et ceci pour tout $i \in \mathbf{I}$, ce qui conclut.
\end{proof}
\begin{exem}
On reprend l'exemple \ref{leexemple}. Dans cet exemple, on obtient trois vecteurs $\id^{(1)}$, $\id^{(2)}$ et $\id^{(4)}$ donnés par les diagrammes respectifs suivants :
\begin{center}
\begin{tikzpicture}[scale=0.3]
\foreach \k in {1,2,...,11}
	{\draw[color=gray!20]  (0,\k)--(12,\k);
	\draw[color=gray!20] (\k,0)--(\k,12);}
\draw (0,0)--(12,0);
\draw (0,0)--(0,12);
\draw (0,12)--(12,12);
\draw (12,0)--(12,12);
\draw[ultra thick] (0,8)--(4,8)--(4,7)--(5,7)--(5,3)--(9,3)--(9,1)--(11,1)--(11,0);
\draw[dashed] (4,12)--(4,8)--(5,8)--(5,7)--(9,7)--(9,3)--(11,3)--(11,1)--(12,1);
\draw (4.5,7.5) node{$1$};
\draw (11.5,0.5) node{$1$};
\end{tikzpicture}
\quad
\begin{tikzpicture}[scale=0.3]
\foreach \k in {1,2,...,11}
	{\draw[color=gray!20]  (0,\k)--(12,\k);
	\draw[color=gray!20] (\k,0)--(\k,12);}
\draw (0,0)--(12,0);
\draw (0,0)--(0,12);
\draw (0,12)--(12,12);
\draw (12,0)--(12,12);
\draw[ultra thick] (0,8)--(4,8)--(4,7)--(5,7)--(5,3)--(9,3)--(9,1)--(11,1)--(11,0);
\draw[dashed] (4,12)--(4,8)--(5,8)--(5,7)--(9,7)--(9,3)--(11,3)--(11,1)--(12,1);
\draw (9.5,2.5) node{$1$};
\draw (10.5,1.5) node{$1$};
\end{tikzpicture}
\quad
\begin{tikzpicture}[scale=0.3]
\foreach \k in {1,2,...,11}
	{\draw[color=gray!20]  (0,\k)--(12,\k);
	\draw[color=gray!20] (\k,0)--(\k,12);}
\draw (0,0)--(12,0);
\draw (0,0)--(0,12);
\draw (0,12)--(12,12);
\draw (12,0)--(12,12);
\draw[ultra thick] (0,8)--(4,8)--(4,7)--(5,7)--(5,3)--(9,3)--(9,1)--(11,1)--(11,0);
\draw[dashed] (4,12)--(4,8)--(5,8)--(5,7)--(9,7)--(9,3)--(11,3)--(11,1)--(12,1);
\draw (0.5,11.5) node{$1$};
\draw (1.5,10.5) node{$1$};
\draw (2.5,9.5) node{$1$};
\draw (3.5,8.5) node{$1$};
\draw (5.5,6.5) node{$1$};
\draw (6.5,5.5) node{$1$};
\draw (7.5,4.5) node{$1$};
\draw (8.5,3.5) node{$1$};
\end{tikzpicture}
\end{center}
\end{exem}
On a finalement
$$\dim \gotq_{\Lambda} \geq \dim \gotq' + \, p = n^2-(s-p)$$
Or on a (équation \eqref{nps}) $\ind \gotq_{\Lambda} \geq n+(s-p)$ d'où $\dim \gotq_{\Lambda} + \ind \gotq_{\Lambda} \geq n^2+n$. Comme (équation \eqref{ovdb}) $\dim \gotq_{\Lambda} + \ind \gotq_{\Lambda}=\dim \gotq + \ind \gotq=n^2+n$, les inégalités précédentes sont en fait des égalités, ce qui implique le théorème \ref{dimind}.

\section{Théorème de polynomialité de semi-centre et application}
\label{sec6}
On a une famille $\mathbf{F}$ de semi-invariants algébriquement indépendants de $\Sy(\gotq)$, qui semble être une bonne candidate pour engendrer $\Sy(\gotq)$. Dans cette section, on donne des théorèmes généraux pour montrer qu'une certaine famille de semi-invariants $\mathbf{f} \subset \Sy(\gotk)$ (où $\gotk$ est une algèbre de Lie quelconque) engendre $\Sy(\gotk)$. On appliquera ensuite ces théorèmes à notre famille $\mathbf{F}$. On appliquera également ces mêmes théorèmes aux autres cas que l'on traitera (notamment en type $C$).

\subsection{Théorème de polynomialité de semi-centre}
Soit $\gotk$ une algèbre de Lie. Dans cette sous-section, on considère $(f_m)_{1 \leq m \leq d}$ une famille d'invariants de $\Sym(\gotk)$ et une famille de semi-invariants de $\Sym(\gotk)$ non constants $\mathbf{f}:=\{f_{m,t} \, | \, 1 \leq m \leq d, 1 \leq t \leq r_{m}\}$ \index{$\mathbf{f}$, $\C[{\mathbf{f}}]$, $\Lambda({\mathbf{f}})$, $\mathbf{f}^\times$}avec $r_m \in \N^*$ pour tout $m$, telles que pour tout $m \in \llbracket 1,d \rrbracket$ et $t \in \llbracket 1,r_m \rrbracket$, il existe des entiers $\nu_{m,t} >0$ tels que
\begin{equation}
f_m \propto \prod_{t=1}^{r_{m}} (f_{m,t})^{\nu_{m,t}} \label{propto}
\end{equation}
On note $\C[{\mathbf{f}}]$ la sous-algèbre de $\Sy(\gotk)$ engendrée par les éléments de $\mathbf{f}$, et on souhaiterait déterminer à quelles conditions on a $\C[{\mathbf{f}}]=\Sy(\gotk)$. Soit $\Lambda({\mathbf{f}})$ le semi-groupe des poids de $\C[{\mathbf{f}}]$. On note aussi $\mathbf{f}^\times$ l'ensemble des éléments de $\mathbf{f}$ de poids non nul.
\begin{rema}
Quitte à retirer un $f_m$ de la liste, on peut d'abord supposer que \emph{les $f_m$ sont deux à deux non colinéaires}. En revanche, on ne peut pas faire de même pour les $f_{m,t}$. Toutefois, pour $f,g \in \mathbf{f}$, quitte à multiplier $f$ et $g$ (et leurs invariants associés) par des constantes, on peut supposer que \emph{$f \propto g \Rightarrow f=g$} (en respectant toujours les hypothèses ci-dessus). On supposera donc ces hypothèses supplémentaires vérifiées dans la suite de cette sous-section. \label{propegal}
\end{rema}
\begin{lem}
Le semi-groupe $\Lambda({\mathbf{f}})$ est un groupe.
\label{premier}
\end{lem} 
\begin{proof}
Pour tous $m,t$, on note $\lambda_{m,t}$ le poids de $f_{m,t}$. Par l'équation \eqref{propto}, on a $-\lambda_{m,t}=\sum_{u \neq t} \nu_{m,u} \lambda_{m,u}+(\nu_{m,t}-1) \lambda_{m,t}$.
\end{proof}
\begin{lem}
On suppose que les $f_m$ sont irréductibles dans $\Y(\gotk)$. Soit $\mathbf{M}_2:=\{m \in \llbracket 1, d \rrbracket \, | \, r_m \geq 2\}$. On a $\mathbf{f}^\times = \{f_{m,t} \, | \, m \in \mathbf{M}_2, 1 \leq t \leq r_m\}$.
\label{lemm}
\end{lem}
\begin{proof}
L'inclusion $\subset$ est claire : si $m \notin \mathbf{M}_2$, alors $r_m=1$ et comme $f_m$ est supposé irréductible dans $\Y(\gotk)$, on a $f_{m,1}=f_m$ qui est de poids nul. Pour l'autre inclusion, comme les éléments de $\mathbf{f}$ sont non constants, un $f_{m,t}$ avec $m \in \mathbf{M}_2$ est nécessairement un facteur non constant de $f_m$. Puisque $f_m$ est irréductible dans $\Y(\gotk)$, le semi-invariant $f_{m,t}$ n'est pas un invariant, donc est de poids non nul.
\end{proof}
Si $x$ est un semi-invariant  de $\Sym(\gotk)$, on voudra montrer (sous certaines hypothèses) que $x \in \C[\mathbf{f}]$. Pour ce faire, le lemme suivant est fondamental.
\begin{lem}
\label{lemcentral}
On suppose que $\GK (\C[{\mathbf{f}}])=\GK \Sy(\gotk)$. Soit $x$ un semi-invariant de $\Sym(\gotk)$. Alors il existe $a \in \N^*$ et $s_{z} \in \N$ pour tout $z \in \mathbf{f}^\times$ tels que
\[x^a \prod_{z \in \mathbf{f}^\times} z^{s_{z}} \in \Y(\gotk)\]
\end{lem}
\begin{proof}
Soit $x$ un semi-invariant de $\Sym(\gotk)$ que l'on peut supposer de poids non nul. Puisque $\GK (\C[{\mathbf{f}}])=\GK \Sy(\gotk)$, il existe un polynôme $P \in \C[{\mathbf{f}}][X]$ non nul tel que $P(x)=0$. Quitte à diviser par une certaine puissance de $X$, on peut supposer que $X$ ne divise pas $P$. Autrement dit, il existe un entier $c$ et des polynômes $P_i((f_{m,t})_{m,t}) \in \C[{\mathbf{f}}]$ avec $P_0((f_{m,t})_{m,t}) \neq 0$ tels que :
$$\sum_{k=0}^c P_k((f_{m,t})_{m,t})x^k=0$$
Il existe alors $a \geq 1$ tel que $P_a((f_{m,t})_{m,t}) \neq 0$. Comme dans la preuve du théorème \ref{alin}, on peut supposer qu'il existe un poids $\lambda$ tel que pour tout $k \in \llbracket 0,c \rrbracket$, pour tout monôme $\mathcal{S}_k$ de $P_k((f_{m,t})_{m,t})$, le semi-invariant $\mathcal{S}_k x^k$ est de poids $\lambda$. En particulier, chaque $P_k((f_{m,t})_{m,t})$ est un semi-invariant et $P_0((f_{m,t})_{m,t})$ est de poids $\lambda$, d'où $\lambda \in \Lambda({\mathbf{f}})$. Puisque $\Lambda({\mathbf{f}})$ est un groupe (lemme \ref{premier}), quitte à multiplier par un certain monôme en les $f \in \mathbf{f}$, on peut supposer que $\lambda=0$. Si $\prod_{f \in \mathbf{f}} {f}^{s_f}$ est (à une constante multiplicative non nulle près) un monôme quelconque de $P_a$, alors $x^a \prod_{f \in \mathbf{f}} {f}^{s_f}$ est un invariant. Ainsi $x^a \prod_{f \in  \mathbf{f}^\times} {f}^{s_{f}}$ est encore un semi-invariant (proposition \ref{dixmier}) et est de poids nul donc est un invariant.
\end{proof}
\begin{coro}
On reprend les hypothèses du lemme \ref{lemcentral}. Alors l'ensemble des poids $\Lambda(\gotk)$ de $\Sy(\gotk)$ est un groupe.\label{groupe}
\end{coro}
\begin{proof}
Soit $x$ un semi-invariant (non nul) de $\Sym(\gotk)$. Alors par le lemme \ref{lemcentral}, il existe $a \in \N^*$ et $s_{z} \in \N$ pour tout $z \in \mathbf{f}^\times$ tels que $x^a \prod_{z \in \mathbf{f}^\times} z^{s_{z}} \in \Y(\gotk)$. Ainsi, il existe un semi-invariant $y$ (non nul) de $\Sym(\gotk)$ tel que $xy \in \Y(\gotk)$. Ainsi, si $x$ est de poids $\lambda$, on a un semi-invariant $y$ de poids $-\lambda$. 
\end{proof}
On a alors le théorème suivant.
\begin{theo}
Soit $\gotk$ une algèbre de Lie. Soit $(f_m)_{1 \leq m \leq d}$ une famille d'invariants de $\Y(\gotk)$. Étant donné pour tout $m \in \llbracket 1,d \rrbracket$ une décomposition de $f_m$ dans $\Sym(\gotk)$ de la forme
\[
f_m = \prod_{t=1}^{r_{m}} (f_{m,t})^{\nu_{m,t}} \label{decompintro} \tag{$\clubsuit$}
\]
avec $\nu_{m,t} \in \N^*$, on suppose pour tout $m$ que l'on est dans au moins un des deux cas suivants :
\begin{itemize}
\item la décomposition \eqref{decompintro} est triviale, c'est-à-dire $r_m=1$ et $\nu_{m,1}=1$, autrement dit $f_m = f_{m,1}$,
\item la décomposition \eqref{decompintro} est la décomposition de $f_m$ en éléments irréductibles dans $\Sym(\gotk)$.
\end{itemize}
On note $\mathbf{f}$ l'ensemble des $f_{m,t}$. Sous les hypothèses suivantes :
\begin{enumerate}
\item[(a)] $\Y(\gotk)$ est une algèbre factorielle,
\item[(b)] $\GK \Sy(\gotk)=\GK \C[\mathbf{f}]$,
\item[(c)] les $f_m$ sont irréductibles dans $\Y(\gotk)$,
\item[(I)] il existe un morphisme de $\C$-algèbres $\vartheta : \Sym(\gotk) \rightarrow \Y(\gotk)$ tel que $\vartheta_{|\Y(\gotk)}$ est un isomorphisme,
\end{enumerate}
alors l'ensemble des poids de $\Sy(\gotk)$ est un groupe et $\Sy(\gotk)$ est engendrée par $\Y(\gotk)$ et les éléments de $\mathbf{f}$. En particulier, si
\begin{itemize}
\item[(d)] les $f_m$ engendrent $\Y(\gotk)$,
\end{itemize}\label{puissant}
alors $\Sy(\gotk)=\C[\mathbf{f}]$.
\end{theo}
\begin{rema}\mbox{}
\begin{itemize}
\item Dans les cas que l'on étudie par la suite, les hypothèses (a), (c) et (d) découleront de l'étude de Panyushev et Yakimova \cite{py13} et l'hypothèse (b) de l'étude précédente (notamment le théorème \ref{dimind}). Les hypothèses (I), (II) et (III) (où (II) et (III) apparaissent au théorème \ref{inter}) demandent une étude plus poussée.
\item L'hypothèse (I) implique que $\Y(\gotk)$ est une algèbre de type fini.
\end{itemize}
\end{rema}
L'hypothèse de factorialité de $\Y(\gotk)$ de ce théorème implique en fait des propriétés fortes sur l'ensemble des poids $\Lambda(\gotk)$, par exemple la proposition suivante.
\begin{propn}
Supposons que $\Y(\gotk)$ soit une algèbre factorielle. Soit $\lambda \in \Lambda(\gotk) \setminus \{0\}$ le poids d'un semi-invariant $x$ irréductible dans $\Sym(\gotk)$. Alors l'un des deux $\Y(\gotk)$-modules parmi $\Sym(\gotk)_\lambda$ et $\Sym(\gotk)_{-\lambda}$ est libre de rang $\leq 1$.
\label{restrictif}
\end{propn}
\begin{proof}
Si $\Sym(\gotk)_\lambda$ ou $\Sym(\gotk)_{-\lambda}$ est réduit à $0$, la proposition est satisfaite. Sinon, soient $x',y,y'$ des semi-invariants non nuls de poids respectifs $\lambda, -\lambda, -\lambda$. On peut supposer que si $s \in \Y(\gotk)$ divise $y$, alors $s \in \C$. On a
$$(xy)(x'y')=(xy')(x'y)$$
L'invariant $xy$ est irréductible dans $\Y(\gotk)$. En effet, si $s,t \in \Y(\gotk)$ sont tels que $xy=st$, alors comme $x$ est irréductible dans $\Sym(\gotk)$, $x$ divise $s$ ou $t$, et donc $y$ est divisible par $t$ ou par $s$. Par hypothèse, on a donc $t$ ou $s$ dans $\C$.\par
L'invariant $xy$ divise alors $(xy')(x'y)$ dans $\Y(\gotk)$, donc divise $xy'$ ou $x'y$. Si $xy$ divise $xy'$ dans $\Y(\gotk)$, alors il existe $s \in \Y(\gotk)$ tel que $y'=sy$ ; si $xy$ divise $x'y$ dans $\Y(\gotk)$, alors il existe $s \in \Y(\gotk)$ tel que $x'=sx$. Autrement dit, on a $\Sym(\gotk)_\lambda=\Y(\gotk)x$ ou $\Sym(\gotk)_{-\lambda}=\Y(\gotk)y$.
\end{proof}
\begin{rema}
On dit que $z,z' \in \Sym(\gotk)$ sont $\Y(\gotk)$-libres si pour tous $s,s' \in \Y(\gotk)$ tels que $sz+s'z'=0$, on a $s=s'=0$. Supposons qu'il existe un poids $\lambda$ non nul et des semi-invariants $x,x'$ de poids $\lambda$ qui sont $\Y(\gotk)$-libres, ainsi que des semi-invariants $y,y'$ de poids $-\lambda$ qui sont également $\Y(\gotk)$-libres. Par contraposée de la proposition \ref{restrictif}, si l'un des semi-invariants $x,x',y,y'$ est irréductible dans $\Sym(\gotk)$, alors $\Y(\gotk)$ n'est pas factorielle, et donc en particulier n'est pas polynomiale.
\end{rema}
\begin{proof}[Preuve du théorème \ref{puissant}]
On note $\mathbf{f}^\times$ l'ensemble des $f_{m,t}$ de poids non nul. Par le lemme \ref{lemm}, $\mathbf{f}^\times$ est l'ensemble des $f_{m,t}$ avec $m$ tel que $r_m \geq 2$.
\begin{lem}
\label{ultratech}
Sous les hypothèses du théorème \ref{puissant}, soit $x \in \Sy(\gotk)$, $f \in \Y(\gotk)$, $\mathbf{f}^\times = \mathbf{f}' \sqcup \mathbf{f}''$ une partition de $\mathbf{f}^\times$, $a \in \N^*$ et $s_{z} \in \N$ pour tout $z \in \mathbf{f}^\times$ tels que
\begin{equation}
x^a \prod_{z \in \mathbf{f}'} z^{s_{z}}=\left(\prod_{z \in \mathbf{f}''} z^{s_{z}}\right) f. \label{eqhypo}
\end{equation}
Alors il existe $f' \in \Y(\gotk)$ et $t_{z} \in \N$ pour tout $z \in \mathbf{f}^\times$ tels que
$$
x^a= \prod_{z \in \mathbf{f}^\times} z^{t_{z}} f'.
$$
\end{lem}
\begin{proof}
%Supposons maintenant la propriété vérifiée pour tous $s_{z}$, $\mathbf{f}'$, $\mathbf{f}''$ tels que $\sum_{z \in \mathbf{f}'} s_{z} < v$, et supposons maintenant que $\sum_{z \in \mathbf{f}'} s_{z} = v$.\par
%On note $\varpi: q \in \gotk^* \mapsto (f_{m}(q))_{1 \leq m \leq d} \in \C^d$ et
Pour toute sous-algèbre de type fini $\A$ de $\Sym(\gotk)$, on note $\Spec(\A)$ le spectre maximal de $\A$, que l'on peut identifier avec un sous-ensemble algébrique affine de $\gotk^*$ (voir \cite[\mbox{}11.6.3]{TYu}). Pour $f \in \Y(\gotk)$, on note $\V'(f):=\{q \in \Spec \Y(\gotk) \, | \, f(q)=0\}$ le lieu d'annulation de $f$. Pour $f \in \Sym(\gotk)$, on note également $\V''(f):=\{q \in \Spec \Sym(\gotk) \simeq \gotk^* \, | \, f(q)=0\}$. On note $i : \Y(\gotk) \hookrightarrow \Sym(\gotk)$ l'inclusion canonique. Le morphisme $\vartheta$ de l'hypothèse (I) est alors une section de $i$, dans le sens où $\vartheta \circ i$ est un isomorphisme. Soit alors $i^*=\Spec(i) : \Spec \Sym(\gotk) \simeq \gotk^* \rightarrow \Spec \Y(\gotk)$ qui est un morphisme dominant.
% Si $f \in \Y(\gotk)$, alors $i^* \V(i(f)) \subset \V'(f)$.\par
\begin{lem}
\label{hypersurf}
Si $\V$ est une hypersurface de $\Spec \Sym(\gotk)$, alors $\overline{i^*(\V)}$ est une sous-variété de codimension au plus un de $\Spec(\Y(\gotk))$.
\end{lem}
\begin{rema}
Plus généralement, si $\V$ est une sous-variété de $\Spec \Sym(\gotk) \simeq \gotk^*$ de codimension $l$, alors $\overline{i^*(\V)}$ est une sous-variété de $\Spec \Y(\gotk)$ de codimension au plus $l$.
\end{rema}
\begin{proof}
On peut supposer $\V$ irréductible. Dans cette preuve, on utilise plusieurs fois le résultat classique de géométrie algébrique donné par \cite[\mbox{}15.5]{TYu}. Soit $i^*_{\V} : \V \rightarrow \overline{i^*(\V)}$ définie par la restriction de $i^*$ à $\V$. Par définition, $i^*_{\V}$ est dominante. Soient $y$ et $y_{\V}$ deux points généraux respectivement de $\Spec \Y(\gotk)$ et de $\overline{i^*(\V)}$. On a $\dim \overline{i^*(\V)} = \dim \V - \dim {i^*_{\V}}^{-1}(y_{\V}) \geq \dim \V - \dim {i^*}^{-1}(y_{\V})$. Pour conclure, on va montrer que toutes les fibres de $i^*$ ont même dimension, de sorte que 
$$\dim  \overline{i^*(\V)} \geq \dim \V - \dim {i^*}^{-1}(y_{\V}) = \dim \gotk^* -1 - \dim {i^*}^{-1}(y)=\dim \Spec \Y(\gotk)-1$$
Par l'hypothèse (I), on a un morphisme $\vartheta : \Sym(\gotk) \rightarrow \Y(\gotk)$ dont le dual $\vartheta^*=\Spec(\vartheta) : \Spec \Y(\gotk) \rightarrow \gotk^*$ est une section de $i^*$, dans le sens où $i^* \circ \vartheta^*$ est un isomorphisme. Soit $\mathcal{K}=\ker \vartheta$ le noyau de $\vartheta$, alors $\Y(\gotk)$ est isomorphe à $\Sym(\gotk) / \mathcal{K}$ et donc $\vartheta^*$ correspond à l'immersion fermée $\vartheta^* : Z \rightarrow \Spec \Sym(\gotk)$, où $Z:=\V''(\mathcal{K}) \simeq \Spec \Y(\gotk)$. Puisque $\vartheta^*$ est une section de $i^*$, pour toute fibre $F$ de $i^*$, l'ensemble $F \cap Z$ est un point. En notant $\codim$ pour la codimension dans $\gotk^*$, on a donc (voir \cite[I, Prop. 5.9]{har77} pour l'inégalité)
\begin{align*}
\dim \gotk^* &= \codim F \cap Z \leq \codim F + \codim Z \\ &= \codim F + \dim \gotk^* - \dim \Spec \Y(\gotk)
\end{align*}
et donc $\dim F \leq \dim \gotk^* - \dim \Spec \Y(\gotk)$. D'un autre côté, toute fibre $F$ de $i^*$ est de dimension supérieure à la dimension d'une fibre générale de $i^*$, qui est de dimension $\dim \gotk^* - \dim \Spec \Y(\gotk)$ (le morphisme $i^*$ est dominant), ce qui donne l'égalité.
\end{proof}
On reprend la preuve du lemme \ref{ultratech}. Si tous les $s_z$ dans \eqref{eqhypo} pour $z \in \mathbf{f}'$ sont nuls, le lemme est bien vérifié. Supposons alors que $\sum_{z \in \mathbf{f}'} s_{z} \geq 1$.  Soit $g=f_{\mu,\tau} \in \mathbf{f}'$ tel que dans l'équation \eqref{eqhypo}, on a $s_{g} \geq 1$. Alors $g$ divise $\prod_{z \in \mathbf{f}''} z^{s_{z}} f$. Comme les éléments de $\mathbf{f}^\times$ sont supposés irréductibles, $g$ divise $f=i(f)$. Pour tout $q \in \gotk^*$, on a donc $g(q)=0 \Rightarrow i(f)(q)=0$. On obtient ainsi $i^*(\V''(g)) \subset \V'(f)$ d'où $\overline{i^*(\V''(g))} \subset \V'(f)$.\par
Or par le lemme \ref{hypersurf}, $\overline{i^*(\V''(g))}$ est de codimension au plus $1$ dans $\Spec(\Y(\gotk))$. De plus, si $\xi \in \V''(g)$, alors $f_\mu(i^*(\xi))=i(f_\mu)(\xi)=f_\mu(\xi) \propto \prod_{\tau'} f_{\mu,\tau'}^{\nu_{\mu,\tau'}}(\xi)=0$ (car $f_{\mu,\tau}(\xi)=0$). Autrement dit, $i^*(\V''(g)) \subset \V'(f_\mu)$. Or $\V'(f_\mu)$ est une hypersurface irréductible de $\Spec \Y(\gotk)$ (car $f_\mu$ est irréductible dans $\Y(\gotk)$) donc $\overline{i^*(\V''(g))}=\V'(f_\mu)$. Ainsi $\V'(f_\mu) \subset \V'(f)$, donc la racine $\sqrt{f \Y(\gotk)}$ de l'idéal $f \Y(\gotk)$ est inclus dans $\sqrt{f_\mu \Y(\gotk)}$. Puisque $\Y(\gotk)$ est supposée factorielle et $f_\mu$ est irréductible, l'idéal $f_\mu \Y(\gotk)$ est premier donc $\sqrt{f_\mu \Y(\gotk)}=f_\mu \Y(\gotk)$ (voir \cite[14.2.3 et 14.3.3]{TYu}). Finalement $f \in \sqrt{f \Y(\gotk)} \subset f_\mu \Y(\gotk)$, d'où $f_\mu$ divise $f$ dans $\Y(\gotk)$.\par
Il existe donc $\widetilde{f} \in \Y(\gotk)$ tel que
$$x^a \prod_{z \in \mathbf{f}'} z^{s_{z}}=\prod_{z \in \mathbf{f}''} z^{s_{z}} \, f_\mu \, \widetilde{f}$$
que l'on peut simplifier avec l'égalité \eqref{decompintro} en une égalité de la forme :
$$x^a \prod_{z \in \mathbf{h}'} z^{t_{z}}=\prod_{z \in \mathbf{h}''} z^{t_{z}} \widetilde{f}$$
avec une partition $\mathbf{f}^\times=\mathbf{h}' \sqcup \mathbf{h}''$ vérifiant
\begin{itemize}
\item $\mathbf{h}' \subset \mathbf{f}'$,
\item $t_{z} \leq s_{z}$ pour tout $z \in \mathbf{h}'$ et il existe au moins un $z \in \mathbf{h}'$ tel que l'inégalité est stricte,
\end{itemize}
ainsi $\sum_{z \in \mathbf{h}'} t_z < \sum_{z \in \mathbf{f}'} s_z $.
On conclut alors par récurrence sur $\sum_{z \in \mathbf{f}'} s_{z}$.
\end{proof}
On peut maintenant conclure pour la preuve du théorème \ref{puissant}. Soit $x$ un semi-invariant de $\Sym(\gotk)$, que l'on suppose irréductible dans $\Sym(\gotk)$. Par le lemme \ref{lemcentral}, il existe $a \in \N^*$, $f \in \Y(\gotk)$ (non nul) et $s_{g} \in \N$ pour tout $g \in \mathbf{f}^\times$ tels que
\begin{equation}
x^a \prod_{g \in \mathbf{f}^\times} g^{s_{g}}=f. \label{pause}
\end{equation}
Par le lemme \ref{ultratech}, il existe donc $f' \in \Y(\gotk)$ et $t_z \in \N$ pour tout $z \in \mathbf{f}^\times$ tels que
$$
x^a= \prod_{z \in \mathbf{f}^\times} z^{t_z} f'.
$$
Si pour tout $z \in \mathbf{f}^\times$, on a $t_z=0$, alors $x^a$ est un semi-invariant de poids nul donc $x$ est lui-même un semi-invariant de poids nul donc un invariant. Sinon il existe un $t_{z}$, $z \in \mathbf{f}^\times$, non nul. Alors $z$ divise $x^a$ d'où $z  \propto x$ (puisque $z$ est supposé irréductible dans $\Sym(\gotk)$).\par
Par la proposition \ref{dixmier}, tout semi-invariant est produit de semi-invariants irréductibles dans $\Sym(\gotk)$, et tout élément de $\Sy(\gotk)$ est somme de semi-invariants, ainsi $\Sy(\gotk)$ est bien engendrée par $\Y(\gotk)$ et les éléments de $\mathbf{f}$.
\end{proof}
Dans le cas général, lorsque l'on a des décompositions $f_m \propto \prod_{t=1}^{r_{m}} (f_{m,t})^{\nu_{m,t}}$, appliquer le théorème \ref{puissant} nécessite de montrer l'irréductibilité de certains $f_{m,t}$. Or ce dernier point est difficile à démontrer a priori. Ainsi, dans les cas qui nous intéresseront ici, on utilisera plutôt le théorème suivant.
\begin{theo}
On reprend les hypothèses et notations du début de cette sous-section. 
On suppose que
\begin{itemize}
\item[(a)] $\Y(\gotk)$ est factorielle,
\item[(b)] $\GK \C[{\mathbf{f}}]=\GK \Sy(\gotk)$,
\item[(c)] les $f_m$ sont irréductibles dans $\Y(\gotk)$,
\item[(I)] il existe un morphisme de $\C$-algèbres $\vartheta : \Sym(\gotk) \rightarrow \Y(\gotk)$ tel que $\vartheta_{|\Y(\gotk)}$ est un isomorphisme.
\item[(II)] tout $f \in \mathbf{f}^\times$ est indivisible dans le semi-groupe multiplicatif $\Sy(\gotk)$ (c'est-à-dire que pour tout $a \geq 2$, $f$ n'est pas une puissance $a$\up{ème} dans $\Sy(\gotk)$),
\item[(III)] pour tout $f \in \mathbf{f}^\times$, il existe une $\C$-algèbre factorielle $\A_f$ et un morphisme de $\C$-algèbres $\vartheta_f : \Sy(\gotk) \rightarrow \A_f$ tel que $\vartheta_f(f)$ n'est pas inversible dans $\A_f$ et est premier avec $\vartheta_f(g)$ pour tout $g \in \mathbf{f}^\times \setminus \{f\}$.
\end{itemize}
Alors les $f_{m,t}$ dans $\mathbf{f}^\times$ sont irréductibles dans $\Sym(\gotk)$. En particulier, d'après le théorème \ref{puissant}, si de plus
\begin{itemize}
\item[(d)] les $f_m$ engendrent $\Y(\gotk)$,
\end{itemize}
alors $\Sy(\gotk)=\C[\mathbf{f}]$.
%$\mathbf{M}_2=\{\mu \in \llbracket 1,\delta \rrbracket \, | \, r'_{\mu} \geq 2 \}$, $\mathcal{E}=\{(\mu,\tau) \, | \, \mu \in \mathbf{M}_2, 1 \leq \tau \leq r'_{\mu}\}$ et
\label{inter}
\label{irreductibilite}
\end{theo}
%\begin{rema}\mbox{}
%\begin{enumerate}
%\label{tutti}
%\item Dans le cas où $\Y(\gotk)$ est librement engendrée par $f_1, \ldots, f_n$, les éléments irréductibles de $\Y(\gotk)$ sont les éléments de la forme $P(f_1, \ldots, f_n)$ où $P$ est un polynôme irréductible de $\C[X_1, \ldots, X_n]$.
%\item Quitte à multiplier certains $f_{m,t}$ et certains $f_m$ par une constante, on peut supposer que si $f_{m,t},g_{m',t'} \in \mathbf{f}$ sont liés, alors ils sont égaux. C'est une hypothèse que l'on supposera vérifiée dans la démonstration de ce théorème.
%\item En types $A$ et $C$, l'algèbre $\Y(\gotq)$ est polynomiale donc factorielle et si $d=\GK(\Y(\gotq))$, on exhibera un morphisme $\vartheta$ sous la forme 
%\begin{align*}
%\vartheta : \Sy(\gotq) \subset \C[\gotq^*] &\longrightarrow \Y(\gotq) \simeq \C[X_1, \ldots, X_d] \\
%f &\longmapsto f(q)
%\end{align*}
%avec $q \in \gotq^*_{\C[X_1, \ldots, X_d]}=\C[X_1, \ldots, X_d] \otimes_{\C} \gotq^*$, tel que pour $f_1, \ldots, f_d$ engendrant librement $\Y(\gotq)$, on a $f_m(q)=X_m$ pour tout $m$.
%\end{enumerate}
%\end{rema}
\begin{proof}
Soit $\Lambda:=\Lambda(\gotk)$ le semi-groupe des poids de $\Sy(\gotk)$. Soit $f=f_{m,t} \in \mathbf{f}^\times$ et $x$ un facteur irréductible de $f$. Appliquons le lemme \ref{lemcentral} à $x$. Il existe donc $a \in \N^*$ et $s_{z} \in \N$ pour tout $z \in \mathbf{f}^\times$ tels que
$$x^a \prod_{g \in \mathbf{f}^\times} g^{s_{g}} \in \Y(\gotk)$$
Soit $\vartheta$ un morphisme vérifiant l'hypothèse (I). Tout $g \in \mathbf{f}^\times$ s'écrit $g=f_{\mu,\tau}$ pour $\mu,\tau$ avec $\mu \in \mathbf{M}_2$ (par le lemme \ref{lemm}). Alors par l'équation \eqref{propto}, $g=f_{\mu,\tau}$ divise $f_\mu$ donc $\vartheta(g)$ divise $\vartheta(f_\mu)$. Comme $f_{\mu}$ est irréductible dans $\Y(\gotk)$ et $\vartheta_{|\Y(\gotk)}$ est un isomorphisme, l'élément $\vartheta(f_{\mu})$ est irréductible dans $\Y(\gotk)$, donc $\vartheta(g)$ est constant ou $\vartheta(g) \propto \vartheta(f_\mu)$. Le même argument s'applique à $x$ ($x$ divise $f_{m,t}$ donc $f_{m}$). Ainsi il existe $s_{\mu} \in \N$ pour tout $\mu \in \mathbf{M}_2$ tels que
$$\vartheta \left(x^a \prod_{g \in \mathbf{f}^\times} g^{s_{g}}\right)=\vartheta(x)^a \prod_{g \in \mathbf{f}^\times} \vartheta(g)^{s_{g}} \propto \prod_{\mu \in \mathbf{M}_2} \vartheta(f_\mu)^{s_{\mu}}=\vartheta \left(\prod_{\mu \in \mathbf{M}_2} f_\mu^{s_{\mu}} \right)$$
Comme $\vartheta_{|\Y(\gotk)}$ est un isomorphisme, avec les équations \eqref{propto}, on obtient une équation de la forme :
$$x^a \prod_{g \in \mathbf{f}^\times} g^{s_{g}} \propto \prod_{g \in \mathbf{f}^\times} g^{s'_{g}}$$
avec $s'_g \in \N$ pour tout $g \in \mathbf{f}^\times$. En simplifiant cette équation, il existe alors une partition $\mathbf{f}^\times=\mathbf{f}' \sqcup \mathbf{f}''$ et des entiers $u_{g} \geq 0$ pour tout $g \in \mathbf{f}^\times$ tels que
\begin{equation}
x^a \prod_{g \in \mathbf{f}'} g^{u_{g}} \propto \prod_{g \in \mathbf{f}''} g^{u_{g}}  \label{disj}
\end{equation}
Par l'hypothèse (III), on a pour tout $h \in \mathbf{f}^\times$, 
$$
\vartheta_h(x)^a \prod_{g \in \mathbf{f}'} \vartheta_h(g)^{u_{g}} \propto \prod_{g \in \mathbf{f}''} \vartheta_h(g)^{u_{g}}
$$
Supposons que $f \in \mathbf{f}'$, alors
\begin{itemize}
\item pour tout $h \in \mathbf{f}'$, on a $\vartheta_h(h)^{u_{h}}$ qui divise $\prod_{g \in \mathbf{f}''} \vartheta_h(g)^{u_{g}}$. Puisque $\vartheta_h(h)$ est non inversible et premier avec les $\vartheta_h(g)$, $g \in \mathbf{f}^\times \setminus \{h\}$, cela implique que $u_{h}=0$,
\item de même, pour tout $h \in \mathbf{f}''$, on a $\vartheta_h(h)^{u_{h}}$ qui divise $\vartheta_h(x)^a \prod_{g \in \mathbf{f}'} \vartheta_h(g)^{u_{g}}$.
Or $\vartheta_h(h)$ est non inversible et
\begin{itemize}
\item $\vartheta_h(h)$ est premier avec les $\vartheta_h(g)$, $g \in \mathbf{f}'$ (car $\mathbf{f}' \subset \mathbf{f}^\times \setminus \{h\})$,
\item $\vartheta_h(h)$ est premier avec $\vartheta_h(f)$ (puisque $h \neq f$), et donc $\vartheta_h(h)$ est premier avec $\vartheta_h(x)$ (puisque $\vartheta_h(x)$ divise $\vartheta_h(f)$).
\end{itemize}
Ainsi $\vartheta_h(h)$ est premier avec $\vartheta_h(x)^a \prod_{g \in \mathbf{f}'} \vartheta_h(g)^{u_{g}}$ donc $u_h=0$.
\end{itemize}
Finalement, l'équation \eqref{disj} devient
$$x^a \propto 1$$
d'où $x \in \C^\times$, ce qui est absurde car on a supposé que $x$ était irréductible (donc non constant).\par
On a donc $f \in \mathbf{f}''$. Alors de la même manière que précédemment, on montre que pour tout $h \in \mathbf{f}^\times \setminus \{f\}$, on a $u_{h}=0$. On ne peut cependant pas conclure pour $f$. L'équation \eqref{disj} devient alors
$$x^a \propto f^{u_{f}}$$
De même que précédemment, si $u_{f}=0$, on aboutit à une absurdité. On a donc $u_{f} > 0$, d'où $f$ divise $x^a$. Puisque $x$ est supposé irréductible, on a alors $f \propto x^{b}$ pour un certain $b \in \N^*$. Par l'hypothèse (II), on a alors $b=1$, donc $x \propto f$.\par
On vient donc de montrer que tout facteur irréductible $x$ de $f$ était associé à $f$. Le semi-invariant $f$ est donc irréductible.
\end{proof}
%\begin{enumerate}
%\item soit $x \in \Y(\gotk)$,
%\item soit il existe $\eta \in \N$ et $m,t$ tel que $x^a=f_{m,t}$.
%\end{enumerate}
Ainsi, lorsqu'on voudra montrer que $\Sy(\gotk)$ est engendrée par certains $f_{m,t}$, on appliquera le théorème \ref{inter}.
\begin{rema}
Dans les cas qui suivront, on appliquera le théorème \ref{inter} dans le cadre suivant :
\label{remaimpo}
\begin{itemize}
\item L'isomorphisme canonique de $\C$-espace vectoriels $f \in \gotk \mapsto [q \in \gotk^* \mapsto q(f)] \in (\gotk^{*})^*$ se prolonge en un unique isomorphisme d'algèbres entre $\Sym(\gotk)$ et l'algèbre $\C[\gotk^*]$ des fonctions polynomiales sur $\gotk^*$. On identifiera dans la suite un élément de $\Sym(\gotk)$ avec son image dans $\C[\gotk^*]$. En particulier, pour tous $x_1, \ldots, x_k \in \gotk$ et tout $q \in \gotk^*$, on a $(x_1 \ldots x_k)(q)=q(x_1) \ldots q(x_k)$.
\item Si $\A$ est une $\C$-algèbre, on pourra voir dans la suite un élément $y \in \Sym(\gotk) \simeq \C[\gotk^*]$ comme un élément de $\A[\A \otimes_{\C} \gotk^*]$ c'est-à-dire une fonction polynomiale sur le $\A$-module $\gotk^*_{\A}:= \A \otimes_{\C} \gotk^*$\index{$\gotk^*_{\A}$} en identifiant $y$ et $\id \otimes y$.
%\item Dans les cas que l'on étudiera, on exhibera les morphismes $\vartheta$ et $\vartheta_g$ (des hypothèses (I) et (III) du théorème \ref{inter}) sous la forme
%$$f \in \Sy(\gotq) \subset \A[\gotq^*] \longmapsto f(q) \in \A$$
%avec $q \in \gotq^*_{\A}$. Puisque l'on a la polynomialité de $\Y(\gotq)$ (par \cite{py13}), pour $\vartheta$, on pourra prendre $\A=\Y(\gotq) \simeq \C[X_1, \ldots, X_d]$. Pour $\vartheta_g$, pour simplifier, on prendra $\A_g=\C[X]$. 
\end{itemize}
\end{rema}
\begin{propn}
%Soit $\gotk$ une algèbre de Lie.
On reprend les hypothèses et notations du début de cette sous-section.  Considérons les hypothèses suivantes :
\label{632}
\begin{enumerate}
\item[(I')] il existe $g_1, \ldots, g_{d'}$ qui engendrent librement $\Y(\gotk)$ et $q \in \gotk^*_{\C[X_1, \ldots, X_{d'}]}$ tel que pour tout $m \in \llbracket 1,d' \rrbracket$, on a $g_m(q) \propto X_m$,
\item[(III')] pour tout $m \in \llbracket 1,d \rrbracket$ tel que $r_m \geq 2$ et $t \in \llbracket 1, r_m \rrbracket$, il existe $q_{m,t} \in \gotk^*_{\C[X]}$ tel que $\deg_X f_{m,t}(q_{m,t}) \geq 1$ et $f_{\mu,\tau}(q_{m,t}) \in \C^\times$ pour $(\mu,\tau) \neq (m,t)$.
\end{enumerate}
Alors les hypothèses (I') et (III') impliquent respectivement les hypothèses (I) et (III) du théorème \ref{inter}.
\end{propn}
Les $g_1, \ldots, g_{d'}$ de l'hypothèse (I') seront souvent les $f_1, \ldots, f_d$ introduits au début de cette sous-section, mais pas toujours.
\begin{proof} \mbox{}
\begin{enumerate}
\item[(I')] Sous l'hypothèse (I'), comme $\Y(\gotk)$ est librement engendrée par les $g_m$, $1 \leq m \leq d'$, l'application
\begin{align*}
\Sym(\gotk) &\longrightarrow \C[X_1, \ldots, X_{d'}] \\
f &\longmapsto f(q)
\end{align*}
restreinte à $\Y(\gotk)$ est un isomorphisme. Il suffit alors de composer cette application avec un isomorphisme $\C[X_1, \ldots, X_{d'}]  \stackrel{\simeq}{\rightarrow}\Y(\gotk)$ pour obtenir un morphisme $\vartheta$ voulu.
\item[(III')] Sous l'hypothèse (III'), l'application 
\begin{align*}
\vartheta_{f_{m,t}} : \Sy(\gotk) &\longrightarrow \C[X] \\
f &\longmapsto f(q_{m,t})
\end{align*}
vérifie l'hypothèse (III).
\qedhere
\end{enumerate}
\end{proof}
Les hypothèses (I') et (III') ont des interprétations plus géométriques. Cependant, on montrera ces propriétés telles quelles dans la suite.
\begin{rema}\mbox{}
\begin{itemize}
\item Plaçons-nous dans l'hypothèse (I'). Alors $q \in \gotk^*_{\C[X_1, \ldots, X_{d'}]}$ implique par évaluation une fonction polynomiale $q : (x_1, \ldots, x_{d'}) \in \C^{d'} \mapsto q(x_1, \ldots, x_{d'}) \in \gotk^*$, notons $\mathcal{Q}$ son image et $\overline{\mathcal{Q}}$ sa fermeture (pour la topologie de Zariski). Soit $q^* : f \in \C[\overline{\mathcal{Q}}] \rightarrow f \circ q \in \C[\C^{d'}]$ l'application duale. Supposons que $q^*$ est un isomorphisme. Cette hypothèse est vérifiée par exemple si $q$ est une fonction polynomiale de degré $1$ injective, de sorte que $\overline{\mathcal{Q}}=\mathcal{Q}$ est un espace affine de dimension $d'$. Cette dernière hypothèse sera notamment vérifiée pour les deux $q$ que l'on exhibe aux sous-sections \ref{sec343} et \ref{sec434}. Soit l'application de restriction %???????? à vérifier
$$\res_{\overline{\mathcal{Q}}} : f \in \Y(\gotk) \longmapsto f_{|\overline{\mathcal{Q}}} \in \C[\overline{\mathcal{Q}}]$$
Par l'hypothèse (I'), l'application $q^* \circ \res_{\overline{\mathcal{Q}}}$ est un isomorphisme, et donc $\res_{\overline{\mathcal{Q}}}$ est un isomorphisme. On dit que le morphisme $\res_{\overline{\mathcal{Q}}}$ est une "tranche" pour $\Y(\gotk)$. Si $\overline{\mathcal{Q}}$ est un espace affine, on parle de section de Kostant-Weierstrass. 
\item Plaçons-nous dans l'hypothèse (III'). Fixons $m,t$. Soit $x$ une racine de $f_{m,t}(q_{m,t})$. Alors en évaluant $q_{m,t}$ en $x$, on obtient un élément $q \in \gotk^*$ tel que $f_{m,t}(q)=0$ et $f_{m',t'}(q) \neq 0$  pour $(m',t') \neq (m,t)$. Autrement dit, si pour toute $f \in \Sym(\gotk) \simeq \C[\gotk^*]$, on note $\mathcal{V}(f)=\{q' \in \gotk^* \, | \, f(q')=0\}$, alors $\mathcal{V}(f_{m,t})$ n'est pas inclus dans l'union des $\mathcal{V}(f_{m',t'})$ pour $(m',t') \neq (m,t)$.
\end{itemize}
\end{rema}
\subsection{Cheminement associé à une forme linéaire sur un sous-module de $\M_n(\A)$}
Soit $\A$ une $\C$-algèbre polynomiale. On introduit une combinatoire pour étudier les $f(q)$ avec $f \in \Sym(\gotq)$ et $q \in \gotq^*_{\A}$. À la manière des graphes de probabilités et des matrices de transitions, dans cette sous-section, on associe un graphe orienté pondéré à toute forme $\A$-linéaire $q : M \rightarrow \A$ où $M$ est un sous-$\A$-module de $\M_n(\A)$.\par
\label{graphe}
\subsubsection{Cheminements : définitions, premières propriétés}
\begin{defi}
On appelle $\A$-\textbf{cheminement} (ou cheminement) un graphe orienté et pondéré par des éléments de $\A$, de sommets $I$, tel que pour tous $x,y \in I$, il existe une et une seule arête de $x$ vers $y$. On dira que $x$ est la base de l'arête et $y$ le but de l'arête.
\end{defi}
\begin{defi}
Soit $\mathcal{G}$ un cheminement.
\begin{itemize}
\item On appelle \textbf{support} de $\mathcal{G}$ et on note $\supp \mathcal{G}$\index{$\supp \mathcal{G}$} l'ensemble des $x \in I$ qui sont la base ou le but d'une arête de $\mathcal{G}$ de poids non nul.
\item Deux cheminements $\mathcal{G}$ et $\mathcal{G}'$ sont dits \textbf{compatibles} si leurs supports sont disjoints, et \textbf{incompatibles} sinon.
\item On appelle \textbf{graphe non trivial induit par $\mathcal{G}$} que l'on note $\widehat{\mathcal{G}}$\index{$\widehat{\mathcal{G}}$} le graphe orienté pondéré d'ensemble de sommets $\supp \mathcal{G}$, et d'ensemble d'arêtes toutes les arêtes de poids non nul de $\mathcal{G}$, avec les mêmes poids.
\end{itemize}
\end{defi}
Autrement dit, le graphe non trivial induit par $\mathcal{G}$ est le graphe dans lequel on ne garde que les arêtes non nulles et les sommets pertinents.
\begin{defi}
Soit $\mathcal{G}$ un cheminement. On appelle une "représentation graphique de $\mathcal{G}$" une représentation graphique de son graphe non trivial induit. Pour $x,y \in I$, on introduit de plus les notations graphiques suivantes :
\begin{itemize}
\item si l'arête de $x$ vers $y$ est de poids $a$, on représentera cette arête de la façon suivante : $x \stackrel{a}{\rightarrow} y$,
\item pour tout $a \in \A \setminus \{0\}$, on note $x \stackrel{(a)}{\rightarrow} y$ si l'on a $x \stackrel{ka}{\rightarrow} y$ avec $k \in \C^\times$,
\item on note $x \rightarrow y$ si l'on a $x \stackrel{(1)}{\rightarrow} y$,
\item on note $x \dashrightarrow y$ si l'on a $x \rightarrow x_1 \rightarrow x_2 \rightarrow \ldots \rightarrow x_l \rightarrow y$ avec $l \in \N$ quelconque,
\item on note $x \stackrel{*}{\rightarrow} y$ si l'on a $x \stackrel{a}{\rightarrow} y$ avec $a \in \A \setminus \{0\}$.
\end{itemize}
\label{repchem}
\end{defi}
\begin{defi}
Soient $\mathcal{G}$ et $\mathcal{G}'$ deux cheminements.
\begin{itemize}
\item On dit que $\mathcal{G}$ et $\mathcal{G}'$ sont \textbf{équivalents} si pour tous sommets $x,y \in I$, l'arête de $x$ vers $y$ dans $\mathcal{G}$ est de poids nul si et seulement si l'arête de $x$ vers $y$ dans $\mathcal{G}'$ est de poids nul.
\item Pour tous $x,y \in I$, soit $a_{x,y}$ (respectivement $a'_{x,y}$) le poids de l'arête de $x$ vers $y$ dans $\mathcal{G}$ (respectivement $\mathcal{G}'$). On appelle \textbf{somme de $\mathcal{G}$ et $\mathcal{G}'$}, et on note $\mathcal{G} + \mathcal{G}'$\index{$\mathcal{G} + \mathcal{G}'$}, le cheminement tel que pour tous $x,y \in I$, l'arête de $x$ vers $y$ est de poids $a_{x,y}+a'_{x,y}$.
\item Le cheminement $\mathcal{G}'$ est un \textbf{sous-cheminement} de $\mathcal{G}$ si $\widehat{\mathcal{G}'}$ est un sous-graphe de $\widehat{\mathcal{G}}$.
\item Le cheminement $\mathcal{G}$ est un \textbf{circuit} si $\widehat{\mathcal{G}}$ est un circuit (c'est-à-dire un graphe cyclique orienté, de poids non nuls quelconques),
\item Le cheminement $\mathcal{G}$ est un \textbf{graphe circuits} si $\widehat{\mathcal{G}}$ est une union de circuits à supports disjoints (d'arêtes de poids non nuls quelconques) ou, de manière équivalente, si $\mathcal{G}$ est une somme de circuits compatibles.
\item Un \textbf{sous-graphe circuits} (respectivement un \textbf{sous-circuit}) de $\mathcal{G}$ est un sous-cheminement de $\mathcal{G}$ qui est également un graphe circuits (respectivement un circuit).
\end{itemize}
\label{circuits}
\end{defi}
\begin{prop}
Pour tout $J \subset \llbracket 1,n \rrbracket$, on a une bijection entre $\mathfrak{S}(J)$ et l'ensemble des graphes circuits de support $J$ à équivalence près.\label{davant}
\end{prop}
\begin{proof}
La bijection consiste à associer à une permutation $\sigma \in \mathfrak{S}(J)$ le graphe $\mathcal{G}^\sigma$ dont les arêtes de poids non nul sont les $l \stackrel{1}{\rightarrow} \sigma(l)$ pour tout $l \in J$.
\end{proof}

\subsubsection{Cheminement associé à $q$}
Soit $M$ un sous-$\A$-module de l'ensemble $\M_n(\A)$ des matrices carrées $n \times n$ à coefficients dans $\A$. La famille $(e_{v,w})_{1 \leq v,w \leq n}$ forme une base canonique du $\A$-module $\M_n(\A)$ et $(e_{v,w}^*)_{1 \leq v,w \leq n}$ est sa base duale. On note $\mor(M,\A)$ l'ensemble des applications $\A$-linéaires de $M$ dans $\A$. Soit $\pr$ une application $\A$-linéaire de $\M_n(\A)$ dans $M$.% telle que $\pr(e_{v,w}) \neq 0$ pour tous $v,w$.
\begin{defi}
Pour tout $q \in \mor(M,\A)$, on définit $\mathcal{G}(q)$\index{$\mathcal{G}(q)$} un cheminement associé à $q$ (ainsi que $M$ et $\pr$) appelé \textbf{graphe de $q$}, tel que l'arête de $x$ vers $y$ a pour poids $q(\pr(e_{x,y}))$.\par
Si $M=\M_n(\A)$ et $\pr=\id$, on dira que $\mathcal{G}(q)$ est le \textbf{graphe de $q$ en type $\gl_n$}. Si $M=\syp_n(\A)$ et $\pr=\pr^C$ (que l'on définit à la section \ref{defC}), on dira que $\mathcal{G}(q)$ est le \textbf{graphe de $q$ en type $C$}.
\end{defi}
\begin{exem}
En type $\gl_n$, pour $q=ae_{x,y}^*$, $a \in \A \setminus \{0\}$, le cheminement $\mathcal{G}(q)$ a une seule arête de poids non nul : $x \stackrel{a}{\rightarrow} y$.
\end{exem}
\begin{prop}
Si $q, q' \in \mor(M,\A)$, alors $\mathcal{G}(q+q')=\mathcal{G}(q)+\mathcal{G}(q')$.
\end{prop}
\begin{defiprop}
Pour tout cheminement $\mathcal{H}$, on définit dans l'algèbre symétrique $\Sym(M)$ le monôme $\mathcal{S}_{\mathcal{H}}=\prod_{x \stackrel{*}{\rightarrow} y \, \in \mathcal{H}} \pr(e_{x,y})$\index{$\mathcal{S}_{\mathcal{H}}$}. On appelle $\mathcal{S}_{\mathcal{H}}$ le \textbf{${\mathcal{H}}$-monôme} associé à $M$ et $\pr$. Un \textbf{${\mathcal{H}}$-monôme en type $\gl_n$} est un ${\mathcal{H}}$-monôme associé à $M=\M_n(\A)$ et $\pr = \id$ ; un \textbf{${\mathcal{H}}$-monôme en type $C$} est un ${\mathcal{H}}$-monôme associé à $M=\syp_n(\A)$ et $\pr=\pr^C$ (voir section \ref{defC}).\par
Pour tout sous-cheminement $\mathcal{H}$ de $\mathcal{G}(q)$, on a 
$$\mathcal{S}_{\mathcal{H}}(q) = \prod_{x \stackrel{*}{\rightarrow} y \, \in \mathcal{H}}  q(\pr(e_{x,y})) \neq 0$$
où $ q(\pr(e_{x,y})) \in \A \setminus \{0\}$ est le poids de l'arête de $x$ vers $y$ dans $\mathcal{G}(q)$.\par \label{neq}
\end{defiprop}
\begin{rema}
Si $\mathcal{H}$ est un sous-cheminement de $\mathcal{G}(q)$, la notion de \emph{type de $\mathcal{H}$-monôme} est indépendante du \emph{type de $\mathcal{G}(q)$}. Plus précisément, si $\mathcal{G}(q)$ est un graphe en type $C$, on pourra définir un $\mathcal{H}$-monôme en type $\gl_n$, et vice versa.
\end{rema}
%On suppose à présent et pour le reste de cette sous-section que pour tous $x,y \in \llbracket 1,n \rrbracket$, on a $e_{x,y}^*(\pr(e_{x,y})) \neq 0$ (on a noté par abus $e_{x,y}^*$ la restriction de la forme linéaire $e_{x,y}^* \in \M_n(A)^*$ à $V$).
Pour $\gl_n$, on a $$F_m=\sum_{\stackrel{J \subset I}{|J|=m}} \sum_{\sigma \in \mathfrak{S}(J)} \varepsilon(\sigma) \prod_{l \in J} e_{l,\sigma(l)}.$$
Cette expression nous incite alors à étudier les termes du type $\prod_{l \in J} \pr(e_{l,\sigma(l)})$, et plus précisément pour pouvoir vérifier les hypothèses (I) et (III) du théorème \ref{inter}, leur valeur en $q$.
\begin{propn}\mbox{}
\label{elemgr}
Soit $\mathcal{S} \in \Sym(M)$ et $q \in \mor(M,\A)$. Les affirmations suivantes sont équivalentes :
\begin{itemize}
\item $\mathcal{S}$ est de la forme $\prod_{l \in J} \pr(e_{l,\sigma(l)})$ avec $J \subset \llbracket 1,n \rrbracket$, $\sigma \in \mathfrak{S}(J)$ et $\mathcal{S}(q) \neq 0$,
\item $\mathcal{S}=\mathcal{S}_{\mathcal{H}}$ avec $\mathcal{H}$ un sous-graphe circuits de $\mathcal{G}(q)$.
\end{itemize}
\end{propn}
\begin{proof}
Montrons les points suivants.
\begin{enumerate}
\item Si $\mathcal{H}$ est un graphe circuits, alors $\mathcal{S}_{\mathcal{H}}$ est de la forme $\prod_{l \in J} \pr(e_{l,\sigma(l)})$ avec $J \subset \llbracket 1,n \rrbracket$ et $\sigma \in \mathfrak{S}(J)$. 
%\item Pour tout $q \in V^*$ et tout monôme $\mathcal{S}$ de la forme $\prod_{(x,y) \in E} \pr(e_{x,y})$, avec $E \subset \llbracket 1,n \rrbracket ^2$, on a $\mathcal{S}(q) \neq 0 \Leftrightarrow \forall (x,y) \in E, x \stackrel{\star}{\rightarrow} y$.
\item Pour tout $\mathcal{S}=\prod_{l \in J} \pr(e_{l,\sigma(l)})$ avec $J \subset \llbracket 1,n \rrbracket$ et $\sigma \in \mathfrak{S}(J)$, et tout $q \in \mor(M,\A)$, on a $\mathcal{S}(q) \neq 0$ si et seulement si $\mathcal{S}=\mathcal{S}_{\mathcal{H}}$ pour un certain sous-graphe circuits $\mathcal{H}$ de $\mathcal{G}(q)$.
\end{enumerate}
Le point (1) est une conséquence directe de la propriété \ref{davant} ($J$ est le support de $\mathcal{H}$). Le sens réciproque du point (2) résulte de la propriété \ref{neq}. Pour le sens direct, si $\mathcal{S}(q) \neq 0$, alors pour tout $l \in J$, on a $q(\pr(e_{l,\sigma(l)})) \neq 0$, donc il existe une arête entre $l$ et $\sigma(l)$. Le sous-cheminement $\mathcal{H}$ de $\mathcal{G}(q)$ de support $J$ et d'arêtes non nulles l'ensemble des arêtes de $l$ vers $\sigma(l)$ pour $l \in J$ forme un sous-graphe circuits de $\mathcal{G}(q)$, tel que $\mathcal{S}_{\mathcal{H}}=\mathcal{S}$.\par
Montrons maintenant la proposition. Soit $\mathcal{S}$ de la forme $\prod_{l \in J} \pr(e_{l,\sigma(l)})$ avec $J \subset \llbracket 1,n \rrbracket$ et $\sigma \in \mathfrak{S}(J)$ tel que $\mathcal{S}(q)\neq 0$. Alors par (2), on a $\mathcal{S}=\mathcal{S}_{\mathcal{H}}$ pour un certain sous-graphe circuits $\mathcal{H}$ de $\mathcal{G}(q)$. Inversement, si $\mathcal{S}=\mathcal{S}_{\mathcal{H}}$ pour un certain sous-graphe circuits $\mathcal{H}$ de $\mathcal{G}(q)$, alors par (1), $\mathcal{S}$ est bien de la forme $\prod_{l \in J} \pr(e_{l,\sigma(l)})$ avec $J \subset \llbracket 1,n \rrbracket$ et $\sigma \in \mathfrak{S}(J)$, et donc en réappliquant (2), on a $\mathcal{S}(q) \neq 0$.
\end{proof}
La proposition \ref{elemgr} affirme donc qu'étudier les monômes $\mathcal{S}$ de la forme $\prod_{l \in J} \pr(e_{l,\sigma(l)})$ avec $J \subset \llbracket 1,n \rrbracket$ et $\sigma \in \mathfrak{S}(J)$ tels que $\mathcal{S}(q)\neq 0$ revient à étudier les sous-graphes circuits de $\mathcal{G}(q)$.
\begin{rema}
La réciproque de (1) dans la démonstration de \ref{elemgr} est fausse en général. Soit $M$ le sous-$\A$-module de $\M_2(\A)$ défini par $M=\A e_{1,1} \oplus \A e_{2,2}$ et $\pr$ l'application $\A$-linéaire de $\M_2(\A)$ dans $M$ définie par $\pr(e_{1,1})=\pr(e_{1,2})=e_{1,1}$ et $\pr(e_{2,1})=\pr(e_{2,2})=e_{2,2}$. Soit $q= e_{1,1}^*+e_{2,2}^*$. Le graphe $\mathcal{G}(q)$ est alors
\vspace{-1cm}
\begin{center}
\begin{tikzpicture}
  \tikzset{LabelStyle/.style = {fill=white}}
  \SetGraphUnit{2}
  \Vertex{1}
  \EA(1){2}
  \Loop[dist = 3cm, dir = WE, label = 1](1.west)
  \Loop[dist = 3cm, dir = EA, label = 1](2.east)
  \tikzset{EdgeStyle/.style = {->,bend right=60}}
  \Edge[label = $1$](1)(2)
  \Edge[label = $1$](2)(1)
\end{tikzpicture}
\end{center}
%\vspace{-1cm}
Le sous-cheminement $\mathcal{H}$ représenté comme suit
%\vspace{-1cm}
\begin{center}
\begin{tikzpicture}
  \useasboundingbox (-3,-2) rectangle (5,2);
  \tikzset{LabelStyle/.style = {fill=white}}
  \SetGraphUnit{2}
  \Vertex{1}
  \EA(1){2}
  \Loop[dist = 3cm, dir = WE, label = 1](1.west)
  \tikzset{EdgeStyle/.style = {->,bend right=60}}
  \Edge[label = $1$](2)(1)
\end{tikzpicture}
\end{center}
\vspace{-1cm}
n'est pas un sous-graphe circuits et donne $\mathcal{S}_{\mathcal{H}}=\pr(e_{1,1})\pr(e_{2,1})=\pr(e_{1,1})\pr(e_{2,2})$, ce qui donne un contre-exemple à la réciproque de (1).
\end{rema}
\subsection{Application du théorème \ref{inter} à $\gotk=\gotq$ en type $\gl_n$}
\label{sec343}
Par le théorème \ref{semiinv}, dans le cas d'une contraction parabolique standard $\gotq$ en type $\gl_n$, l'ensemble $\mathbf{F}$ des $F_{m,t}$ forme une famille d'éléments algébriquement indépendants. On veut donc conclure à la polynomialité de $\Sy(\gotk)$ lorsque $\gotk=\gotq$ en appliquant le théorème \ref{inter} avec $\mathbf{f}=\mathbf{F}$, où les égalités de la forme \eqref{propto} sont ici les égalités \eqref{decomp}. On vérifie d'abord les hypothèses (a), (b), (c) et (d) :
\begin{enumerate}
\item[(a)] L'algèbre $\Y(\gotq)$ est polynomiale (proposition \ref{fdroit}) donc en particulier factorielle.
\item[(b)] Par le théorème \ref{dimind}, on a $\GK \Sy(\gotq)=\ind \gotq_\Lambda = \Card \mathbf{F}$. Comme l'ensemble $\mathbf{F}$ des $F_{m,t}$ forme une famille d'éléments algébriquement indépendants (théorème \ref{semiinv} (4)), on a $\GK \Sy(\gotq)= \GK \C [\mathbf{F}]$.
\item[(c,d)] Les $f_m=F_m^\bullet$ engendrent librement $\Y(\gotq)$ (proposition \ref{fdroit}) donc en particulier, sont irréductibles dans $\Y(\gotq)$.
\end{enumerate}
Il reste donc à vérifier les hypothèses (I), (II) et (III).
\subsubsection{Hypothèse (II)}
Tous les $F_m^\bullet$ sont des sommes de monômes de la forme $\prod_{l \in J} e_{l,\sigma(l)}$ avec $J \subset I$ et $\sigma \in \mathfrak{S}(J)$. En particulier, aucun monôme de $F_m^\bullet$ ne contient de terme $e_{v,w}$ au carré. Autrement dit, pour tous $v,w$, on a $\deg_{e_{v,w}} F_m^\bullet \leq 1$, et on a l'égalité si et seulement si $e_{v,w}$ apparaît dans un monôme de $F_m^\bullet$. Pour un tel $e_{v,w}$ vérifiant ces conditions équivalentes, on a $\deg_{e_{v,w}} F_{m,t} \leq 1$, et on a l'égalité si et seulement si $e_{v,w}$ apparaît dans un monôme de $F_{m,t}$. Par relation sur les degrés partiels, ceci implique que $F_{m,t}$ ne peut pas être une puissance $a$\up{ème} dans $\Sym(\gotq)$ pour $a \geq 2$, et est donc indivisible dans $\Sym(\gotq)$.
\begin{rema}
Soit $F$ un semi-invariant de $\Sy(\gotq)$ de poids $\lambda$. Si $\lambda$ est indivisible dans le semi-groupe additif $\Lambda(\gotq)$ (voir notations de la section \ref{sec2}), alors $F$ est indivisible dans le semi-groupe multiplicatif $\Sym(\gotq)$. On rappelle que par le lemme \ref{poidsgen}, on a $\Lambda(\gotq) \subset \bigoplus_{\ell \in I_{\gotp}} \rat \varpi_\ell$. Pour montrer qu'un poids $\lambda$ est indivisible dans $\Lambda(\gotq)$, il suffit donc de montrer qu'il est indivisible dans $\bigoplus_{\ell \in I_{\gotp}} \rat \varpi_\ell$. Dans notre cas, les poids des $F_{m,t}$ sont presque tous indivisibles dans $\bigoplus_{\ell \in I_{\gotp}} \rat \varpi_\ell$. Il existe cependant une exception : le \textit{cas de la racine centrale}, c'est-à-dire $n$ pair et $\pi \setminus \pi'=\{\alpha_{n/2}\}$. Dans ce cas particulier, on rappelle que (exemple \ref{raccenetc}) $F_{n,1}=\Delta_{I_2, I_1}$ et $F_{n,2}=\Delta_{I_1, I_2}$ sont de poids respectifs $-2\varpi_{n/2}$ et $2\varpi_{n/2}$ ; pour $m \neq n$, on a $r_m=1$ et $F_{m,1}=F_m^\bullet$ donc $\lambda_{m,1}=0$.
Or $F_{n,1}$ et $F_{n,2}$ sont deux polynômes déterminants sur $\gl_n$, donc irréductibles d'où indivisibles dans le semi-groupe multiplicatif $\Sym(\gotq)$.
\end{rema}

\subsubsection{Hypothèses (I) et (III)}
Dans notre cas, on va montrer que $\gotk=\gotq$ vérifie les hypothèses (I') et (III') de la proposition \ref{632}, avec $g_m=f_m=F_m^\bullet$ et $f_{m,t}=F_{m,t}$ pour tous $m,t$.\par
%Soit $\iota \in I$ et $t$ tel que $\iota \in I_t$.
Fixons $\xi \in \llbracket 1,s \rrbracket$. On considère le graphe orienté non pondéré cyclique $\mathcal{C}$ suivant (que l'on ne considère pas comme un cheminement) :
\begin{center}
\begin{tikzpicture}[scale=0.7]
\node (1) at (0,3) {$1$};
\node (2) at (2.598,1.5) {$2$};
\draw[->] (0.366,2.978) arc (83:37:3);
\node (3) at (2.598,-1.5) {$3$};
\draw[->] (2.762,1.172) arc (23:-23:3);
\node (s) at (-2.598,1.5) {$s$};
\draw[->] (-2.396,1.805) arc (143:97:3);
\draw[->, dashed] (2.396,-1.805) arc (-37:-203:3);
\end{tikzpicture}
\end{center}
On pondère cette fois-ci les sommets de ce graphe : le sommet $k$ est pondéré par l'entier $p_{0,k}:=i_k=|I_k|$. La somme des poids est $\sum_k i_k=n$. On définit alors des suites $(v_m)_{1 \leq m \leq n}$, $(t_m)_{1 \leq m \leq n}$ et $(p_{m,k})_{1 \leq m \leq n, 1 \leq k \leq s}$ comme suit.
\begin{itemize}
\item À l'étape $1$,
\begin{enumerate}
\item[(1-a)] on se pose au sommet $t_1:=\xi$ du graphe $\mathcal{C}$, puis
\item[(1-b)] on choisit $v_1 \in I_{t_1}$, puis
\item[(1-c)] on diminue de $1$ le poids du sommet $t_1$.
\end{enumerate}
Après cette étape (1-c), tout sommet $k$ est pondéré par l'entier $p_{1,k}=|I_k \setminus (I_k \cap \{v_1\})|$, et $\sum_k p_{1,k}=n-1$.
%\item À l'étape $2$,
%\begin{enumerate}
%\item[(2-a)] depuis $t_1$, on avance dans le graphe $\mathcal{C}$  suivant l'orientation jusqu'au prochain sommet $t_2$ tel que $p_{1,t_2}>0$ (si $t_2=t_1$, alors on doit parcourir suivant l'orientation tout le graphe $\mathcal{C}$), puis
%\item[(2-b)] on choisit $v_2 \in I_{t_2} \setminus (I_{t_2} \cap \{v_1\})$ (de cardinal $p_{1,t_2}>0$), puis
%\item[(2-c)] on diminue de $1$ le poids du sommet $t_2$.
%\end{enumerate}
%Après cette étape (2-c), tout sommet $k$ est pondéré par l'entier $p_{2,k}=|I_k \setminus (I_k \cap \{v_1, v_2\})|$, et $\sum_k p_{2,k}=n-2$.
\item On définit alors ce processus par récurrence, de sorte qu'à l'étape $u$,
\begin{enumerate}
\item[($u$-a)] depuis $t_{u-1}$, on avance dans le graphe $\mathcal{C}$  suivant l'orientation jusqu'au prochain sommet $t_u$ tel que $p_{u-1,t_u}>0$ (si $t_u=t_{u-1}$, alors on doit parcourir tout le graphe $\mathcal{C}$ suivant l'orientation), puis
\item[($u$-b)] on choisit $v_u$ dans $I_{t_u} \setminus (I_{t_u} \cap \{v_1, \ldots, v_{u-1}\})$, qui est un ensemble de cardinal $p_{u-1,t_u}>0$, puis
\item[($u$-c)] on diminue de $1$ le poids du sommet $t_u$.
\end{enumerate}
Après l'étape ($u$-c), tout sommet $k$ est pondéré par l'entier $p_{u,k}=|I_k \setminus (I_k \cap \{v_1, \ldots, v_{u}\})|$, et $\sum_k p_{u,k}=n-u$.
\item On continue ce procédé jusqu'à la fin de l'étape $n$, c'est-à-dire après l'étape ($n$-c) où la somme des poids est nulle, donc tous les poids des sommets sont nuls.
%$t_{m+1} = \left\{ \begin{array}{ll}\min \{ u > t_m \, | \, I_{u}\setminus\{v_1, \ldots,v_m \}\neq \emptyset \} &\text{si } \{ u > t_m \, | \, I_{[u]}\setminus\{v_1, \ldots,v_m \}\neq \emptyset \} \nonvide \\ \min \{ u \, | \, I_{u}\setminus\{v_1, \ldots,v_m \}\neq \emptyset \} & \text{sinon}\end{array}\right. $ et 
%$v_{m+1} =\left\{ \begin{array}{ll}\min \{ v \succ v_m \, | \, v \notin \{v_1, \ldots,v_m \} \} &\text{si } \{ v \succ v_m \, | \, v \notin \{v_1, \ldots,v_m \} \}  \nonvide \\ \min \{ v \, | \, v \notin \{v_1, \ldots,v_m \} \} & \text{sinon}\end{array}\right.$.
\end{itemize}
L'application $m \in \llbracket 1,n \rrbracket \mapsto v_m \in I$ est  bijective. En particulier, $\{v_1, \ldots, v_n \}=I$.
\begin{exem}
Dans l'exemple \ref{leexemple}, supposons $\xi=t_1=4$.
On reprend les schémas de l'exemple \ref{exemple}. Cette fois-ci, on remplit les cases des nombres de $1$ à $n$ dans l'ordre croissant \textbf{de gauche à droite puis de haut en bas}. Par exemple, ici
\begin{center}
\begin{tikzpicture}[scale=0.5]
\draw (0,0)--(0,5)--(4,5)--(4,4)--(1,4)--(1,3)--(4,3)--(4,2)--(2,2)--(2,1)--(1,1)--(1,0)--(0,0);
\draw (0,1)--(1,1);
\draw (0,2)--(2,2);
\draw (0,3)--(1,3);
\draw (0,4)--(1,4);
\draw (1,1)--(1,3);
\draw (1,4)--(1,5);
\draw (2,2)--(2,3);
\draw (2,4)--(2,5);
\draw (3,4)--(3,5);
\draw (3,2)--(3,3);
\draw (0.5,4.5) node{$1$};
\draw (1.5,4.5) node{$2$};
\draw (2.5,4.5) node{$3$};
\draw (3.5,4.5) node{$4$};
\draw (0.5,3.5) node{$5$};
\draw (0.5,2.5) node{$6$};
\draw (1.5,2.5) node{$7$};
\draw (2.5,2.5) node{$8$};
\draw (3.5,2.5) node{$9$};
\draw (0.5,1.5) node{$10$};
\draw (1.5,1.5) node{$11$};
\draw (0.5,0.5) node{$12$};
\draw (-0.5,4.5) node{$I_1$};
\draw (-0.5,3.5) node{$I_2$};
\draw (-0.5,2.5) node{$I_3$};
\draw (-0.5,1.5) node{$I_4$};
\draw (-0.5,0.5) node{$I_5$};
\end{tikzpicture}
\end{center}
On construit alors une suite $(v_1, \ldots, v_{12})$ en choisissant et cochant un élément de la ligne $t_1$, puis en choisissant et cochant un élément de la ligne inférieure non déjà coché, et ainsi de suite. S'il n'y a plus de case non cochée sur une ligne, on passe à la ligne suivante. Si on était sur la dernière ligne, on revient à la première ligne. Ici une suite possible est $(10,12,1,5,6,11,2,7,3,8,4,9)$.
%On permute les cases de la ligne contenant $\iota$ de sorte que $\iota$ soit sur la première colonne. Pour cela, on place $\iota$ sur la première colonne, puis les autres nombres dans l'ordre croissant de gauche à droite. On effectue ensuite une permutation cyclique (de longueur $s$) sur les lignes de cette figure de sorte que le nombre $v_1$ soit sur la case en haut à gauche. Si ici, $v_1=10$, alors le schéma ressemble désormais à ceci :
%\begin{center}
%\begin{tikzpicture}
%\draw (0,0)--(0,5)--(2,5)--(2,4)--(1,4)--(1,3)--(4,3)--(4,2)--(1,2)--(1,1)--(4,1)--(4,0)--(0,0);
%\draw (0,1)--(1,1);
%\draw (0,2)--(1,2);
%\draw (0,3)--(1,3);
%\draw (0,4)--(1,4);
%\draw (1,0)--(1,1);
%\draw (1,2)--(1,3);
%\draw (1,4)--(1,5);
%\draw (2,0)--(2,1);
%\draw (2,2)--(2,3);
%\draw (3,0)--(3,1);
%\draw (3,2)--(3,3);
%\draw (0.5,4.5) node{$10$};
%\draw (1.5,4.5) node{$11$};
%\draw (0.5,3.5) node{$12$};
%\draw (0.5,2.5) node{$1$};
%\draw (1.5,2.5) node{$2$};
%\draw (2.5,2.5) node{$3$};
%\draw (3.5,2.5) node{$4$};
%\draw (0.5,1.5) node{$5$};
%\draw (0.5,0.5) node{$6$};
%\draw (1.5,0.5) node{$7$};
%\draw (2.5,0.5) node{$8$};
%\draw (3.5,0.5) node{$9$};
%\draw (-0.5,4.5) node{$I_4$};
%\draw (-0.5,3.5) node{$I_5$};
%\draw (-0.5,2.5) node{$I_1$};
%\draw (-0.5,1.5) node{$I_2$};
%\draw (-0.5,0.5) node{$I_3$};
%\end{tikzpicture}
%\end{center}
%On lit alors les termes de la suite $v_j$ en parcourant les cases \textbf{de haut en bas puis de gauche à droite}. Par exemple, ici $(v_j)_{1 \leq j \leq 10}=(10,12,1,5,6,11,2,7,3,8,4,9)$.
\end{exem}
On définit alors $q \in \C[X_1, \ldots, X_n] \otimes_{\C} \gotq^*$ de manière similaire aux matrices compagnons :
\begin{equation}
q=\sum_{\ell=1}^n \, X_\ell \, e_{v_1, v_\ell}^* + \sum_{\ell=1}^{n-1} e_{v_{\ell+1},v_{\ell}}^*
\label{defprec}
\end{equation}
\begin{exem}
Dans notre exemple, avec la suite $(v_\ell)_\ell $ choisie précédemment, on obtient $q$ donné par la matrice suivante :
\begin{center}
\begin{tikzpicture}[scale=0.7]
\foreach \k in {1,2,...,11}
	{\draw[color=gray!20]  (0,\k)--(12,\k);
	\draw[color=gray!20] (\k,0)--(\k,12);}
\draw (0,0)--(12,0);
\draw (0,0)--(0,12);
\draw (0,12)--(12,12);
\draw (12,0)--(12,12);
\draw[ultra thick] (0,8)--(4,8)--(4,7)--(5,7)--(5,3)--(9,3)--(9,1)--(11,1)--(11,0);
%\draw (2,2) node{\LARGE $\gotn^-$};
%\draw (8,8) node{\LARGE $\gotp$};
%\draw[dashed] (3,10)--(3,7)--(4,7)--(4,6)--(7,6)--(7,3)--(9,3)--(9,1)--(10,1);
%\draw (1.5,8.5) node{\Large $\g_{I_1}$};
%\draw (3.5,6.5) node{\Large $\g_{I_2}$};
%\draw (5.5,4.5) node{\Large $\g_{I_3}$};
%\draw (8,2) node{\Large $\g_{I_4}$};
%\draw (9.5,0.5) node{\Large $\g_{I_5}$};
\draw (9.5,2.5) node{\small $ X_1$};
\draw (9.5,0.5) node{\Large $1$};
\draw (11.5,2.5) node{\small $X_2$};
\draw (11.5,11.5) node{\Large $1$};
\draw (0.5,2.5) node{\small $ X_3$};
\draw (0.5,7.5) node{\Large $1$};
\draw (4.5,2.5) node{\small $ X_4$};
\draw (4.5,6.5) node{\Large $1$};
\draw (5.5,2.5) node{\small $ X_5$};
\draw (5.5,1.5) node{\Large $1$};
\draw (10.5,2.5) node{\small $ X_6$};
\draw (10.5,10.5) node{\Large $1$};
\draw (1.5,2.5) node{\small $ X_7$};
\draw (1.5,5.5) node{\Large $1$};
\draw (6.5,2.5) node{\small $ X_8$};
\draw (6.5,9.5) node{\Large $1$};
\draw (2.5,2.5) node{\small $ X_9$};
\draw (2.5,4.5) node{\Large $1$};
\draw (7.5,2.5) node{\small $X_{10}$};
\draw (7.5,8.5) node{\Large $1$};
\draw (3.5,2.5) node{\small $ X_{11}$};
\draw (3.5,3.5) node{\Large $1$};
\draw (8.5,2.5) node{\small $X_{12}$};
\end{tikzpicture}
\end{center}
\end{exem}
%\begin{prop}
%Pour toute $\C$-algèbre $A$, $(\gl_n(A),\id)$ associe les cycles (voir la définition \ref{ascy}).
%\end{prop}
%\begin{proof}
%Supposons que $\mathcal{S}_{\mathcal{H}}$ est de la forme $\prod_{l \in J} e_{l,\sigma(l)}$ avec $J \subset \llbracket 1,n \rrbracket$ et $\sigma \in \mathfrak{S}(J)$. Alors toute arête $x \stackrel{\star}{\rightarrow} y$ de $\mathcal{H}$ correspond à un élément $e_{x,y}$ qui doit être égal à un certain $e_{l,\sigma(l)}$ , et inversement. Autrement dit, l'ensemble des arêtes $x \stackrel{\star}{\rightarrow} y$ de $\mathcal{H}$ correspond à l'ensemble des arêtes $l \stackrel{\star}{\rightarrow} \sigma(l)$ avec $l \in J$. L'ensemble de ces arêtes forme alors un graphe circuits (où chaque circuit correspond à une orbite de $\sigma$).
%\end{proof}
L'élément $q$ est choisi de sorte que son graphe $\mathcal{G}(q)$ en type $\gl_n$ est :
\begin{center}
\begin{tikzpicture}
  \tikzset{LabelStyle/.style = {fill=white}}
  \tikzset{VertexStyle/.style = {%
  shape = circle, minimum size = 28pt,draw}}
  \SetGraphUnit{2.5}
  \Vertex[L=$v_1$]{1}
  \EA[L=$v_2$](1){2}
  \EA[L=$v_3$](2){3}
  \Loop[dist = 2cm, dir = WE, label = $X_1$](1.west)
  \Edge[style= {->}, label = 1](2)(1)
  \Edge[style= {->}, label = 1](3)(2)
  
  \tikzset{EdgeStyle/.style = {->,bend left=30}}
  \Edge[label=$X_2$](1)(2)
  \tikzset{EdgeStyle/.style = {->,bend left=34}}
  \Edge[label=$X_3$](1)(3)
  \SetVertexNormal[Shape = circle, LineColor=white, MinSize=28pt]
  \tikzset{EdgeStyle/.style = {->}}
  \EA[L=$\ldots$](3){4}
  \Edge[style= {->}, label = 1](4)(3)
  %\SetUpVertex[LineColor=black]
  \SetVertexNormal[Shape = circle, LineColor=black, MinSize=28pt]
  \EA[L=$v_{n-1}$](4){n-1}
  \Edge[style= {->}, label = 1](n-1)(4)
  \EA[L=$v_n$](n-1){n}
  \Edge[style= {->}, label = 1](n)(n-1)
  \tikzset{EdgeStyle/.style = {->,bend left=38}}
  \Edge[label=$X_{n-1}$](1)(n-1)
  \tikzset{EdgeStyle/.style = {->,bend left=42}}
  \Edge[label=$X_n$](1)(n)
\end{tikzpicture}
\end{center}
On remarque que pour chaque $m \in \llbracket 1,n \rrbracket$, le cheminement $\mathcal{G}(q)$ admet un seul sous-graphe circuits à $m$ sommets, qui est
\begin{center}
\begin{tikzpicture}
  \tikzset{LabelStyle/.style = {fill=white}}
  \tikzset{VertexStyle/.style = {%
  shape = circle, minimum size = 28pt,draw}}
  \SetGraphUnit{2.5}
  \Vertex[L=$v_1$]{1}
  \EA[L=$v_2$](1){2}
  \EA[L=$v_3$](2){3}
  %\Loop[dist = 2cm, dir = WE, label = $X_1$](1.west)
  \Edge[style= {->}, label = $1$](2)(1)
  \Edge[style= {->}, label = $1$](3)(2)
  
  \tikzset{EdgeStyle/.style = {->,bend left=30}}
  %\Edge[label=$-X_2$](1)(2)
  %\tikzset{EdgeStyle/.style = {->,bend left=34}}
  %\Edge[label=$X_3$](1)(3)
  \SetVertexNormal[Shape = circle, LineColor=white, MinSize=28pt]
  \tikzset{EdgeStyle/.style = {->}}
  \EA[L=$\ldots$](3){4}
  \Edge[style= {->}, label = $1$](4)(3)
  %\SetUpVertex[LineColor=black]
  \SetVertexNormal[Shape = circle, LineColor=black, MinSize=28pt]
  \EA[L=$v_{m-1}$](4){m-1}
  \Edge[style= {->}, label = $1$](m-1)(4)
  \EA[L=$v_m$](m-1){m}
  \Edge[style= {->}, label = $1$](m)(m-1)
%  \tikzset{EdgeStyle/.style = {->,bend left=38}}
%  \Edge[label=$(-1)^m X_{m-1}$](1)(m-1)
  \tikzset{EdgeStyle/.style = {->,bend left=42}}
  \Edge[label=$X_m$](1)(m)
\end{tikzpicture}
\end{center}
Appelons ce cheminement $\mathcal{G}_m$ et $\mathcal{S}_m:=\mathcal{S}_{\mathcal{G}_m}$ le $\mathcal{G}_m$-monôme en type $\gl_n$. Par la proposition \ref{elemgr}, à une constante multiplicative non nulle près, le monôme $\mathcal{S}_m$ est le seul monôme de $F_m$ tel que $\mathcal{S}_m(q) \neq 0$. Par la propriété \ref{neq}, on a 
\begin{equation}
F_m(q)=(-1)^{m+1} \mathcal{S}_m(q) \propto X_m \label{egalite}
\end{equation}
Pour conclure pour le point (I'), il reste à montrer que $(-1)^{m+1} \mathcal{S}_m$ est en fait un monôme de $F_m^\bullet$, et donc pour cela, à montrer que $\mathcal{S}_m = e_{v_1, v_m} e_{v_m,v_{m-1}} \ldots e_{v_2,v_1}$ vérifie $\deg_{\gotn^-} \mathcal{S}_m=\deg_{\gotn^-} F_m$.\par
Pour $m=1$, on a $F_1= \id$ de degré $0$ en $\gotn^-$, et $e_{v_1,v_1}$ est également de degré $0$ en $\gotn^-$. On va alors montrer que la suite $\deg_{\gotn^-} \mathcal{S}_m$ vérifie la même relation de récurrence que $\deg_{\gotn^-} F_m$ (équation \eqref{degrec} du corollaire \ref{coro}). Pour tout $m \in \llbracket 1,n-1 \rrbracket$, on a
$$
\deg_{\gotn^-}\mathcal{S}_{m+1}=\deg_{\gotn^-}\mathcal{S}_{m}+ \deg_{\gotn^-}e_{v_1, v_{m+1}} + \deg_{\gotn^-} e_{v_{m+1},v_m}-\deg_{\gotn^-}e_{v_1, v_{m}}
$$
Si $v_m \prec v_{m+1} \prec v_1$ (voir notation \ref{ordre}), alors par la propriété \ref{pn}, on a $e_{v_1, v_{m+1}}, e_{v_{m+1},v_m}, e_{v_{1},v_m} \in \gotn^-$. Alors $\deg_{\gotn^-}\mathcal{S}_{m+1}=\deg_{\gotn^-}\mathcal{S}_{m}+1$. On a comme ceci six cas à étudier :
\begin{enumerate}[label=(\arabic*)]
\item $v_m \prec v_{m+1} \prec v_1$,
\item $v_{m+1} \preceq v_m \prec v_1$,
\item $v_{m+1} \prec v_1 \preceq v_m$,
\item $v_m \prec v_1 \preceq v_{m+1}$,
\item $v_1 \preceq v_m \prec v_{m+1}$,
\item $v_1 \preceq v_{m+1} \preceq v_m$.
\end{enumerate}
et on obtient le résultat suivant :
$$\deg_{\gotn^-}\mathcal{S}_{m+1}-\deg_{\gotn^-}\mathcal{S}_{m}=\left\{\begin{array}{ll}
1 &\text{dans les cas (1), (3) et (5)}  \\
0&\text{dans les cas (2), (4) et (6)} 
\end{array}\right.$$
Pour conclure, on veut montrer que l'on est dans l'un des cas (2), (4), ou (6) si et seulement si $m=m_i$, $1 \leq i < \imax$. Pour tous $x,y \in \llbracket 1,s \rrbracket$, considérons dans le graphe orienté $\mathcal{C}$ le trajet non trivial le plus court de $x$ vers $y$. On note $] x \circlearrowright y]$ l'ensemble des sommets de $\mathcal{C}$ parcourus via ce trajet, en excluant $x$ et en incluant $y$ (par exemple, si $x=y$, on parcourt tout le graphe $\mathcal{C}$).
\begin{prop}
Soit $m \in \llbracket 1,n \rrbracket$. Les propositions suivantes sont équivalentes :
\begin{enumerate}[label=(\alph*)]
\item l'entier $m$ vérifie l'un des trois cas (2), (4) ou (6),
\item on a $t_1 \in ]t_m \circlearrowright t_{m+1}]$,
\item on a $m=m_i$, $1 \leq i \leq \imax-1$.
\end{enumerate}
\end{prop}
\begin{proof}
En termes de $t_k$, les cas (1) à (6) se réécrivent 
\begin{enumerate}[label=(\arabic*)]
\item $t_m < t_{m+1} < t_1$,
\item $t_{m+1} \leq t_m < t_1$,
\item $t_{m+1} < t_1 \leq t_m$,
\item $t_m < t_1 \leq t_{m+1}$,
\item $t_1 \leq t_{m} < t_{m+1}$,
\item $t_1 \leq t_{m+1} \leq t_m$.
\end{enumerate}
En remarquant que $]t_m \circlearrowright t_{m+1}]$ s'écrit $\llbracket t_m+1, t_{m+1} \rrbracket$ si $t_m < t_{m+1}$ et $\llbracket t_m+1, s \rrbracket \cup \llbracket 1,t_{m+1} \rrbracket$ si $t_m \geq t_{m+1}$, on montre en disjoignant les cas que (a) équivaut à (b).\par

Montrons maintenant que (b) équivaut à (c). On reprend la définition combinatoire de la suite $(v_k)_k$.\par
Dans le graphe, on s'arrête d'abord une première fois à tous les sommets, puisque les poids $p_{0,k}=i_k$ sont nécessairement $ \geq 1$. Puisqu'il y a $s=m_1$ sommets, on se sera arrêté à tous les sommets à la fin de l'étape $m_1$. Plus précisément, à la fin de cette étape, on se trouve juste avant le sommet $t_1$, et on a $p_{m_1,k}=p_{0,k}-1=\max(0,p_{0,k}-1)$ pour tout $k$. Ainsi, c'est à l'étape $m_1+1$ que l'on va repasser pour la deuxième fois par $t_1$. Comme $p_{m_1,t_1}$ peut être nul, à l'étape $m_1+1$, on ne s'arrêtera pas nécessairement au sommet $t_1$, mais on a $t_1 \in ]t_{m_1} \circlearrowright t_{m_1+1}]$.\par
Reprenons juste après ($m_1$-c) et juste avant ($m_1+1$-a). À partir de ce moment, on s'arrête à tous les sommets $k$ tels que $p_{m_1,k}=p_{0,k}-1 \geq 1$, autrement dit, tel que $p_{0,k}=i_k \geq 2$. Puisqu'il y a $m_2-m_1$ tels sommets (propriété \ref{propelem} (3)), on se sera arrêté à tous ces sommets à la fin de l'étape $m_1+(m_2-m_1)=m_2$. À la fin de cette étape, on se trouve au sommet $t_{m_2}$, qui est le dernier sommet avant $t_1$ tel que $p_{0,t_{m_2}} \geq 2$, et on a $p_{m_2,k}=\max(0,p_{0,k}-2)$ pour tout $k$. Ainsi, c'est à l'étape $m_2+1$ que l'on va repasser pour la troisième fois par $t_1$. Comme $p_{m_2,t_1}$ peut être nul, on ne s'arrêtera pas nécessairement au sommet $t_1$, mais on a $t_1 \in ]t_{m_2} \circlearrowright t_{m_2+1}]$.\par
On continue par récurrence ce raisonnement jusqu'à $m_{\imax}=n$. Ainsi les seuls entiers $m \in \llbracket 1,n \rrbracket$ tels que $t_1 \in ]t_m \circlearrowright t_{m+1}]$ sont les entiers de la forme $m_i$ avec $1 \leq i \leq \imax-1$, ce qui donne (b) $\Leftrightarrow$ (c).
%\begin{enumerate}[label=(\alph*)]
%\item Les conditions (1) à (6) s'écrivent
%\begin{enumerate}[label=(\arabic*)]
%\item $t_m < t_{m+1} < t_1$,
%\item $t_{m+1} \leq t_m < t_1$,
%\item $t_{m+1} < t_1 \leq t_m$,
%\item $t_m < t_1 \leq t_{m+1}$,
%\item $t_m \leq t_{m+1} < t_1$,
%\item $t_1 \leq t_{m+1} \leq t_m$.
%\end{enumerate}
%et on vérifie que seuls les cas (2), (4) et (6) peuvent correspondre à $t_1 \in t_m \circlearrowright t_{m+1}$. Inversement, si $t_1 \in t_m \circlearrowright t_{m+1}$, selon les positionnements de $t_1$, $t_m$ et $t_{m+1}$ dans $\llbracket 1,s \rrbracket$, on se situe dans l'un des cas (2), (4) ou (6).
%\end{enumerate}
\end{proof}

On conclut donc pour le point (I').\par
Pour (III'), on rappelle que $\mathbf{f}^\times=\{F_{m,t} \, | \, m \in \mathbf{M}_2, 1 \leq t \leq r_m\}$ (lemme \ref{lemm}). Posons $m,t$ avec $m \in \mathbf{M}_2$ et $t \in \llbracket 1,r_m \rrbracket$. On a $F_m^\bullet \propto \prod_{u=1}^{r_m} F_{m,u}$. Notons $q'_{m,t}$ un $q$ de la forme de l'équation \eqref{defprec} tel que le paramètre $\xi$ (dont dépend la suite $(v_k)_k$) est choisi dans $\llbracket k_t+1, k_{t+1} \rrbracket$, c'est-à-dire $I_\xi \subset I^{(t)}$ (voir notation \ref{nota46}). On rappelle l'expression de $F_{m,u}$ (équation \eqref{expmochesemiinv}) :
$$F_{m,t}= \sum_{(J_k) \in \prod_{k \in K_t} \mathcal{J}_{k}} \; \Delta_{J^{(t)},J^{[t]}}^\bullet$$
On a donc (voir notation \ref{simp})
\begin{equation}
F_{m,u} \in \Sym\left(\g_{I^{(u)},I^{[u]}}\right) \label{somme}
\end{equation}
Puisque les $\g_{I^{(u)},I^{[u]}}$ sont en somme directe (les $I^{(u)}$ étant deux à deux disjoints), on obtient le lemme suivant :
\begin{lem}
\label{monome}
Soit $J \subset I$ et $\sigma \in \mathfrak{S}(J)$ tels que $\prod_{l \in J} e_{l,\sigma(l)}$ est (au signe près) un monôme de $F_m^{\bullet}$. On a $\prod_{l \in J} e_{l,\sigma(l)} = \prod_{u=1}^{r_m} \left(\prod_{l \in J^{(u)}}e_{l,\sigma(l)}\right)$. Alors pour tout $u$, le facteur $\prod_{l \in J^{(u)}} e_{l,\sigma(l)}$ est (au signe près) un monôme de $F_{m,u}$.
\end{lem}
Puisqu'il existe un unique monôme $\mathcal{S}$ de $F_m^\bullet$ tel que $\mathcal{S}(q'_{m,t}) \neq 0$, il existe un unique monôme $\mathcal{S}_u$ de $F_{m,u}$ tel que $\mathcal{S}_u(q'_{m,t}) \neq 0$ pour tout $u$. Ainsi pour tout $u \in \llbracket 1,r_m \rrbracket$, on a $F_{m,u}(q'_{m,t}) \propto \prod_{v \in J \cap I^{(t)}} e_{v,\sigma(v)}(q'_{m,t})$ où $J = \{v_1, \ldots, v_m\}$ et $\sigma=(v_m \, v_{m-1} \, \ldots \, v_1)$. Par construction de $q'_{m,t}$, on a alors $e_{v_1,v_m}(q'_{m,t})= X_m$, et pour tout $\ell > 1$, $e_{v_\ell,v_{\ell-1}}(q'_{m,t})= 1$. Comme par hypothèse, $v_1$ a été pris dans $I^{(t)}$ et que les $I^{(u)}$ sont disjoints, on a $F_{m,t}(q'_{m,t})\propto X_m$ et pour tout $u \neq t$, on a $F_{m,u}(q'_{m,t}) \propto 1$.\par
On rappelle que pour tout $\mu$, on a
\begin{equation}
F_\mu^\bullet(q'_{m,t}) \propto X_\mu \label{fmu}
\end{equation}
Notons alors $q_{m,t} \in \gotq^*_{\C[X]}$ la spécialisation de $q'_{m,t}$ en $X_\mu=1$ pour tout $\mu \neq m$ et $X_m=X$. On a alors $F_{m,t}(q_{m,t}) \propto X$, $F_{m,u}(q_{m,t}) \propto 1$ pour tout $u \neq t$, et pour tout $(\mu,\tau)$ avec $\mu \neq m$, le polynôme $F_{\mu,\tau}(q_{m,t})$ divise $F_{\mu}(q_{m,t}) \in \C^\times$, donc $F_{\mu,\tau}(q_{m,t}) \in \C^\times$. Ainsi $q_{m,t}$ vérifie l'hypothèse (III') de la proposition \ref{632}.\par
Les hypothèses (I), (II) et (III) sont donc vérifiées et on obtient
%\begin{rema}
%Renotons $\prod_{f \in \mathcal{F}} {f}^{s_f}=\prod_{m=1}^n \prod_{t=1}^{r'_m} \left(f_m^{(t)}\right)^{s_m^{(t)}}$. Soit $m \in \llbracket 1,n \rrbracket$ et $r':=r'_m$. On peut prendre les $s^{(t)}:=s_m^{(t)}$ de sorte qu'il existe $q$ tel que $s^{(q)}=0$. En effet, si $s=\min(s^{(q)})$, alors $\left(\prod_{t=1}^{r'} f_m^{(t)}\right)^s=\left(F_m^{\bullet}\right)^s$ est un invariant. Ainsi $\prod_{t=1}^{r'} \left(f_m^{(t)}\right)^{s^{(t)}}$ est de même poids que $\prod_{t=1}^{r'} \left(f_j^{(k)}\right)^{s^{(t)}-s}$ et on peut donc remplacer le premier par le second dans le produit.
%\end{rema}
\begin{theo}
Soit $\gotq$ une contraction parabolique standard de $\gl_n$, associée à un facteur de Levi $\gotl= \g_{I_1} \times \ldots \times \g_{I_s}$. Alors
\begin{itemize}
\item la troncation canonique $\gotq_\Lambda$ est égale à $\gotq' \oplus \bigoplus_{i} \C \id^{(i)}$, où $\id^{(i)}= \sum_{|I_k|=i} \id_{\g_{I_k}}$,
\item l'indice de $\gotq_\Lambda$ est $\ind \gotq_\Lambda=n+s-p$, où $p$ est le nombre de classes d'isomorphisme des blocs $\g_{I_k}$ dans $\gotl$,
\item l'algèbre des semi-invariants $\Sy(\gotq)=\Y(\gotq_\Lambda)$ est polynomiale et librement engendrée par les semi-invariants $F_{m,t}$, $1 \leq m \leq n$, $1 \leq t \leq r_m$ (voir théorème \ref{semiinv}). En particulier, l'ensemble des poids $\Lambda(\gotq)$ de $\Sy(\gotq)$ est un groupe.
\end{itemize}
\label{finalgl}
\end{theo}
\begin{rema}
On vient de montrer que l'algèbre $\Sy(\gotq)$ est librement engendrée par les $F_{m,t}$. On a déjà montré par le théorème \ref{irreductibilite} que tous les $F_{m,t}$ tels que $r_m \geq 2$ sont irréductibles. En fait, tous les $F_{m,t}$ sont irréductibles : si $x$ est un facteur irréductible non constant de $F_{m,t}$, alors $x$ est un semi-invariant, donc un polynôme en les $F_{\mu,\tau}$, qui divise $F_{m,t}$. Par algébrique indépendance des $F_{\mu,\tau}$, le semi-invariant $x$ est donc associé à $F_{m,t}$. Pour tout $m \in \llbracket 1,n \rrbracket$, l'équation \eqref{decomp}
$$F_m^\bullet=c_m \prod_{t=1}^{r_m} F_{m,t}$$
est donc la décomposition en produit d'éléments irréductibles de $F_m^\bullet$.
\end{rema}
\chapter{Étude de la polynomialité en type $A$ et $C$}
\label{chap4}
\section{Polynomialité en type $A$}
\label{sec7}
L'étude précédente sur $\g=\gl_n$ permet de conclure assez facilement pour le type $A$, c'est-à-dire le cas où l'algèbre de Lie semi-simple est ${\g}^A:=\spl_n$.\par
Dans cette section et uniquement dans cette section, on notera $\pr$ la projection $\pr^A$ de la définition \ref{prA}. On rappelle que si $\gotq=\gotp \ltimes \gotn^-$ est une contraction parabolique standard sur $\gl_n$, on a défini la contraction parabolique $\gotq^A=\gotp^A \ltimes \gotn^-$ en type $A$ avec $\gotp^A=\pr(\gotp)$. Une telle contraction parabolique est appelée \textbf{contraction parabolique standard de $\spl_n$}. Toute contraction parabolique en type $A$ est la conjuguée d'une telle contraction parabolique.
\begin{theo}
Soit $\gotq^A=\gotp^A \ltimes (\gotn^-)^\text{a}$ une contraction parabolique standard en type $A$. Soit $\mathbf{F}$ la famille des $F_{m,t}$ de la définition \ref{familleF}.\par
Alors $\Sy(\gotq^A)$ est polynomiale et $\pr(\mathbf{F}):=\{\pr(f) \, | \, f \in \mathbf{F} \setminus \{F_{1,1}\}\}$ engendre librement $\Sy(\gotq^A)$.
\end{theo}
\begin{rema}\mbox{}
\begin{itemize}
\item On a $F_{1,1}=\id$ donc $\pr(F_{1,1})=0$, c'est pourquoi on exclut $F_{1,1}$ de $\pr(\mathbf{F})$.
\item En type $A$, on a $\deg_{\gotn^-} \pr(F_m)=\deg_{\gotn^-} F_m$ pour tout $m$, c'est-à-dire $\pr(F_m)^\bullet=\pr(F_m^\bullet)$ (voir propriété \ref{ftypeA}).
\end{itemize}
\end{rema}
\begin{proof}
Les $\mathcal{F}_m^\bullet=\pr(F_{m+1}^\bullet)$ engendrent librement $\Y(\gotq^A)$ (théorème \ref{fipt}). En appliquant $\pr$ à $F_{m+1}^\bullet \propto \prod_{t=1}^{r_{m+1}} F_{m+1,t}$ (équation \eqref{decomp}), on obtient
$$\mathcal{F}_m^\bullet \propto \prod_{t=1}^{r_{m+1}} \pr(F_{m+1,t})$$
et le poids de $\pr(F_{m+1,t})$ est le poids de $F_{m+1,t}$ restreint à $\spl_n$. On vérifie alors aisément que les $\pr(F_{m+1,t})$ sont algébriquement indépendants par le théorème \ref{alin}.\par
Si $x \in \Sy(\gotq^A)$, alors $x \in \Sy(\gotq)$ puisque $\C \id \subset \gotz(\gotq)$, ainsi $x$ est un polynôme en les $F_{m,t}$ par le théorème \ref{finalgl}. On conclut alors en appliquant $\pr$.
\end{proof}
\begin{rema}
Puisque $\gotq = \gotq^A \times \C \id$, on peut reprendre l'étude pour $\gotq$ et on trouve facilement que :
\begin{itemize}
\item la troncation canonique $\gotq_\Lambda^A$ est égale à $\gotq' \oplus \bigoplus_{i < \imax} \C \pr(\id^{(i)})$ (on rappelle que les $\id^{(i)}$ pour $i \in \mathbf{I}$ forment une famille libre dans $\gotq_\Lambda$ et que $\sum_{i \in \mathbf{I}} \id^{(i)} = \id$),
\item l'indice de $\gotq_\Lambda^A$ est $\ind \gotq_\Lambda^A=\ind \gotq_\Lambda -1$,
\item l'ensemble des poids $\Lambda(\gotq^A)$ de $\Sy(\gotq^A)$ est $\Lambda(\gotq^A)=\pr(\Lambda(\gotq)):=\{\lambda_{|\g^A} \, | \, \lambda \in \Lambda(\gotq)\}$ et est un groupe.
\end{itemize}
\end{rema}
\section{Généralités en type $C$}
\label{defC}
L'algèbre de Lie $\gotq$ est toujours une contraction parabolique de $\gl_n$ par une sous-algèbre parabolique standard $\gotp$.
\begin{nota}
On définit l'involution $\gamma \in \mathfrak{S}_n$ par $\gamma(k)=n+1-k$. Si $L \subset \llbracket 1,n \rrbracket$, on notera également $L^\gamma=\gamma(L):=\{\gamma(l) \, | \, l \in L\}$\index{$L^\gamma, M^\gamma, V^\gamma, f^\gamma$}. Si $M \in \gl_n$, on note $M^\gamma$ la matrice symétrique de $M$ par rapport à l'antidiagonale (c'est-à-dire $ M^\gamma_{i,j}=M_{j^\gamma, i^\gamma}$), et si $V \subset \gl_n$, on note $V^\gamma=\{M^\gamma \, | \, M \in V\}$. L'application $\C$-linéaire $M \in \gl_n \mapsto  M^\gamma \in \gl_n$ s'étend naturellement en un morphisme d'algèbres de $\Sym(\gl_n)$ (et donc de $\Sym(\gotq)$) dans lui-même. On note $f^\gamma$ l'image de $f \in \Sym(\gl_n)$ par ce morphisme.
\end{nota}
\begin{nota}
On reprend globalement les conventions de \cite[chap. VIII]{bou06}. On suppose ici que $n$ est pair et on pose $n=2n'$. On définit $\syp_{n}$\index{$\syp_{n}$} comme l'ensemble des matrices de la forme
$$ \begin{pmatrix}
M_1&M_2 \\ M_3&-  M_1^\gamma
\end{pmatrix}$$
avec $M_1,M_2,M_3$ des matrices $n' \times n'$ telles que $M_2^\gamma= M_2$ et $M_3^\gamma= M_3$. L'espace vectoriel $\syp_n$ est de dimension $n(n+1)/2$. On appellera abusivement \textbf{la} base canonique de $\syp_{n}$ une base de $\syp_{n}$ constituée des $e_{u,v}-e_{v^\gamma,u^\gamma}$ pour $1 \leq u,v \leq n'$, des $e_{u,v}+e_{v^\gamma,u^\gamma}$ pour $(u,v) \notin \llbracket 1,n' \rrbracket^2$ tels que $u+v \leq n$, et des $e_{u,u^\gamma}$ pour $1 \leq u \leq n$. On complète la représentation graphique de $\gl_{n}$ :
\begin{center}
\begin{tikzpicture}[scale=0.4]
\draw (0,0)--(9,0)--(9,9)--(0,9)--(0,0);
\draw[double, ultra thick] (0,4.5)--(4.5,4.5)--(4.5,0);
\draw[double, ultra thick] (4.5,9)--(4.5,4.5)--(9,4.5);
\draw (2.25,2.25) node{$\gl_{n}^+$};
\draw (2.25,6.75) node{$\gl_{n}^-$};
\draw (6.75,6.75) node{$\gl_{n}^+$};
\draw (6.75,2.25) node{$\gl_{n}^-$};
\end{tikzpicture}
\end{center}
où $\gl_{n}^+$ (respectivement $\gl_{n}^-$)\index{$\gl_{n}^+$, $\gl_{n}^-$} est le sous-espace vectoriel de $\gl_{n}$ constitué de l'ensemble des $x \in \gl_{n}$ tels que $x +  x^\gamma \in \syp_{n}$ (respectivement $x - x^\gamma \in \syp_{n}$).
\end{nota}
\begin{defi}
On reprend les notations combinatoires d'une contraction parabolique de $\gl_n$ (définition \ref{canonique}). Soit $\gotq=\gotp \ltimes \gotn^-$ une contraction parabolique standard de $\gl_n$. On dit que $\gotq$ est contraction parabolique standard \textbf{symétrique} (ou CPSS) si le sous-ensemble $\pi'\subset \pi$ associé à $\gotq$ est symétrique, dans le sens où pour tout $i \in \llbracket 1, n-1 \rrbracket$, on a $\alpha_i \in \pi' \Leftrightarrow \alpha_{n-i} \in \pi'$. De manière équivalente, $\gotq$ est symétrique si et seulement si $\gotp^\gamma=\gotp$ et $(\gotn^-)^\gamma = \gotn^-$.\label{cpcs}
\end{defi}
\begin{defi}
Soit $\gotq=\gotp \ltimes \gotn^-$ une contraction parabolique standard symétrique de $\gl_n$. Soient $\gotp^C=\gotp \cap \syp_n$ et $(\gotn^-)^C=\gotn^- \cap \syp_n$. On a $\syp_n=\gotp^C \oplus (\gotn^-)^C$ (comme $\gotq$ est symétrique, tout vecteur de la base canonique de $\syp_n$ est soit dans $\gotp$, soit dans $\gotn^-$). On pose alors $\gotq^C = \gotp^C \ltimes (\gotn^-)^C$. L'algèbre de Lie $\gotq^C$\index{$\gotq^C$, $\gotp^C$, $(\gotn^-)^C$} est une contraction parabolique de $\syp_n$, que l'on appelle \textbf{contraction parabolique standard de $\syp_n$}. Toute contraction parabolique en type $C$ est la conjuguée d'une telle contraction. Son crochet de Lie $[\; , \; ]_{\gotq^C}$ vérifie $[x,y]_{\gotq^C}=[x,y]_{\gotq}$ pour tous $x,y \in \gotq^C$. On dit également que $\gotq$ est une \textbf{contraction parabolique standard de $\gl_{n}$ au-dessus de $\gotq^C$}.\label{dessus}
\end{defi}
\begin{defi}
Soit $\pr^C=\pr_{\syp_{n}, \syp_{n}^{\perp} \oplus \C \id}$\index{$\pr^C$} la projection de $\gl_{n}$ sur $\syp_{n}$ parallèlement à $\syp_{n}^{\perp} \oplus \C \id$, où $\syp_{n}^{\perp}$ est l'orthogonal de $\syp_{n} \subset \spl_{n}$ pour la forme de Killing de $\spl_{n}$. Jusqu'à la fin de ce chapitre, on note $\pr$ pour $\pr^C$. La projection $\pr$ induit un morphisme d'algèbres $\Sym(\gl_{n}) \rightarrow \Sym(\syp_{n})$ que par abus, on note toujours $\pr$.
\end{defi}
\begin{rema}
Les sous-espaces vectoriels $\gl_n^{\pm}$ sont les ensembles des $x \in \gl_n$ tels que $\pr(x^\gamma)=\pm \pr(x)$.
\end{rema}
Précisons comment la combinatoire de la contraction parabolique standard symétrique $\gotq$ de $\gl_n$ (définition \ref{canonique}) se projette en type $C$. Soit $\h^C:=\pr^C(\h)$ l'ensemble des matrices diagonales de $\syp_n$, qui forment une sous-algèbre de Cartan de $\syp_{n}$. On note $R^C:=\{\beta_{|\syp_n} \, | \, \beta \in R\}$\index{$\varpi_i^C$, $\alpha^C_i$, $\pi^C$, $(\pi^C)'$} le système de racines associé. Pour tout $i \in \llbracket 1,n' \rrbracket$, on note $\varpi_i^C={\varpi_i}_{|\syp_n}$ les poids fondamentaux. Pour tout $i \in \llbracket n'+1,n-1 \rrbracket$, on a ${\varpi_i}_{|\syp_n}={\varpi_{n-i}}_{|\syp_n}=\varpi_{n-i}^C$. Pour $i \in \llbracket 1, n' \rrbracket$, on note $\epsilon_i^C={\epsilon_i}_{|\syp_n}$ et $\alpha^C_i={\alpha_i}_{|\syp_n}$. Pour $i \in \llbracket 1, n'-1 \rrbracket$, on a $\alpha^C_i=\epsilon_i^C-\epsilon_{i+1}^C$, et $\alpha^C_{n'}=2\epsilon_{n'}^C$. On considère la base $\pi^C=\{\beta_{|\syp_n} \, | \, \beta \in \pi\}=\left\{\alpha^C_1, \ldots, \alpha^C_{n'}\right\}$ de $R^C$. Soit $(\pi^C)' =\{\beta_{|\syp_n} \, | \, \beta \in \pi'\}$. Alors $\gotp^C$ est la sous-algèbre parabolique associée à $\h^C$, $\pi^C$ et $(\pi^C)'$.\par
Dans toute la suite, on pose $\gotq$ une CPSS au-dessus d'une contraction parabolique standard $\gotq^C$ de $\syp_n$.
\begin{prop}
La projection $\pr^C$ vue de $\gotq$ dans $\gotq^C$ (qui sont égales respectivement à $\gl_n$ et $\syp_n$ comme espaces vectoriels) est un morphisme de $\gotq^C$-modules.\label{mormor}
\end{prop}
\begin{proof}
On note $\gotq^\perp=\syp_n^\perp \oplus \C \id$. Soient $q_1 \in \gotq^C$ et $q_2=\pr(q_2) + q_2^\perp \in \gotq$ avec $q_2^\perp \in \gotq^\perp$. Par définition de $\gotq^\perp$, on a $[q_1,q_2^\perp]_{\g} \in \gotq^\perp$. En utilisant la définition \ref{conpara} du crochet de Lie sur $\gotq$, on obtient $[q_1,q_2^\perp]_{\gotq} \in \gotq^\perp$ (on pourra définir $\gotp^\perp:=\gotp \cap \gotq^\perp$ et $(\gotn^-)^\perp:=\gotn^- \cap \gotq^\perp$, de sorte que $\gotp=\gotp^C \oplus \gotp^\perp$, $\gotn^-=(\gotn^-)^C \oplus (\gotn^-)^\perp$ et $\gotq^\perp = \gotp^\perp \oplus (\gotn^-)^\perp$). Ainsi
\[\pr([q_1,q_2]_{\gotq})=\pr([q_1,\pr(q_2)]_{\gotq} + [q_1,q_2^\perp]_{\gotq})=\pr([q_1,\pr(q_2)]_{\gotq})=[q_1,\pr(q_2)]_{\gotq^C} \qedhere\]
\end{proof}
En type $C$, les monômes d'un polynôme de $\Sym(\gotq^C)$ et les degrés en $(\gotn^-)^C$ et $\gotp^C$ sont définis et sous-entendus par rapport à la base canonique décrite plus haut. Aussi, le bidegré d'un élément $F \in \Sym(\syp_n)$ bihomogène en $(\gotn^-)^C$ est le couple $(\deg_{\gotp^C} F, \deg_{(\gotn^-)^C} F)$.
\begin{prop}
La projection $\pr$ "préserve les bidegrés" dans le sens où si $s$ est un élément bihomogène en $\gotn^-$ de $\Sym(\gotq)$ (isomorphe à $\Sym(\gl_n)$ comme algèbre), alors soit $\bideg_{\gotp^C,(\gotn^-)^C} \pr(s) = \bideg_{\gotp, \gotn^-} s$, soit $\pr(s)=0$. En particulier, pour tout $f \in \Sym(\gotq)$, on a $\deg_{(\gotn^-)^C} \pr(f) \leq \deg_{\gotn^-} f$.
\label{preserv}
\end{prop}
\begin{proof}
Cette propriété vient du fait que $\pr(\gotp) =\gotp^C$ et $\pr(\gotn^-)=(\gotn^-)^C$.
\end{proof}
%On évoquera souvent des "degrés en $\gotn^-$ pour parler abusivement de degrés en $(\gotn^-)^C$, et on pourra également noter $\deg_{\gotn^-}$ au lieu de $\deg_{(\gotn^-)^C}$.
%\begin{prop}
%Soit $\gotq^C$ une contraction parabolique canonique de $\syp_n$. Il existe une unique contraction parabolique canonique symétrique $\gotq$ de $\gl_n$ telle que $\gotq^C=\gotq \cap \syp_n$.
%\end{prop}
On utilise alors cette projection et nos connaissances en type $\gl_n$ pour étudier le type $C$.
\begin{prop}
Soit $F$ un semi-invariant de $\gotq$ de poids $\lambda$. Alors $\pr(F)$ est un semi-invariant de $\gotq^C$, de poids $\lambda_{|\gotq^C}$. \label{yes}
\end{prop}
\begin{proof}
On a une décomposition de $\Sym(\gl_n)$ comme somme \emph{de $\gotq^C$-modules} : $\Sym(\gotq)= \Sym(\gotq^C) \oplus \left( \gotq^\perp \otimes \Sym(\gotq) \right)$. La projection $\pr : \Sym(\gotq) \rightarrow \Sym(\gotq^C)$ est alors une projection parallèlement à $\gotq^\perp \otimes \Sym(\gotq)$.
Pour tout $q \in \gotq^C$, on a donc $q \cdot \pr(F) = \pr(q \cdot F)= \pr (\lambda(q) F)=\lambda(q) \pr(F)$ par la propriété \ref{mormor}.
\end{proof}
Comme en type $A$, on reprend l'étude de l'algèbre des invariants de Panyushev et Yakimova \cite{py13} :
\begin{theo}
Pour toute contraction parabolique $\gotq^C$ de $\syp_n$, l'algèbre des invariants $\Y(\gotq^C)$ est polynomiale et librement engendrée par les $f_{m'}:=\pr(F_{2m'})^\bullet$, $1 \leq m' \leq n'$. \label{pyC}
\end{theo}
On va donc étudier les projetés des semi-invariants $F_{m,t}$ du théorème \ref{semiinv} par la projection $\pr$. Dans toute la suite, on notera abusivement $e_{x,y}^*$\index{$e_{p,q}$, $e^*_{p,q}$} pour la forme linéaire restreinte à $\syp_n$.
\begin{lem}
Pour tous $u,u',v,v' \in \llbracket 1,n \rrbracket$, on a $\pr(e_{u,v})(e_{u',v'}^*) \neq 0$ si et seulement si $(u',v') \in \{(u,v),(v^\gamma,u^\gamma)\}$.\par
De même, on a $\C \pr(e_{u,v}) = \C \pr(e_{u',v'})$ si et seulement si $(u',v') \in \{(u,v),(v^\gamma,u^\gamma)\}$.
\label{lemC}
\end{lem}
\begin{proof}
Le point crucial est de voir que pour tous $u,v$, on a $\pr(e_{u,v})=\frac{1}{2}(e_{u,v}+\varepsilon \, e_{v^\gamma,u^\gamma})$, où $\varepsilon= \pm 1$ si $e_{u,v} \in \gl^{\pm}_{n}$.
\end{proof}
On rappelle que les graphes de forme linéaire sont définis à la sous-section \ref{graphe}.
\begin{coro}
Pour tous $x, y \in \llbracket 1,n \rrbracket$, pour $q = e_{x,y}^*$, le graphe $\mathcal{G}(q)$ est :
\begin{itemize}
\item si $y = x^\gamma$, composé d'une seule arête $x \rightarrow y$,
\item si $y \neq x^\gamma$, composé de deux arêtes $x \rightarrow y $ et $ y^\gamma \rightarrow x^\gamma$.
\label{graphC}
\end{itemize}
\end{coro}
Soit $m \in \llbracket 1, n \rrbracket$ pair et $\mathcal{E}_{m}$ l'ensemble des monômes $\mathcal{S}$ de $F_{m}$, c'est-à-dire l'ensemble des $\varepsilon(\sigma)\prod_{l \in J} e_{l,\sigma(l)}$ pour $J \subset I$ de cardinal $m$ et $\sigma \in \mathfrak{S}(J)$. On note aussi $\mathcal{E}'_{m}$ l'ensemble des $\pr(\mathcal{S})$ avec $\mathcal{S}$ un monôme de $F_{m}$. On considère l'application $\pr : \mathcal{E}_{m} \longrightarrow \mathcal{E}'_{m}$.
\begin{rema}
L'application $\pr : \mathcal{E}_{m} \longrightarrow \mathcal{E}'_{m}$ est surjective (par définition). Cependant, si $\mathcal{S} \in \mathcal{E}_{m}$ est un monôme de $F_m$, alors $\pr(\mathcal{S}) \in \mathcal{E}'_{m}$ n'est pas nécessairement un monôme de $\pr(F_m)$. Cela tient à la non injectivité de $\pr$ sur $\mathcal{E}_{m}$, autrement dit au fait que l'ensemble des monômes $\mathcal{S}_1, \ldots, \mathcal{S}_d \in \mathcal{E}_{m}$ tels que $\pr(\mathcal{S}_k) \propto \pr(\mathcal{S})$ peut avoir plus d'un élément, et donc qu'il est possible que $\sum_k \pr(\mathcal{S}_k)=0$, ce qui implique que $\pr(\mathcal{S})$ ne soit pas un monôme de $\pr(F_m)$.
\end{rema}
\begin{propn}
Si $\mathcal{S} \in \mathcal{E}_{m}$ est symétrique, c'est-à-dire $\mathcal{S}^\gamma=\mathcal{S}$ alors pour tout $\mathcal{T} \in \mathcal{E}_m$, on a $\pr(\mathcal{T}) \propto \pr(\mathcal{S}) \Longrightarrow \mathcal{T}=\mathcal{S}$. En particulier, si $\mathcal{S}$ est un monôme symétrique de $F_{m}$, alors $\pr(\mathcal{S})$ est un monôme de $\pr(F_{m})$.
\label{tivafor}
\end{propn}
\begin{proof}
Remarquons déjà que $\mathcal{S}$ est symétrique si et seulement si $\mathcal{S}$ est tel que $\mathcal{S} \propto \prod_{l \in J} e_{l,\sigma(l)}$ avec $J \subset \llbracket 1,n \rrbracket$ symétrique (c'est-à-dire tel que $J^\gamma=J$) et $\sigma \in \mathfrak{S}(J)$ tel que $\sigma \circ \gamma \circ \sigma = \gamma$. En effet, par définition de la symétrie de $\mathcal{S}$, pour tout $x,y \in I$, $e_{x,y}$ divise $\mathcal{S}$ si et seulement si $e_{y^\gamma,x^\gamma} $ divise $\mathcal{S}$. Ainsi, si $x,y \in J$, alors $x^\gamma, y^\gamma \in J$ et si $y=\sigma(x)$, alors $x^\gamma=\sigma(y^\gamma)$.\par
Supposons alors que $\pr(\mathcal{T}) \propto \pr(\mathcal{S})$. On écrit $\mathcal{S} \propto \prod_{l \in J} e_{l,\sigma(l)}$ avec $J \subset \llbracket 1,n \rrbracket$ symétrique de cardinal $m$ et $\sigma \in \mathfrak{S}(J)$ tel que $\sigma \circ \gamma \circ \sigma = \gamma$, et $\mathcal{T}\propto\prod_{l' \in J'} \pr(e_{l',\sigma(l')})$ avec $J' \subset \llbracket 1,n \rrbracket$ de cardinal $m$ et $\sigma \in \mathfrak{S}(J')$. On a alors
$$\prod_{l \in J} \pr(e_{l,\sigma(l)}) \propto \prod_{l' \in J'} \pr(e_{l',\sigma'(l')}).$$
Soit $l' \in J'$. Alors il existe $l \in J$ tel que $\pr(e_{l,\sigma(l)}) \propto \pr(e_{l',\sigma'(l')}) $, et donc par le lemme \ref{lemC}, soit $(l',\sigma'(l'))=(l,\sigma(l))$, soit $(l',\sigma'(l'))=(\sigma(l)^\gamma, l^\gamma) $.\par
D'abord, on a soit $l'=l$, soit $l'=\sigma(l)^\gamma$, et donc comme $J$ est symétrique, on a en particulier $l' \in J$. Ainsi $J' \subset J$ et donc $J'=J$ par égalité des cardinaux.\par
Ensuite, on a soit $\sigma'(l')=\sigma(l)=\sigma(l')$, soit $\sigma'(l')=l^\gamma=\sigma^{-1}(\sigma(l))^\gamma=\sigma^{-1}(l'^\gamma)^\gamma=\sigma(l')$ (car $\sigma \circ \gamma \circ \sigma = \gamma$). Ainsi $\sigma'(l')=\sigma(l')$, donc $\sigma'=\sigma$ ce qui conclut.
\end{proof}
\begin{lem}
Pour tout $m \in \llbracket 1,n \rrbracket$ pair, il existe un monôme symétrique $\mathcal{S}$ de $F_{m}^\bullet$.
\end{lem}
\begin{proof}
Il existe $J \in \mathcal{J}(m)$ (définition \ref{typej}) qui est symétrique : en effet, pour tous $k,k'$ tels que $k'=s+1-k$, par symétrie de $\gotq$, on a $I_{k'}=I_k^\gamma$ et on peut choisir $J_k$ et $J_{k'}$ de sorte que $J_{k'}=J_k^\gamma$. Par la proposition \ref{typej}, si $\deg_{\gotn^-} F_{m}=m-i$, alors $\max_k |J_k|=i$. Si $J=\{a_1 < \ldots < a_{m}\}$, comme pour tout $l \in \llbracket 1, m-i \rrbracket$ on a $a_{l+i} \geq a_l$, par la propriété \ref{propordre}, on obtient $a_{l+i} \succeq a_l$ (voir définition \ref{ordre}) ; de plus, on ne peut pas avoir $a_{l+i} \sim a_l$ car sinon $a_l \sim a_{l+1} \sim \ldots \sim a_{l+i}$ ce qui contredit le fait que $\max_k |J_k|=i$, donc $a_{l+i} \succ a_l$. Finalement le terme $e_{a_1, a_{m-i+1}} \ldots e_{a_i, a_{m}} e_{a_{i+1}, a_1} \ldots e_{a_{m}, a_{m-i}}$ est un terme symétrique de degré $m-i$ en $\gotn^-$ (car $a_{m+1-l}=a_l^\gamma$ pour tout $l$).
%Il suffit de montrer qu'il existe $J \subset \llbracket 1,n \rrbracket$ symétrique de cardinal $2m'$ et $\sigma \in \mathfrak{S}(J)$ vérifiant $\sigma=\gamma \circ \sigma \circ \gamma$ tels que $\prod_{l \in J} \pr(e_{l,\sigma(l)})$ est de degré maximal en $\gotn^-$.
%Si $J=\{a_1 < \ldots < a_{2m'}\}$, on définit alors $\sigma \in \mathfrak{S}(J)$ comme étant le cycle $(a_{2m'} \, \ldots \, a_1)$. Puisque $a_{2m'+1-l}=a_l^\gamma$ pour tout $l$, on a bien $\sigma=\gamma \circ \sigma \circ \gamma$.
\end{proof}
D'après la proposition \ref{tivafor}, si $\mathcal{S}$ est un monôme symétrique de $F_m^\bullet$, alors $\pr(\mathcal{S})$ est un monôme de $\pr(F_{m})$. Par la propriété \ref{preserv}, on obtient ainsi $\deg_{(\gotn^-)^C} \pr(\mathcal{S}) \leq \deg_{(\gotn^-)^C} \pr(F_{m}) \leq \deg_{\gotn^-} F_{m} = \deg_{\gotn^-} \mathcal{S}$. Or $\pr(\mathcal{S})$ est un monôme non nul, donc bihomogène en $(\gotn^-)^C$, donc toujours par la propriété \ref{preserv}, on a $\deg_{(\gotn^-)^C} \pr(\mathcal{S})=\deg_{\gotn^-} \mathcal{S}$, d'où $\deg_{(\gotn^-)^C} \pr(F_{m}) = \deg_{\gotn^-} F_{m}$. Ainsi on a :
$$\pr(F_{m})^\bullet=\pr(F_{m}^\bullet)$$
Pour tout $m' \in \llbracket 1,n' \rrbracket$, on note désormais $f_{m'}:=\pr(F_{2m'})^\bullet=\pr(F_{2m'}^\bullet)$.
\begin{rema}
Pour $m$ impair, on a $\pr(F_m^\bullet)=\pr(F_m)^\bullet=0$ car $\pr(F_m)=0$ (voir \cite{bou06}).
\end{rema}

\section{Cas où le facteur de Levi est de type $A$}
\label{sec9}
On rappelle que $n'=n/2$. Soit $\gotq^C$ une contraction parabolique standard de $\syp_n$ telle que $\alpha_{n'}^C \notin (\pi^C)'$. Soit $\gotq$ la CPSS au-dessus de $\gotq^C$. Dans ce cas, on a $e_{{n'}^\gamma,n'}=e_{n'+1,n} \in (\gotn^-)^C$, de sorte que l'algèbre de Lie $\gotq$ admet un facteur de Levi $\gotl$ de la forme
$$\gotl=\g_{I_1} \times \ldots \times \g_{I_s}$$
où l'on a $I_{s+1-k}=I_k^\gamma$ pour tout $k \in \llbracket 1,s \rrbracket$, et où \textbf{$s$ est pair}. On note $s=2s'$.
\begin{rema}
Réciproquement, si $\gotq$ admet un facteur de Levi d'une telle forme, la contraction $\gotq^C$ en type $C$ déduite de $\gotq$ vérifie $\alpha_{n'}^C \notin (\pi^C)'$.
\end{rema}
On a alors $\gotl^C:=\pr(\gotl) \simeq \gl_{i_1} \times \ldots \times \gl_{i_{s'}}$ qui est un facteur de Levi de $\gotq^C$.
\begin{notaprop}
On reprend pour $\gotq$ toutes les notations de \ref{nota}.
\begin{itemize}
\item Pour tout $i \in \llbracket 0, \imax \rrbracket$, les entiers $\rho_i$ et $m_i$ sont pairs. En particulier, $\mathbf{M}_1=\mathbf{M}_2=:\mathbf{M}$\index{$\mathbf{M}$, $\mathbf{M}'$} et tous les éléments de $\mathbf{M}$ sont pairs. Pour tout $i$, on définit alors $\rho'_i:=\rho_i/2$ et $m'_i:=m_i/2$\index{$\rho'_i$, $m'_i$, $r'_{m'}$}. On note $\mathbf{M}':=\{m/2 \, | \, m \in \mathbf{M}\}$. En particulier, pour tout $m=m_i \in \mathbf{M}$, l'entier $r_m=\rho_i$ est pair. Pour tout $m' \in \mathbf{M}'$, on note alors $r'_{m'}:=r_{2m'}/2$.% Si $m' \in \llbracket 1,n' \rrbracket \setminus \mathbf{M}'$, alors $r_{2m'}=1$ et on note $r'_{m'}:=0$.
\item Pour tout $i \in \mathbf{I}$, l'ensemble $\kappa_i=\left\{k_{i,1} < \ldots < k_{i,\rho_i}\right\}$ est symétrique dans le sens où $k_{i,\rho_i+1-t}=s+1-k_{i,t}$ pour tout $t$. Comme précédemment, on omet souvent les indices $i$.
\end{itemize}
\end{notaprop}
\begin{exem}
\label{exempleC}
Pour illustrer les raisonnements de cette section, on se placera dans le cas d'une CPSS où $n=12, s=8, I_1=\{1\}, I_2=\{2,3\}, I_3=\{4,5\}, I_4=\{6\}$ et $I_k=I_{9-k}^\gamma$ pour $5 \leq k \leq 8$, c'est-à-dire schématiquement :
\begin{center}
\begin{tikzpicture}[scale=0.5]
\foreach \k in {1,2,...,11}
	{\draw[color=gray!20]  (0,\k)--(12,\k);
	\draw[color=gray!20] (\k,0)--(\k,12);}
\draw (0,0)--(12,0);
\draw (0,0)--(0,12);
\draw (0,12)--(12,12);
\draw (12,0)--(12,12);
\draw[ultra thick] (0,11)--(1,11)--(1,9)--(3,9)--(3,7)--(5,7)--(5,6)--(6,6)--(6,5)--(7,5)--(7,3)--(9,3)--(9,1)--(11,1)--(11,0);
\draw (3,3) node{\LARGE $\gotn^-$};
\draw (9,9) node{\LARGE $\gotp$};
\draw[dashed] (1,12)--(1,11)--(3,11)--(3,9)--(5,9)--(5,7)--(6,7)--(6,6)--(7,6)--(7,5)--(9,5)--(9,3)--(11,3)--(11,1)--(12,1);
\draw (0.5,11.5) node{\small $\g_{I_1}$};
\draw (2,10) node{\small $\g_{I_2}$};
\draw (4,8) node{\small $\g_{I_3}$};
\draw (5.5,6.5) node{\small $\g_{I_4}$};
\draw (6.5,5.5) node{\small $\g_{I_5}$};
\draw (8,4) node{\small $\g_{I_6}$};
\draw (10,2) node{\small $\g_{I_7}$};
\draw (11.5,0.5) node{\small $\g_{I_8}$};
\end{tikzpicture}
\end{center}
\end{exem}
\subsection{Semi-invariants}
On pose maintenant $i \in \mathbf{I}$, $m=m_i$, $r=\rho_i$, $m'=m/2$ et $r'=\rho'_i=r/2$. D'après l'équation \eqref{expmochesemiinv}, on a $F_m^\bullet \propto F_{m,1} \ldots F_{m,r}$ avec
$$F_{m,t}= \sum_{\stackrel{J_k \in \mathcal{J}_{k} }{ \forall k \in K_t}} \; \Delta_{J^{(t)},J^{[t]}}^\bullet$$
où $K_t=\left\{ \begin{array}{ll}
\llbracket k_t+1, k_{t+1}-1 \rrbracket&\text{si } t<r\\\llbracket 1, k_1-1 \rrbracket \sqcup \llbracket k_r+1, s \rrbracket& \text{si } t=r
\end{array}\right.$. De $I_{s+1-k}=I_k^\gamma$ pour tout $k \in \llbracket 1,s \rrbracket$ et $k_{r+1-t}=s+1-k_{t}$ pour tout $t$, on tire $\left(I^{(t)}\right)^\gamma=I^{[r-t]}$ pour tout $t<r$ et $\left(I^{(r)}\right)^\gamma=I^{[r]}$ par définition des $J^{(t)}, J^{[t]}$ (notation \ref{nota46}).
\begin{propn}
Pour tout $t \in \llbracket 1, r' -1 \rrbracket$, on a $\pr(F_{m,r-t})=(-1)^{\deg F_{m,t}} \pr(F_{m,t})$.
\end{propn}
\begin{proof}
Comme $\alpha_{n'}^C \notin (\pi^C)'$, on a $F_{m,t} \in \gl_{n}^-$ pour tout $t \in \llbracket 1, r'-1\rrbracket$. Pour tout $k \in K_t=\llbracket k_t+1, k_{t+1}-1 \rrbracket$, pour tout choix de $J_k \in \mathcal{J}_k$, son symétrique $J_k^\gamma$ est un choix possible de $J'_l \in \mathcal{J}_l$ pour  $l=s+1-k \in K_{r-t}=\llbracket k_{r-t}+1, k_{r-t+1}-1 \rrbracket=s+1-K_t$, et la correspondance est bijective.
\begin{center}
\begin{tikzpicture}[scale=0.6]
\draw (0,0)--(18,0)--(18,18)--(0,18)--(0,0);
\draw[double, ultra thick] (0,9)--(9,9)--(9,0);
\draw[double, ultra thick] (9,18)--(9,9)--(18,9);
\draw(4.5,4.5) node{$\displaystyle \gotn^-$};
\draw(13.5,13.5) node{$\displaystyle \gotp$};
\draw[dotted] (0,18)--(1,17);
\draw[dotted] (3,15)--(4,14);
\draw[dotted] (6,12)--(6.5,11.5);
\draw[dotted] (11.5,6.5)--(12,6);
\draw[dotted] (14,4)--(15,3);
\draw[dotted] (17,1)--(18,0);
\draw (1,17)--(1,15)--(3,15);
\draw (4,14)--(4,12)--(6,12);
\draw (6.5,11.5)--(6.5,9)--(9,9)--(9,6.5)--(11.5,6.5);
\draw (12,6)--(12,4)--(14,4);
\draw (15,3)--(15,1)--(17,1);
\draw[dashed] (1,17)--(3,17)--(3,15);
\draw[dashed] (4,14)--(6,14)--(6,12);
\draw[dashed] (6.5,11.5)--(9,11.5)--(9,9)--(11.5,9)--(11.5,6.5);
\draw[dashed] (12,6)--(14,6)--(14,4);
\draw[dashed] (15,3)--(17,3)--(17,1);
\draw (2,16) node{$\g_{J_{k_t}}$};
\draw (5,13) node{$\g_{J_{k_{t+1}}}$};
\draw (7.75,10.25) node{$\g_{J_{s'}}$};
\draw (10.25,7.75) node{$\g_{J_{s'+1}}$};
\draw (13,5) node{$\g_{J_{k_{r-t}}}$};
\draw (16,2) node{\tiny $ \g_{J_{k_{r+1-t}}}$};
\draw[red] (1,15)--(1,12)--(4,12)--(4,15)--(1,15);
\draw[red] (12,4)--(12,1)--(15,1)--(15,4)--(12,4);
\draw[red] (15,4)--(16,6);
\draw (16,6.5) node[red]{$\g_{J^{(r-t)}, J^{[r-t]}}$};
\draw[red] (4,15)--(6,16);
\draw (6,16.5) node[red]{$\g_{J^{(t)}, J^{[t]}}$};
\end{tikzpicture}
\end{center}
Ainsi, comme l'illustre le diagramme ci-dessus, on a $\Delta_{J^{(r-t)},J^{[r-t]}}^\bullet =  (\Delta_{J^{(t)},J^{[t]}}^\bullet)^\gamma$ sur $\gl_n$, de sorte que
\[\pr(F_{m,r-t})=(-1)^{\deg F_{m,t}} \pr(F_{m,t})\qedhere\]
\end{proof}
\begin{exem}
Dans l'exemple \ref{exempleC}, on a $F_8^\bullet \propto F_{8,1} F_{8,2}F_{8,3} F_{8,4}$ et $F_{12}^\bullet \propto F_{12,1} F_{12,2} F_{12,3} F_{12,4}$. Pour tout $m$ pair différent de $8$ et $12$, on a $F_m^\bullet = F_{m,1}$. La projection donne alors $\pr(F_{8,3})=-\pr(F_{8,1})$ et $\pr(F_{12,3})=\pr(\Delta^\bullet_{\{10,11\}, \{8,9\}})=\pr(\Delta^\bullet_{\{4,5\}, \{2,3\}})=\pr(F_{12,1})$.
\end{exem}
\begin{nota}\mbox{}
\begin{itemize}
\item Soient $i \in \mathbf{I}$, $m=m_i \in \mathbf{M}$, $r=\rho_i$, $m'=m/2 \in \mathbf{M}'$ et $r'=r'_{m'}=\rho'_i=r/2$. On note $f_{m',t}=\pr(F_{2m',t})$\index{$f_{m',t}$} pour $t \in \llbracket 1, r' \rrbracket$ et $f_{m',r'+1}=\pr(F_{m,r})$. On pose également $c'_{m'}:=c_{m}\prod_{t=1}^{r'-1}(-1)^{\deg F_{m,t}}$\index{$c'_{m'}$} (équation \eqref{decomp}). On a alors
\begin{equation}
f_{m'}=c'_{m'} \left(\prod_{t=1}^{r'-1} f_{m',t}^2\right) \times f_{m',r'} \times f_{m',r'+1} \label{decompC}
\end{equation}
avec $f_{m',t} \in \Sym((\gotn^-)^C)$ pour tout $t \leq r'$ par le théorème \ref{semiinv}).
\item Pour tout $m' \in \llbracket 1,n' \rrbracket \setminus \mathbf{M}'$, on note $c'_{m'}=1$, $f_{m',1}:=f_{m'}$ et $r'_{m'}=0$. On note alors $\mathbf{f}$ la famille de semi-invariants de $\Sy(\gotq^C)$ composée de l'ensemble des $f_{m',t}$ pour $1 \leq m' \leq n'$ et $1 \leq t \leq r'_{m'}+1$.
\end{itemize}
\label{notag}
\end{nota}
On obtient $\Card(\mathbf{f})=\sum_{m'=1}^{n'} (r'_{m'}+1) = n'+s'$.
\subsection{Poids}
On reprend les notations de la sous-section \ref{ssec4.2}. Soit $w_k^C=\varpi^C_{\iota_{k-1}}-\varpi^C_{\iota_{k}}$ pour $1 \leq k \leq s'$.
\begin{propn}
À $m'=m'_i \in \mathbf{M}'$ fixé, pour tout $t \in \llbracket 1, r'_{m'} +1\rrbracket$, le semi-invariant $f_{m',t}$ est de poids
$$\lambda^C_{m',t}=\left\{ \begin{array}{ll} w^C_{k_t}-w^C_{k_{t+1}} &  \text{si } t<r'_{m'} \\ 2w^C_{k_t} & \text{si }t=r'_{m'} \\ -2w^C_{k_1} & \text{si }t=r'_{m'}+1 \end{array} \right. $$
Si $m' \notin \mathbf{M}'$, on a $\lambda^C_{m',1}=0$. À $m'$ fixé, la famille $(\lambda^C_{m',t})_{1 \leq t \leq r'_{m'}+1}$ est de rang $r'_{m'}$.\par
Les espaces vectoriels $\vect((\lambda^C_{m',t})_{1 \leq t \leq r'_{m'}+1})$  pour $m' \in \llbracket 1,n' \rrbracket$ sont en somme directe.
\end{propn}
\begin{proof}
L'expression des poids vient du fait que $\lambda^C_{m',t}={\lambda_{2m',t}}_{|\syp_n}$.
%\begin{center}
%\begin{tikzpicture}[scale=0.3]
%\draw (0,0)--(18,0)--(18,18)--(0,18)--(0,0);
%\draw[double, ultra thick] (0,9)--(9,9)--(9,0);
%\draw[double, ultra thick] (9,18)--(9,9)--(18,9);
%\draw(2.25,2.25) node{$\displaystyle \gotn^-$};
%\draw(13.5,13.5) node{$\displaystyle \gotp$};
%\draw[dotted] (0,18)--(4.5,13.5);
%\draw[dotted] (13.5,4.5)--(18,0);
%\draw (4.5,13.5)--(4.5,9)--(9,9)--(9,4.5)--(13.5,4.5);
%\draw[dashed] (4.5,13.5)--(9,13.5)--(9,9)--(13.5,9)--(13.5,4.5);
%\draw (6.75,11.25) node{$\g_{I_{s'}}$};
%\draw (11.25,6.75) node{$\g_{I_{s'+1}}$};
%\draw[red] (4.5,9)--(9,9)--(9,4.5)--(4.5,4.5)--(4.5,9);
%\end{tikzpicture}
%\end{center}
Les poids $\lambda^C_{m',t}$, $1 \leq t \leq r'_{m'}$ sont linéairement indépendants. En effet, posons $e^C_k=\varpi^C_{\iota_k}$, avec $e^C_0=0$. La famille des $e^C_k$, $1 \leq k \leq s'$ est libre, donc la famille des $w^C_k=e^C_{k-1} - e^C_{k}$, $1 \leq k \leq s'$ est libre.
\begin{exem}
On reprend l'exemple \ref{exempleC}.
\begin{center}
\begin{tikzpicture}[scale=0.5]
\foreach \k in {1,2,...,11}
	{\draw[color=gray!20]  (0,\k)--(12,\k);
	\draw[color=gray!20] (\k,0)--(\k,12);}
\draw (0,0)--(12,0);
\draw (0,0)--(0,12);
\draw (0,12)--(12,12);
\draw (12,0)--(12,12);
\draw[ultra thick] (0,11)--(1,11)--(1,9)--(3,9)--(3,7)--(5,7)--(5,6)--(6,6)--(6,5)--(7,5)--(7,3)--(9,3)--(9,1)--(11,1)--(11,0);
\draw (3,3) node{\LARGE $\gotn^-$};
\draw (9,9) node{\LARGE $\gotp$};
\draw[dashed] (1,12)--(1,11)--(3,11)--(3,9)--(5,9)--(5,7)--(6,7)--(6,6)--(7,6)--(7,5)--(9,5)--(9,3)--(11,3)--(11,1)--(12,1);
%\draw (0.5,11.5) node{\small $\g_{I_1}$};
%\draw (2,10) node{\small $\g_{I_2}$};
%\draw (4,8) node{\small $\g_{I_3}$};
%\draw (5.5,6.5) node{\small $\g_{I_4}$};
%\draw (6.5,5.5) node{\small $\g_{I_5}$};
%\draw (8,4) node{\small $\g_{I_6}$};
%\draw (10,2) node{\small $\g_{I_7}$};
%\draw (11.5,0.5) node{\small $\g_{I_8}$};
\draw[->] (0.5,11.5)--(1.5,10.5);
\draw (1.5,11.5) node{\small $\varpi_1^C$};
\draw[->] (2.5,9.5)--(3.5,8.5);
\draw (3.5,9.5) node{\small $\varpi_3^C$};
\draw[->] (4.5,7.5)--(5.5,6.5);
\draw (5.5,7.5) node{\small $\varpi_5^C$};
\draw[->] (5.5,6.5)--(6.5,5.5);
\draw (6.5,6.5) node{\small $\varpi_6^C$};
\draw[->] (6.5,5.5)--(7.5,4.5);
\draw (7.5,5.5) node{\small $\varpi_5^C$};
\draw[->] (8.5,3.5)--(9.5,2.5);
\draw (9.5,3.5) node{\small $\varpi_3^C$};
\draw[->] (10.5,1.5)--(11.5,0.5);
\draw (11.5,1.5) node{\small $\varpi_1^C$};
\end{tikzpicture}
\end{center}
On a alors
\begin{itemize}
\item $e^C_0=0$, $e^C_1=\varpi_1^C$, $e^C_2=\varpi_3^C$, $e^C_3=\varpi_5^C$ et $e^C_4=\varpi_6^C$,
\item $w^C_1=-\varpi_1^C$, $w^C_2=\varpi_1^C-\varpi_3^C$, $w^C_3=\varpi_3^C-\varpi_5^C$ et $w^C_4=\varpi_5^C-\varpi_6^C$,
\item $\lambda^C_{8,1}=w^C_1-w^C_4$, $\lambda^C_{8,2}=2w^C_4$, $\lambda^C_{12,1}=w^C_2-w^C_3$ et $\lambda^C_{12,2}=2w^C_3$.
\end{itemize}
\end{exem}
Montrons alors que quelque soit $\kappa=\{k_1 < \ldots < k_r \} \subset \llbracket 1,s \rrbracket$, la famille des $\lambda^C_{m',t}=w_{k_t}^C-w_{k_{t+1}}^C$, $1 \leq t \leq r'-1$ et $\lambda^C_{m',r'}=2w_{k_{r'}}$ est libre. Supposons que $\sum_{t=1}^{r'} x_t \lambda^C_{m',t}=0$ avec $x_t \in \C$ et posons $x_0=0$. On a donc $\sum_{t=1}^{r'-1} x_t (w_{k_t}^C-w_{k_{t+1}}^C) + 2 x_{r'} w_{k_{r'}}^C=0$, c'est-à-dire
$$\sum_{t=1}^{r'-1} (x_t-x_{t-1}) w_{k_t}^C+(2x_{r'}-x_{r'-1}) w_{k_{r'}}^C=0$$
Comme la famille des $w_k^C$ est libre on a donc $x_1-x_0 = x_2-x_1= \ldots = x_{r'-1}-x_{r'-2}=2x_{r'}-x_{r'-1}=0$, ce qui implique que tous les $x_t$ sont nuls. En particulier, les $\lambda^C_{m',t}$ sont linéairement indépendants.\par
Soient $m_i \neq m_{i'}$ dans $\mathbf{M}_1$. Alors l'ensemble des $w_k^C$ qui apparaissent dans l'écriture des $\lambda^C_{m'_{i'},t}$ et l'ensemble des $w_k^C$ qui apparaissent dans l'écriture des $\lambda_{m'_{i'},t}^C$ sont disjoints. Comme les $w^C_k$ forment une famille libre, les espaces vectoriels $\vect((\lambda^C_{m'_{i'},t})_t)$ pour $1 \leq i \leq \imax$ sont en somme directe.
\end{proof} 
\begin{coro}
Par le théorème \ref{alin}, les $f_{m',t}$, pour $m' \in \llbracket 1, n' \rrbracket$, et $t \in \llbracket 1, r'_{m'} +1 \rrbracket$ sont algébriquement indépendants, de sorte que $\ind \gotq^C_\Lambda \geq n'+ \sum_{i \in \mathbf{I}} \rho'_i=n'+s'$.
\end{coro}

\subsection{Troncation canonique}
L'algèbre dérivée $(\gotq^C)'$ de $\gotq^C$ est
$$(\gotq^C)'=\widehat{\gotq^C} \oplus \bigoplus_{\stackrel{l \in \llbracket 1,n'-1 \rrbracket}{l \notin I_{\gotp}}} \C \pr(h_l)$$
où $\widehat{\gotq^C}$ est le sous-espace vectoriel des matrices de $\syp_{n}$ de diagonale nulle. Ainsi $(\gotq^C)'$ est de dimension $\dim \gotq^C-s'$.\par
Or $(\gotq^C)' \subset \gotq^C_\Lambda$ (propriété \ref{inclusion}) donc $\dim \gotq^C_\Lambda \geq \dim (\gotq^C)'= \dim \gotq^C-s'$ et $\ind \gotq^C_\Lambda \geq n'+s'=\ind \gotq^C+s'$. Mais alors par \eqref{ovdb}, on a les égalités, de sorte que
\begin{propn}
On a $\gotq^C_\Lambda=(\gotq^C)'$. Ceci implique que la famille des $f_{m',t}$, $m' \in \llbracket 1,n' \rrbracket$, $t \in \llbracket 1, r'_{m'}+1 \rrbracket$ est une base de transcendance de $\Sy(\gotq^C)$.
\label{conparaC}
\end{propn}
%Soit $\A$ une $\C$-algèbre. La projection $\pr=\pr^C$ induit la projection de $\A$-modules $\id_{\A} \otimes \pr : \gotq_{\A} \rightarrow \gotq^C_{\A}$, que l'on continue par abus de noter $\pr$.\par
%$$\shorthandoff{;:!?}
%  \xymatrix{\Sym(\gotq) \ar[r]^{\simeq} & \C[\gotq^*] \ar@{-->}[d] \ar@/^1pc/[r] & \gotq^* \ar[d]^{\res} \\ \Sym(\gotq^C) \ar[r]^{\simeq} & \C[{\gotq^C}^*] \ar@/^1pc/[r] & {\gotq^C}^*
%  }$$
%On définit le morphisme $\phi: \A[\gotq^*] \rightarrow \A[{\gotq^C}^*]$ comme le morphisme qui fait commuter le diagramme suivant
%$$\shorthandoff{;:!?}
%  \xymatrix{\Sym(\gotq)_{\A} \ar[r]^{\simeq} \ar[d]_{\pr} & \A[\gotq^*] \ar[d]^{\phi} \\ \Sym(\gotq^C)_{\A} \ar[r]^{\simeq} & \A[{\gotq^C}^*] 
%  }$$
%\begin{rema}
%Si l'on s'intéresse à l'autre partie du diagramme, le morphisme $\phi$ ne peut pas vérifier $\phi(f) \circ \res = f$ pour toute $f$. Supposons le contraire. Soit $V$ un supplémentaire de $\gotq^C$ dans $\gotq$. On a alors $\gotq^* \simeq {\gotq^C}^* \oplus V^*$. Pour tout $\lambda \in V^*$, on a donc $f(\lambda)=\phi(f)(0)$. Autrement dit, toute $f \in \C[\gotq^*]$ est constante sur $V^*$, ce qui est faux en général.
%\end{rema}
%Dans la suite, on considèrera que $\Sym(\gotq)$ (respectivement $\Sym(\gotq^C)$) agit sur $\gotq^*$ (respectivement ${\gotq^C}^*$).
%Puisque la projection $\pr$ ne peut pas vérifier $\pr(f) \circ \res = f$ pour toute $f$, on veut savoir à quel point cette égalité est fausse.

\subsection{Application du théorème \ref{inter} à $\gotk=\gotq^C$}
\label{sec434}
On sait que l'ensemble $\mathbf{f}$ des $f_{m',t}$ forme une base de transcendance de $\Sy(\gotq^C)$. On veut donc conclure à la polynomialité lorsque $\gotk=\gotq^C$ est une contraction parabolique de $\syp_n$ en appliquant le théorème \ref{inter} avec $\mathbf{f}$ l'ensemble des $f_{m',t}$. On vérifie d'abord les hypothèses (a), (b), (c) et (d) :
\begin{itemize}
\item[(a)] L'algèbre $\Y(\gotq^C)$ est polynomiale (théorème \ref{pyC}) donc en particulier factorielle.
\item[(b)] On a $\GK \Sy(\gotq^C)=\ind \gotq_\Lambda^C$ (théorème \ref{ai}). Comme l'ensemble $\mathbf{f}$ des $f_{m',t}$ est de cardinal $\ind \gotq^C_\Lambda$ (proposition \ref{conparaC}), on a $\GK \Sy(\gotq^C)= \GK \C[\mathbf{f}]$.
\item[(c,d)] Les $f_{m'}=\pr(F_{2m'}^\bullet)$ pour tout $m' \in \llbracket 1,n' \rrbracket$ engendrent librement l'algèbre $\Y(\gotq^C)$ donc sont irréductibles dans $\Y(\gotq^C)$.
\end{itemize}
Il reste donc à vérifier les hypothèses (I), (II) et (III) du théorème \ref{inter}.
\subsubsection{Hypothèse (I)}
\label{1051}
Pour montrer les hypothèses (I) et (III), on va comme précédemment montrer que les hypothèses (I') et (III') (de la proposition \ref{632}) sont vérifiées. On prendra ici $g_{m'}=f_{m'}=\pr(F_{2m'}^\bullet)$ et $f_{m',t}$ défini comme précédemment pour tous $m',t$.\par
\begin{nota}
On définit alors une partition $I= \bigsqcup_{i=1}^{\imax} I(i)$\index{$I(i)$, $I(i)^+$, $I(i)^-$} telle que pour tout $i \in \llbracket 1, \imax \rrbracket$ :
\begin{itemize}
\item $I(i)^\gamma=I(i)$,
\item pour tout $k \in \llbracket 1,s \rrbracket$, $|I(i) \cap I_k |=\left\{ \begin{array}{ll}
1&\text{si } k \in K(i) \\ 0& \text{sinon}
\end{array}\right.$.
\end{itemize}
On a $|I(i)|=|K(i)|=m_i-m_{i-1}$. Pour tout $i$, on note aussi $I(i)^-:=I(i) \cap \llbracket 1,n' \rrbracket$ et $I(i)^+:=I(i) \cap \llbracket n'+1,2n' \rrbracket$.
\label{KetI}
\end{nota}
On pose alors la suite $(v_l)_{1 \leq l \leq n'}$ définie par les deux points suivants :
\begin{itemize}
\item pour tout $i \in \llbracket 1, \imax \rrbracket$, l'ensemble des $v_l$ avec $l \in \llbracket m'_{i-1}+1,m'_i \rrbracket$ est l'ensemble $I(i)^-$ si $i$ est impair et $I(i)^+$ si $i$ est pair,
\item la suite $(v_l)_{ m'_{i-1} \, < \, l \, \leq \, m'_i}$ est strictement croissante.
\end{itemize}
On pose alors
$$q= \sum_{j=1}^{n'} X_{j} e_{v_{j}^\gamma, v_{j}}^* + e_{v_1, v_1^\gamma}^*+\sum_{j=2}^{n'} e_{v_{j},v_{j-1}}^*  $$
\begin{exem}
Dans l'exemple \ref{exempleC}, un choix possible pour les $I(i)$ est de prendre $I(1)=\{1,3,5,6,7,8,10,12\}$ et $I(2)=\{2,4,9,11\}$. La suite $(v_l)$ est alors $(v_l)_l=(1,3,5,6,9,11)$.\par \label{exCI}
En type $C$, on voit un élément $e_{u,v}^*$ dans $(\gotq^C)^*$ par restriction, il admet une décomposition dans la base canonique de $(\gotq^C)^*$ (la base duale de la base canonique de $\gotq^C$). Dans l'exemple, on représente alors $q$ par une matrice via l'isomorphisme $\gotq^C \rightarrow (\gotq^C)^*$ donné par la base canonique de $\gotq^C$.
\begin{center}
\begin{tikzpicture}[scale=0.7]
\foreach \k in {1,2,...,11}
	{\draw[color=gray!20]  (0,\k)--(12,\k);
	\draw[color=gray!20] (\k,0)--(\k,12);}
\draw (0,0)--(12,0);
\draw (0,0)--(0,12);
\draw (0,12)--(12,12);
\draw (12,0)--(12,12);
\draw[ultra thick] (0,11)--(1,11)--(1,9)--(3,9)--(3,7)--(5,7)--(5,6)--(6,6)--(6,5)--(7,5)--(7,3)--(9,3)--(9,1)--(11,1)--(11,0);
\draw (0.5,9.5) node{\Large $1$};
\draw (9.5,0.5) node{\Large $-1$};
\draw (2.5,7.5) node{\Large $1$};
\draw (7.5,2.5) node{\Large $-1$};
\draw (4.5,6.5) node{\Large $1$};
\draw (6.5,4.5) node{\Large $-1$};
\draw (5.5,3.5) node{\Large $1$};
\draw (3.5,5.5) node{\Large $1$};
\draw (8.5,1.5) node{\Large $1$};
\draw (1.5,8.5) node{\Large $-1$};
\draw (0.5,0.5) node{\Large $X_1$};
\draw (2.5,2.5) node{\Large $X_2$};
\draw (4.5,4.5) node{\Large $X_3$};
\draw (5.5,5.5) node{\Large $X_4$};
\draw (8.5,8.5) node{\Large $X_5$};
\draw (10.5,10.5) node{\Large $X_6$};
%\draw (3,3) node{\LARGE $\gotn^-$};
%\draw (9,9) node{\LARGE $\gotp$};
%\draw[dashed] (1,12)--(1,11)--(3,11)--(3,9)--(5,9)--(5,7)--(6,7)--(6,6)--(7,6)--(7,5)--(9,5)--(9,3)--(11,3)--(11,1)--(12,1);
\draw (11.5,11.5) node{\Large $1$};
%\draw (0.5,0.5) node{\small $-c_1X_1$};
%\draw (1.5,1.5) node{\small $-c_2X_2$};
%\draw (3.5,3.5) node{\small $-c_3X_3$};
%\draw (5.5,5.5) node{\small $-c_4X_4$};
%\draw (0.5,6.5) node{\Large $1$};
%\draw (6.5,0.5) node{\Large $-1$};
%\draw (1.5,4.5) node{\Large $1$};
%\draw (4.5,1.5) node{\Large $-1$};
%\draw (2.5,3.5) node{\Large $1$};
%\draw (3.5,2.5) node{\Large $1$};
\end{tikzpicture}
\end{center}
\end{exem}
Dans le cas général, le graphe $\mathcal{G}(q)$ en type $C$ est alors de la forme (voir le corollaire \ref{graphC})
\begin{center}
\begin{tikzpicture}
  \tikzset{LabelStyle/.style = {fill=white}}
  \tikzset{VertexStyle/.style = {%
  shape = circle, minimum size = 33pt,draw}}
  \SetGraphUnit{2.5}
  \Vertex[L=$v_1^\gamma$]{1}
  \EA[L=$v_2^\gamma$](1){2}
  \EA[L=$v_3^\gamma$](2){3}
  %\Loop[dist = 2cm, dir = WE, label = $X_1$](1.west)
  \Edge[style= {->}](1)(2)
  \Edge[style= {->}](2)(3)
  \SetVertexNormal[Shape = circle, LineColor=white, MinSize=33pt]
  \tikzset{EdgeStyle/.style = {->}}
  \EA[L=$\ldots$](3){4}
  \NO[L=$\ldots$](4){4'}
  \Edge(3)(4)
  %\SetUpVertex[LineColor=black]
  \SetVertexNormal[Shape = circle, LineColor=black, MinSize=33pt]
  \EA[L=$v_{n'-1}^\gamma$](4){m-1}
  \Edge[style= {->}](4)(m-1)
  \EA[L=$v_{n'}^\gamma$](m-1){m}
  \Edge[style= {->}](m-1)(m)
  \NO[L=$v_1$](1){1'}
  \NO[L=$v_2$](2){2'}
  \NO[L=$v_3$](3){3'}
  \NO[L=$v_{n'-1}$](m-1){m-1'}
  \NO[L=$v_{n'}$](m){m'}
  \Edge[style= {->}](2')(1')
  \Edge[style= {->}](3')(2')
  \Edge[style= {->}](4')(3')
  \Edge[style= {->}](m-1')(4')
  \Edge[style= {->}](m')(m-1')
  \Edge[style= {->,bend right=30}, label =$1$](1')(1)
  \Edge[style= {->,bend right=30}, label =$X_1$](1)(1')
  \Edge[style= {->,bend right=30}, label =$X_2$](2)(2')
  \Edge[style= {->,bend right=30}, label =$X_3$](3)(3')
  \Edge[style= {->,bend right=30}, label =$X_{n'-1}$](m-1)(m-1')
  \Edge[style= {->,bend right=30}, label =$X_{n'}$](m)(m')
\end{tikzpicture}
\end{center}
On se référera dans la suite aux définitions générales introduites dans la sous-partie \ref{graphe}. Les sous-graphes circuits de $\mathcal{G}(q)$ sont les graphes de la forme
\begin{center}
\begin{tikzpicture}[scale=0.9]
  \tikzset{LabelStyle/.style = {fill=white}}
  \tikzset{VertexStyle/.style = {%
  shape = circle, minimum size = 33pt,draw}}
  \SetGraphUnit{2.5}
  \Vertex[L=$v_1^\gamma$]{1}
  \EA[L=$v_2^\gamma$](1){2}
  \EA[L=$v_3^\gamma$](2){3}
  %\Loop[dist = 2cm, dir = WE, label = $X_1$](1.west)
  \Edge[style= {->}](1)(2)
  \Edge[style= {->}](2)(3)
  \SetVertexNormal[Shape = circle, LineColor=white, MinSize=33pt]
  \tikzset{EdgeStyle/.style = {->}}
  \EA[L=$\ldots$](3){4}
  \NO[L=$\ldots$](4){4'}
  \Edge(3)(4)
  %\SetUpVertex[LineColor=black]
  \SetVertexNormal[Shape = circle, LineColor=black, MinSize=33pt]
  \EA[L=$v_{m-1}^\gamma$](4){m-1}
  \Edge[style= {->}](4)(m-1)
  \EA[L=$v_{m}^\gamma$](m-1){m}
  \Edge[style= {->}](m-1)(m)
  \NO[L=$v_1$](1){1'}
  \NO[L=$v_2$](2){2'}
  \NO[L=$v_3$](3){3'}
  \NO[L=$v_{m-1}$](m-1){m-1'}
  \NO[L=$v_{m}$](m){m'}
  \Edge[style= {->}](2')(1')
  \Edge[style= {->}](3')(2')
  \Edge[style= {->}](4')(3')
  \Edge[style= {->}](m-1')(4')
  \Edge[style= {->}](m')(m-1')
  \Edge[style= {->,bend right=30}, label =$1$](1')(1)
  \Edge[style= {->,bend right=30}, label =$X_{m}$](m)(m')
\end{tikzpicture}
\end{center}
pour $ m \in \llbracket 1,n' \rrbracket$. On note $\mathcal{H}_{2m}$ ce sous-graphe circuits, $\mathcal{S}_{2m}:=\mathcal{S}_{\mathcal{H}_{2m}}$ le $\mathcal{H}_{2m}$-monôme en type $\gl_n$ (voir définition \ref{neq}), et $\mathcal{S}^C_{2m}:=\pr(\mathcal{S}_{2m})$ le $\mathcal{H}_{2m}$-monôme en type $C$. On a $\mathcal{S}_{2m}(q) \propto X_{m}$.\par
Le sous-graphe $\mathcal{H}_{2m}$ est symétrique au sens où si l'on a une arête de $x$ vers $y$, alors on a une arête de $y^\gamma$ vers $x^\gamma$. Ceci implique que $\mathcal{S}_{2m}$ est symétrique. Ainsi par la proposition \ref{tivafor}, à une constante multiplicative non nulle près, $\mathcal{S}^C_{2m}$ est un monôme de $\pr(F_{2m})$ en les $\pr(e_{x,y})$ et, par la proposition \ref{elemgr}, le seul monôme $\mathcal{S}$ de $\pr(F_{2m})$ tel que $\mathcal{S}(q) \neq 0$. On a donc $\pr(F_{2m})(q) \propto \mathcal{S}^C_{2m}(q) \propto X_m$.\par
Pour conclure quant à l'hypothèse (I'), il faut maintenant vérifier que $\mathcal{S}^C_{2m}$ est (à une constante multiplicative non nulle près) un monôme de $\pr(F_{2m}^\bullet)$. Par la remarque \ref{preserv}, il suffit de montrer que $\deg_{\gotn^-} \mathcal{S}_{2m} = \deg_{\gotn^-} F_{2m}$.\par
Comme sur $\gl_n$, on va montrer que la suite $\deg_{\gotn^-} \mathcal{S}_{2m}$ vérifie la même relation de récurrence que $\deg_{\gotn^-} F_{2m}$ , sachant que (voir équation \eqref{degrec})
$$\deg_{\gotn^-} F_{2m+2}=\left\{ \begin{array}{ll}
\deg_{\gotn^-} F_{2m} +1 &\text{si } 2m=m_i, \, 1 \leq i \leq \imax-1\\ \deg_{\gotn^-} F_{2m} +2& \text{sinon}.
\end{array}\right.$$
Premièrement, on a $\deg_{\gotn^-} \mathcal{S}_{2}=1=\deg_{\gotn^-} F_{2}$. En effet, puisque $v_1 \in I(1)^-$, on a $v_1 \leq n'$, et donc comme on est dans le cas où $\alpha_{n'}^C \notin (\pi^C)'$, on a $v_1^\gamma \succ v_1$ (voir propriété \ref{pn}). Ainsi $e_{v_1,v_1^\gamma} \in \gotp$ et $e_{v_1^\gamma,v_1} \in \gotn^-$.\par
On a ensuite
\begin{align*}
\deg_{\gotn^-} \mathcal{S}_{2m+2}&=\deg_{\gotn^-} \mathcal{S}_{2m} + \deg_{\gotn^-} e_{v_{m+1},v_{m}} + \deg_{\gotn^-} e_{v_{m}^\gamma,v_{m+1}^\gamma} \\ &+ \deg_{\gotn^-} e_{v_{m+1}^\gamma,v_{m+1}}-\deg_{\gotn^-} e_{v_{m}^\gamma,v_{m}}
\end{align*}
Déjà, puisque $\gamma$ est décroissante pour le préordre $\prec$ sur $ \llbracket 1,n \rrbracket$, on a $ \deg_{\gotn^-} e_{v_{m}^\gamma,v_{m+1}^\gamma}=\deg_{\gotn^-} e_{v_{m+1},v_{m}} $, d'où 
$$
\deg_{\gotn^-} \mathcal{S}_{2m+2}=\deg_{\gotn^-} \mathcal{S}_{2m} + 2\deg_{\gotn^-} e_{v_{m+1},v_{m}}  + \deg_{\gotn^-} e_{v_{m+1}^\gamma,v_{m+1}}-\deg_{\gotn^-} e_{v_{m}^\gamma,v_{m}}
$$
Si $2m$ n'est pas égal à un $m_i$, $1 \leq i \leq \imax$, alors $ m$ et $m+1$ appartiennent à un même intervalle $\llbracket m'_{i-1} +1, m'_i \rrbracket$ pour un certain $i$ (propriété \ref{propelem} (2)). Ainsi par définition de la suite $(v_{k})_{m'_{i-1}+1 \leq k \leq m'_i}$, on a $v_{m} < v_{m+1}$ et $v_{m} \nsim v_{m+1}$ (car l'ensemble des $v_k$ avec $m'_{i-1}+1 \leq k \leq m'_i$ est inclus dans $I(i)$ et $|I(i) \cap I_k| \leq 1$ pour tout $k$). Ainsi $v_{m} \prec v_{m+1}$, d'où $\deg_{\gotn^-} e_{v_{m+1},v_{m}}=1$. De plus, $v_{m}$ et $v_{m+1}$ sont soit tous les deux dans $I(i)^-$, soit tous les deux dans $I(i)^+$, de sorte que ${v_{m}}^\gamma \succ v_{m}$ si et seulement si ${v_{m+1}}^\gamma \succ v_{m+1}$. Finalement, $\deg_{\gotn^-} \mathcal{S}_{2m+2}=\deg_{\gotn^-} \mathcal{S}_{2m} + 2$.\par
Si $2m$ est égal à un $m_i$, $1 \leq i < \imax$, alors $ m$ et $m+1$ n'appartiennent pas à un même intervalle $\llbracket m'_{j-1} +1, m'_j \rrbracket$, de sorte que soit $v_{m} \in I(i)^+$ et $v_{m+1} \in I(i+1)^-$, soit $v_{m} \in I(i)^-$ et $v_{m+1} \in I(i+1)^+$. Dans le premier cas, on a :
\begin{itemize}
\item $v_{m+1} \prec v_{m}$, donc $\deg_{\gotn^-} e_{v_{m+1},v_{m}}=0$,
\item $v_{m+1}^\gamma \succ v_{m+1}$, donc $\deg_{\gotn^-} e_{v_{m+1}^\gamma,v_{m+1}}=1$,
\item $v_{m}^\gamma \prec v_{m}$, donc $\deg_{\gotn^-} e_{v_{m}^\gamma,v_{m}}=0$.
\end{itemize}
Dans le second cas, on a :
\begin{itemize}
\item $v_{m+1} \succ v_{m}$, donc $\deg_{\gotn^-} e_{v_{m+1},v_{m}}=1$,
\item $v_{m+1}^\gamma \prec v_{m+1}$, donc $\deg_{\gotn^-} e_{v_{m+1}^\gamma,v_{m+1}}=0$,
\item $v_{m}^\gamma \succ v_{m}$, donc $\deg_{\gotn^-} e_{v_{m}^\gamma,v_{m}}=1$.
\end{itemize}
Dans ces deux cas, on obtient $\deg_{\gotn^-} \mathcal{S}_{2m+2}=\deg_{\gotn^-} \mathcal{S}_{2m} + 1$. Ainsi, $\mathcal{S}_{2m}$ et $F_{2m}$ ont même degré en $\gotn^-$, ce qui conclut pour l'hypothèse (I').
\subsubsection{Hypothèse (III)}
Contrairement au type $\gl_n$ où le $q$ que l'on a exhibé pour l'hypothèse (I') pouvait facilement être adapté pour l'hypothèse (III'), ici ce n'est pas le cas, puisque pour le $q$ ci-dessus, on ne contrôle pas quels $t$ vérifient $f_{m',t}(q) \notin \C^\times$.\par
On reprend les définitions et notations de la sous-section \ref{1051}. Pour tout $i \in \llbracket 1,\imax \rrbracket$, on définit
$$q'_i:=\left(\sum_{l=m'_{i-1}+1}^{m'_i-1} e_{v_{l+1},v_l}^* \right) + e_{v_{m'_{i-1}+1},v_{m'_{i-1}+1}^\gamma}^* + e_{v_{m'_i}^\gamma, v_{m'_i}}^*$$
Le graphe $\mathcal{G}(q'_i)$ en type $C$ est de la forme (voir le corollaire \ref{graphC})
\begin{center}
\begin{tikzpicture}
  \tikzset{LabelStyle/.style = {fill=white}}
  \tikzset{VertexStyle/.style = {%
  shape = circle, minimum size = 45pt,draw}}
  \SetGraphUnit{2.5}
  \Vertex[L=$\scriptscriptstyle v_{m'_{i-1}+1}^\gamma$]{1}
  \EA[L=$\scriptscriptstyle v_{m'_{i-1}+2}^\gamma$](1){2}
  \EA[L=$\scriptscriptstyle v_{m'_{i-1}+3}^\gamma$](2){3}
  %\Loop[dist = 2cm, dir = WE, label = $X_1$](1.west)
  \Edge[style= {->}](1)(2)
  \Edge[style= {->}](2)(3)
  \SetVertexNormal[Shape = circle, LineColor=white, MinSize=45pt]
  \tikzset{EdgeStyle/.style = {->}}
  \EA[L=$\ldots$](3){4}
  \NO[L=$\ldots$](4){4'}
  \Edge(3)(4)
  %\SetUpVertex[LineColor=black]
  \SetVertexNormal[Shape = circle, LineColor=black, MinSize=45pt]
  \EA[L=$v_{m'_{i}-1}^\gamma$](4){m-1}
  \Edge[style= {->}](4)(m-1)
  \EA[L=$v_{m'_i}^\gamma$](m-1){m}
  \Edge[style= {->}](m-1)(m)
  \NO[L=$\scriptscriptstyle v_{m'_{i-1}+1}$](1){1'}
  \NO[L=$\scriptscriptstyle v_{m'_{i-1}+2}$](2){2'}
  \NO[L=$\scriptscriptstyle v_{m'_{i-1}+3}$](3){3'}
  \NO[L=$v_{m'_i-1}$](m-1){m-1'}
  \NO[L=$v_{m'_i}$](m){m'}
  \Edge[style= {->}](2')(1')
  \Edge[style= {->}](3')(2')
  \Edge[style= {->}](4')(3')
  \Edge[style= {->}](m-1')(4')
  \Edge[style= {->}](m')(m-1')
  \Edge[style= {->,bend right=30}](1')(1)
  \Edge[style= {->,bend right=30}](m)(m')
\end{tikzpicture}
\end{center}
%\begin{exem}
%
%\end{exem}
%Par la suite, on pourra noter ce graphe comme ceci :
%\begin{center}
%\begin{tikzpicture}
%\draw (0,0) ellipse(1cm and 1.5cm);
%\draw (0,1.5) node{$<$};
%\draw (0,-2) node{$i$};
%\end{tikzpicture}
%\end{center}
et est donc un circuit (voir définition \ref{circuits}). Soit $\overline{\mathcal{T}}_i$ le $\mathcal{G}(q'_i)$-monôme\index{$\overline{\mathcal{T}}_i$, $\mathcal{T}_i$} en type $\gl_n$ et $\mathcal{T}_i=\pr(\overline{\mathcal{T}}_i)$ le $\mathcal{G}(q'_i)$-monôme en type $C$. Par la suite, on pourra noter $\mathcal{G}(q'_i)$ sous la forme
\begin{center}
\begin{tikzpicture}
\draw[dashed] (0,0) ellipse(1.5cm and 1cm);
\draw (0,1) node{$<$};
\draw (-2.5,0) node{$i$};
\end{tikzpicture}
\end{center}
où le $i$ signale qu'il s'agit du graphe $\mathcal{G}(q'_i)$.
\begin{prop}
Le monôme $\mathcal{T}_i$ est de la forme $\prod_{l \in J} \pr(e_{l,\sigma(l)})$ avec $J \subset I$ et $\sigma \in \mathfrak{S}(J)$, et vérifie $\mathcal{T}_i(q'_i) \in \C^\times$. Parmi les facteurs $\pr(e_{l,\sigma(l)})$ :
\begin{itemize}
\item exactement deux ($e_{v_{m'_{i-1}+1},v_{m'_{i-1}+1}^\gamma}$ et $e_{v_{m'_i}^\gamma, v_{m'_i}}$) sont tels que $\sigma(l)=l^\gamma$ (autrement dit sont sur l'antidiagonale). Suivant la parité de $i$, l'un est tel que $\sigma(l) > l$ (et est donc sur la partie supérieure de l'antidiagonale, donc dans $\gotp^C$) et l'autre est tel que $\sigma(l)< l$ (et est donc sur la partie inférieure de l'antidiagonale, donc dans $(\gotn^-)^C$),
\item tous les autres facteurs vont par paire de la forme $\pr(e_{l,\sigma(l)})$ et $\pr(e_{\sigma(l)^\gamma,l^\gamma})$ avec $l+\sigma(l) <n+1$ ; ils sont tous dans $(\gotn^-)^C$,
\item tous les facteurs $\pr(e_{l,\sigma(l)})$ dans $(\gotn^-)^C$ sont tels que $\sigma(l)$ et $l$ sont deux termes consécutifs (pour l'ordre croissant) de l'ensemble $I(i)$.
\end{itemize}
Si $\pr(e_{l,\sigma(l)})$ est un facteur de $\mathcal{T}_i$, alors l'un des deux termes $e_{l,\sigma(l)}^*$ ou $e_{\sigma(l)^\gamma,l^\gamma}^*$ est un terme de $q'_i$.
\label{elemSi}
\end{prop}
On définit alors $q'=\sum_{i=1}^{\imax}q'_i$ qui a donc pour graphe en type $C$ la somme des graphes compatibles $\mathcal{G}(q'_i)$, $1 \leq i \leq \imax$. Le graphe $\mathcal{G}(q')$ est un graphe circuits.
\begin{exem}
Dans l'exemple \ref{exempleC} avec la suite $(v_l)_l$ définie en \ref{exCI}, on obtient $q'$ sous la forme suivante :
\begin{center}
\begin{tikzpicture}[scale=0.5]
\foreach \k in {1,2,...,11}
	{\draw[color=gray!20]  (0,\k)--(12,\k);
	\draw[color=gray!20] (\k,0)--(\k,12);}
\draw (0,0)--(12,0);
\draw (0,0)--(0,12);
\draw (0,12)--(12,12);
\draw (12,0)--(12,12);
\draw[ultra thick] (0,11)--(1,11)--(1,9)--(3,9)--(3,7)--(5,7)--(5,6)--(6,6)--(6,5)--(7,5)--(7,3)--(9,3)--(9,1)--(11,1)--(11,0);
\draw (0.5,9.5) node{\large $\mathbf{1}$};
\draw (9.5,0.5) node{\large $\mathbf{-1}$};
\draw (2.5,7.5) node{\large $\mathbf{1}$};
\draw (7.5,2.5) node{\large $\mathbf{-1}$}; 
\draw (4.5,6.5) node{\large $\mathbf{1}$};
\draw (6.5,4.5) node{\large $\mathbf{-1}$};
\draw (5.5,5.5) node{\large $\mathbf{1}$};
\draw (11.5,11.5) node{\large $\mathbf{1}$};
\draw (8.5,1.5) node{\large $\mathit{1}$};
\draw (1.5,8.5) node{\large $\mathit{-1}$};
\draw (3.5,3.5) node{\large $\mathit{1}$};
\draw (10.5,10.5) node{\large $\mathit{1}$};
%\draw[blue] (5.5,3.5) node{\large $1$};
%\draw (0.5,0.5) node{\tiny $-c_1X_1$};
%\draw (2.5,2.5) node{\tiny $-c_2X_2$};
%\draw (4.5,4.5) node{\tiny $-c_3X_3$};
%\draw (5.5,5.5) node{\tiny $-c_4X_4$};
%\draw (8.5,8.5) node{\tiny $-c_5X_5$};
%\draw (10.5,10.5) node{\tiny $-c_6X_6$};
%\draw (3,3) node{\LARGE $\gotn^-$};
%\draw (9,9) node{\LARGE $\gotp$};
%\draw[dashed] (1,12)--(1,11)--(3,11)--(3,9)--(5,9)--(5,7)--(6,7)--(6,6)--(7,6)--(7,5)--(9,5)--(9,3)--(11,3)--(11,1)--(12,1);

%\draw (0.5,0.5) node{\small $-c_1X_1$};
%\draw (1.5,1.5) node{\small $-c_2X_2$};
%\draw (3.5,3.5) node{\small $-c_3X_3$};
%\draw (5.5,5.5) node{\small $-c_4X_4$};
%\draw (0.5,6.5) node{\Large $1$};
%\draw (6.5,0.5) node{\Large $-1$};
%\draw (1.5,4.5) node{\Large $1$};
%\draw (4.5,1.5) node{\Large $-1$};
%\draw (2.5,3.5) node{\Large $1$};
%\draw (3.5,2.5) node{\Large $1$};
\end{tikzpicture}
\end{center}
Les coefficients en gras correspondent à $q'_1$ et les coefficients en italique correspondent à $q'_2$.
\end{exem}
%\begin{center}
%\begin{tikzpicture}
%\draw (0,0) ellipse(1cm and 1.5cm);
%\draw (0,1.5) node{$<$};
%\draw (0,-2) node{$1$};
%\draw (3,0) ellipse(1cm and 1.5cm);
%\draw (3,1.5) node{$<$};
%\draw (3,-2) node{$2$};
%\draw (6,0) ellipse(1cm and 1.5cm);
%\draw (6,1.5) node{$<$};
%\draw (6,-2) node{$3$};
%\draw (9,0) node{$\ldots$};
%\draw (12,0) ellipse(1cm and 1.5cm);
%\draw (12,1.5) node{$<$};
%\draw (12,-2) node{$\imax$};
%\end{tikzpicture}
%\end{center}
\begin{propn}
Pour tout $m=m_i \in \mathbf{M}$, le monôme $\prod_{j=1}^i \mathcal{T}_j$ est à une constante multiplicative non nulle près un monôme de $\pr(F_m^\bullet)$ et pour tout monôme $\mathcal{S}$ de $\pr(F_m^\bullet)$ on a $\mathcal{S}(q') \neq 0 \Leftrightarrow \mathcal{S} \propto \prod_{j=1}^i \mathcal{T}_j$. En particulier, on a $f_{m'_i}(q') \propto \prod_{j=1}^i \mathcal{T}_j(q') \in \C^\times$.
\label{pedagogique}
\end{propn}
Cette proposition n'est pas essentielle pour obtenir l'hypothèse (III') en tant que telle, mais sa démonstration servira de modèle dans les multiples cas que l'on exhibera plus tard.
\begin{proof}
On montre le sens direct en plusieurs étapes :
\begin{enumerate}[label=(\arabic*)]
\item \textbf{Déterminer l'ensemble des $\overline{\mathcal{S}} \in \Sym(\gl_n)$ qui sont des monômes de $F_m$ pour un certain $m$ et qui vérifient $\pr(\overline{\mathcal{S}})(q') \neq 0$.}\par
Par la proposition \ref{elemgr}, les monômes $\mathcal{S}$ de la forme $\varepsilon(\sigma) \prod_{l \in J} \pr(e_{l,\sigma(l)})$ avec $J \subset I$, $\sigma \in \mathfrak{S}(J)$ tels que $\mathcal{S}(q') \neq 0$ sont les monômes $\mathcal{S}$ tels que $\mathcal{S} \propto \mathcal{S}_{\mathcal{H}}$ avec $\mathcal{H}$ un sous-graphe circuits de $\mathcal{G}(q')$ (voir définition \ref{circuits}), c'est-à-dire, étant donné la forme de $\mathcal{G}(q')$, les monômes $\mathcal{S}$ tels que $\mathcal{S} \propto \prod_{j \in K} \mathcal{T}_j$ avec $K \subset \llbracket 1, \imax \rrbracket$.
\item \textbf{Parmi ces monômes, déterminer ceux qui sont des monômes de $F_m^\bullet$ à $m=m_i \in \mathbf{M}$ fixé.}\par
Pour cela, il s'agit de calculer les bidegrés de ces monômes. On calculera en fait leurs degrés et leurs degrés en $\gotp$. D'après la propriété \ref{elemSi}, on a
$$\deg \mathcal{T}_j=|I(j)|=m_j-m_{j-1} \quad \text{et} \quad \deg_{\gotp} \mathcal{T}_j = 1,$$
ce qui donne
$$\deg \prod_{j \in K} \mathcal{T}_j=\sum_{j \in K} (m_j-m_{j-1}) \quad \text{et} \quad \deg_{\gotp} \prod_{j \in K} \mathcal{T}_j = |K|.$$
Rappelons que $\deg \pr(F_m^\bullet)=m_i$ et $\deg_{\gotp} \pr(F_m^\bullet)=i$ (proposition \ref{oui}). On cherche donc $K \subset \llbracket 1, \imax \rrbracket$ tel que
$$\left\{ \begin{array}{l}
\sum_{j \in K} (m_j-m_{j-1}) = m_i
\\ |K|=i
\end{array}\right.
$$
Remarquons déjà que $K=\llbracket 1,i \rrbracket$ vérifie ces conditions. Ensuite, considérons un $K$ qui convient. Par la propriété \ref{propelem}, la suite $(m_{j}-m_{j-1})_{1 \leq j \leq \imax}$ est décroissante, et pour tout $j > i$, on a $m_j-m_{j-1} < m_i-m_{i-1}$. Or si l'on suppose que $K \neq \llbracket 1,i \rrbracket$, puisque $|K|=i$, on a $K \cap \llbracket i+1,\imax \rrbracket \nonvide$, et donc $\sum_{j \in K} (m_j-m_{j-1}) <\sum_{j=1}^i (m_j-m_{j-1})=m_i$ ce qui est absurde. On a donc montré que $K=\llbracket 1,i \rrbracket$ est le seul qui convient.
\end{enumerate}
Ainsi, si $\mathcal{S}$ est un monôme de $\pr(F_m^\bullet)$, alors $\mathcal{S}=\pr(\overline{\mathcal{S}})$ avec $\overline{\mathcal{S}}$ un monôme de $F_m^\bullet$. Avec les points (1) et (2), si $\mathcal{S}(q') \neq 0$, alors $\mathcal{S} \propto \prod_{j=1}^i \mathcal{T}_j$.\par
Réciproquement, si $\mathcal{S} \propto \prod_{j=1}^i \mathcal{T}_j$, comme $\overline{\mathcal{S}}:=\prod_{j=1}^i \overline{\mathcal{T}}_j$ est symétrique, par la proposition \ref{tivafor}, $\mathcal{S}$ est bien un monôme de $\pr(F_m)$, qui est de degré maximal en $(\gotn^-)^C$ par le point (2), donc est un monôme de $\pr(F_m^\bullet)$. De plus $\mathcal{S}(q') \neq 0$ par le point (1).
\end{proof}
Rappelons que pour tout $i \in \mathbf{I}$, on a $f_{m'_i}(q') \propto \prod_{j=1}^i \mathcal{T}_j(q') \in \C^\times$. Fixons $i \in \mathbf{I}$, $m'=m'_i \in \mathbf{M}'$, $m=2m'$, $r'=r'_{m'}$, $r=2r'$ et $t \in \llbracket 1, r'+1 \rrbracket$. Pour montrer (III'), une première idée est de multiplier un terme de $q'_i$ convenable par $X$, de sorte que si $q$ est la forme linéaire modifiée, on a $f_{m',t}(q) \propto X^p$ avec $p \in \{1,2\}$ et pour $\tau \neq t$, $f_{m',\tau}(q) \in \C^\times$. Par la proposition \ref{pedagogique}, pour $j < i$ on a également $f_{m'_j}(q) \in \C^\times$, donc $f_{m'_j, \tau}(q) \in \C^\times$ pour tout $\tau$. Avec cette seule modification, si $i<\imax$ on a $f_{m'_j}(q)\propto X^p$ pour $j > i$, ce qui n'est pas le résultat voulu, car on souhaiterait que $f_{m'_j}(q)\propto 1$ pour $j > i$. Si $i=\imax$, la condition $j>i$ est vide et ce problème ne se pose pas.\par
Dans le cas où $i < \imax$, on va donc introduire des modifications supplémentaires sur $q'$ et donc sur son graphe $\mathcal{G}(q')$ (en type $C$) pour faire en sorte que $f_{m'_j}(q) \in \C^\times$ pour $j > i$ (on ne gardera pas nécessairement l'hypothèse $f_{m',t}(q)\propto X^p$, on se contentera de $\deg_X f_{m',t}(q)>0$). Lorsque cette modification sera faite, comme dans la preuve de la proposition \ref{pedagogique}, on va 
\begin{itemize}
\item déterminer l'ensemble des monômes $\overline{\mathcal{S}}$ qui sont des monômes de $F_m$ pour un certain $m$ et qui vérifient $\pr(\overline{\mathcal{S}})(q) \neq 0$, c'est-à-dire \textbf{identifier les sous-graphes circuits $\mathcal{H}$ de $\mathcal{G}(q)$},
\item parmi ces monômes, déterminer ceux qui sont des monômes de $F_m^\bullet$ à $m=m_i$ fixé, c'est-à-dire \textbf{identifier quels sous-graphes circuits $\mathcal{H}$ vérifient $\deg \mathcal{S}_{\mathcal{H}}=m$ et $\deg_{\gotp} \mathcal{S}_{\mathcal{H}}=i$}, et pour ces sous-graphes circuits $\mathcal{H}$, calculer $\mathcal{S}_{\mathcal{H}}(q)$.
\end{itemize}
On rappelle que pour tout $u \in \llbracket 1,r' \rrbracket$, on a $F_{m,u} \in \Sym\left(\g_{I^{(u)},I^{[u]}}\right)$ (équation \eqref{somme}), d'où $f_{m',u}=\pr(F_{m,u}) \in \Sym\left(\pr(\g_{I^{(u)},I^{[u]}})\right)$ et $f_{m',r'+1}=\pr(F_{m,r}) \in \Sym\left(\pr(\g_{I^{(r)},I^{[r]}})\right)$. Or on a $\pr(\g_{I^{(u)},I^{[u]}}) \subset \g_{I^{(u)},I^{[u]}} \, + \, \g_{I^{(r-u)},I^{[r-u]}}$ pour $u \in \llbracket 1, r'-1 \rrbracket$ et $\pr(\g_{I^{(u)},I^{[u]}}) \subset \g_{I^{(u)},I^{[u]}}$ pour $u \in \{ r', r\}$. Les $\g_{I^{(u)},I^{[u]}} \, + \, \g_{I^{(r-u)},I^{[r-u]}}$ pour $u \in \llbracket 1, r' \rrbracket$ et $\g_{I^{(r)},I^{[r]}}$ sont en somme directe. Le produit $\prod_{j=1}^i \mathcal{T}_j$ est à une constante multiplicative près un monôme de $\pr(F_{m}^\bullet)$ en les $\pr(e_{x,y})$ (proposition \ref{pedagogique}). Ainsi de la même manière que dans le lemme \ref{monome}, à une constante multiplicative près il existe $y_u$\index{$y_u$} monôme de $f_{m',u}$ pour $1 \leq u \leq r'+1$, tel que
$$\prod_{j=1}^i \mathcal{T}_j=\prod_{u=1}^{r'-1} y_u^2 \times y_{r'} \times y_{r'+1}$$
(voir l'équation \ref{decompC}). À $m'$ fixé, les $y_u$ sont donc premiers entre eux deux à deux.
\begin{propn}\mbox{}
Pour tout $u \in \llbracket 1, r'+1 \rrbracket$, il existe un facteur $\pr(e_{v,w})$ du monôme $\mathcal{T}_i$ qui divise également le monôme $y_u$. \label{convenable}
\end{propn}
\begin{proof}
On suppose que $u \in \llbracket 1, r' \rrbracket$. Le raisonnement est similaire si $u = r'+1$. L'ensemble des facteurs irréductibles de $\mathcal{T}_i$ qui divisent $y_u$ est l'ensemble des facteurs irréductibles $\pr(e_{v,w})$ de $\mathcal{T}_i$ avec $e_{v,w} \in \g_{I^{(u)},I^{[u]}} \oplus \g_{I^{(r-u)},I^{[r-u]}}$. On rappelle que (notation \ref{nota46}) $I^{(u)}$ et $I^{[u]}$ sont de la forme $I^{(u)}=\bigsqcup_{k=a+1}^b I_k$ et $I^{[u]}=\bigsqcup_{k=a}^{b-1} I_k$ où $a<b$ vérifient $|I_a|=|I_b|=i$. Puisque $I(i)$ est défini de telle sorte que $I(i) \cap I_k \neq \emptyset$ si et seulement si $|I_k| \geq i$, on a $I(i) \cap I_b \neq \emptyset$. Soit alors $v \in I(i) \cap I_b$, et $\pr(e_{v,w})$ le facteur irréductible de $\mathcal{T}_i$ correspondant. Ce facteur appartient à $(\gotn^-)^C$ car $u \leq r'$. Alors par définition de $\mathcal{T}_i$, l'indice $w$ appartient à $I_c$ avec $c = \max\{k < b \, | \, I(i) \cap I_k \nonvide \}=\max\{k < b \, | \, |I_k| \geq i \}$ (propriété \ref{elemSi}). Comme $|I_a| =i$, on a $c \geq a$, et donc $I_c \subset I^{[u]}$. Ainsi $\pr(e_{v,w})$ divise nécessairement $y_u$.
\end{proof}
On rappelle que $i$ et $t$ sont fixés. On considère un facteur $\pr(e_{v,w})$ commun à $\mathcal{T}_i$ et $y_t$. Comme $\pr(e_{w^\gamma,v^\gamma})=\pm\pr(e_{v,w})$, on peut supposer que $v+w \leq n+1$. Par la propriété \ref{elemSi}, le terme $e_{v,w}^*$ est (au signe près) un terme de $q'$.\par
\begin{nota}
Les $e_{a,b}^*$ avec $a+b \leq n+1$ forment une base de $(\gotq^C)^*$. Si $a+b \leq n+1$, on note $p_{a,b}$\index{$p_{a,b}$} le coefficient de $e_{a,b}^*$ dans $q'$. Si $a+b > n+1$, il existe $\epsilon_{a,b} \in \{ \pm 1 \}$ tel que $e_{a,b}^*=\epsilon_{a,b} e_{b^\gamma,a^\gamma}^*$ avec $b^\gamma + a^\gamma \leq n+1$. On pose donc $p_{a,b}=\epsilon_{a,b} p_{b^\gamma, a^\gamma}$. Ainsi, pour tous $a,b$, on a $p_{a,b} \in \{-1,0,1\}$, et si $p_{a,b} \neq 0$, le terme $p_{a,b} e_{a,b}^*$ est un terme de $q'$.
\end{nota}
On modifie alors $q'_i$ et $q'_{i+1}$, et on note $q$ (respectivement $q_i$, $q_{i+1}$) le $q'$ (respectivement $q'_i$, $q'_{i+1}$) modifié.  Pour tous $j \notin \{i,i+1\}$, on note également $q_j=q'_j$.
\paragraph{Cas 1 : Si $i=\imax$,} comme on l'a évoqué précédemment, on multiplie simplement le terme $p_{v,w} e_{v,w}^*$ de $q'_i$ par $X$. En termes de graphes, on passe de $\mathcal{G}(q')$ à $\mathcal{G}(q)$ en multipliant par $X$ les poids des arêtes $v \rightarrow w$ et $w^\gamma \rightarrow v^\gamma$ (qui sont éventuellement une seule et même arête). Alors pour tout $i' \in \llbracket 1, \imax \rrbracket$, le seul monôme $\mathcal{S}$ de $\pr(F_{m_{i'}}^\bullet)$ tel que $\mathcal{S}(q) \neq 0$ vérifie $\mathcal{S} \propto \prod_{j=1}^{i'} \mathcal{S}_{\mathcal{G}(q_j)}$. Pour tout $m=m_{i'} \in \mathbf{M} \setminus \{n\}$, on a alors $\pr(F_m^\bullet)(q) \propto \prod_{j=1}^{i'} \mathcal{S}_{\mathcal{G}(q_j)}(q) \in \C^\times$ et $\pr(F_n^\bullet)(q) \propto \prod_{j=1}^{\imax} \mathcal{S}_{\mathcal{G}(q_j)}(q) \propto X^p$ avec $p \in \{1,2\}$. Puisque le seul terme modifié est $p_{v,w} e_{v,w}^*$ et que $\pr(e_{v,w})$ est un facteur de $y_t$, le seul facteur de $\pr(F_n^\bullet)$ qui appliqué en $q$ est de degré non nul en $X$ est le terme $f_{m'_{\imax},t}$. On a donc $f_{m'_{\imax},t}(q) \propto X^{p'}$ avec $p' \in \{1,2\}$ et $f_{m'_{\imax},u}(q)\propto 1$ pour $u \neq t$. De plus, pour tout $\mu' \in \llbracket 1, n'-1 \rrbracket$, $f_{\mu',\tau}(q)$ divise $\pr(F_{2\mu'}^\bullet)(q) \in \C^\times$. Ainsi l'hypothèse (III') est vérifiée.\par
Dans tous les cas suivants, on supposera que $i \neq \imax$ (dans l'exemple \ref{exempleC}, cela impose $i=1$). On modifiera $q'_i$ et $q'_{i+1}$ ; on pourra leur retirer et ajouter des termes. Les graphes $\mathcal{G}(q_i)$ et $\mathcal{G}(q_{i+1})$ qui en résultent ne seront pas équivalents à $\mathcal{G}(q'_i)$ et $\mathcal{G}(q'_{i+1})$. Il faudra alors vérifier que l'hypothèse (III') est satisfaite, donc calculer $\pr(F_m^\bullet)(q)$ pour tout $m \in \mathbf{M}$, et donc déterminer pour tout $m \in \mathbf{M}$ quels monômes $\mathcal{S}^C$ de $\pr(F_m^\bullet)$ vérifient $\mathcal{S}^C(q) \neq 0$ et calculer ces $\mathcal{S}^C(q)$, c'est-à-dire avoir un "analogue" de la proposition \ref{pedagogique} pour $q$.
\begin{defi}\mbox{}
\begin{itemize}
\item On utilise les notations de la définition \ref{repchem}. Dans les graphes qui suivront, on notera des arêtes sous la forme
\begin{center}
\begin{tikzpicture}
\node (v) at (3,0) {$x$};
\node (w) at (4.5,0) {$y$};
\draw[o->,, ultra thick, >=stealth] (v)--(w);
\node (et) at (6,0) {$\text{et}$};
\node (x) at (7.5,0) {$x$};
\node (y) at (9,0) {$y$};
\draw[o->,, ultra thick, dashed, >=stealth] (x)--(y);
\end{tikzpicture}
\end{center}
pour signifier que l'arête représente un élément $\pr(e_{x,y})$ dans $\gotp$ (respectivement qu'exactement une des arêtes impliquées dans la suite d'arêtes représente un élément de $\gotp$). Une arête (respectivement une arête pointillée) qui ne sera pas sous cette forme représentera un élément de $(\gotn^-)^C$  (respectivement une suite d'éléments de $(\gotn^-)^C$).
\item On parlera parfois de degré ou de bidegré d'un graphe $\mathcal{G}$ pour signifier le degré ou le bidegré de $\mathcal{S}_{\mathcal{G}}$.
\end{itemize}
\end{defi}
\paragraph{Cas 2 : Si $w=v^\gamma$,} alors par la propriété \ref{elemSi},
\begin{enumerate}[label=(2.\alph*)]
\item soit $w < v$ (c'est-à-dire $e_{v,w} \in (\gotn^-)^C$), dans ce cas, on définit $x,y$ de sorte que $\pr(e_{x,y})$ soit l'unique facteur de $\mathcal{T}_{i+1}$ tel que $y=x^\gamma$ et $y < x$,
\item soit $w > v$ (c'est-à-dire $e_{v,w} \in \gotp$), dans ce cas, on définit $x,y$ de sorte que $\pr(e_{x,y})$ soit l'unique facteur de $\mathcal{T}_{i+1}$ tel que $y=x^\gamma$ et $y > x$.
\end{enumerate}
Dans les deux cas, $q$ est obtenu en remplaçant dans $q'$ les termes $p_{v,v^\gamma} e_{v,v^\gamma}^*+p_{x,x^\gamma} e_{x,x^\gamma}^*$ par les termes $p_{v,v^\gamma}Xe_{v,v^\gamma}^*+ e_{v,x^\gamma}^*$.
\begin{exem}
Dans l'exemple \ref{exempleC},
\begin{itemize}
\item le sous-cas (2.a) correspond à $\pr(e_{v,w})=e_{7,6}$, et alors $\pr(e_{x,y})=e_{9,4}$,
\item le sous-cas (2.b) correspond à $\pr(e_{v,w})=e_{1,12}$, et alors $\pr(e_{x,y})=e_{2,11}$.
\end{itemize}
\end{exem}
\subparagraph{Cas (2.a)}\par
On passe alors du graphe $\mathcal{G}(q')$ au graphe $\mathcal{G}(q)$ en effectuant les modifications suivantes :\par
\begin{center}
\begin{tikzpicture}
\draw (0,0) node{\huge $\Longrightarrow$};
\draw (1.5,1) node{$i$};
\draw (1.5,-1) node{$i+1$};
\node (v) at (2.5,.5) {$v$};
\node (w) at (4.5,.5) {$v^\gamma$};
\draw[->, >=stealth, ultra thick] (v)--(w) node[midway,above] {$(X)$};
\draw[o->, dashed, >=stealth, ultra thick] (4.5,.7) arc (30:150:1.155);
\node (x) at (2.5,-.5) {$x$};
\node (y) at (4.5,-.5) {$x^\gamma$};
\draw[->, >=stealth, ultra thick] (v)--(y);
\draw[->, >=stealth, ultra thick] (x)--(w);
\draw[o->, dashed, >=stealth, ultra thick] (4.5,-.7) arc (-30:-150:1.155);
\draw (-4.5,1) node{$i$};
\draw (-4.5,-1) node{$i+1$};
\node (v) at (-3.5,.5) {$v$};
\node (w) at (-1.5,.5) {$v^\gamma$};
\draw[->, >=stealth, ultra thick] (v)--(w);
\draw[o->, dashed, >=stealth, ultra thick] (-1.5,.7) arc (30:150:1.155);
\node (x) at (-3.5,-.5) {$x$};
\node (y) at (-1.5,-.5) {$x^\gamma$};
\draw[->, >=stealth, ultra thick] (x)--(y);
\draw[o->, dashed, >=stealth, ultra thick] (-1.5,-.7) arc (-30:-150:1.155);
\end{tikzpicture}
\end{center}
%\textit{Cas (b)}
%\begin{center}
%\begin{tikzpicture}
%\draw (0,0) node{\huge $\Longrightarrow$};
%\draw (1.5,1) node{$i$};
%\draw (1.5,-1) node{$i+1$};
%\node (v) at (2.5,.5) {$v$};
%\node (w) at (4.5,.5) {$v^\gamma$};
%\draw[o->] (v)--(w) node[midway,above] {$(X)$};
%\draw[->, dashed] (4.5,.7) arc (30:150:1.155);
%\node (x) at (2.5,-.5) {$x$};
%\node (y) at (4.5,-.5) {$x^\gamma$};
%\draw[o->] (v)--(y);
%\draw[o->] (x)--(w);
%\draw[->, dashed] (4.5,-.7) arc (-30:-150:1.155);
%\draw (-4.5,1) node{$i$};
%\draw (-4.5,-1) node{$i+1$};
%\node (v) at (-3.5,.5) {$v$};
%\node (w) at (-1.5,.5) {$v^\gamma$};
%\draw[o->] (v)--(w);
%\draw[->, dashed] (-1.5,.7) arc (30:150:1.155);
%\node (x) at (-3.5,-.5) {$x$};
%\node (y) at (-1.5,-.5) {$x^\gamma$};
%\draw[o->] (x)--(y);
%\draw[->, dashed] (-1.5,-.7) arc (-30:-150:1.155);
%\end{tikzpicture}
%\end{center}
\textbf{Identifions les sous-graphes circuits $\mathcal{H}$ de $\mathcal{G}(q)$.} Identifions d'abord les sous-circuits de $\mathcal{G}(q)$. En dehors des circuits $\mathcal{G}(q'_j)$ pour $j \notin \{i,i+1\}$, on a ces deux sous-circuits incompatibles :
\begin{center}
\begin{tikzpicture}
\draw (1.5,1) node{$i$};
\draw (1.5,-1) node{$i+1$};
\node (v) at (2.5,.5) {$v$};
\node (w) at (4.5,.5) {$v^\gamma$};
%\draw[->] (v)--(w);
\draw[o->, dashed, >=stealth, ultra thick] (4.5,.7) arc (30:150:1.155);
\node (x) at (2.5,-.5) {$x$};
\node (y) at (4.5,-.5) {$x^\gamma$};
\draw[->, >=stealth, ultra thick] (v)--(y);
\draw[->, >=stealth, ultra thick] (x)--(w);
\draw[o->, dashed, >=stealth, ultra thick] (4.5,-.7) arc (-30:-150:1.155);
\draw (-4.5,1) node{$i$};
\draw (-4.5,-1) node{$i+1$};
\node (v) at (-3.5,.5) {$v$};
\node (w) at (-1.5,.5) {$v^\gamma$};
\draw[->, >=stealth] (v)--(w) node[midway,above] {$(X)$};
\draw[o->, dashed, >=stealth, ultra thick] (-1.5,.7) arc (30:150:1.155);
%\node (x) at (-3.5,-.5) {$x$};
%\node (y) at (-1.5,-.5) {$x^\gamma$};
%\draw[->] (x)--(y);
%\draw[->, dashed] (-1.5,-.7) arc (0:-180:1);
\end{tikzpicture}
\end{center}
Le premier est équivalent à $\mathcal{G}(q'_i)$, on le notera $\mathcal{G}'_1$. On note le second $\mathcal{G}'_2$. Ils sont de bidegré respectifs $(1, |I(i)|-1)$ et $(2, |I(i)|+|I(i+1)|-2)$.\par
Un sous-graphe circuits de $\mathcal{G}(q)$ est donc de la forme suivante
$$\mathcal{H}=\sum_{j \in D_1} \mathcal{G}(q'_j) + \sum_{j \in D_2} \mathcal{G}'_j $$
avec $D_1 \subset \llbracket 1,\imax \rrbracket \setminus \{i,i+1\}$ et $D_2 \varsubsetneq \{1,2\}$.\par
\textbf{Identifions quels sous-graphes circuits $\mathcal{H}$ vérifient $\deg \mathcal{S}_{\mathcal{H}}=m$ et $\deg_{\gotp} \mathcal{S}_{\mathcal{H}}=i$.}
On a $\bideg \mathcal{G}'_{1} = \bideg \mathcal{G}(q'_i)$ et $\bideg \mathcal{G}'_{2} = \bideg \mathcal{G}(q'_i) + \bideg \mathcal{G}(q'_{i+1})$. Définissons une fonction $\Upsilon$ qui associe à chaque sous-circuit de $\mathcal{H}$ un sous-graphe circuits de $\mathcal{G}(q')$ de même bidegré :
\begin{itemize}
\item pour $j \notin \{i,i+1\}$, on pose $\Upsilon(\mathcal{G}(q'_j))=\mathcal{G}(q'_j)$,
\item $\Upsilon(\mathcal{G}'_1)=\mathcal{G}(q'_i)$,
\item $\Upsilon(\mathcal{G}'_2)=\mathcal{G}(q'_i) + \mathcal{G}(q'_{i+1})$.
\end{itemize}
Puisque $\mathcal{H}$ ne peut contenir à la fois $\mathcal{G}'_{1}$ et $\mathcal{G}'_{2}$, tous les $\Upsilon(\mathcal{G})$ pour $\mathcal{G}$ un sous-circuit de $\mathcal{H}$ sont compatibles deux à deux, ainsi la somme $\sum_{j \in D_1} \Upsilon(\mathcal{G}(q'_j)) + \sum_{j \in D_2} \Upsilon(\mathcal{G}'_j)$ est bien un sous-graphe circuits de $\mathcal{G}(q')$, qui est de même bidegré que $\mathcal{H}$. Autrement dit, étudier le bidegré de $\mathcal{H}$ revient à étudier le bidegré d'un certain sous-graphe circuits de $\mathcal{G}(q')$, ce qui a déjà été fait précédemment (démonstration de la proposition \ref{pedagogique} (2)). Avec cette considération, on obtient alors à $m_{i'} \in \mathbf{M}$ fixé,
\begin{itemize}
\item si $i' < i$, alors le seul sous-graphe circuits $\mathcal{H}$ tel que $\deg \mathcal{S}_{\mathcal{H}}=m_{i'}$ et $\deg_{\gotp} \mathcal{S}_{\mathcal{H}}=i'$ est $\sum_{j \in \llbracket 1,i' \rrbracket} \mathcal{G}(q'_j)$,
\item si $i' = i$, alors le seul sous-graphe circuits $\mathcal{H}$ tel que $\deg \mathcal{S}_{\mathcal{H}}=m_{i'}$ et $\deg_{\gotp} \mathcal{S}_{\mathcal{H}}=i'$ est $\left(\sum_{j=1}^{i-1} \mathcal{G}(q'_j)\right) + \mathcal{G}'_1$,
\item si $i' \geq i+1$, alors le seul sous-graphe circuits $\mathcal{H}$ tel que $\deg \mathcal{S}_{\mathcal{H}}=m_{i'}$ et $\deg_{\gotp} \mathcal{S}_{\mathcal{H}}=i'$ est $\left(\sum_{j \in \llbracket 1,i-1 \rrbracket \sqcup \llbracket i+2, i' \rrbracket} \mathcal{G}(q'_j)\right) + \mathcal{G}'_{2}$.
\end{itemize}
%on notera à chaque fois $\mathcal{H}_i'$ l'unique sous-graphe exhibé.
On montre alors grâce aux propriétés \ref{neq} et \ref{elemgr} que pour tout $i' \neq i$, la quantité $\pr(F_{m_{i'}})(q)$ est un polynôme de degré nul en $X$, et $\pr(F_{m_{i}})(q) \propto X$. Comme dans le cas $i=\imax$, puisque seul le coefficient en $e_{v,v^\gamma}^*$ est non constant et que $\pr(e_{v,v^\gamma})$ est un facteur de $y_t$ lui-même un monôme de $f_{m'_i,t}$, le seul facteur de $\pr(F_{m_i}^\bullet)$ qui appliqué en $q$ est de degré $\geq 1$ en $X$ est le terme $f_{m'_i,t}$. Ainsi on a $\deg_X f_{m'_i,t}(q)=1$ et $\deg_X f_{m'_i,u}(q)=0$ pour $u \neq t$. Pour tout $\mu \neq m'_i$, on a également $\deg_X f_{\mu,\tau} = 0$.

\subparagraph{\textbf{Cas (2.b)}} La démonstration est identique au cas (2.a), seuls les vecteurs de $\gotp$ sont différents.
\begin{center}
\begin{tikzpicture}
\draw (0,0) node{\huge $\Longrightarrow$};
\draw (1.5,1) node{$i$};
\draw (1.5,-1) node{$i+1$};
\node (v) at (2.5,.5) {$v$};
\node (w) at (4.5,.5) {$v^\gamma$};
\draw[o->, >=stealth, ultra thick] (v)--(w) node[midway,above] {$(X)$};
\draw[->, dashed, >=stealth, ultra thick] (4.5,.7) arc (30:150:1.155);
\node (x) at (2.5,-.5) {$x$};
\node (y) at (4.5,-.5) {$x^\gamma$};
\draw[o->, >=stealth, ultra thick] (v)--(y);
\draw[o->, >=stealth, ultra thick] (x)--(w);
\draw[->, dashed, >=stealth, ultra thick] (4.5,-.7) arc (-30:-150:1.155);
\draw (-4.5,1) node{$i$};
\draw (-4.5,-1) node{$i+1$};
\node (v) at (-3.5,.5) {$v$};
\node (w) at (-1.5,.5) {$v^\gamma$};
\draw[o->, >=stealth, ultra thick] (v)--(w);
\draw[->, dashed, >=stealth, ultra thick] (-1.5,.7) arc (30:150:1.155);
\node (x) at (-3.5,-.5) {$x$};
\node (y) at (-1.5,-.5) {$x^\gamma$};
\draw[o->, >=stealth, ultra thick] (x)--(y);
\draw[->, dashed, >=stealth, ultra thick] (-1.5,-.7) arc (-30:-150:1.155);
\end{tikzpicture}
\end{center}

\paragraph{\textbf{Cas (3) : Si $w \neq v^\gamma$,}} alors par la propriété \ref{elemSi}, on a $\pr(e_{v,w}) \in (\gotn^-)^C$, et donc $v \succ w$ (propriété \ref{pn}). Comme par hypothèse $v+w \leq n+1$, on a $v,w \in I(i)^-$.\par
\subparagraph{\textbf{Cas (3.a) : Supposons qu'il existe un facteur $\pr(e_{x,y})$ de $\mathcal{T}_{i+1}$ tel que $y \preceq w \prec v \preceq x$.}} Alors $e_{x,y} \in (\gotn^-)^C$. On distingue encore deux cas.\par
\textbf{Cas (3.a.i) : Si $y = x^\gamma$,} alors on remplace $p_{v,w} e_{v,w}^*+p_{x,x^\gamma} e_{x,x^\gamma}^*$ par $p_{v,w} Xe_{v,w}^*+e_{v,x^\gamma}^*+e_{w^\gamma,w}^*$.
\begin{exem}
Un exemple de ce cas dans l'exemple \ref{exempleC} est le cas où $\pr(e_{v,w})=\pr(e_{6,5})$ et $\pr(e_{x,y})=e_{9,4}$.
\end{exem}
On passe alors du graphe $\mathcal{G}(q')$ au graphe $\mathcal{G}(q)$ en effectuant les modifications suivantes :
\begin{center}
\begin{tikzpicture}[scale=0.9]
%\draw (1.5,1) node{$i$};
%\draw (1.5,-1) node{$i+1$};
%\node (v) at (2.5,.5) {$v$};
%\node (w) at (4.5,.5) {$v^\gamma$};
%\draw[->] (v)--(w) node[midway,above] {$c X$};
%\draw[->, dashed] (4.5,.7) arc (0:180:1);
%\node (x) at (2.5,-.5) {$x$};
%\node (y) at (4.5,-.5) {$x^\gamma$};
%\draw[->] (v)--(y);
%\draw[->] (x)--(w);
%\draw[->, dashed] (4.5,-.7) arc (0:-180:1);
\draw (-4.5,1) node{$i$};
\draw (-4.5,-1) node{$i+1$};
\node (v) at (-3.5,.5) {$v$};
\node (w) at (-1.5,.5) {$w$};
\node (wg) at (.5,.5) {$w^\gamma$};
\node (vg) at (2.5,.5) {$v^\gamma$};
\draw[o->, dashed, >=stealth, ultra thick] (w)--(wg);
\draw[->, >=stealth, ultra thick] (wg)--(vg);
\draw[->, >=stealth, ultra thick] (v)--(w);
\draw[->, dashed, >=stealth, ultra thick] (2.5,.7) arc (30:150:3.464);
\node (x) at (-1.5,-.5) {$x$};
\node (y) at (.5,-.5) {$x^\gamma$};
\draw[->, >=stealth, ultra thick] (x)--(y);
\draw[o->, dashed, >=stealth, ultra thick] (.5,-.7) arc (-30:-150:1.155);

\draw (3.5,0) node{\huge $\Rightarrow$};

\draw (4.5,1) node{$i$};
\draw (4.5,-1) node{$i+1$};
\node (v) at (5.5,.5) {$v$};
\node (w) at (7.5,.5) {$w$};
\node (wg) at (9.5,.5) {$w^\gamma$};
\node (vg) at (11.5,.5) {$v^\gamma$};
\draw[o->, dashed, >=stealth, ultra thick] (w)--(wg);
\draw[->, >=stealth, ultra thick] (wg)--(vg) node[midway,above] {$(X)$};
\draw[->, >=stealth, ultra thick] (v)--(w) node[midway,above] {$(X)$};
\draw[->, dashed, >=stealth, ultra thick] (11.5,.7) arc (30:150:3.464);
\node (x) at (7.5,-.5) {$x$};
\node (y) at (9.5,-.5) {$x^\gamma$};
%\draw[->] (x)--(y);
\draw[->, >=stealth, ultra thick] (x)--(vg);
\draw[->, >=stealth, ultra thick] (v)--(y);
\draw[->, >=stealth, ultra thick] (9.5,.7) arc (30:150:1.155);
\draw[o->, dashed, >=stealth, ultra thick] (9.5,-0.7) arc (-30:-150:1.155);
\end{tikzpicture}
\end{center}
\textbf{Identifions les sous-graphes circuits $\mathcal{H}$ de $\mathcal{G}(q)$.}  Identifions d'abord les sous-circuits de $\mathcal{G}(q)$. En dehors des circuits $\mathcal{G}(q'_j)$ pour $j \notin \{i,i+1\}$, on a ces trois sous-circuits :
\begin{center}
\begin{tikzpicture}[scale=0.7]
%\draw (1.5,1) node{$i$};
%\draw (1.5,-1) node{$i+1$};
%\node (v) at (2.5,.5) {$v$};
%\node (w) at (4.5,.5) {$v^\gamma$};
%\draw[->] (v)--(w) node[midway,above] {$c X$};
%\draw[->, dashed] (4.5,.7) arc (0:180:1);
%\node (x) at (2.5,-.5) {$x$};
%\node (y) at (4.5,-.5) {$x^\gamma$};
%\draw[->] (v)--(y);
%\draw[->] (x)--(w);
%\draw[->, dashed] (4.5,-.7) arc (0:-180:1);
\draw (-4.5,1) node{$i$};
\draw (-4.5,-1) node{$i+1$};
\node (v) at (-3.5,.5) {$v$};
\node (w) at (-1.5,.5) {$w$};
\node (wg) at (.5,.5) {$w^\gamma$};
\node (vg) at (2.5,.5) {$v^\gamma$};
\draw[o->, dashed, >=stealth, ultra thick] (w)--(wg);
\draw[->, >=stealth, ultra thick] (wg)--(vg) node[midway,above] {$(X)$};
\draw[->, >=stealth, ultra thick] (v)--(w) node[midway,above] {$(X)$};
\draw[->, dashed, >=stealth, ultra thick] (2.5,.7) arc (30:150:3.464);
%\node (x) at (-1.5,-.5) {$x$};
%\node (y) at (.5,-.5) {$x^\gamma$};
%\draw[->] (x)--(y);
%\draw[->, dashed] (.5,-.7) arc (0:-180:1);
%\draw (-.5,-2.5) node{$\Downarrow$};
\end{tikzpicture}
\begin{tikzpicture}[scale=0.7]
\draw (-9,-4) node{$i$};
\draw (-9,-6) node{$i+1$};
\node (v) at (-8,-4.5) {$v$};
%\node (w) at (-1.5,-4.5) {$w$};
%\node (wg) at (.5,-4.5) {$w^\gamma$};
\node (vg) at (-2,-4.5) {$v^\gamma$};
%\draw[->, dashed] (w)--(wg);
%\draw[->] (wg)--(vg) node[midway,above] {$\epsilon c X$};
%\draw[->] (v)--(w) node[midway,above] {$c X$};
\draw[->, dashed, >=stealth, ultra thick] (-2,-4.3) arc (30:150:3.464);
\node (x) at (-6,-5.5) {$x$};
\node (y) at (-4,-5.5) {$x^\gamma$};
%\draw[->] (x)--(y);
\draw[->, >=stealth, ultra thick] (x)--(vg);
\draw[->, >=stealth, ultra thick] (v)--(y);
%\draw[] (.5,-4.3) arc (0:90:.5);
%\draw[] (0,-3.8)--(-1,-3.8);
%\draw[->] (-1,-3.8) arc (90:180:.5);
\draw[o->, dashed, >=stealth, ultra thick] (-4,-5.7) arc (-30:-150:1.155);
\end{tikzpicture}
\begin{tikzpicture}[scale=0.7]
\draw (1,-4) node{$i$};
\draw (1,-6) node{$i+1$};
%\node (v) at (-8,-4.5) {$v$};
\node (w) at (2,-4.5) {$w$};
\node (wg) at (4,-4.5) {$w^\gamma$};
%\node (vg) at (-1,-4.5) {$v^\gamma$};
\draw[o->, dashed, >=stealth, ultra thick] (w)--(wg);
%\draw[->] (wg)--(vg) node[midway,above] {$\epsilon c X$};
%\draw[->] (v)--(w) node[midway,above] {$c X$};
%\draw[->, dashed] (-7,-3.3) arc (90:180:1);
%\draw[dashed] (-2,-4.3) arc (0:90:1);
%\draw[dashed] (-3,-3.3)--(-7,-3.3);
%\node (x) at (-6,-5.5) {$x$};
%\node (y) at (-4,-5.5) {$x^\gamma$};
%\draw[->] (x)--(y);
%\draw[->] (x)--(vg);
%\draw[->] (v)--(y);
\draw[->, >=stealth, ultra thick] (4,-4.3) arc (30:150:1.155);
%\draw[->, dashed] (-4,-5.7) arc (0:-180:1);
%\draw[white] (9,-4)--(9,-6);
\end{tikzpicture}
\end{center}
Le premier circuit est équivalent à $\mathcal{G}(q'_i)$, on le notera $\mathcal{G}'_1$. On note les deux autres respectivement $\mathcal{G}'_2$ et $\mathcal{G}'_3$. Le graphe $\mathcal{G}'_1$ est incompatible avec les graphes $\mathcal{G}'_2$ et $\mathcal{G}'_3$ (les graphes $\mathcal{G}'_2$ et $\mathcal{G}'_3$ sont compatibles). Un sous-graphe circuits de $\mathcal{G}(q)$ est donc de la forme
$$\sum_{j \in D_1} \mathcal{G}(q'_j) + \sum_{j \in D_2} \mathcal{G}'_j $$
avec $D_1 \subset \llbracket 1,\imax \rrbracket \setminus \{i,i+1\}$ et $D_2 \subset \{1\}$ ou $D_2 \subset \{2,3\}$.\par
\textbf{Identifions quels sous-graphes circuits $\mathcal{H}$ vérifient $\deg \mathcal{S}_{\mathcal{H}}=m$ et $\deg_{\gotp} \mathcal{S}_{\mathcal{H}}=i$.}
On a $\bideg \mathcal{G}'_{1} = \bideg \mathcal{G}(q'_i)$ et $\bideg (\mathcal{G}'_{2} + \mathcal{G}'_{3}) = \bideg \mathcal{G}(q'_i) + \bideg \mathcal{G}(q'_{i+1})$. Alors comme dans le cas précédent, à $m_{i'} \in \mathbf{M}$ fixé,
\begin{itemize}
\item pour $i' < i$, le sous-graphe $\mathcal{H}_{i'}=\sum_{j \in \llbracket 1,i' \rrbracket} \mathcal{G}(q'_j)$ vérifie $\bideg \mathcal{S}_{\mathcal{H}_{i'}} = \bideg \pr(F_{m_{i'}})$,
\item pour $i' = i$, le sous-graphe $\mathcal{H}_{i'}=\left(\sum_{j=1}^{i-1} \mathcal{G}(q'_j)\right) + \mathcal{G}'_1$ vérifie $\bideg \mathcal{S}_{\mathcal{H}_{i'}} = \bideg \pr(F_{m_{i'}})$,
\item pour $i' \geq i+1$ le sous-graphe $\left(\sum_{j \in \llbracket 1,i-1 \rrbracket \sqcup \llbracket i+2, i' \rrbracket} \mathcal{G}(q'_j)\right) + \mathcal{G}'_{2} + \mathcal{G}'_{3}$ vérifie $\bideg \mathcal{S}_{\mathcal{H}_{i'}} = \bideg \pr(F_{m_{i'}})$.
\end{itemize}
mais contrairement au cas précédent, il n'y a pas unicité. Autrement dit, il peut exister d'autres sous-graphes circuits donnant une égalité de bidegré. Cela vient du fait que cette fois-ci, le sous-graphe circuits incompatible avec $\mathcal{G}'_1$ est une somme de deux circuits compatibles et non plus un seul circuit.\par
Remarquons que tous les circuits de $\mathcal{G}(q)$ ont chacun au moins une \textbf{arête propre}, c'est-à-dire une arête dont le poids dans ce circuit est non nul et le poids dans tout autre circuit de $\mathcal{G}(q)$ est nul. Plus précisément,
\begin{itemize}
\item toute arête de tout circuit $\mathcal{G}(q_j)$ avec $j \in \llbracket 1,i-1 \rrbracket \sqcup \llbracket i+2, \imax \rrbracket$ est une arête propre,
\item les arêtes $v \stackrel{(X)}{\rightarrow} w$ et $w^\gamma \stackrel{(X)}{\rightarrow} v^\gamma$ de $\mathcal{G}'_1$ sont propres,
\item les arêtes $v \rightarrow x^\gamma$, $x \rightarrow v^\gamma$, ainsi que la suite d'arêtes $x^\gamma \dashrightarrow x$ de $\mathcal{G}'_2$ sont propres,
\item l'arête $w^\gamma \rightarrow w$ de $\mathcal{G}'_3$ est propre.
\end{itemize}
Soit $\mathbf{z}:=\{z_1, \ldots, z_{i-1}, z'_2, z'_3, z_{i+2}, \ldots, z_{\imax}\}$ un ensemble d'indéterminées. On choisit $c_1 \stackrel{*}{\rightarrow} c_2$ une arête propre de $\mathcal{G}(q_j)$ pour $j \in \llbracket 1,i-1 \rrbracket \sqcup \llbracket i+2, \imax \rrbracket$ (respectivement de $\mathcal{G}'_j$ pour $j \in \{2,3\}$). On multiplie le coefficient (non nul) de $e_{c_1,c_2}^*$ dans $q$ par $z_j$ (respectivement par $z'_j$).\par
On a alors $\mathcal{S}_{\mathcal{G}(q_j)}(q) \propto z_j^{p_j}$ pour tout $j \in \llbracket 1,i-1 \rrbracket \sqcup \llbracket i+2, \imax \rrbracket$ et $\mathcal{S}_{\mathcal{G}'_j}(q) \propto (z'_j)^{p'_j}$pour $j \in \{2,3\}$, où $p_j,p'_j \in \{1,2\}$ (à $e_{c_1,c_2}^*$ peuvent correspondre deux arêtes). Les $\mathcal{S}_{C}(q)$ pour $C$ parcourant l'ensemble des sous-graphes circuits de $\mathcal{G}(q)$ sont linéairement indépendants comme polynômes en $X$ et en les éléments de $\mathbf{z}$. Pour $C$ un sous-graphe circuits de $\mathcal{G}(q)$, on a $\deg_X \mathcal{S}_{C}(q) > 0$ si et seulement si $\mathcal{G}'_1$ est un sous-circuit de $C$. On calcule alors les $\pr(F_{m_{i'}}^\bullet)(q)$.\par
Si $i' =i$, alors il existe un unique sous-graphe circuits dont $\mathcal{G}'_1$ est un sous-circuit qui vérifie les propriétés de bidegré : il s'agit de $\mathcal{G}=\sum_{j=1}^{i-1} \mathcal{G}(q_j) + \mathcal{G}'_1$. On a $\mathcal{S}_{\mathcal{G}}(q) \propto \left(\prod_{j=1}^{i-1} z_j^{p_j} \right) X^2$. Ainsi, si $\mathcal{H}$ est un autre sous-graphe circuits de $\mathcal{G}(q)$ vérifiant les égalités de bidegré, alors $\deg_X \mathcal{S}_{\mathcal{H}}(q)=0$. Par conséquent, $\pr(F_{m_i}^\bullet)(q)$ est de la forme $\left(\prod_{j=1}^{i-1} z_j\right) X^2 + P(\mathbf{z})$ avec $P(\mathbf{z})$ un polynôme en les éléments de $\mathbf{z}$. Comme précédemment, puisque $e_{v,w}^*$ est le seul terme dont le coefficient est de degré $\geq 1$ en $X$ et que $\pr(e_{v,w}) \propto \pr(e_{w^\gamma,v^\gamma})$ est un terme de $y_t$ lui-même un monôme de $f_{m'_i,t}$, le seul facteur de $\pr(F_{m_i}^\bullet)$ qui appliqué en $q$ est de degré $\geq 1$ en $X$ est le facteur $f_{m'_i,t}$. Ainsi on a bien $\deg_X f_{m'_{i},t}(q) > 0$ et $f_{m'_i,u}(q)\in \C^\times$ pour $u \neq t$.\par
Si $i' \neq i$, alors il n'existe pas de sous-graphe circuits contenant $\mathcal{G}'_1$ vérifiant les propriétés de bidegré. En effet, si un tel sous-graphe circuits $C$ existe, $C$ ne peut pas contenir les graphes $\mathcal{G}'_2$ et $\mathcal{G}'_3$ (par incompatibilité avec $\mathcal{G}'_1$). Comme $\mathcal{G}(q_j)$ est équivalent à $\mathcal{G}(q'_j)$ pour $j \in \llbracket 1, \imax \rrbracket \setminus \{i,i+1\}$ et $\mathcal{G}'_1$ est équivalent à $\mathcal{G}(q'_i)$, le cheminement $C$ est donc équivalent à un sous-graphe circuits $C'$ de $\mathcal{G}(q')$, qui vérifie l'égalité de bidegré. D'après la preuve de la proposition \ref{pedagogique}, on a alors $C'= \sum_{j=1}^{i'} \mathcal{G}(q'_j)$. On distingue alors deux cas :
\begin{itemize}
\item si $i'<i$, alors $\mathcal{G}(q'_i)$ n'est pas un sous-cheminement de $C'$, donc $\mathcal{G}'_1$ n'est pas un sous-cheminement de $C$, ce qui est absurde,
\item si $i' > i$, alors $\mathcal{G}(q'_{i+1})$ est un sous-cheminement de $C'$, ce qui est absurde car il n'y a pas de circuit équivalent à $\mathcal{G}(q'_{i+1})$ dans $\mathcal{G}(q)$.
\end{itemize}
Ainsi, si $\mathcal{H}$ est un sous-graphe circuits de $\mathcal{G}(q)$ vérifiant les égalités de bidegré, alors $\mathcal{S}_{\mathcal{H}}(q)$ est un monôme uniquement en les éléments de $\mathbf{z}$. Comme les $\mathcal{S}_{C}(q)$ pour $C$ un sous-graphe circuits de $\mathcal{G}(q)$ sont linéairement indépendants par construction, la quantité $\pr(F_{m_{i'}}^\bullet)(q)$ est un polynôme $P_{i'}(\mathbf{z})$ non nul.\par

Pour conclure quant à l'hypothèse (III'), on évalue alors $(z_1, \ldots, z_{i-1}, z'_2, z'_3, z_{i+2}, \ldots, z_{\imax})$ en $(b_1, \ldots, b_{i-1}, b'_{2}, b'_{3}, b_{i+2}, \ldots, b_{\imax}) \in \C^{\imax}$ qui n'appartient
\begin{itemize}
\item ni à $\V(z_j)$ pour tout $1 \leq j \leq i-1$ (pour que $\pr(F_{m_i}^\bullet)(q) \neq 0$), 
\item ni à $\V(P_{j})$ pour tout $j \in \llbracket 1, \imax \rrbracket \setminus \{i\}$ (pour que $\pr(F_{m_j}^\bullet)(q) \neq 0$ pour $j \neq i$),
\end{itemize}
où $\V(Q) \subset \C^{\imax}$ est le lieu d'annulation de $Q \in \C[\mathbf{z}]$.\par
\begin{rema}
Récapitulons les points cruciaux de la modification : on modifie $q'_i$ et $q'_{i+1}$ dans $q'$ de sorte que dans le graphe $\mathcal{G}(q)$, on a trois nouveaux sous-circuits $\mathcal{G}'_1$, $\mathcal{G}'_2$, $\mathcal{G}'_3$ tels que
\begin{itemize}
\item $\supp \mathcal{G}'_1 \cup \supp \mathcal{G}'_2 \cup \supp \mathcal{G}'_3 = \supp \mathcal{G}(q'_i) \cup \supp \mathcal{G}(q'_{i+1})$,
\item $\mathcal{G}'_1$ est équivalent à $\mathcal{G}(q'_i)$ et tous les poids des arêtes de $\mathcal{G}'_1$ sont égaux à ceux des arêtes de $\mathcal{G}(q'_i)$, excepté ceux des arêtes $v \rightarrow w$ et $w^\gamma \rightarrow v^\gamma$ (qui peuvent être la même arête) qui sont multipliés par $X$,
\item les circuits $\mathcal{G}'_2, \mathcal{G}'_3$
\begin{itemize}
\item ne sont pas équivalents à $\mathcal{G}(q'_{i+1})$ dans $\mathcal{G}(q)$,
\item sont incompatibles avec $\mathcal{G}'_1$,
\item ont leur somme ayant le même bidegré que $\mathcal{G}(q'_{i}) + \mathcal{G}(q'_{i+1})$,
\end{itemize}
\item chacun des trois circuits $\mathcal{G}'_1$, $\mathcal{G}'_2$, $\mathcal{G}'_3$ a des arêtes propres, et les arêtes $v \rightarrow w$ et $w^\gamma \rightarrow v^\gamma$ sont propres pour $\mathcal{G}'_1$.
\end{itemize}
et ces propriétés suffisent à conclure pour l'hypothèse (III'). Il suffira de montrer ces propriétés dans les autres cas. \label{recap}
\end{rema}
\textbf{Cas (3.a.ii) : Si $y \neq x^\gamma$,} alors on remplace $p_{v,w} e_{v,w}^*+p_{x,y} e_{x,y}^*$ par $p_{v,w} Xe_{v,w}^*+ e_{v,y}^*+ e_{x,w}^*$.
\begin{exem}
Un exemple de ce cas dans l'exemple \ref{exempleC} est le cas où $\pr(e_{v,w})=\pr(e_{5,3})$ et $\pr(e_{x,y})=\pr(e_{4,2})$.
\end{exem}
On passe alors du graphe $\mathcal{G}(q')$ au graphe $\mathcal{G}(q)$ en effectuant les modifications suivantes :
\begin{center}
\begin{tikzpicture}[scale=0.9]
%\draw (1.5,1) node{$i$};
%\draw (1.5,-1) node{$i+1$};
%\node (v) at (2.5,.5) {$v$};
%\node (w) at (4.5,.5) {$v^\gamma$};
%\draw[->] (v)--(w) node[midway,above] {$c X$};
%\draw[->, dashed] (4.5,.7) arc (0:180:1);
%\node (x) at (2.5,-.5) {$x$};
%\node (y) at (4.5,-.5) {$x^\gamma$};
%\draw[->] (v)--(y);
%\draw[->] (x)--(w);
%\draw[->, dashed] (4.5,-.7) arc (0:-180:1);
\draw (-4.5,1) node{$i$};
\draw (-4.5,-1) node{$i+1$};
\node (v) at (-3.5,.5) {$v$};
\node (w) at (-1.5,.5) {$w$};
\node (wg) at (.5,.5) {$w^\gamma$};
\node (vg) at (2.5,.5) {$v^\gamma$};
\draw[o->, dashed, >=stealth, ultra thick] (w)--(wg);
\draw[->, >=stealth, ultra thick] (wg)--(vg);
\draw[->, >=stealth, ultra thick] (v)--(w);
\draw[->, dashed, >=stealth, ultra thick] (2.5,.7) arc (30:150:3.464);
\node (x) at (-3.5,-.5) {$x$};
\node (y) at (-1.5,-.5) {$y$};
\node (yg) at (.5,-.5) {$y^\gamma$};
\node (xg) at (2.5,-.5) {$x^\gamma$};
\draw[->, >=stealth, ultra thick] (x)--(y);
\draw[o->, dashed, >=stealth, ultra thick] (y)--(yg);
\draw[->, >=stealth, ultra thick] (yg)--(xg);
\draw[->, dashed, >=stealth, ultra thick] (2.5,-.7) arc (-30:-150:3.464);

\draw (3.5,0) node{\huge $\Rightarrow$};

\draw (4.5,1) node{$i$};
\draw (4.5,-1) node{$i+1$};
\node (v) at (5.5,.5) {$v$};
\node (w) at (7.5,.5) {$w$};
\node (wg) at (9.5,.5) {$w^\gamma$};
\node (vg) at (11.5,.5) {$v^\gamma$};
\draw[o->, dashed, >=stealth, ultra thick] (w)--(wg);
\draw[->, >=stealth, ultra thick] (wg)--(vg) node[midway,above] {$(X)$};
\draw[->, >=stealth, ultra thick] (v)--(w) node[midway,above] {$(X)$};
\draw[->, dashed, >=stealth, ultra thick] (11.5,.7) arc (30:150:3.464);
\node (x) at (5.5,-.5) {$x$};
\node (y) at (7.5,-.5) {$y$};
\node (yg) at (9.5,-.5) {$y^\gamma$};
\node (xg) at (11.5,-.5) {$x^\gamma$};
%\draw[->] (x)--(y);
\draw[o->, dashed, >=stealth, ultra thick] (y)--(yg);
\draw[->, >=stealth, ultra thick] (x)--(w);
\draw[->, >=stealth, ultra thick] (wg)--(xg);
\draw[->, >=stealth, ultra thick] (v)--(y);
\draw[->, >=stealth, ultra thick] (yg)--(vg);
%\draw[->] (9.5,.7) arc (30:150:1.155) node[midway, above] {$(z''_{i+1})$};
\draw[->, dashed, >=stealth, ultra thick] (11.5,-.7) arc (-30:-150:3.464); 
\end{tikzpicture}
\end{center}
On peut vérifier que les circuits $\mathcal{G}'_1, \mathcal{G}'_2, \mathcal{G}'_3$ respectivement définis par
\begin{center}
\begin{tikzpicture}[scale=0.6]
%\draw (1.5,1) node{$i$};
%\draw (1.5,-1) node{$i+1$};
%\node (v) at (2.5,.5) {$v$};
%\node (w) at (4.5,.5) {$v^\gamma$};
%\draw[->] (v)--(w) node[midway,above] {$c X$};
%\draw[->, dashed] (4.5,.7) arc (0:180:1);
%\node (x) at (2.5,-.5) {$x$};
%\node (y) at (4.5,-.5) {$x^\gamma$};
%\draw[->] (v)--(y);
%\draw[->] (x)--(w);
%\draw[->, dashed] (4.5,-.7) arc (0:-180:1);
%\draw (-4.5,1) node{$i$};
%\draw (-4.5,-1) node{$i+1$};
%\node (v) at (-3.5,.5) {$v$};
%\node (w) at (-1.5,.5) {$w$};
%\node (wg) at (.5,.5) {$w^\gamma$};
%\node (vg) at (2.5,.5) {$v^\gamma$};
%\draw[->, dashed] (w)--(wg);
%\draw[->] (wg)--(vg) node[midway,above] {$\epsilon c$};
%\draw[->] (v)--(w) node[midway,above] {$c$};
%\draw[->, dashed] (2.5,.7) arc (30:150:3.464);
%\node (x) at (-3.5,-.5) {$x$};
%\node (y) at (-1.5,-.5) {$y$};
%\node (yg) at (.5,-.5) {$y^\gamma$};
%\node (xg) at (2.5,-.5) {$x^\gamma$};
%\draw[->] (x)--(y);
%\draw[->, dashed] (y)--(yg);
%\draw[->] (yg)--(xg);
%\draw[->, dashed] (2.5,-.7) arc (-30:-150:3.464);
%
%\draw (-.5,-3) node{$\Downarrow$};

\draw (-13.5,-5) node{$i$};
\draw (-13.5,-7) node{$i+1$};
\node (v) at (-12.5,-5.5) {$v$};
\node (w) at (-10.5,-5.5) {$w$};
\node (wg) at (-8.5,-5.5) {$w^\gamma$};
\node (vg) at (-6.5,-5.5) {$v^\gamma$};
\draw[o->, dashed, >=stealth, ultra thick] (w)--(wg);
\draw[->, >=stealth, ultra thick] (wg)--(vg) node[midway,above] {$(X)$};
\draw[->, >=stealth, ultra thick] (v)--(w) node[midway,above] {$(X)$};
\draw[->, dashed, >=stealth, ultra thick] (-6.5,-5.3) arc (30:150:3.464);
\draw (-13.5,-9) node[white]{0};
%\node (x) at (-3.5,-6.5) {$x$};
%\node (y) at (-1.5,-6.5) {$y$};
%\node (yg) at (.5,-6.5) {$y^\gamma$};
%\node (xg) at (2.5,-6.5) {$x^\gamma$};
%\draw[->] (x)--(y);
%\draw[->, dashed] (y)--(yg);
%\draw[->] (x)--(w);
%\draw[->] (wg)--(xg);
%\draw[->] (v)--(y);
%\draw[->] (yg)--(vg);
%\draw[->] (.5,-5.3) arc (30:150:1.155); %node[midway, above] {$z''_{i+1}$};
%\draw[->, dashed] (2.5,-6.7) arc (-30:-150:3.464);

\draw (-4.5,-5) node{$i$};
\draw (-4.5,-7) node{$i+1$};
\node (v) at (-3.5,-5.5) {$v$};
%\node (w) at (-1.5,-5.5) {$w$};
%\node (wg) at (.5,-5.5) {$w^\gamma$};
\node (vg) at (2.5,-5.5) {$v^\gamma$};
%\draw[->, dashed] (w)--(wg);
%\draw[->] (wg)--(vg) node[midway,above] {$\epsilon c X$};
%\draw[->] (v)--(w) node[midway,above] {$c X$};
\draw[->, dashed, >=stealth, ultra thick] (2.5,-5.3) arc (30:150:3.464);
%\node (x) at (-3.5,-6.5) {$x$};
\node (y) at (-1.5,-6.5) {$y$};
\node (yg) at (.5,-6.5) {$y^\gamma$};
%\node (xg) at (2.5,-6.5) {$x^\gamma$};
%\draw[->] (x)--(y);
\draw[o->, dashed, >=stealth, ultra thick] (y)--(yg);
%\draw[->] (x)--(w);
%\draw[->] (wg)--(xg);
\draw[->, >=stealth, ultra thick] (v)--(y);
\draw[->, >=stealth, ultra thick] (yg)--(vg);
\draw (-4.5,-9) node[white]{0};
%\draw[->] (.5,-5.3) arc (30:150:1.155); %node[midway, above] {$z''_{i+1}$};
%\draw[->, dashed] (2.5,-6.7) arc (-30:-150:3.464);

\draw (4.5,-5) node{$i$};
%\draw (3.5,-5) node[white]{0};
\draw (4.5,-7) node{$i+1$};
%\node (v) at (-3.5,-5.5) {$v$};
\node (w) at (7.5,-5.5) {$w$};
\node (wg) at (9.5,-5.5) {$w^\gamma$};
%\node (vg) at (2.5,-5.5) {$v^\gamma$};
\draw[o->, dashed, >=stealth, ultra thick] (w)--(wg);
%\draw[->] (wg)--(vg) node[midway,above] {$\epsilon c X$};
%\draw[->] (v)--(w) node[midway,above] {$c X$};
%\draw[->, dashed] (2.5,-5.3) arc (30:150:3.464);
\node (x) at (5.5,-6.5) {$x$};
%\node (y) at (-1.5,-6.5) {$y$};
%\node (yg) at (.5,-6.5) {$y^\gamma$};
\node (xg) at (11.5,-6.5) {$x^\gamma$};
%\draw[->] (x)--(y);
%\draw[->, dashed] (y)--(yg);
\draw[->, >=stealth, ultra thick] (x)--(w);
\draw[->, >=stealth, ultra thick] (wg)--(xg);
%\draw[->] (v)--(y);
%\draw[->] (yg)--(vg);
%\draw[->] (8.5,-5.3) arc (30:150:1.155); %node[midway, above] {$z''_{i+1}$};
\draw[->, dashed, >=stealth, ultra thick] (11.5,-6.7) arc (-30:-150:3.464);
\end{tikzpicture}
\end{center}
vérifient les propriétés de la remarque \ref{recap}.
\subparagraph{\textbf{Cas (3.b) : Dans le cas où on n'a $y \preceq w \prec v \preceq x$ pour aucun facteur $\pr(e_{x,y})$ de $\mathcal{T}_{i+1}$,}} montrons que pour tout $y \in I(i+1)$, on a $v \preceq y$. On rappelle que $v,w \in I(i)^-$ \par
Soit $y \in I(i+1)$. On reprend les notations \ref{nota}.
\begin{itemize}
\item Premièrement, on ne peut pas avoir $w \prec y \prec v$. En effet, si l'on reprend la notation \ref{KetI}, on a $K(i+1) \subset K(i)$ et, par la propriété \ref{elemSi}, $K(i) \cap \llbracket k(w) +1, k(v) -1 \rrbracket = \emptyset$. 
\item Il reste donc à vérifier que l'on ne peut pas avoir $y \preceq w$. S'il existe $y \in I(i+1)$ tel que $y \preceq w$, alors il existe $z \in I(i+1)$ tel que $y \preceq w \prec v \preceq z$ (il suffit de prendre $z=y^\gamma$, on a $z \geq n'+1$ et $y \leq n'$). Quitte alors à augmenter $y$ et diminuer $z$, grâce au premier point, on peut supposer que $y$ et $z$ sont deux termes consécutifs pour l'ordre croissant de $I(i+1)$, et donc que $\pr(e_{z,y})$ est un facteur de $\mathcal{T}_{i+1}$, ce qui est absurde par hypothèse.
\end{itemize}
Considérons donc, d'après la propriété \ref{elemSi}, le facteur de la forme $e_{x,x^\gamma}$ de $\mathcal{T}_{i+1}$ avec $x^\gamma > x$ (c'est-à-dire le terme sur la partie haute de l'antidiagonale). On a donc $w \preceq v \preceq x \preceq x^\gamma$. On remplace alors $p_{v,w} e_{v,w}^*+p_{x,x^\gamma} e_{x,x^\gamma}^*$ par $p_{v,w} X e_{v,w}^*+e_{x,w}^*+e_{v,v^\gamma}^*$.
\begin{exem}
Un exemple de ce cas dans l'exemple \ref{exempleC} est le cas où $\pr(e_{v,w})=\pr(e_{3,1})$ et $\pr(e_{x,x^\gamma})=e_{2,11}$.
\end{exem}
On passe de $\mathcal{G}(q')$ à $\mathcal{G}(q)$ en effectuant les modifications suivantes :
\begin{center}
\begin{tikzpicture}[scale=0.9]
%\draw (1.5,1) node{$i$};
%\draw (1.5,-1) node{$i+1$};
%\node (v) at (2.5,.5) {$v$};
%\node (w) at (4.5,.5) {$v^\gamma$};
%\draw[->] (v)--(w) node[midway,above] {$c X$};
%\draw[->, dashed] (4.5,.7) arc (0:180:1);
%\node (x) at (2.5,-.5) {$x$};
%\node (y) at (4.5,-.5) {$x^\gamma$};
%\draw[->] (v)--(y);
%\draw[->] (x)--(w);
%\draw[->, dashed] (4.5,-.7) arc (0:-180:1);
\draw (-4.5,1) node{$i$};
\draw (-4.5,-1) node{$i+1$};
\node (v) at (-3.5,.5) {$v$};
\node (w) at (-1.5,.5) {$w$};
\node (wg) at (.5,.5) {$w^\gamma$};
\node (vg) at (2.5,.5) {$v^\gamma$};
\draw[o->, dashed, >=stealth, ultra thick] (w)--(wg);
\draw[->, >=stealth, ultra thick] (wg)--(vg);
\draw[->, >=stealth, ultra thick] (v)--(w);
\draw[->, dashed, >=stealth, ultra thick] (2.5,.7) arc (30:150:3.464);
\node (x) at (-1.5,-.5) {$x$};
\node (y) at (.5,-.5) {$x^\gamma$};
\draw[o->, >=stealth, ultra thick] (x)--(y);
\draw[->, dashed, >=stealth, ultra thick] (.5,-.7) arc (-30:-150:1.155);

\draw (3.5,0) node{\huge $\Rightarrow$};

\draw (4.5,1) node{$i$};
\draw (4.5,-1) node{$i+1$};
\node (v) at (5.5,.5) {$v$};
\node (w) at (7.5,.5) {$w$};
\node (wg) at (9.5,.5) {$w^\gamma$};
\node (vg) at (11.5,.5) {$v^\gamma$};
\draw[o->, dashed, >=stealth, ultra thick] (w)--(wg);
\draw[->, >=stealth, ultra thick] (wg)--(vg) node[midway,above] {$(X)$};
\draw[->, >=stealth, ultra thick] (v)--(w) node[midway,above] {$(X)$};
\draw[->, dashed, >=stealth, ultra thick] (11.5,.7) arc (30:150:3.464);
\node (x) at (7.5,-.5) {$x$};
\node (y) at (9.5,-.5) {$x^\gamma$};
%\draw[->] (x)--(y);
\draw[->, >=stealth, ultra thick] (x)--(w);
\draw[->, >=stealth, ultra thick] (wg)--(y);
\draw[o->, >=stealth, ultra thick] (5.5,.3) arc (-150:-30:3.464);
\draw[->, dashed, >=stealth, ultra thick] (9.5,-.7) arc (-30:-150:1.155);
\end{tikzpicture}
\end{center}
On peut encore vérifier que les circuits $\mathcal{G}'_1, \mathcal{G}'_2, \mathcal{G}'_3$ respectivement définis par
\begin{center}
\begin{tikzpicture}[scale=0.7]
%\draw (1.5,1) node{$i$};
%\draw (1.5,-1) node{$i+1$};
%\node (v) at (2.5,.5) {$v$};
%\node (w) at (4.5,.5) {$v^\gamma$};
%\draw[->] (v)--(w) node[midway,above] {$c X$};
%\draw[->, dashed] (4.5,.7) arc (0:180:1);
%\node (x) at (2.5,-.5) {$x$};
%\node (y) at (4.5,-.5) {$x^\gamma$};
%\draw[->] (v)--(y);
%\draw[->] (x)--(w);
%\draw[->, dashed] (4.5,-.7) arc (0:-180:1);
%\draw (-4.5,1) node{$i$};
%\draw (-4.5,-1) node{$i+1$};
%\node (v) at (-3.5,.5) {$v$};
%\node (w) at (-1.5,.5) {$w$};
%\node (wg) at (.5,.5) {$w^\gamma$};
%\node (vg) at (2.5,.5) {$v^\gamma$};
%\draw[->, dashed] (w)--(wg);
%\draw[->] (wg)--(vg);
%\draw[->] (v)--(w);
%\draw[->, dashed] (2.5,.7) arc (30:150:3.464);
%\node (x) at (-1.5,-.5) {$x$};
%\node (y) at (.5,-.5) {$x^\gamma$};
%\draw[->] (x)--(y);
%\draw[->, dashed] (.5,-.7) arc (-30:-150:1.155);
%
%\draw (3.5,0) node{\huge $\Rightarrow$};

\draw (4.5,1) node{$i$};
\draw (4.5,-1) node{$i+1$};
\node (v) at (5.5,.5) {$v$};
\node (w) at (7.5,.5) {$w$};
\node (wg) at (9.5,.5) {$w^\gamma$};
\node (vg) at (11.5,.5) {$v^\gamma$};
\draw[o->, dashed, >=stealth, ultra thick] (w)--(wg);
\draw[->, >=stealth, ultra thick] (wg)--(vg) node[midway,above] {$(X)$};
\draw[->, >=stealth, ultra thick] (v)--(w) node[midway,above] {$(X)$};
\draw[->, dashed, >=stealth, ultra thick] (11.5,.7) arc (30:150:3.464);
%\node (x) at (7.5,-.5) {$x$};
%\node (y) at (9.5,-.5) {$x^\gamma$};
%\draw[->] (x)--(y);
%\draw[->] (x)--(w);
%\draw[->] (wg)--(y);
%\draw[->] (5.5,.3) arc (-150:-30:3.464) node[midway, below] {$(z''_{i+1})$};
%\draw[->, dashed] (9.5,-.7) arc (-30:-150:1.155);
%\draw (1.5,1) node{$i$};
%\draw (1.5,-1) node{$i+1$};
%\node (v) at (2.5,.5) {$v$};
%\node (w) at (4.5,.5) {$v^\gamma$};
%\draw[->] (v)--(w) node[midway,above] {$c X$};
%\draw[->, dashed] (4.5,.7) arc (0:180:1);
%\node (x) at (2.5,-.5) {$x$};
%\node (y) at (4.5,-.5) {$x^\gamma$};
%\draw[->] (v)--(y);
%\draw[->] (x)--(w);
%\draw[->, dashed] (4.5,-.7) arc (0:-180:1);
%\draw (-4.5,1) node{$i$};
%\draw (-4.5,-1) node{$i+1$};
%\node (v) at (-3.5,.5) {$v$};
%\node (w) at (-1.5,.5) {$w$};
%\node (wg) at (.5,.5) {$w^\gamma$};
%\node (vg) at (2.5,.5) {$v^\gamma$};
%\draw[->, dashed] (w)--(wg);
%\draw[->] (wg)--(vg);
%\draw[->] (v)--(w);
%\draw[->, dashed] (2.5,.7) arc (30:150:3.464);
%\node (x) at (-1.5,-.5) {$x$};
%\node (y) at (.5,-.5) {$x^\gamma$};
%\draw[->] (x)--(y);
%\draw[->, dashed] (.5,-.7) arc (-30:-150:1.155);
%
%\draw (3.5,0) node{\huge $\Rightarrow$};

\draw (13.5,1) node{$i$};
\draw (13.5,-1) node{$i+1$};
\node (v) at (14.5,.5) {$v$};
%\node (w) at (7.5,.5) {$w$};
%\node (wg) at (9.5,.5) {$w^\gamma$};
\node (vg) at (20.5,.5) {$v^\gamma$};
%\draw[->, dashed] (w)--(wg);
%\draw[->] (wg)--(vg) node[midway,above] {$(X)$};
%\draw[->] (v)--(w) node[midway,above] {$(X)$};
\draw[->, dashed, >=stealth, ultra thick] (20.5,.7) arc (30:150:3.464);
%\node (x) at (7.5,-.5) {$x$};
%\node (y) at (9.5,-.5) {$x^\gamma$};
%\draw[->] (x)--(y);
%\draw[->] (x)--(w);
%\draw[->] (wg)--(y);
\draw[o->, >=stealth, ultra thick] (14.5,.3) arc (-150:-30:3.464);
%\draw[->, dashed] (9.5,-.7) arc (-30:-150:1.155);
%\draw (1.5,1) node{$i$};
%\draw (1.5,-1) node{$i+1$};
%\node (v) at (2.5,.5) {$v$};
%\node (w) at (4.5,.5) {$v^\gamma$};
%\draw[->] (v)--(w) node[midway,above] {$c X$};
%\draw[->, dashed] (4.5,.7) arc (0:180:1);
%\node (x) at (2.5,-.5) {$x$};
%\node (y) at (4.5,-.5) {$x^\gamma$};
%\draw[->] (v)--(y);
%\draw[->] (x)--(w);
%\draw[->, dashed] (4.5,-.7) arc (0:-180:1);
%\draw (-4.5,1) node{$i$};
%\draw (-4.5,-1) node{$i+1$};
%\node (v) at (-3.5,.5) {$v$};
%\node (w) at (-1.5,.5) {$w$};
%\node (wg) at (.5,.5) {$w^\gamma$};
%\node (vg) at (2.5,.5) {$v^\gamma$};
%\draw[->, dashed] (w)--(wg);
%\draw[->] (wg)--(vg);
%\draw[->] (v)--(w);
%\draw[->, dashed] (2.5,.7) arc (30:150:3.464);
%\node (x) at (-1.5,-.5) {$x$};
%\node (y) at (.5,-.5) {$x^\gamma$};
%\draw[->] (x)--(y);
%\draw[->, dashed] (.5,-.7) arc (-30:-150:1.155);
%
%\draw (3.5,0) node{\huge $\Rightarrow$};

\draw (22.5,1) node{$i$};
\draw (22.5,-1) node{$i+1$};
%\node (v) at (5.5,.5) {$v$};
\node (w) at (23.5,.5) {$w$};
\node (wg) at (25.5,.5) {$w^\gamma$};
%\node (vg) at (11.5,.5) {$v^\gamma$};
\draw[o->, dashed, >=stealth, ultra thick] (w)--(wg);
%\draw[->] (wg)--(vg) node[midway,above] {$(X)$};
%\draw[->] (v)--(w) node[midway,above] {$(X)$};
%\draw[->, dashed] (11.5,.7) arc (30:150:3.464);
\node (x) at (23.5,-.5) {$x$};
\node (y) at (25.5,-.5) {$x^\gamma$};
%\draw[->] (x)--(y);
\draw[->, >=stealth, ultra thick] (x)--(w);
\draw[->, >=stealth, ultra thick] (wg)--(y);
%\draw[->] (5.5,.3) arc (-150:-30:3.464) node[midway, below] {$(z''_{i+1})$};
\draw[->, dashed, >=stealth, ultra thick] (25.5,-.7) arc (-30:-150:1.155);
\end{tikzpicture}
\end{center}
vérifient les propriétés de la remarque \ref{recap}.
\subsubsection{Hypothèse (II)}
Soit $m'=m'_i \in \mathbf{M}'$. On pose $r'=r'_{m'}$, $m=2m' \in \mathbf{M}$ et $r=2r'$. On rappelle que pour tout $t \in \llbracket 1,r'+1 \rrbracket$, on a
$$\lambda^C_{m',t}=\left\{ \begin{array}{ll} w^C_{k_t}-w^C_{k_{t+1}} &  \text{si } t<r' \\ 2w^C_{k_{r'}} & \text{si } t=r' \\ -2w^C_{k_1} & \text{si } t=r'+1 \end{array} \right. $$
avec $w_k^C=\varpi^C_{\iota_{k-1}}-\varpi^C_{\iota_{k}}$ et $\varpi_l^C \in (\syp_n)^*$ les poids fondamentaux. Comme en type $\gl_n$, si $\Lambda^C$ est l'ensemble des poids de $\Sy(\gotq^C)$, alors $\Lambda^C \subset \bigoplus_{\ell=1}^{n'} \rat \varpi_\ell^C$. ainsi tous les $\lambda^C_{m',t}$, $t<r'$ sont indivisibles dans $\Lambda^C$, de sorte que leurs semi-invariants associés $f_{m',t}$ sont bien indivisibles dans $\Sy(\gotq^C)$. Restent les cas où $t=r'$ et $t=r'+1$, qui correspondent à $f_{m',r'}=\pr(F_{m,r'})$ et $f_{m',r'+1}=\pr(F_{m,r})$.\par
Notons $\overline{D}$ le sous-espace vectoriel de $\gl_n$ des matrices antidiagonales, engendré par les $e_{v,v^\gamma}$. On a en fait $\overline{D} \subset \syp_n$. Pour tout $e_{v,v^\gamma} \in \overline{D}$, l'invariant $F_m$ est de degré au plus $1$ en $e_{v,v^\gamma}$, donc $F_m^\bullet$ aussi (c'est une somme de certains monômes de $F_m$). Or pour tout $v,w$, on a $e_{v,w} \in \overline{D} \Leftrightarrow \pr(e_{v,w}) \in \overline{D}$. Ainsi pour tout $e_{v,v^\gamma} \in \overline{D}$, la projection $\pr(F_m^\bullet)$ est de degré au plus $1$ en $e_{v,v^\gamma}$, et donc les $f_{m',t}$ aussi (ce sont des facteurs de $\pr(F_m^\bullet)$).
\begin{prop}
Il existe $e_{v,v^\gamma} \in \overline{D}$ tel que le semi-invariant $f_{m',t}$ est de degré $1$ en $e_{v,v^\gamma}$ si et seulement si $t \in \{r',r'+1\}$.
\end{prop}
\begin{proof}
On rappelle que $F_{m,t} \in \Sym(\g_{I^{(t)},I^{[t]}})$. Si $F_{m,t}$ est de degré $1$ en des éléments de $\overline{D}$, alors $\overline{D} \cap \g_{I^{(t)},I^{[t]}} \neq \emptyset$. Mais alors $\left(I^{(t)}\right)^\gamma \cap I^{[t]} \neq \emptyset$. Or si $t < r$, on a $\left(I^{(t)}\right)^\gamma = I^{[r-t]}$, donc $\left(I^{(t)}\right)^\gamma \cap I^{[t]} = I^{[r-t]}\cap I^{[t]}$.  Comme les $I^{[t]}$ sont disjoints, l'ensemble $I^{[r-t]}\cap I^{[t]}$ est non vide si et seulement si $t=r/2=r'$. Si $t=r$, on a $\left(I^{(r)}\right)^\gamma = I^{[r]}$, donc $\left(I^{(t)}\right)^\gamma \cap I^{[t]} \neq \emptyset$. Ainsi $t \in \{r',r\}$. Autrement dit, avec la notation \ref{notag}, si $f_{m',t}$ est de degré $1$ en des éléments de $\overline{D}$, alors $t \in \{r',r'+1\}$.\par
Prenons $\mathcal{T}_i$, qui est un facteur d'un monôme de $\pr(F_m^\bullet)$ si $m=m_i$ (proposition \ref{pedagogique}). D'après la propriété \ref{elemSi}, il existe exactement deux facteurs irréductibles de $\mathcal{T}_i$ qui sont sur l'antidiagonale, dont un dans $\gotp^C$ et l'un dans $(\gotn^-)^C$. D'après ce qui précède et le lemme \ref{monome}, ils apparaissent donc dans un monôme de $f_{m',r'}$ ou $f_{m',r'+1}$. Pour conclure, on va montrer que l'un des facteurs apparaît dans $f_{m',r'}$ et l'autre dans $f_{m',r'+1}$. Le facteur dans $\gotp^C$ apparaît nécessairement dans $f_{m',r'+1}$ car $f_{m',r'}=\pr(F_{m,r'}) \in \Sym((\gotn^-)^C)$. Par la propriété \ref{elemSi}, le facteur dans $(\gotn^-)^C$ est du type $\pr(e_{v,w})$ (avec $w=v^\gamma$) où $w,v$ sont deux termes consécutifs pour l'ordre croissant de $I(i)$. Autrement dit, on a $k(w)<k(v)$ avec $k(w), k(v) \in K(i)$ et $\llbracket k(w)+1,k(v)-1 \rrbracket \cap K(i) = \emptyset$, c'est-à-dire $|I_k| <i$ pour tout $k \in \llbracket k(w)+1,k(v)-1 \rrbracket$. Or, comme $i \in \mathbf{I}$, il existe $c \in \kappa_i$, c'est-à-dire tel que $|I_c|=i$. Comme $k(w) \leq s' <s'+1 \leq k(v)$, par symétrie, il existe donc $d,d'$ tels que $d\leq k(w) < k(v) \leq d'$ et $d,d' \in \kappa_i$. En particulier, on a $k_1 < k(v) \leq k_r$ (où l'on rappelle que $\kappa_i=\{k_1 <  \ldots < k_r\}$ en omettant les indices $i$).\par
Si $e_{v,w}$ apparaissait dans $f_{m',r'+1}$, puisque $f_{m',r'+1} \in \Sym(\g_{I^{(r)},I^{[r]}})$, on aurait $v \in I^{(r)}=I_{\llbracket 1,k_1 \rrbracket \sqcup \llbracket k_r+1,s \rrbracket}$, donc $k(v) \in \llbracket 1,k_1 \rrbracket \sqcup \llbracket k_r+1,s \rrbracket$, ce qui est contradiction avec les inégalités précédentes.
\end{proof}
Ainsi les semi-invariants $f_{m,r'}$ et $f_{m,r'+1}$ sont de degré $1$ en des éléments de l'antidiagonale (qui sont des vecteurs de la base canonique de $\syp_n$ définie dans la partie \ref{defC}), donc en particulier sont indivisibles par additivité des degrés partiels.\par
Les hypothèses (I), (II) et (III) du théorème \ref{inter} sont donc vérifiées. On obtient :
\begin{theo}
Soit $\gotq^C$ une contraction parabolique standard de $\syp_{2n'}=\syp_{n}$ telle que $\alpha_{n'}^C \notin (\pi^C)'$. Soit $\gotq$ une contraction parabolique standard symétrique de $\gl_n$ au-dessus de $\gotq^C$, associée à un facteur de Levi $\gotl= \g_{I_1} \times \ldots \times \g_{I_s}$. Alors
\begin{itemize}
\item la troncation canonique $\gotq_\Lambda^C$ de $\gotq^C$ est $\gotq_\Lambda^C=(\gotq^C)'$,
\item l'indice de $\gotq_\Lambda^C$ est $\ind \gotq_\Lambda^C=(n+s)/2$,
\item l'algèbre des semi-invariants $\Sy(\gotq^C)=\Y(\gotq^C_\Lambda)$ est polynomiale et librement engendrée par les semi-invariants $\pr(F_{2m',t})$, $m' \in \llbracket 1,n/2 \rrbracket$ et $t \in \llbracket 1, r'_{m'} +1 \rrbracket$ où les semi-invariants $F_{m,t}$ en type $\gl_n$ sont ceux de la contraction parabolique standard symétrique $\gotq$ au-dessus de $\gotq^C$.  En particulier, l'ensemble des poids $\Lambda(\gotq^C)$ de $\Sy(\gotq^C)$ est un groupe.
\end{itemize}
\end{theo}
%Grâce à ce lemme, on montre la propriété calculatoire suivante :
%\begin{prop}
%Soit $(q_{m,t})$ la précision de $(F_m)_m$ en $(F_{m,t})_{m,t}$ de l'équation \eqref{defprec} (section \ref{sec5}). Fixons $m$ et $t$, soit $j \in \llbracket 1,n \rrbracket$. On a vu au lemme \ref{calculatoire} qu'il existe un unique $J \subset \llbracket 1,n \rrbracket$ de cardinal $j$ et un unique $\sigma \in \mathfrak{S}_J$ tel que
%$$\left(\prod_{k \in J} e_{k,\sigma(k)}\right)(q_{m,t}) \neq 0.$$
%Alors pour tout $J' \subset \llbracket 1,n \rrbracket$ de cardinal $j$ et tout $\sigma' \in \mathfrak{S}_{J'}$, on a
%$$\pr\left(\prod_{k \in J'} e_{k,\sigma'(k)}\right)(\res(q_{m,t})) \neq 0$$
%si et seulement si
%\begin{itemize}
%\item soit $J'=J$ et $\sigma'=\sigma$,
%\item soit $J'=\gamma(J)$ et $\sigma'=\gamma \sigma^{-1}\gamma^{-1}$.
%\end{itemize}
%On a alors $\pr\left(\prod_{k \in J} e_{k,\sigma(k)}\right)=\pr\left(\prod_{k \in \gamma(J)} e_{k,\gamma \sigma^{-1}\gamma^{-1}(k)}\right)$.
%\end{prop}
%\xymatrix{\Sym(\gotq) \ar[d]_{\simeq} &\Sym(\gotq^C) \ar[d]^{\simeq} \\ \C[\gotq^*] \ar@{-->}[r] &\C[{\gotq^C}^*]\\ \gotq^* \ar@(ul,ur) &{\gotq^C}^* \ar@(ul,ur)}
\section{Un cas de non-polynomialité lorsque le facteur de Levi n'est pas de type $A$}
\label{sec10}
On considère la CPSS $\gotq$ (voir définition \ref{cpcs}) de $\gl_8$ définie par $\pi'=\{\alpha_2, \alpha_4, \alpha_6\}$, autrement dit, la contraction parabolique associée au diagramme suivant :
\begin{center}
\begin{tikzpicture}[scale=0.7]
\foreach \k in {1,2,...,7}
	{\draw[color=gray!20]  (0,\k)--(8,\k);
	\draw[color=gray!20] (\k,0)--(\k,8);}
\draw (0,0)--(8,0);
\draw (0,0)--(0,8);
\draw (0,8)--(8,8);
\draw (8,0)--(8,8);
\draw[ultra thick] (0,7)--(1,7)--(1,5)--(3,5)--(3,3)--(5,3)--(5,1)--(7,1)--(7,0);
\draw (2,2) node{\LARGE $\gotn^-$};
\draw (6,6) node{\LARGE $\gotp$};
\draw[dashed] (1,8)--(1,7)--(3,7)--(3,5)--(5,5)--(5,3)--(7,3)--(7,1)--(8,1);
\draw (0.5,7.5) node{\Large $\g_{I_1}$};
\draw (2,6) node{\Large $\g_{I_2}$};
\draw (4,4) node{\Large $\g_{I_3}$};
\draw (6,2) node{\Large $\g_{I_4}$};
\draw (7.5,0.5) node{\Large $\g_{I_5}$};
\end{tikzpicture}
\end{center}
et $\gotq^C=\pr(\gotq)$ la contraction de $\syp_8$ associée. On a alors $(\pi^C)'=\{\alpha_2^C, \alpha_4^C\}$. On va tenter de suivre un cheminement analogue au type $A$ et au type $C$ où $\alpha_{n'}^C \notin (\pi^C)'$. On verra qu'il faudra l'adapter au vu de quelques problèmes non rencontrés jusqu'ici.
\subsection{Semi-invariants}
Pour trouver les semi-invariants de $\gotq^C$ en type $C$, on a regardé la contraction parabolique $\gotq$ au-dessus de $\gotq^C$, factorisé les invariants en semi-invariants, puis projeté ces semi-invariants, en retirant les termes qui se télescopent du fait de la projection.\par
Si l'on s'intéresse à notre cas, on a $F_5^\bullet=F_{5,1}F_{5,2}$ et $F_8^\bullet=F_{8,1}F_{8,2}F_{8,3}$. Lorsque l'on veut projeter la première égalité, on obtient $\pr(F_5^\bullet)=0$ (d'après \cite[chapitre 8]{bou06}). On trouve bien $\pr(F_{5,1})=0$, mais $\pr(F_{5,2})=e_{1,8} \neq 0$. Autrement dit, en ne considérant que les projections des $F_m^\bullet$, on perd la donnée du semi-invariant $\pr(F_{5,2})$.\par
On a $\deg_{\gotn^-} F_8 =6$ et $\deg_{\gotn^-} F_4=3$ (corollaire \ref{coro}), de sorte que $F_8^\bullet$ et $(F_4^{\bullet})^2$ ont les mêmes bidegrés. Par un calcul, on trouve alors que
$ \pr(F_{5,2})$ divise $\pr(F_8^\bullet)-\frac{1}{4} \pr(F_4^\bullet)^2$. On rappelle que $f_{m'}=\pr(F_{2m'}^\bullet)$. On note alors $f_5:=\pr(F_8^\bullet)-\frac{1}{4} \pr(F_4^\bullet)^2 = f_4-\frac{1}{4} f_2^2$ et $f_5=f_{5,1}f_{5,2}$ où $f_{5,2}:=\pr(F_{5,2})$. Pour $f_4$, on a $\pr(F_8^\bullet) \propto \pr(F_{8,1}) \pr(F_{8,2}) \pr(F_{8,3})$. Posant $f_{4,1}:=\pr(F_{8,2})=\pr(F_{8,1})$ et $f_{4,2}:=\pr(F_{8,3})$, on a $f_4 \propto f_{4,1}^2 f_{4,2}$. Quitte à multiplier $f_{4,2}$, on peut et on va supposer que $f_4 = f_{4,1}^2 f_{4,2}$.
\begin{prop}
On a $f_4 = f_{4,1}^2f_{4,2}$, avec
$$f_{4,1}=\pr(\Delta_{\{4,5\},\{2,3\}}) \qquad \text{et} \qquad f_{4,2}=\pr(\Delta^\bullet_{\{1,2,3,8\},\{1,6,7,8\}})$$
On a également $f_5:=f_4-\frac{1}{4} f_2^2=f_{5,1}f_{5,2}$, de sorte que 
\begin{equation}
f_{5,1}f_{5,2}+\frac{1}{4} f_2^2 = f_{4,1}^2f_{4,2} \label{eqbiz}
\end{equation}
\label{calculceC}
\end{prop}

\subsection{Poids et troncation canonique}
\label{sec102}
\begin{prop}
Les poids de $f_{4,1}$ et $f_{4,2}$ sont les restrictions des poids de $F_{8,1}$ et $F_{8,3}$ à $\syp_8$ (propriété \ref{yes}) c'est-à-dire respectivement $\lambda_{4,1}^C= \varpi_1^C-\varpi_3^C$ et $\lambda_{4,2}^C=2(\varpi_3^C-\varpi_1^C)$. Le poids de $f_{5,2}=e_{1,8}$ est $2\varpi_1^C$ et donc le poids de $f_{5,1}$ est $-2\varpi_1^C$.
\label{pdsceC}
\end{prop}
Comme la famille $(f_1,f_2,f_3,f_4)$ est algébriquement indépendante \cite{py13}, la famille $(f_1, f_3, f_4, f_4-\frac{1}{4}f_2^2)$ est algébriquement indépendante, donc par le théorème \ref{alin} appliqué aux facteurs $f_1$, $f_3$, $f_{4,1}$, $f_{4,2}$, $f_{5,1}$, $f_{5,2}$ de $(f_1, f_3, f_4, f_4-\frac{1}{4}f_2^2)$, on obtient :
\begin{prop}
La famille $(f_1, f_3, f_{4,1},f_{4,2},f_{5,1},f_{5,2})$ est algébriquement indépendante.
\label{aiceC}
\end{prop}
Ainsi $\GK \Sy(\gotq^C) = \ind \gotq^C_\Lambda \geq 6 = \ind \gotq^C + 2$. Or par la propriété \ref{inclusion}, on a $\gotq^C_\Lambda \supset (\gotq^C)'$, avec $\dim (\gotq^C)'=\dim \gotq^C-2$. Alors
$$\ind \gotq^C_\Lambda + \dim \gotq^C_\Lambda \geq \ind \gotq^C+2 + \dim \gotq^C-2= \ind \gotq^C + \dim \gotq^C$$
et on a les égalités par le résultat de Ooms et van den Bergh (équation \eqref{ovdb})
\begin{prop}
On a $\gotq_\Lambda^C=(\gotq^C)'$ et $\ind \gotq_\Lambda^C=6$. Ainsi $\GK \Sy(\gotq^C)=6$ et donc par exemple, la famille $\{f_1, f_3, f_{4,1},f_{4,2},f_{5,1},f_{5,2}\}$ est une base de transcendance de $\Sy(\gotq^C)$.
\label{indceC}
\end{prop}
Par \eqref{eqbiz}, la famille $f_1, f_2, f_3, f_{4,1}, f_{4,2}, f_{5,1}, f_{5,2}$ n'est pas algébriquement indépendante et la famille $f_1, f_3, f_{4,1}, f_{4,2}, f_{5,1}, f_{5,2}$ est bien algébriquement indépendante, mais a priori ne semble pas engendrer l'algèbre (par exemple, elle ne semble pas engendrer $f_2$).

\subsection{L'algèbre $\C[f_1, f_2, f_3, f_{4,1}, f_{4,2}, f_{5,1}, f_{5,2}]$ n'est pas polynomiale}
\begin{propn}
On a un isomorphisme de $\C$-algèbres
$$\C[f_1, f_2, f_3, f_{4,1}, f_{4,2}, f_{5,1}, f_{5,2}] \simeq \dfrac{\C[X_1, X_2, X_3, X_{4,1}, X_{4,2}, X_{5,1}, X_{5,2}]}{\left(X_{5,1}X_{5,2}+\frac{1}{4} X_2^2-X_{4,1}^2X_{4,2}\right)}$$
(où $X_1, X_2, X_3, X_{4,1}, X_{4,2}, X_{5,1}, X_{5,2}$ sont des indéterminées).
\end{propn}
\begin{proof}
Puisque l'on a $f_{4,1}^2f_{4,2}-\frac{1}{4}f_2^2=f_{5,1}f_{5,2}$, on a bien une application surjective 
$$\dfrac{\C[X_1, X_2, X_3, X_{4,1}, X_{4,2}, X_{5,1}, X_{5,2}]}{\left(X_{5,1}X_{5,2}+\frac{1}{4} X_2^2-X_{4,1}^2X_{4,2}\right)} \longtwoheadrightarrow \C[f_1, f_2, f_3, f_{4,1}, f_{4,2}, f_{5,1}, f_{5,2}]$$
(qui est l'évaluation en $f_1, f_2, f_3, f_{4,1}, f_{4,2}, f_{5,1}, f_{5,2}$). Il reste à montrer que l'application est bien bijective.\par
Soit $P \in \C[X_1, X_2, X_3, X_{4,1}, X_{4,2}, X_{5,1}, X_{5,2}]$ tel que $P(f_1, f_2, f_3, f_{4,1}, f_{4,2}, f_{5,1}, f_{5,2})=0$. Soit $P=\left(\frac{1}{4} X_2^2+X_{5,1}X_{5,2}-X_{4,1}^2X_{4,2}\right)Q+R$ la division euclidienne de $P$ par $\frac{1}{4} X_2^2+X_{5,1}X_{5,2}-X_{4,1}^2X_{4,2}$ par rapport à l'indéterminée $X_2$ (le polynôme $\frac{1}{4} X_2^2+X_{5,1}X_{5,2}-X_{4,1}^2X_{4,2}$ est de coefficient dominant inversible comme polynôme en $X_2$, ce qui justifie la division euclidienne). Ainsi $R$ est un polynôme de degré au plus $1$ en $X_2$, que l'on note
$$R=R_1(X_1, X_3, X_{4,1}, X_{4,2}, X_{5,1}, X_{5,2})X_2+R_0(X_1, X_3, X_{4,1}, X_{4,2}, X_{5,1}, X_{5,2})$$
et il s'agit de montrer que $R_1=R_0=0$. Évaluons $P$ en $f_1, f_2, f_3, f_{4,1}, f_{4,2}, f_{5,1}, f_{5,2}$, on obtient
\begin{equation}
R_1(f_1, f_3, f_{4,1}, f_{4,2}, f_{5,1}, f_{5,2})f_2+R_0(f_1, f_3, f_{4,1}, f_{4,2}, f_{5,1}, f_{5,2})=0
\label{jsp}
\end{equation}
Comme $f_1, f_3, f_{4,1}, f_{4,2}, f_{5,1}, f_{5,2}$ est une famille algébriquement indépendante (propriété \ref{aiceC}), il suffit de montrer que $\mathcal{R}_1:=R_1(f_1, f_3, f_{4,1}, f_{4,2}, f_{5,1}, f_{5,2})$ et $\mathcal{R}_0:=R_0(f_1, f_3, f_{4,1}, f_{4,2}, f_{5,1}, f_{5,2})$ sont nuls. On raisonne alors de la même manière que dans la démonstration du théorème \ref{alin}. Supposons que $\mathcal{R}_1$ et $\mathcal{R}_1$ sont non nuls. On peut supposer que tous les monômes dans l'égalité \eqref{jsp}, qui sont des semi-invariants, sont de même poids. Puisque le semi-groupe engendré par les poids de $f_{4,1}, f_{4,2}, f_{5,1}, f_{5,2}$ est un groupe (on le vérifie à la main), on peut supposer que ce poids est nul. Ainsi on peut supposer que $\mathcal{R}_1$ et $\mathcal{R}_0$ sont des invariants. Or, puisque les poids de $f_{4,1}$ et de $f_{5,1}$ ne sont pas liés, un monôme en $f_1, f_3, f_{4,1}, f_{4,2}, f_{5,1}, f_{5,2}$ de poids nul est nécessairement de la forme
$$ f_1^{s_1} f_3^{s_3} (f_{4,1}^2 f_{4,2})^{s_4} (f_{5,1} f_{5,2})^{s_5}=f_1^{s_1} f_3^{s_3} f_4^{s_4} \left(f_4-\frac{1}{4}f_2^2\right)^{s_5}.$$
avec les $s_i \in \N$. Alors $\mathcal{R}_1$ et $\mathcal{R}_0$ sont des polynômes en $f_1, f_2, f_3, f_4$, de degré pair en $f_2$. Comme $f_1, f_2, f_3, f_4$ est une famille algébriquement indépendante (théorème \ref{pyC}), cela contredit l'égalité \eqref{jsp} qui se réécrit $\mathcal{R}_1f_2 +\mathcal{R}_0=0$. Ainsi $\mathcal{R}_1=0$ ou $\mathcal{R}_0=0$, d'où $\mathcal{R}_1=\mathcal{R}_0=0$, et donc $R_1=R_0=0$.
\end{proof}
\begin{propn}
L'algèbre $\dfrac{\C[X_1, X_2, X_3, X_{4,1}, X_{4,2}, X_{5,1}, X_{5,2}]}{\left(X_{5,1}X_{5,2}+\frac{1}{4} X_2^2-X_{4,1}^2X_{4,2}\right)}$ n'est pas polynomiale.
\end{propn}
\begin{proof}
On vérifie que sur la variété $\Spec\left(\frac{\C[X_1, X_2, X_3, X_{4,1}, X_{4,2}, X'_{4,1}, X'_{4,2}]}{\left(X'_{4,1}X'_{4,2}+\frac{1}{4} X_2^2-X_{4,1}^2X_{4,2}\right)}\right)$, le point $0$ est singulier en calculant son espace tangent.
\end{proof}

\subsection{Les semi-invariants $f_1, f_2, f_3, f_{4,1}, f_{4,2}, f_{5,1}, f_{5,2}$ engendrent $\Sy(\gotq^C)$}
Pour montrer que $\Sy(\gotq^C)=\C[f_1, f_2, f_3, f_{4,1}, f_{4,2}, f_{5,1}, f_{5,2}]$, on applique encore une fois le théorème \ref{inter}, cette fois-ci à $\mathbf{f}:=\{f_1, f_2, f_3, f_{4,1}, f_{4,2}, f_{5,1}, f_{5,2}\}$. On pose $f_{m,1}:=f_m$ pour $m \in \{1,2,3\}$. On vérifie d'abord les hypothèses (a), (b), (c) et (d) :
\begin{itemize}
\item[(a)] L'algèbre $\Y(\gotq^C)$ est polynomiale \cite{py13} donc en particulier factorielle.
\item[(b)] L'équation \eqref{eqbiz} montre que $\GK \C[\mathbf{f}] \leq 6$. Comme l'ensemble $\{f_1, f_3, f_{4,1}, f_{4,2}, f_{5,1}, f_{5,2}\}$ forme une base de transcendance de $\Sy(\gotq^C)$ (propriété \ref{indceC}), on a $\GK \Sy(\gotq^C)= \GK \C[\mathbf{f}]=6$.
\item[(c,d)] Comme $f_1, f_2, f_3, f_4$ engendrent librement $\Y(\gotq^C)$ (théorème \ref{pyC}), les invariants $f_1$, $f_2$, $f_3$, $f_4$, $f_5:=f_4-\frac{1}{4}f_2^2$ engendrent l'algèbre $\Y(\gotq^C)$, et $f_1, f_2, f_3, f_4$ sont irréductibles dans $\Y(\gotq^C)$. L'invariant $f_5$ est un polynôme irréductible en $f_1, f_2, f_3, f_4$ donc est également irréductible dans $\Y(\gotq^C)$.
\end{itemize}
Il reste donc à vérifier les hypothèses (I), (II) et (III).
\subsubsection{Hypothèse (II)}
On rappelle que $\lambda^C_{4,1}=\varpi_1^C-\varpi_3^C$, $\lambda_{4,2}^C=2(\varpi_3^C-\varpi_1^C)$, $\lambda_{5,1}^C=-2\varpi_1^C$ et $\lambda_{5,2}^C=2\varpi_1^C$ (sous-section \ref{sec102}). Or comme dans le cas $\alpha_{n'}^C \notin (\pi^C)'$, on a $\Lambda^C \subset \bigoplus_{\ell=1}^{4} \rat \varpi_\ell^C$. Ainsi $\lambda_{4,1}^C$ est donc bien indivisible dans $\Lambda^C$, donc $f_{4,1}$ est indivisible dans $\Sy(\gotq^C)$.\par
Pour $f_{4,2}$, $f_{5,1}$ et $f_{5,2}$, on raisonne comme dans le cas $\alpha_{n'}^C \notin (\pi^C)'$, c'est-à-dire que l'on montre que ces semi-invariants sont de degré $1$ en certains vecteurs $e_{v,v^\gamma}$ appartenant à l'antidiagonale. Le monôme $-\pr(e_{2,1}) \pr(e_{8,7}) \pr(e_{1,8})  \pr(e_{3,6})$ est un monôme de $f_{4,2}$ (par la propriété \ref{calculceC}) qui est de degré $1$ en $\pr(e_{1,8})=e_{1,8}$, où $e_{1,8}$ est un vecteur de la base canonique appartenant à l'antidiagonale. De même, $\pr(e_{1,8})=e_{1,8}$ est un terme de $f_{5,2}$ (en fait, $f_{5,2}=e_{1,8}$). Enfin
$$\mathcal{S}=-\pr(e_{1,8})\pr(e_{8,6})\pr(e_{6,4})\pr(e_{4,2})\pr(e_{2,7})\pr(e_{7,5})\pr(e_{5,3})\pr(e_{3,1})$$
est un monôme de $f_4=\pr(F_8^\bullet)$ qui n'est pas un monôme de $\frac{1}{4}f_2^2$. En effet, dans le cas contraire, $\mathcal{S}$ serait un produit de deux monômes de $f_2$, et on vérifie manuellement qu'il n'existe pas de facteur de degré $4$ de $\mathcal{S}$ de la forme $\prod_{l \in J} \pr(e_{l,\sigma(l)})$ avec $J \subset \llbracket 1,8 \rrbracket$ et $\sigma \in \mathfrak{S}(J)$.\par
Ainsi $\mathcal{S}$ est un monôme de $f_4-\frac{1}{4}f_2^2$ d'où
$$\pr(e_{8,6})\pr(e_{6,4})\pr(e_{4,2})\pr(e_{2,7})\pr(e_{7,5})\pr(e_{5,3})\pr(e_{3,1})$$
est un facteur de $f_{5,1}$, qui est bien de degré $1$ en $\pr(e_{2,7})=e_{2,7}$, où $e_{2,7}$ est un vecteur de la base canonique appartenant à l'antidiagonale.

\subsubsection{Hypothèse (I)}
Comme en type $A$ et dans le cas $\alpha_{n'}^C \notin (\pi^C)'$ du type $C$, pour montrer l'hypothèse (I), on montre en fait l'hypothèse (I') de la proposition \ref{632}. Ici l'ensemble $\{g_1, g_2, g_3, g_4\}$ de la proposition \ref{inter} correspond aux invariants $f_1,f_2,f_3,f_4$ introduits dans cette section, et on construit $q$ de manière similaire au cas $\alpha_{n'}^C \notin ({\pi^C})'$ du type $C$ comme suit :
\begin{center}
\begin{tikzpicture}[scale=0.7]
\foreach \k in {1,2,...,7}
	{\draw[color=gray!20]  (0,\k)--(8,\k);
	\draw[color=gray!20] (\k,0)--(\k,8);}
\draw (0,0)--(8,0);
\draw (0,0)--(0,8);
\draw (0,8)--(8,8);
\draw (8,0)--(8,8);
\draw[ultra thick] (0,7)--(1,7)--(1,5)--(3,5)--(3,3)--(5,3)--(5,1)--(7,1)--(7,0);
%\draw (2,2) node{\LARGE $\gotn^-$};
%\draw (6,6) node{\LARGE $\gotp$};
%\draw[dashed] (1,8)--(1,7)--(3,7)--(3,5)--(5,5)--(5,3)--(7,3)--(7,1)--(8,1);
%\draw (0.5,7.5) node{\Large $\f_{I_1}$};
%\draw (2,6) node{\Large $\g_{I_2}$};
%\draw (4,4) node{\Large $\g_{I_3}$};
%\draw (6,2) node{\Large $\g_{I_4}$};
%\draw (7.5,0.5) node{\Large $\g_{I_5}$};
\draw (7.5,7.5) node{\Large $1$};
\draw (0.5,6.5) node{\Large $1$};
\draw (6.5,0.5) node{\Large $-1$};
\draw (1.5,4.5) node{\Large $1$};
\draw (4.5,1.5) node{\Large $-1$};
\draw (2.5,3.5) node{\Large $1$};
\draw (3.5,2.5) node{\Large $1$};
\draw (0.5,0.5) node{\Large $X_1$};
\draw (1.5,1.5) node{\Large $X_2$};
\draw (3.5,3.5) node{\Large $X_3$};
\draw (5.5,5.5) node{\Large $X_4$};
\end{tikzpicture}
\end{center}
de sorte que son graphe $\mathcal{G}(q)$ (voir sous-section \ref{graphe}) est
\begin{center}
\begin{tikzpicture}
  \tikzset{LabelStyle/.style = {fill=white}}
  \tikzset{VertexStyle/.style = {%
  shape = circle, minimum size = 20pt,draw}}
  \SetGraphUnit{2.5}
  \Vertex[L=$8$]{1}
  \EA[L=$7$](1){2}
  \EA[L=$5$](2){3}
  \EA[L=$3$](3){4}
  %\Loop[dist = 2cm, dir = WE, label = $X_1$](1.west)
  \Edge[style= {->}](1)(2)
  \Edge[style= {->}](2)(3)
  \Edge[style= {->}](3)(4)
  \tikzset{EdgeStyle/.style = {->}}
  %\SetUpVertex[LineColor=black]
  \NO[L=$1$](1){1'}
  \NO[L=$2$](2){2'}
  \NO[L=$4$](3){3'}
  \NO[L=$6$](4){4'}
  \Edge[style= {->}](2')(1')
  \Edge[style= {->}](3')(2')
  \Edge[style= {->}](4')(3')
  \Edge[style= {->,bend right=30}, label =$1$](1')(1)
  \Edge[style= {->,bend right=30}, label =$X_1$](1)(1')
  \Edge[style= {->,bend right=30}, label =$X_2$](2)(2')
  \Edge[style= {->,bend right=30}, label =$X_3$](3)(3')
  \Edge[style= {->,bend right=30}, label =$X_4$](4)(4')
\end{tikzpicture}
\end{center}
\subsubsection{Hypothèse (III)}
Dans notre cas, on a $\mathbf{f}^\times=\{f_{4,1}, f_{4,2}, f_{5,1}, f_{5,2}\}$. Comme dans les cas précédents, on va exhiber les morphismes $\vartheta_f : \Sy(\gotq^C) \rightarrow \C[X]$ sous la forme $g \mapsto g(q_f)$ avec $q_f \in (\gotq^C)^*_{\C[X]}$, mais contrairement aux cas précédents, on ne demandera pas de condition restrictive sur les degrés des $\vartheta_f(h)$ pour $h \in \mathbf{f}^\times$. On se contentera de vérifier que $\vartheta_f(f)$ est bien non constant et premier avec les $\vartheta_f(h)$, $h \neq f$. On définit $q'$ par
\begin{center}
\begin{tikzpicture}[scale=0.7]
\foreach \k in {1,2,...,7}
	{\draw[color=gray!20]  (0,\k)--(8,\k);
	\draw[color=gray!20] (\k,0)--(\k,8);}
\draw (0,0)--(8,0);
\draw (0,0)--(0,8);
\draw (0,8)--(8,8);
\draw (8,0)--(8,8);
\draw[ultra thick] (0,7)--(1,7)--(1,5)--(3,5)--(3,3)--(5,3)--(5,1)--(7,1)--(7,0);
%\draw (2,2) node{\LARGE $\gotn^-$};
%\draw (6,6) node{\LARGE $\gotp$};
%\draw[dashed] (1,8)--(1,7)--(3,7)--(3,5)--(5,5)--(5,3)--(7,3)--(7,1)--(8,1);
%\draw (0.5,7.5) node{\Large $\g_{I_1}$};
%\draw (2,6) node{\Large $\g_{I_2}$};
%\draw (4,4) node{\Large $\g_{I_3}$};
%\draw (6,2) node{\Large $\g_{I_4}$};
%\draw (7.5,0.5) node{\Large $\g_{I_5}$};
\draw (7.5,7.5) node{\Large $1$};
\draw (1.5,1.5) node{\Large $1$};
\draw (6.5,7.5) node{\Large $1$};
\draw (7.5,6.5) node{\Large $1$};
\draw (0.5,6.5) node{\Large $1$};
\draw (6.5,0.5) node{\Large $-1$};
\draw (0.5,5.5) node{\Large $1$};
\draw (5.5,0.5) node{\Large $-1$};
\draw (1.5,4.5) node{\Large $1$};
\draw (4.5,1.5) node{\Large $-1$};
\draw (2.5,3.5) node{\Large $1$};
\draw (3.5,2.5) node{\Large $1$};
\end{tikzpicture}
\end{center}
(on rappelle que l'on identifie les espaces vectoriels $(\gotq^C)^*$ et $\gotq^C$ par l'isomorphisme $\gotq^C \simeq (\gotq^C)^*$ donné par la base canonique de $\gotq^C$). Son graphe $\mathcal{G}(q')$ est de la forme :
\begin{center}
\begin{tikzpicture}[scale=0.75]
\node[draw, circle] (4) at (0,0){$4$};
\node[draw, circle] (2) at (2,0){$2$};
\node[draw, circle] (1) at (4,0){$1$};
\node[draw, circle] (3) at (6,0){$3$};
\node[draw, circle] (6) at (0,2){$6$};
\node[draw, circle] (8) at (2,2){$8$};
\node[draw, circle] (7) at (4,2){$7$};
\node[draw, circle] (5) at (6,2){$5$};
\draw[->,>=stealth, ultra thick] (4)--(2);
\draw[->,>=stealth, ultra thick] (2)--(1);
\draw[->,>=stealth, ultra thick] (3)--(1);
\draw[->,>=stealth, ultra thick] (6)--(4);
\draw[o->,>=stealth, ultra thick] (2)--(8);
\draw[o->,>=stealth, ultra thick] (1)--(8);
\draw[->,>=stealth, ultra thick] (7)--(2);
\draw[o->,>=stealth, ultra thick] (1)--(7);
\draw[->,>=stealth, ultra thick] (5)--(3);
\draw[->,>=stealth, ultra thick] (8)--(6);
\draw[->,>=stealth, ultra thick] (8)--(7);
\draw[->,>=stealth, ultra thick] (7)--(5);
\end{tikzpicture}
\end{center}
Puisque nos semi-invariants sont des facteurs de $\pr(F_8^\bullet)$ et $\pr(F_8^\bullet)-\frac{1}{4} \pr(F_4^\bullet)^2$, on s'intéresse aux sous-graphes circuits de $\mathcal{G}(q')$ de longueur $4$ et $8$. On vérifie que les seuls sous-graphes circuits de longueur $4$ sont :
\begin{center}
\begin{tikzpicture}[scale=0.75]
%\node[draw, circle] (4) at (0,0){$4$};
%\node[draw, circle] (2) at (2,0){$2$};
\node[draw, circle] (1) at (4,0){$1$};
\node[draw, circle] (3) at (6,0){$3$};
%\node[draw, circle] (6) at (0,2){$6$};
%\node[draw, circle] (8) at (2,2){$8$};
\node[draw, circle] (7) at (4,2){$7$};
\node[draw, circle] (5) at (6,2){$5$};
%\draw[->,>=stealth] (4)--(2);
%\draw[->,>=stealth] (2)--(1);
\draw[->,>=stealth, ultra thick] (3)--(1);
%\draw[->,>=stealth] (6)--(4);
%\draw[->,>=stealth] (2)--(8) node[midway, left] {$+$};
%\draw[->,>=stealth] (1)--(8) node[near end, above] {$+$};
%\draw[->,>=stealth] (7)--(2);
\draw[o->,>=stealth, ultra thick] (1)--(7);
\draw[->,>=stealth, ultra thick] (5)--(3);
%\draw[->,>=stealth] (8)--(6);
%\draw[->,>=stealth] (8)--(7);
\draw[->,>=stealth, ultra thick] (7)--(5);
\end{tikzpicture}
\qquad
\begin{tikzpicture}[scale=0.75]
\node[draw, circle] (4) at (0,0){$4$};
\node[draw, circle] (2) at (2,0){$2$};
%\node[draw, circle] (1) at (4,0){$1$};
%\node[draw, circle] (3) at (6,0){$3$};
\node[draw, circle] (6) at (0,2){$6$};
\node[draw, circle] (8) at (2,2){$8$};
%\node[draw, circle] (7) at (4,2){$7$};
%\node[draw, circle] (5) at (6,2){$4$};
\draw[->,>=stealth, ultra thick] (4)--(2);
%\draw[->,>=stealth] (2)--(1);
%\draw[->,>=stealth] (3)--(1);
\draw[->,>=stealth, ultra thick] (6)--(4);
\draw[o->,>=stealth, ultra thick] (2)--(8);
%\draw[->,>=stealth] (1)--(8) node[near end, above] {$+$};
%\draw[->,>=stealth] (7)--(2);
%\draw[->,>=stealth] (1)--(7) node[midway, right] {$+$};
%\draw[->,>=stealth] (5)--(3);
\draw[->,>=stealth, ultra thick] (8)--(6);
%\draw[->,>=stealth] (8)--(7);
%\draw[->,>=stealth] (7)--(5);
\end{tikzpicture}
\qquad
\begin{tikzpicture}[scale=0.75]
%\node[draw, circle] (4) at (0,0){$4$};
\node[draw, circle] (2) at (2,0){$2$};
\node[draw, circle] (1) at (4,0){$1$};
%\node[draw, circle] (3) at (6,0){$3$};
%\node[draw, circle] (6) at (0,2){$6$};
\node[draw, circle] (8) at (2,2){$8$};
\node[draw, circle] (7) at (4,2){$7$};
%\node[draw, circle] (5) at (6,2){$4$};
%\draw[->,>=stealth] (4)--(2);
\draw[->,>=stealth, ultra thick] (2)--(1);
%\draw[->,>=stealth] (3)--(1);
%\draw[->,>=stealth] (6)--(4);
%\draw[->,>=stealth] (2)--(8) node[midway, left] {$+$};
\draw[o->,>=stealth, ultra thick] (1)--(8);
\draw[->,>=stealth, ultra thick] (7)--(2);
%\draw[->,>=stealth] (1)--(7) node[midway, right] {$+$};
%\draw[->,>=stealth] (5)--(3);
%\draw[->,>=stealth] (8)--(6);
\draw[->,>=stealth, ultra thick] (8)--(7);
%\draw[->,>=stealth] (7)--(5);
\end{tikzpicture}
\end{center}
que l'on note $\mathcal{G}'_1$, $\mathcal{G}'_2$ et $\mathcal{G}'_3$. Les deux premiers sont compatibles, le troisième est incompatible avec les deux premiers. On a $\mathcal{S}_{\mathcal{G}'_1}=\mathcal{S}_{\mathcal{G}'_2}$. On vérifie également que le seul sous-graphe circuits de longueur $8$ est $\mathcal{G}'_1 + \mathcal{G}'_2$.
\begin{center}
\begin{tikzpicture}[scale=0.75]
\node[draw, circle] (4) at (0,0){$4$};
\node[draw, circle] (2) at (2,0){$2$};
\node[draw, circle] (1) at (4,0){$1$};
\node[draw, circle] (3) at (6,0){$3$};
\node[draw, circle] (6) at (0,2){$6$};
\node[draw, circle] (8) at (2,2){$8$};
\node[draw, circle] (7) at (4,2){$7$};
\node[draw, circle] (5) at (6,2){$5$};
\draw[->,>=stealth, ultra thick] (4)--(2);
%\draw[->,>=stealth] (2)--(1);
\draw[->,>=stealth, ultra thick] (3)--(1);
\draw[->,>=stealth, ultra thick] (6)--(4);
\draw[o->,>=stealth, ultra thick] (2)--(8);
%\draw[->,>=stealth] (1)--(8) node[near end, above] {$+$};
%\draw[->,>=stealth] (7)--(2);
\draw[o->,>=stealth, ultra thick] (1)--(7);
\draw[->,>=stealth, ultra thick] (5)--(3);
\draw[->,>=stealth, ultra thick] (8)--(6);
%\draw[->,>=stealth] (8)--(7);
\draw[->,>=stealth, ultra thick] (7)--(5);
\end{tikzpicture}
\end{center}
Les monômes correspondants ont bien les bons bidegrés. Posons $S:=\mathcal{S}_{\mathcal{G}'_1}(q')=\mathcal{S}_{\mathcal{G}'_2}(q') \in \C^\times$ et $T=\mathcal{S}_{\mathcal{G}'_3}(q') \in \C^\times$, par la proposition \ref{elemgr}, on a :
\begin{enumerate}[label=(\roman*)]
\item $f_2(q')=-(2 \mathcal{S}_{\mathcal{G}'_1}+\mathcal{S}_{\mathcal{G}'_3})(q')=-(2S+T)$,
\item $f_4(q')=(\mathcal{S}_{\mathcal{G}'_1 + \mathcal{G}'_2})(q')=S^2$.
\end{enumerate}
d'où $(f_4-\frac{1}{4}f_2^2)(q')=S^2-\frac{1}{4}(2S+T)^2=T\left(S-\frac{T}{4}\right)$
Comme dans le cas $\alpha_{n'}^C \notin (\pi^C)'$, pour obtenir $q$, on multiplie un terme \emph{convenable} de $q'$ par $X$.
\paragraph{(1) : Pour $(m,t)=(4,1)$,} on transforme $q'$ en $q$ en multipliant le coefficient en $(e_{4,2}^*)_{|\syp_8}=(e_{4,2}-e_{7,5})^*$ de $q'$ par $X$. Pour passer du graphe $\mathcal{G}(q')$ au graphe $\mathcal{G}(q)$, on multiplie alors les poids des arêtes $4 \rightarrow 2$ et $7 \rightarrow 5$ par $X$.
\begin{center}
\begin{tikzpicture}[scale=0.9]
\node[draw, circle] (4) at (0,0){$4$};
\node[draw, circle] (2) at (2,0){$2$};
\node[draw, circle] (1) at (4,0){$1$};
\node[draw, circle] (3) at (6,0){$3$};
\node[draw, circle] (6) at (0,2){$6$};
\node[draw, circle] (8) at (2,2){$8$};
\node[draw, circle] (7) at (4,2){$7$};
\node[draw, circle] (5) at (6,2){$5$};
\draw[->,>=stealth, ultra thick] (4)--(2) node[midway, above] {$(X)$};
\draw[->,>=stealth, ultra thick] (2)--(1);
\draw[->,>=stealth, ultra thick] (3)--(1);
\draw[->,>=stealth, ultra thick] (6)--(4);
\draw[o->,>=stealth, ultra thick] (2)--(8);
\draw[o->,>=stealth, ultra thick] (1)--(8);
\draw[->,>=stealth, ultra thick] (7)--(2);
\draw[o->,>=stealth, ultra thick] (1)--(7);
\draw[->,>=stealth, ultra thick] (5)--(3);
\draw[->,>=stealth, ultra thick] (8)--(6);
\draw[->,>=stealth, ultra thick] (8)--(7);
\draw[->,>=stealth, ultra thick] (7)--(5) node[midway, below] {$(X)$};
\end{tikzpicture}
\end{center}
De la même manière que dans les calculs (i) et (ii), on obtient alors $f_4(q)=(SX)^2$ et $(f_4-\frac{1}{4}f_2^2)(q)=T\left(SX-\frac{T}{4}\right)$. Les polynômes $(SX)^2$ et $T\left(SX-\frac{T}{4}\right)$ sont premiers entre eux. Puisque seul le coefficient de $(e_{4,2}-e_{7,5})^*$ de $q$ est de degré $>0$ en $X$, on a nécessairement $f_{4,2}(q) \in \C^\times$ et $f_{4,1}^2(q) \propto X^2$ par la propriété \ref{calculceC}, c'est-à-dire $f_{4,1}(q) \propto X$, et donc nécessairement premier avec $f_{5,1}(q)$ et $f_{5,2}(q)$ qui divisent $T\left(SX-\frac{T}{4}\right)$.\par
\paragraph{(2) : Pour $(m,t)=(4,2)$,} on effectue le même type de transformation mais en multipliant par exemple le coefficient en $(e_{3,1}-e_{8,6})^*$ de $q'$ par $X$ (propriété \ref{calculceC}). Le raisonnement est similaire.\par
\paragraph{(3) : Pour $(m,t)=(5,1)$,} on peut en fait reprendre un des $q$ construits au-dessus. En effet, puisque $f_{5,2}=e_{1,8}$, dans les deux cas ci-dessus, on a $f_{5,2}(q) \in \C^\times$, donc $f_{5,1}(q)$ est associé à $T\left(SX-\frac{T}{4}\right)$, qui est premier avec $(SX)^2$.
\paragraph{(4) : Pour $(m,t)=(5,2)$,} il faut multiplier par $X$ le coefficient en $e_{1,8}^*$ de $q'$, et on a alors $f_4(q)=S^2$ et $(f_4-\frac{1}{4}f_2^2)(q)=TX\left(S-\frac{TX}{4}\right)$. Puisque l'on a $f_{5,2}(q)=X$ par construction, $f_{5,2}(q)$ est bien premier avec $f_{4,1}(q)$ et $f_{4,2}(q)$ qui sont constants et on a $(f_{5,1})(q)=-\frac{1}{4}T^2X+ST$ qui est bien premier avec $X$.

\subsection{Conclusion}
On obtient donc le théorème suivant :
\begin{theo}
On considère la contraction parabolique standard $\gotq$ de $\gl_8$ définie par $\pi'=\{\alpha_2, \alpha_4, \alpha_6\}$, autrement dit, la contraction parabolique associée au diagramme suivant :
\begin{center}
\begin{tikzpicture}[scale=0.7]
\foreach \k in {1,2,...,7}
	{\draw[color=gray!20]  (0,\k)--(8,\k);
	\draw[color=gray!20] (\k,0)--(\k,8);}
\draw (0,0)--(8,0);
\draw (0,0)--(0,8);
\draw (0,8)--(8,8);
\draw (8,0)--(8,8);
\draw[ultra thick] (0,7)--(1,7)--(1,5)--(3,5)--(3,3)--(5,3)--(5,1)--(7,1)--(7,0);
\draw (2,2) node{\LARGE $\gotn^-$};
\draw (6,6) node{\LARGE $\gotp$};
\draw[dashed] (1,8)--(1,7)--(3,7)--(3,5)--(5,5)--(5,3)--(7,3)--(7,1)--(8,1);
\draw (0.5,7.5) node{\Large $\g_{I_1}$};
\draw (2,6) node{\Large $\g_{I_2}$};
\draw (4,4) node{\Large $\g_{I_3}$};
\draw (6,2) node{\Large $\g_{I_4}$};
\draw (7.5,0.5) node{\Large $\g_{I_5}$};
\end{tikzpicture}
\end{center}
et $\gotq^C=\pr^C(\gotq)$ la contraction de $\syp_8$ associée.\par
L'algèbre des semi-invariants $\Sy(\gotq^C)$ n'est pas polynomiale et engendrée par $f_1$, $f_2$, $f_3$, $f_{4,1}$, $f_{4,2}$, $f_{5,1}$, $f_{5,2}$. En particulier, $\Sy(\gotq^C)$ est isomorphe à l'algèbre
$$\dfrac{\C[X_1, X_2, X_3, X_{4,1}, X_{4,2}, X'_{4,1}, X'_{4,2}]}{\left(X'_{4,1}X'_{4,2}+\frac{1}{4} X_2^2-X_{4,1}^2X_{4,2}\right)}$$
\end{theo}

\chapter{Considérations diverses}
\label{sec11}
\section{Conjectures en type $C$}
Tout au long de l'étude en type $C$, plusieurs résultats récurrents sont apparus. Avec certains calculs dans d'autres cas que ceux étudiés, ces considérations amènent quelques conjectures.
\subsection{Troncation canonique}
\begin{conj}
Soit $\gotq^C$ une contraction parabolique en type $C$. Alors $\gotq^C_\Lambda=(\gotq^C)'$. En particulier, dans le cas d'une contraction parabolique standard, on a $\ind \gotq^C_\Lambda = \ind \gotq^C + \Card\left((\pi^C) \setminus (\pi^C)'\right)$.
\end{conj}
On a montré ceci dans le cas $\alpha_{n'}^C \notin (\pi^C)'$ ainsi que dans le contre-exemple de la section \ref{sec10}. Il a fallu pour cela minorer $\ind \gotq^C_\Lambda$ en exhibant une famille d'éléments algébriquement indépendants de $\Y(\gotq^C_\Lambda)=\Sy(\gotq^C)$ de cardinal suffisamment grand. Dans le cas général en type $C$, faute d'étude concluante, on ne dispose plus d'une telle famille, de sorte que l'on ne peut plus conclure que $\gotq^C_\Lambda=(\gotq^C)'$.

\subsection{Condition nécessaire et suffisante à la polynomialité}
Empiriquement, sur tous les cas en type $C$ que l'on a pu étudier, le seul obstacle à la polynomialité est celui du type du contre-exemple que l'on a exhibé en section \ref{sec10}. On pose $\gotq^C$ une contraction parabolique standard en type $C$ et $\gotq$ la CPSS au-dessus de $\gotq^C$ (voir la définition \ref{cpcs}).\par
Dans la CPSS $\gotq$ en type $\gl_n$, lorsque $\alpha_{n'}^C \in (\pi^C)'$, il peut exister $m \in \mathbf{M}_2$ impair, et donc une décomposition non triviale du type $F_m^\bullet=F_{m,1} \ldots, F_{m,r_m}$. Puisque pour tout $J$ de cardinal impair, on a $\Delta_{J^\gamma}=-\Delta_J$, on obtient $\pr(F_m^\bullet)=0$. Comme dans le contre-exemple, il n'y a toutefois aucune raison que $\pr(F_{m,t})=0$ pour tout $t$.\par
Il existe des cas où ce genre de considération n'empêche pas la polynomialité. Par exemple, prenons le facteur de Levi $\gotl$ de la forme $\gotl=\gl_{\{1\}} \times \gl_{\llbracket 2,n-1 \rrbracket} \times \gl_{\{n\}}$ avec $n$ pair $\geq 6$. On a $F_3^\bullet=F_{3,1}F_{3,2}$ avec $F_{3,2}=e_{1,n}$. Dans ce cas, on a $\pr(F_4^\bullet)-\frac{1}{4}\pr(F_2^\bullet)^2 = g \times e_{1,n}$ pour un certain semi-invariant $g$. Avec le même schéma de preuve que précédemment, on montre que la famille $(\pr(F_2^\bullet), g,e_{1,n},\pr(F_6^\bullet), \ldots, \pr(F_n^\bullet))$ engendre librement $\Sy(\gotq^C)$.\par
Le point crucial qui a fait échouer la polynomialité dans le contre-exemple est le fait que l'on ait simultanément $m=5$ impair dans $\mathbf{M}_2$ et $2m-2=8$ dans $\mathbf{M}_2$. Ainsi, dans le cas général, on pense obtenir deux décompositions du type
$$\pr(F_{2m-2}^\bullet)=f_{2m-2,1}^2 \, \ldots \, f_{2m-2,r'_{m-1}-1}^2 \, f_{2m-2,r'_{m-1}} \, f_{2m-2,r'_{m-1}+1}$$ 
$$\pr(F_{2m-2}^\bullet)-\frac{1}{4}\pr(F_{m-1}^\bullet)^2=g_{2m-2,1}^{s_1} \ldots g_{2m-2,r'_m+1}^{s_m}$$
où $s_k \in \N^*$ et les espaces de poids sont en somme directe. De manière générale, dans tous les cas que l'on a considérés, on a constaté le fait suivant :
\begin{conj}
Si $m \in \mathbf{M}_2$ est impair, alors il existe une décomposition $f=\prod_{k=1}^{r'_m+1} g_k^{s_k}$ avec les $g_k$ non constants et non liés deux à deux, où $f=\left\{ \begin{array}{ll} \pr(F_{2m-2}^\bullet)-\frac{1}{4}\pr(F_{m-1}^\bullet)^2 &  \text{si } 2m-2 \leq n \\ \pr(F_{m-1}^\bullet) & \text{si } 2m-2 > n \end{array} \right.$.
\end{conj}
Cette conjecture peut faire apparaître un problème du type du contre-exemple de la sous-section \ref{sec10} dans le cas où $m \in \mathbf{M}_2$ est impair et $2m-2 \in \mathbf{M}_2$. En revanche, dans l'autre sens, si $m \in \mathbf{M}_2$ est impair, alors $m-1 \notin \mathbf{M}_1$ (propriété \ref{propelem} (5)), ce qui ne posera donc pas le même type de problème. On conjecture alors :
\begin{conj}
En type $C$, l'algèbre $\Sy(\gotq^C)$ des semi-invariants n'est pas polynomiale si et seulement si il existe $m \in \mathbf{M}_2$ impair tel que $2m-2 \in \mathbf{M}_2$.
\end{conj} 
%\subsection{Propriétés générales des facteurs d'invariants lorsque $\Y(\gotk)$ est polynomiale}
%On pose $\gotk$ une algèbre de Lie, et dans cette sous-section, on suppose que $\Y(\gotk)$ est polynomiale engendrée par $f_1, \ldots, f_d$.
%\begin{prop}
%Soit $P \in \C[X_1, \ldots, X_d]$ irréductible et $y=P(f_1, \ldots, f_d)$. Supposons qu'il existe un facteur non trivial $x$ de $y$ (c'est-à-dire ni constant, ni associé à $y$). Alors $x$ est de poids non nul.
%\end{prop}
%\begin{proof}
%Si $x$ est de poids nul, alors c'est un invariant donc un polynôme $Q(f_1, \ldots, f_d)$. Alors $Q$ divise nécessairement $P$ donc soit $Q$ est constant, soit $Q$ est associé à $P$.
%\end{proof}

\section{Limites de l'approche en type différent de $A$ et $C$}
\label{sec12}
En type autre que $A$ et $C$, il n'existe plus de théorème général donnant la polynomialité de $\Y(\gotq)$. On pourrait toutefois espérer comprendre les cas de non-polynomialité en s'intéressant aux semi-invariants obtenus comme facteurs d'invariants connus dans ces deux types.\par
Aussi, on ne peut plus simplement étudier les projections sur $\so_n$ des semi-invariants en type $\gl_n$. Cela vient du fait que l'on n'a plus nécessairement $\pr(F_m)^\bullet=\pr(F_m^\bullet)$ pour tout $m$. Par exemple, en type $D$, il arrive que
$$\deg_{(\gotn^-)^D} \pr(F_m) < \deg_{\gotn^-} F_m$$
\subsection*{Le cas inconclusif de Panyushev et Yakimova}
\label{inconclusif}
Panyushev et Yakimova montrent la polynomialité de $\Y(\gotq)$ lorsque $\gotq$ est une contraction parabolique en type $A$ ou $C$. Dans leur article, ils exhibent également un cas où leur approche ne permet pas de conclure \cite[remarque 4.6]{py13}. Reprenons leur cas dans la perspective d'y chercher des semi-invariants.\par
On reprend les conventions de \cite[chapitre 8]{bou06}. Pour tout $n$, on définit $\so_n$ comme l'ensemble des matrices $M$ de taille $n \times n$ antisymétriques par rapport à l'antidiagonale, autrement dit, telles que $ M^\gamma = -M$. Le type $D$ correspond à $\so_n$ avec $n$ pair. Comme en type $C$, on définit en type $D$ une projection $\pr^D : \gl_n \rightarrow \so_n$ qui est la projection par rapport à $\C \id \oplus \so_n^\perp$ où $\so_n^\perp$ est l'orthogonal de $\so_n$ pour la forme de Killing sur $\spl_n$. Dans cette sous-section, on notera $\pr$ pour $\pr^D$.\par
Dans l'exemple de Panyushev et Yakimova, on se place en type $D_6$, et on regarde la contraction parabolique $\gotq^D$ de $\so_{12}$ image par $\pr$ de la contraction parabolique $\gotq$ de $\gl_{12}$ donnée par le diagramme suivant :
\begin{center}
\begin{tikzpicture}[scale=0.7]
\foreach \k in {1,2,...,11}
	{\draw[color=gray!20]  (0,\k)--(12,\k);
	\draw[color=gray!20] (\k,0)--(\k,12);}
\draw (0,0)--(12,0);
\draw (0,0)--(0,12);
\draw (0,12)--(12,12);
\draw (12,0)--(12,12);
\draw[ultra thick] (0,11)--(1,11)--(1,10)--(2,10)--(2,6)--(6,6)--(6,2)--(10,2)--(10,1)--(11,1)--(11,0);
\draw (3,3) node{\LARGE $\gotn^-$};
\draw (9,9) node{\LARGE $\gotp$};
\draw[dashed] (1,12)--(1,11)--(2,11)--(2,10)--(6,10)--(6,6)--(10,6)--(10,2)--(11,2)--(11,1)--(12,1);
\draw (0.5,11.5) node{\Large $\g_{I_1}$};
\draw (1.5,10.5) node{\Large $\g_{I_2}$};
\draw (4,8) node{\Large $\g_{I_3}$};
\draw (8,4) node{\Large $\g_{I_4}$};
\draw (10.5,1.5) node{\Large $\g_{I_5}$};
\draw (11.5,0.5) node{\Large $\g_{I_6}$};
\end{tikzpicture}
\end{center}
%\begin{rema}
%Contrairement aux types $A$ et $C$, en type $D$, toutes les contractions paraboliques ne sont pas données par projection d'une contraction parabolique de $\gl_n$.
%\end{rema}
Dans ce cas, on a $\pr(F_6^\bullet)=0$ (car $e_{1,12}$ divise $F_6^\bullet$ et la projection des éléments antidiagonaux est nulle) et $\deg_{(\gotn^-)^D} \pr(F_6)=4$ (alors que $\deg_{\gotn^-} F_6=5$). En revanche, pour $m \in \{2,4,8,10,12\}$, on a bien $\deg_{(\gotn^-)^D} \pr(F_m) = \deg_{\gotn^-} F_m$.\par
Par \cite[Chap. VIII]{bou06}, l'algèbre des invariants $\Y(\so_{12})$ est polynomiale et engendrée par $\pr(F_2)$, $\pr(F_4)$, $\pr(F_6)$, $\pr(F_8)$, $\pr(F_{10})$ et $f$, où $f$ est une racine carrée de $\pr(F_{12})$. Panyushev et Yakimova ont étudié la famille $\pr(F_2)^\bullet$, $\pr(F_4)^\bullet$, $\pr(F_6)^\bullet$, $\pr(F_8)^\bullet$, $\pr(F_{10})^\bullet$, $f^\bullet$ et ont conclu dans ce cas que la famille était algébriquement liée dans $\Y(\gotq^D)$. Comme ils ont montré que $\ind \gotq^D = \rg \so_{12} = 6$, cette famille ne peut pas non plus engendrer $\Y(\gotq^D)$.\par 
On peut tout de même essayer de tirer des semi-invariants de cette famille. En type $A$, on a une décomposition $F_{12}^\bullet =F_{12,1} F_{12,2}$ ce qui en projetant, donne une projection de la forme $\pr(F_{12}^\bullet) =x_3^2 y_3^2$ où $\pr(F_{12,1})=x_3^2$ et $\pr(F_{12,2})=y_3^2$. Comme $\ind \gotq^D_\Lambda-\ind \gotq^D = \dim \gotq^D-\dim \gotq^D_\Lambda \leq \dim \gotq^D - \dim (\gotq^D)' = 3$, il manque soit des éléments dans $\gotq^D_\Lambda$, soit des semi-invariants. On peut obtenir des semi-invariants indépendamment de l'étude des $\pr(F_2)^\bullet$, $\pr(F_4)^\bullet$, $\pr(F_6)^\bullet$, $\pr(F_8)^\bullet$, $\pr(F_{10})^\bullet$, $f^\bullet$.
\begin{propn}
Soit $\gotq=\gotp \ltimes \gotn^-$ une contraction parabolique d'une algèbre simple $\g$ associé à une certaine sous-algèbre de Cartan, une certaine base $\pi$ du système de racines $R$ associé et un certain sous-ensemble $\pi' \varsubsetneq \pi$. On note $\theta$ la plus grande racine de $R$. Soit $e \in \gotq \setminus \{0\}$. Alors $e$ est un semi-invariant (de degré $1$) de $\Sym(\gotq)$ si et seulement si $e \in \g_{-\beta}$ où $\beta \in (\pi \setminus \pi') \cup \{-\theta\}$ est telle que le sommet associé à $\beta$ dans le diagramme de Dynkin étendu de $R$ n'est relié à aucun sommet associé à un élément de $\pi'$. \label{degre1}
\end{propn}
\begin{proof}
On rappelle que par définition du crochet de Lie sur $\gotq$, on a pour tous $p_1, p_2 \in \gotp, n_1, n_2 \in \gotn^-$ :
$$[p_1,p_2]_{\gotq}=[p_1,p_2]_{\g} \qquad [p_1,n_1]_{\gotq}=\pr_{\gotn^-,\gotp} \left([p_1,n_1]_{\g}\right) \qquad [n_1,n_2]_{\gotq}=0$$
Le sous-espace vectoriel $\left(\bigoplus_{\alpha \in \pi \sqcup -\pi'} \g_\alpha\right) \oplus \g_{-\theta} \oplus \h$ engendre l'algèbre de Lie $\gotq$ : en effet, $\gotn^-$ appartient à l'algèbre de Lie engendrée par $\left(\bigoplus_{\alpha \in \pi} \g_\alpha\right) \oplus \g_{-\theta}$ et $\gotp$ appartient à l'algèbre de Lie engendrée par $\left(\bigoplus_{\alpha \in \pi \sqcup -\pi'}\g_\alpha\right) \oplus \h$.\par
Soit $\beta \in (-\pi \setminus -\pi') \cup \{\theta\}$. Alors $e \in \g_{\beta}$ est un semi-invariant et seulement si $[\g_{\beta}, \g_{\beta'}]_{\gotq}=0$ pour $\beta' \in \pi \sqcup -\pi' \sqcup \{-\theta\}$. Si $\beta = \theta$, alors 
\begin{itemize}
\item pour tout $\beta' \in \pi$, on a $[\g_{\beta}, \g_{\beta'}]_{\g} =0$, d'où $[\g_{\beta}, \g_{\beta'}]_{\gotq}=0$,
\item pour $\beta'=-\theta$, on a $[\g_{\beta}, \g_{\beta'}]_{\g} \subset \g_{0}=\h \subset \gotp$, ainsi $[\g_{\beta}, \g_{\beta'}]_{\gotq} =0$,
\item pour tout $\beta' \in -\pi'$, on a $\g_{\beta}, \g_{\beta'} \in \gotp$ donc $[\g_{\beta}, \g_{\beta'}]_{\gotq}=[\g_{\beta}, \g_{\beta'}]_{\g}$. Par conséquent, $[\g_{\beta}, \g_{\beta'}]_{\gotq}=0$ si et seulement si les sommets du diagramme de Dynkin étendu associés à $-\beta$ et $-\beta'$ ne sont pas reliés.
\end{itemize}
Si $\beta \in -\pi \setminus -\pi'$, on raisonne de même. On a donc montré que dans le cas où $\beta \in (-\pi \setminus -\pi') \cup \{\theta\}$, l'élément $e \in \g_\beta$ est un semi-invariant si et seulement si tous les sommets reliés au sommet du diagramme de Dynkin étendu associé à $- \beta$ sont associés à des éléments de $\pi \setminus \pi'$.
%\begin{itemize}
%\item pour tout $\beta' \in \pi$, on a $[\g_{-\beta}, \g_{\beta'}]_{\g} \subset \h \subset \gotp$ avec $\g_{-\beta} \subset \gotn^-$ et $\g_{\beta'} \subset \gotp$, donc $[\g_{-\beta}, \g_{\beta'}]_{\gotq}=0$,
%\item pour $\beta'=-\theta$, on a $[\g_{-\beta}, \g_{\beta'}]_{\g} =0$, d'où $[\g_{-\beta}, \g_{\beta'}]_{\gotq}=0$,
%\item pour tout $\beta' \in -\pi$, si le sommet associé à $-\beta'$ n'est pas relié au sommet associé à $\beta$, alors $[\g_{-\beta}, \g_{\beta'}]_{\g}=0$. Si le sommet associé à $-\beta'$ n'est pas relié au sommet associé à $\beta$, alors $[\g_{-\beta}, \g_{\beta'}]_{\g} \neq 0$. Donc $[\g_{-\beta}, \g_{\beta'}]_{\gotq} = 0$ si et seulement si $\g_{\beta'} \subset \gotn^-$ donc $\beta' \notin -\pi'$,
%\end{itemize}
Il reste à montrer que si $x \in \gotq$ est un semi-invariant, alors $x \in \g_\beta$ avec $\beta \in (-\pi \setminus -\pi') \cup \{\theta\}$.\par
On peut déjà remarquer que si $x \in \gotq$ est un semi-invariant, alors $x \in \g_\beta$ pour un certain $\beta \in R \sqcup \{0\}$  (l'action de $\h$ sur $\gotq$ est diagonalisable et les sous-espaces propres sont les $\g_\beta$). On peut supposer que $x \neq 0$. Soit $R=R^+ \sqcup R^-$ la décomposition de $R$ en racines positives et négatives, $R_{\pi'}$ le système de racines engendré par $\pi'$ et $R_{\pi'}=R_{\pi'}^+ \sqcup R_{\pi'}^-$ la décomposition de $R_{\pi'}$ en racines positives et négatives.\par
On a $\beta \neq 0$ puisque pour tout $h \in \h$, on a $[\gotq,h]_{\gotq}=[\g,h]_{\g} \nsubseteq \C h$.\par
Si $\beta \in (R^+ \sqcup R_{\pi'}^-) \setminus \{\theta\}$, alors $\g_\beta \subset \gotp$ et il existe $\beta' \in R^+ \sqcup R_{\pi'}^-$ tel que $\beta+\beta' \in R^+ \sqcup R_{\pi'}^-$, d'où $[\g_\beta,\g_{\beta'}]_{\gotq} = [\g_\beta,\g_{\beta'}]_{\g}=\g_{\beta + \beta'} \neq 0$.\par
Si $\beta \in (R^- \setminus R_{\pi'}^-) \setminus (-\pi \setminus -\pi')$, alors $\g_{\beta} \subset \gotn^-$ et il existe $\beta' \in R^+ \sqcup R_{\pi'}^-$ tel que $\beta + \beta' \in R^- \setminus R_{\pi'}^-$, d'où $[\g_\beta,\g_{\beta'}]_{\gotq} = [\g_\beta,\g_{\beta'}]_{\g}=\g_{\beta + \beta'} \neq 0$.\par
Ainsi $\beta \in (-\pi \setminus -\pi') \cup \{\theta\}$, ce qui conclut la démonstration.
\end{proof}
Dans notre cas, le diagramme de Dynkin étendu est
\begin{center}
\begin{tikzpicture}
\node[draw,circle,minimum width=15pt,inner sep=0pt] (1) at (0,0.5) {$\alpha_1$};
\node[draw,circle,minimum width=15pt,inner sep=0pt] (m) at (0,-0.5) {$- \theta$};
\node[draw,circle,minimum width=15pt,inner sep=0pt] (2) at (1.5,0) {$\alpha_2$};
\node[draw,circle,minimum width=15pt,inner sep=0pt, fill=gray!30] (3) at (3,0) {$\alpha_3$};
\node[draw,circle,minimum width=15pt,inner sep=0pt, fill=gray!30] (4) at (5,0) {$\alpha_4$};
\node[draw,circle,minimum width=15pt,inner sep=0pt, fill=gray!30] (5) at (6.5,0.5) {$\alpha_5$};
\node[draw,circle,minimum width=15pt,inner sep=0pt] (6) at (6.5,-0.5) {$\alpha_6$};
\draw (1)--(2)--(3)--(4)--(5);
\draw (4)--(6);
\draw (m)--(2);
\end{tikzpicture}
\end{center}
(les sommets grisés sont ceux représentant les éléments de $\pi'$) donc, par la proposition \ref{degre1}, les semi-invariants de degré $1$ de $\Sym(\gotq^D)$ sont associés soit à $x_1:=\pr(e_{2,1}) \in \g_{- \alpha_1}$ soit à $x_2:=\pr(e_{1,11}) \in \g_{\theta}$. On calcule alors les poids $\lambda_1$, $\lambda_2$ et $\lambda_3$ de $x_1$, $x_2$ et $x_3$, qui sont linéairement indépendants. Si $g_1, \ldots, g_d$ est une famille algébriquement indépendante de $\Y(\gotq^D)$, alors la famille $g_1, \ldots, g_d, x_1, x_2, x_3$ est une famille algébriquement indépendante de $\Sy(\gotq^D)$. On montre cela de manière similaire au théorème \ref{alin} : à un poids $\lambda$ fixé, tous les monômes $\mathbf{m}$ en $g_1, \ldots, g_d, x_1, x_2, x_3$ de poids $\lambda$ auront leur partie $x_1^{b_1} x_2^{b_2} x_3^{b_3}$ en commun (autrement dit, $b_1$, $b_2$, $b_3$ ne dépendent pas du monôme de poids $\lambda$ considéré) ; cela vient du fait que le poids de $g_1^{a_1} \ldots g_d^{a_d} x_1^{b_1} x_2^{b_2} x_3^{b_3}$ est le poids de $x_1^{b_1} x_2^{b_2} x_3^{b_3}$ et que les poids de $x_1$, $x_2$ et $x_3$ sont linéairement indépendants.\par %Si une somme de monômes deux à deux non colinéaires en les $g_1, \ldots, g_d, x_1, x_2, x_3$ est nulle, on peut supposer que tous ces monômes ont même poids $\lambda$ et donc que la partie en $x_1, x_2, x_3$ est commune à tous les monômes. En simplifiant par ce monôme en $x_1, x_2, x_3$, on obtient donc une somme de monômes deux à deux non colinéaires en les $g_1, \ldots, g_d$ qui est nulle, et donc par indépendance algébrique des $g_1, \ldots, g_d$, cette somme n'a aucun terme.\par
Ainsi $\ind \gotq^D_\Lambda \geq \ind \gotq^D + 3$. Puisque $\dim \gotq^D_\Lambda \geq \dim \gotq^D -3$, on peut conclure par \eqref{ovdb} que $\ind \gotq^D_\Lambda = \ind \gotq^D + 3=9$ et $\gotq^D_\Lambda=(\gotq^D)'$.\par
On peut calculer que $x_1$ et $x_2$ ne divisent pas $\pr(F_2)^\bullet$, $\pr(F_4)^\bullet$, $\pr(F_6)^\bullet$, $\pr(F_8)^\bullet$, $\pr(F_{10})^\bullet$, $f^\bullet$. En revanche, on constate que $\pr(F_4)^\bullet$ appartient à l'idéal engendré par $x_1$ et $x_2$. En fait, un calcul montre que\footnote{on détaille ce calcul en donnant les expressions exactes de $x_1$, $x_2$, $y_1$, $y_2$, $y_{1,2}$ en annexe}
$$\pr(F_4)^\bullet=x_1y_1+x_2y_2+x_1x_2y_{1,2}$$
avec $y_1$, $y_2$, $y_{1,2}$ des semi-invariants irréductibles bihomogènes en $(\gotn^-)^D$ de bidegrés respectifs $(1,2)$, $(0,3)$, $(0,2)$ (voir définition \ref{invbas}) et de poids respectifs $-\lambda_1$, $-\lambda_2$, $-\lambda_1-\lambda_2$. En particulier, $f_4^{(1)}:=x_1y_1$, $f_4^{(2)}:=x_2y_2$ et $f_4^{(3)}:=x_1x_2y_{1,2}$ sont des invariants bihomogènes en $(\gotn^-)^D$ de même bidegré que $\pr(F_4)^\bullet$.\par
Aussi, contrairement à tous les cas précédents, il existe $\lambda$ tel que $\Sym(\gotq^D)_\lambda$ n'est pas un $\Y(\gotq^D)$-module libre de rang $1$ (on prend $\lambda=-\lambda_1-\lambda_2$, on a les semi-invariants $y_1y_2$ et $y_{1,2}$).
\begin{rema}
On ne peut pas appliquer la proposition \ref{restrictif} pour conclure à la non-factorialité (donc à la non-polynomialité) de $\Y(\gotq^D)$ puisque les seuls semi-invariants de poids $\lambda_1+\lambda_2$ que l'on connaisse sont dans $\Y(\gotq^D)x_1x_2$.
\end{rema}
On soupçonne que la famille $\pr(F_2^\bullet)$, $f_4^{(1)}$, $f_4^{(2)}$, $f_4^{(3)}$, $\pr(F_6)^\bullet$, $f^\bullet$ est candidate à engendrer librement $\Y(\gotq^D)$, et que la famille $\pr(F_2^\bullet)$, $x_1$, $x_2$, $y_1$, $y_2$, $y_{1,2}$, $\pr(F_6)^\bullet$, $x_3$, $y_3$ est une famille candidate à engendrer librement $\Sy(\gotq^D)$. On a déjà calculé\footnote{calcul en annexe} que $\pr(F_8^\bullet)=-\left(f_4^{(1)}-f_4^{(2)}\right)^2$.\par
Il se trouve également que dans toutes les contractions paraboliques que l'on a étudiées jusqu'ici, l'ensemble des poids est un groupe, ce qui amène la conjecture suivante :
\begin{conj}
Soit $\gotq$ une contraction parabolique et $\Lambda$ le semi-groupe des poids de $\Sy(\gotq)$. Alors $\Lambda$ est un groupe.
\end{conj}
Si cette conjecture est vérifiée, alors tout semi-invariant est bien un facteur d'un invariant (et la réciproque provient de la proposition \ref{dixmier}).

\chapter{Annexe}
\section*{Calculs dans le cas inconclusif de Panyushev et Yakimova}
On reprend les notations utilisées dans la sous-section \ref{inconclusif} dans le cas inconclusif de Panyushev et Yakimova en type $D_6$. On détaille ici les calculs montrant les résultats suivants annoncés dans la sous-section \ref{inconclusif} :
\begin{propn0}\mbox{}
\begin{enumerate}
\item On a
$$\pr(F_4)^\bullet=x_1y_1+x_2y_2+x_1x_2y_{1,2}$$
avec $y_1$, $y_2$, $y_{1,2}$ des semi-invariants irréductibles bihomogènes en $(\gotn^-)^D$ de bidegrés respectifs $(1,2)$, $(0,3)$, $(0,2)$ (voir définition \ref{invbas}) et de poids respectifs $-\lambda_1$, $-\lambda_2$, $-\lambda_1-\lambda_2$.
\item On a $\pr(F_8^\bullet)=-\left(f_4^{(1)}-f_4^{(2)}\right)^2$.
\end{enumerate}
\end{propn0}
\subsection*{Calcul de $\pr(F_4)^\bullet$}
On montre ici le point (1) de la proposition. On rappelle que
$$F_4=\sum_{\Card(J)=4} \Delta_J$$
On commence par calculer $\pr(F_4^\bullet)$. Si $\pr(F_4^\bullet) \neq 0$, alors $\pr(F_4)^\bullet=\pr(F_4^\bullet)$. On va donc regarder les $\Delta_J$ avec $J$ de cardinal $4$ de degré maximal en $\gotn^-$ et calculer leurs projections. D'après le corollaire \ref{coro} et le lemme \ref{clé}, les $\Delta_J$ de degré maximal en $\gotn^-$ sont ceux qui vérifient $\max_k j_k=1$. On a alors $15$ types différents de mineurs à étudier. 
$$ \Delta_{1,2,a,b}, \Delta_{1,2,a,11}, \Delta_{1,2,a,12}, \Delta_{1,2,b,11}, \Delta_{1,2,b,12}, \Delta_{1,2,11,12}, \Delta_{1,a,b,11}, \Delta_{1,a,b,12}, \Delta_{1,a,11,12}, \Delta_{1,b,11,12},$$
$$\Delta_{2,a,b,11}, \Delta_{2,a,b,12}, \Delta_{2,a,11,12}, \Delta_{2,b,11,12}, \Delta_{a,b,11,12} $$
Ici, $a$ désigne un élément de $\llbracket 3,6 \rrbracket$ et $b$ un élément de $\llbracket 7,10 \rrbracket$. Par exemple, la notation $\Delta_{1,2,a,b}$ désigne les mineurs de la forme $\Delta_{\{1,2,a,b\}}$ avec $a \in \llbracket 3,6 \rrbracket$ et $b \in \llbracket 7,10 \rrbracket$. Comme $ \Delta_{1,2,a,b}$ est inclus dans $\Sym(\g_{\{1,2,a,b\}})$ avec $\g_{\{1,2,a,b\}}$ représenté comme suit :
\begin{center}
\begin{tikzpicture}[scale=0.7]
\foreach \k in {1,2,...,3}
	{\draw[color=gray!20]  (0,\k)--(4,\k);
	\draw[color=gray!20] (\k,0)--(\k,4);}
\draw (0,0)--(4,0);
\draw (0,0)--(0,4);
\draw (0,4)--(4,4);
\draw (4,0)--(4,4);
\draw[ultra thick] (0,3)--(1,3)--(1,2)--(2,2)--(2,1)--(3,1)--(3,0);
\draw (1,1) node{\LARGE $\gotn^-$};
\draw (3,3) node{\LARGE $\gotp$};
\draw[dashed] (1,4)--(1,3)--(2,3)--(2,2)--(3,2)--(3,1)--(4,1);
\draw (0.5,4.5) node{\Large $1$};
\draw (1.5,4.5) node{\Large $2$};
\draw (2.5,4.4) node{\Large $a$};
\draw (3.5,4.5) node{\Large $b$};
\draw (-0.5,3.5) node{\Large $1$};
\draw (-0.5,2.5) node{\Large $2$};
\draw (-0.5,1.5) node{\Large $a$};
\draw (-0.5,0.5) node{\Large $b$};
\end{tikzpicture}
\end{center}
on a $\Delta_{1,2,a,b}^\bullet=e_{2,1} \, e_{a,2} \, e_{b,a} \, e_{1,b}$. Le même raisonnement se généralise à tous les cas. En particulier, $e_{1,12}$ divise $\Delta_{1,2,a,12}^\bullet$, $\Delta_{1,2,b,12}^\bullet$, $\Delta_{1,2,11,12}^\bullet$, $\Delta_{1,a,b,12}^\bullet$, $\Delta_{1,a,11,12}^\bullet$ et $\Delta_{1,b,11,12}^\bullet$. De même, $e_{2,11}$ divise $\Delta_{2,a,b,11}^\bullet$. Puisque $\pr(e_{1,12})=\pr(e_{2,11})=0$, cela implique que les projections de tous ces termes sont nulles. Il reste donc à étudier les mineurs des types suivants :
$$ \Delta_{1,2,a,b}, \Delta_{1,2,a,11}, \Delta_{1,2,b,11}, \Delta_{1,a,b,11}, \Delta_{2,a,b,12}, \Delta_{2,a,11,12}, \Delta_{2,b,11,12}, \Delta_{a,b,11,12} $$
Remarquons que ces types peuvent être associés par paires via l'application d'anti-transposition $s \in \Sym(\so_n) \mapsto s^\gamma \in \Sym(\so_n)$. Par exemple, un mineur du type $\Delta_{1,2,a,b}$ sera associé par anti-transposition à un mineur du type $\Delta_{a,b,11,12}$. On obtient ainsi une bijection entre les mineurs du type $\Delta_{1,2,a,b}$ et les mineurs du type $\Delta_{a,b,11,12}$. Ainsi, après projection, ces mineurs seront égaux. En procédant ainsi pour tous les types, il reste donc 4 types de mineurs :
$$ \Delta_{1,2,a,b}, \Delta_{1,a,b,11}, \Delta_{1,2,a,11}, \Delta_{1,2,b,11}$$
et on a
\begin{align*}
\pr(F_4^\bullet)&=2 \times \sum_{(a,b) \in \llbracket 3,6 \rrbracket \times \llbracket 7,10 \rrbracket} \pr(e_{2,1})\pr(e_{a,2})\pr(e_{b,a})\pr(e_{1,b})\\ &+ 2 \times \sum_{(a,b) \in \llbracket 3,6 \rrbracket \times \llbracket 7,10 \rrbracket} \pr(e_{a,1})\pr(e_{b,a})\pr(e_{11,b})\pr(e_{1,11})\\ &+ 2 \times \sum_{x \in \llbracket 3,10 \rrbracket} \pr(e_{2,1})\pr(e_{x,2})\pr(e_{11,x})\pr(e_{1,11})
\end{align*}
Tous les monômes sont bien linéairement indépendants, ce qui montre que $\pr(F_4^\bullet) \neq 0$ et donc que $\pr(F_4)^\bullet = \pr(F_4^\bullet)$. Comme $x_1=\pr(e_{2,1})$ et $x_2=\pr(e_{1,11})$, on obtient
$$\pr(F_4)^\bullet=x_1y_1+x_2y_2+x_1x_2y_{1,2}$$
avec
$$y_1:= 2 \times \sum_{(a,b) \in \llbracket 3,6 \rrbracket \times \llbracket 7,10 \rrbracket} \pr(e_{a,2})\pr(e_{b,a})\pr(e_{1,b})$$
$$y_2:= 2 \times \sum_{(a,b) \in \llbracket 3,6 \rrbracket \times \llbracket 7,10 \rrbracket} \pr(e_{a,1})\pr(e_{b,a})\pr(e_{11,b})$$
$$y_{1,2}:=2 \times \sum_{x \in \llbracket 3,10 \rrbracket} \pr(e_{x,2})\pr(e_{11,x})$$
et on vérifie que $y_1$, $y_2$ et $y_{1,2}$ satisfont les propriétés annoncées.

\subsection*{Calcul de $\pr(F_8)^\bullet$}
Pour le point (2) de la proposition, on raisonne de même que précédemment. On calcule donc $\pr(F_8^\bullet)$, et on regarde les $\Delta_J$ avec $J$ de cardinal $8$ de degré maximal en $\gotn^-$, égal à $6$ par le corollaire \ref{coro}, dont on calcule leurs projections. Ces $\Delta_J$ sont ceux de la forme $\Delta_{\{1,2,a,b,c,d,11,12\}}$ avec $\{a,b\} \subset \llbracket 3,6 \rrbracket$ et $\{c,d\} \subset \llbracket 7,10 \rrbracket$. Représentons $\g_{\{1,2,a,b,c,d,11,12\}}$ :
\begin{center}
\begin{tikzpicture}[scale=0.7]
\foreach \k in {1,2,...,7}
	{\draw[color=gray!20]  (0,\k)--(8,\k);
	\draw[color=gray!20] (\k,0)--(\k,8);}
\draw (0,0)--(8,0);
\draw (0,0)--(0,8);
\draw (0,8)--(8,8);
\draw (8,0)--(8,8);
\draw[ultra thick] (0,7)--(1,7)--(1,6)--(2,6)--(2,4)--(4,4)--(4,2)--(6,2)--(6,1)--(7,1)--(7,0);
\draw (2,2) node{\LARGE $\gotn^-$};
\draw (6,6) node{\LARGE $\gotp$};
\draw[dashed] (1,8)--(1,7)--(2,7)--(2,6)--(4,6)--(4,4)--(6,4)--(6,2)--(7,2)--(7,1)--(8,1);
\draw (0.5,8.5) node{\LARGE $1$};
\draw (1.5,8.5) node{\LARGE $2$};
\draw (2.5,8.4) node{\LARGE $a$};
\draw (3.5,8.5) node{\LARGE $b$};
\draw (4.5,8.4) node{\LARGE $c$};
\draw (5.5,8.5) node{\LARGE $d$};
\draw (6.5,8.5) node{\LARGE $11$};
\draw (7.5,8.5) node{\LARGE $12$};
\draw (-0.5,7.5) node{\LARGE $1$};
\draw (-0.5,6.5) node{\LARGE $2$};
\draw (-0.5,5.5) node{\LARGE $a$};
\draw (-0.5,4.5) node{\LARGE $b$};
\draw (-0.5,3.5) node{\LARGE $c$};
\draw (-0.5,2.5) node{\LARGE $d$};
\draw (-0.5,1.5) node{\LARGE $11$};
\draw (-0.5,0.5) node{\LARGE $12$};
\end{tikzpicture}
\end{center}
Par \eqref{deltajpt}, on a donc
$$F_8^\bullet = \Delta_{\{1,2,a,b,11,12\},\{1,2,c,d,11,12\}}^\bullet \Delta_{\{c,d\}, \{a,b\}}$$
Calculons $\Delta_{\{1,2,a,b,11,12\},\{1,2,c,d,11,12\}}^\bullet$, qui est de degré $4$ en $\gotn^-$. On remarque que la sous-algèbre de Lie $\g_{\{1,2,a,b,11,12\},\{1,2,c,d,11,12\}}$ est de la forme
\begin{center}
\begin{tikzpicture}[scale=0.7]
\foreach \k in {1,2,...,5}
	{\draw[color=gray!20]  (0,\k)--(6,\k);
	\draw[color=gray!20] (\k,0)--(\k,6);}
\draw (0,0)--(6,0);
\draw (0,0)--(0,6);
\draw (0,6)--(6,6);
\draw (6,0)--(6,6);
\draw[ultra thick] (0,5)--(1,5)--(1,4)--(2,4)--(2,2)--(4,2)--(4,1)--(5,1)--(5,0);
\draw (1.5,1.5) node{\LARGE $\gotn^-$};
\draw (4.5,4.5) node{\LARGE $\gotp$};
\draw[dashed] (1,6)--(1,5)--(2,5)--(2,4)--(4,4)--(4,2)--(5,2)--(5,1)--(6,1);
\draw (0.5,6.5) node{\LARGE $1$};
\draw (1.5,6.5) node{\LARGE $2$};
\draw (2.5,6.4) node{\LARGE $c$};
\draw (3.5,6.5) node{\LARGE $d$};
\draw (4.5,6.5) node{\LARGE $11$};
\draw (5.5,6.5) node{\LARGE $12$};
\draw (-0.5,5.5) node{\LARGE $1$};
\draw (-0.5,4.5) node{\LARGE $2$};
\draw (-0.5,3.5) node{\LARGE $a$};
\draw (-0.5,2.5) node{\LARGE $b$};
\draw (-0.5,1.5) node{\LARGE $11$};
\draw (-0.5,0.5) node{\LARGE $12$};
\end{tikzpicture}
\end{center}
On développe $\Delta_{\{1,2,a,b,11,12\},\{1,2,c,d,11,12\}}$ selon la deuxième ligne. Certains des termes que l'on obtient ne sont pas de degré $4$ en $\gotn^-$. Par exemple, $e_{2,2} \Delta_{\{1,a,b,11,12\},\{1,c,d,11,12\}}$ n'est pas de degré $4$ en $\gotn^-$ : pour s'en convaincre, on peut visualiser $\g_{\{1,a,b,11,12\},\{1,c,d,11,12\}}$
\begin{center}
\begin{tikzpicture}[scale=0.7]
\foreach \k in {1,2,...,4}
	{\draw[color=gray!20]  (0,\k)--(5,\k);
	\draw[color=gray!20] (\k,0)--(\k,5);}
\draw (0,0)--(5,0);
\draw (0,0)--(0,5);
\draw (0,5)--(5,5);
\draw (5,0)--(5,5);
\draw[ultra thick] (0,4)--(1,4)--(1,2)--(3,2)--(3,1)--(4,1)--(4,0);
\draw (1,1) node{\LARGE $\gotn^-$};
\draw (4,4) node{\LARGE $\gotp$};
\draw[dashed] (1,5)--(1,4)--(3,4)--(3,2)--(4,2)--(4,1)--(5,1);
\draw (0.5,5.5) node{\LARGE $1$};
\draw (1.5,5.4) node{\LARGE $c$};
\draw (2.5,5.5) node{\LARGE $d$};
\draw (3.5,5.5) node{\LARGE $11$};
\draw (4.5,5.5) node{\LARGE $12$};
\draw (-0.5,4.5) node{\LARGE $1$};
\draw (-0.5,3.5) node{\LARGE $a$};
\draw (-0.5,2.5) node{\LARGE $b$};
\draw (-0.5,1.5) node{\LARGE $11$};
\draw (-0.5,0.5) node{\LARGE $12$};
\end{tikzpicture}
\end{center}
Aussi, puisque $\pr(e_{2,11})=0$, le terme $e_{2,11} \Delta_{\{1,a,b,11,12\},\{1,2,c,d,12\}}$ a sa projection nulle et on peut donc l'écarter. Il reste deux termes à considérer, qui sont $-e_{2,1} \Delta_{\{1,a,b,11,12\},\{2,c,d,11,12\}}$ et $e_{2,12} \Delta_{\{1,a,b,11,12\},\{1,2,c,d,11\}}$. Il reste donc à calculer $\Delta_{\{1,a,b,11,12\},\{2,c,d,11,12\}}^\bullet$ (qui est de degré $3$ en $\gotn^-$) et $\Delta_{\{1,a,b,11,12\},\{1,2,c,d,11\}}^\bullet$ (qui est de degré $4$ en $\gotn^-$). On trouve (en développant selon la colonne numérotée $11$)
$$\Delta_{\{1,a,b,11,12\},\{2,c,d,11,12\}}^\bullet=-e_{12,11} \Delta_{\{a,b\},\{2,12\}}\Delta_{\{1,11\},\{c,d\}}-e_{1,11} \Delta_{\{a,b\},\{2,12\}} \Delta_{\{11,12\},\{c,d\}}$$
$$\Delta_{\{1,a,b,11,12\},\{1,2,c,d,11\}}^\bullet=e_{12,11} \Delta_{\{a,b\},\{1,2\}}\Delta_{\{1,11\},\{c,d\}}+e_{1,11} \Delta_{\{a,b\},\{1,2\}} \Delta_{\{11,12\},\{c,d\}}$$
d'où, en notant $\sum_{(a,b),(c,d)}$ la somme sur l'ensemble des $(a,b) \in \llbracket 3,6 \rrbracket^2$ et $(c,d) \in \llbracket 7,10 \rrbracket^2$, la formule
\begin{align*}
\pr(F_8^\bullet)=&-x_1^2\pr\left(\sum_{(a,b),(c,d)} \Delta_{\{a,b\},\{2,12\}}\Delta_{\{1,11\},\{c,d\}} \Delta_{\{c,d\},\{a,b\}}\right) \\
&-x_2^2 \pr\left(\sum_{(a,b),(c,d)} \Delta_{\{a,b\},\{1,2\}} \Delta_{\{11,12\},\{c,d\}} \Delta_{\{c,d\},\{a,b\}}\right) \\
&+x_1x_2 \pr\left(\sum_{(a,b),(c,d)} \Delta_{\{a,b\},\{2,12\}} \Delta_{\{11,12\},\{c,d\}} \Delta_{\{c,d\},\{a,b\}}\right) \\
&+x_1x_2 \pr\left(\sum_{(a,b),(c,d)} \Delta_{\{a,b\},\{1,2\}}\Delta_{\{1,11\},\{c,d\}} \Delta_{\{c,d\},\{a,b\}}\right)
\end{align*}
Or par anti-transposition, on a
$$ \pr\left(\sum_{(a,b),(c,d)} \Delta_{\{a,b\},\{2,12\}} \Delta_{\{11,12\},\{c,d\}} \Delta_{\{c,d\},\{a,b\}}\right) = \pr\left(\sum_{(a,b),(c,d)} \Delta_{\{a,b\},\{1,2\}}\Delta_{\{1,11\},\{c,d\}} \Delta_{\{c,d\},\{a,b\}}\right)$$
d'où
\begin{align*}
\pr(F_8^\bullet)=&-x_1^2\pr\left(\sum_{(a,b),(c,d)} \Delta_{\{a,b\},\{2,12\}}\Delta_{\{1,11\},\{c,d\}} \Delta_{\{c,d\},\{a,b\}}\right) \\
&-x_2^2 \pr\left(\sum_{(a,b),(c,d)} \Delta_{\{a,b\},\{1,2\}} \Delta_{\{11,12\},\{c,d\}} \Delta_{\{c,d\},\{a,b\}}\right) \\
&+2x_1x_2 \pr\left(\sum_{(a,b),(c,d)} \Delta_{\{a,b\},\{2,12\}} \Delta_{\{11,12\},\{c,d\}} \Delta_{\{c,d\},\{a,b\}}\right)
\end{align*}
Pour conclure, il nous suffit de montrer que
\begin{equation}
\pr\left(\sum_{(a,b),(c,d)} \Delta_{\{a,b\},\{2,12\}}\Delta_{\{1,11\},\{c,d\}} \Delta_{\{c,d\},\{a,b\}}\right) = y_1^2 \label{y12}
\end{equation}
\begin{equation}
\pr\left(\sum_{(a,b),(c,d)} \Delta_{\{a,b\},\{1,2\}} \Delta_{\{11,12\},\{c,d\}} \Delta_{\{c,d\},\{a,b\}}\right) = y_2^2 \label{y22}
\end{equation}
\begin{equation}
\pr\left(\sum_{(a,b),(c,d)} \Delta_{\{a,b\},\{2,12\}} \Delta_{\{11,12\},\{c,d\}} \Delta_{\{c,d\},\{a,b\}}\right)= y_1y_2 \label{y1y2}
\end{equation}
Pour l'égalité \eqref{y12}, si l'on développe le terme
$$\Delta_{\{a,b\},\{2,12\}}\Delta_{\{1,11\},\{c,d\}} \Delta_{\{c,d\},\{a,b\}}$$
on a $8$ monômes qui apparaissent, et modulo un échange des variables $a$ et $b$ ou\footnote{Le "ou" n'est pas exclusif, on peut échanger les variables $a$ et $b$ en même temps qu'échanger les variables $c$ et $d$.} des variables $c$ et $d$, on a deux classes de monômes, dont des représentants sont
$$e_{1,c}e_{c,a}e_{a,2}e_{11,d}e_{d,b}e_{b,12} \qquad \text{et} \qquad -e_{1,c}e_{c,a}e_{a,12}e_{11,d}e_{d,b}e_{b,2}$$
Quatre de ces monômes sont de la première classe et quatre des monômes sont de la deuxième classe. Après sommation, on obtient alors
\begin{align*}
&\pr\left(\sum_{(a,b),(c,d)} \Delta_{\{a,b\},\{2,12\}}\Delta_{\{1,11\},\{c,d\}} \Delta_{\{c,d\},\{a,b\}}\right) = \\ &4 \sum_{(a,b),(c,d)} \pr(e_{1,c})\pr(e_{c,a})\pr(e_{a,2})\pr(e_{11,d})\pr(e_{d,b})\pr(e_{b,12}) \\ - &4 \sum_{(a,b),(c,d)} \pr(e_{1,c})\pr(e_{c,a})\pr(e_{a,12})\pr(e_{11,d})\pr(e_{d,b})\pr(e_{b,2})
\end{align*}
On a
\begin{align*}
&4 \sum_{(a,b),(c,d)} \pr(e_{1,c})\pr(e_{c,a})\pr(e_{a,2})\pr(e_{11,d})\pr(e_{d,b})\pr(e_{b,12}) \\
&=\left(2\sum_{(a,c) \in \llbracket 3,6 \rrbracket \times \llbracket 7,10 \rrbracket} \pr(e_{1,c})\pr(e_{c,a})\pr(e_{a,2})\right)\left(2\sum_{(b,d) \in \llbracket 3,6 \rrbracket \times \llbracket 7,10 \rrbracket} \pr(e_{11,d})\pr(e_{d,b})\pr(e_{b,12})\right)
\end{align*}
et
\begin{align*}
&\sum_{(a,b),(c,d)} \pr(e_{1,c})\pr(e_{c,a})\pr(e_{a,12})\pr(e_{11,d})\pr(e_{d,b})\pr(e_{b,2}) \\
&=\left(\sum_{(a,c) \in \llbracket 3,6 \rrbracket \times \llbracket 7,10 \rrbracket} \pr(e_{1,c})\pr(e_{c,a})\pr(e_{a,12})\right)\left(\sum_{(b,d) \in \llbracket 3,6 \rrbracket \times \llbracket 7,10 \rrbracket} \pr(e_{11,d})\pr(e_{d,b})\pr(e_{b,2})\right)
\end{align*}
Or
$$2\sum_{(a,c) \in \llbracket 3,6 \rrbracket \times \llbracket 7,10 \rrbracket} \pr(e_{1,c})\pr(e_{c,a})\pr(e_{a,2})=2\sum_{(b,d) \in \llbracket 3,6 \rrbracket \times \llbracket 7,10 \rrbracket} \pr(e_{11,d})\pr(e_{d,b})\pr(e_{b,12})=y_1$$
et
$$\sum_{(a,c) \in \llbracket 3,6 \rrbracket \times \llbracket 7,10 \rrbracket} \pr(e_{1,c})\pr(e_{c,a})\pr(e_{a,12})=\sum_{(b,d) \in \llbracket 3,6 \rrbracket \times \llbracket 7,10 \rrbracket} \pr(e_{11,d})\pr(e_{d,b})\pr(e_{b,2})=0$$
Pour la deuxième égalité, montrons par exemple que $\sum_{(a,c) \in \llbracket 3,6 \rrbracket \times \llbracket 7,10 \rrbracket} \pr(e_{1,c})\pr(e_{c,a})\pr(e_{a,12})=0$. Il suffit de remarquer que :
\begin{itemize}
\item si $c=a^\gamma$, on a $\pr(e_{c,a})=0$ de sorte que $\pr(e_{1,c})\pr(e_{c,a})\pr(e_{a,12})=0$,
\item les autres termes $\pr(e_{1,c})\pr(e_{c,a})\pr(e_{a,12})$ pour $c \neq a^\gamma$, s'associent par paires de la forme $\pr(e_{1,c})\pr(e_{c,a})\pr(e_{a,12})$ et $\pr(e_{1,a^\gamma})\pr(e_{a^\gamma,c^\gamma})\pr(e_{c^\gamma,12})=-\pr(e_{1,c})\pr(e_{c,a})\pr(e_{a,12})$, avec $c < a^\gamma$.
\end{itemize}
Les calculs pour les deux autres équations \eqref{y22} et \eqref{y1y2} sont similaires.

\newpage
\addcontentsline{toc}{chapter}{Bibliographie}
\emergencystretch=10em
\bibliographystyle{plainurl}
\bibliography{bibliotheque}
\renewcommand{\indexname}{Index des notations}
\printindex
\end{document}